\documentclass[10pt,notitlepage]{amsart} 
\usepackage[french,english]{babel}
\usepackage[utf8]{inputenc}
\usepackage[T1]{fontenc}
\usepackage{amsmath}
\usepackage{amsthm}
\usepackage{amsfonts}
\usepackage{amssymb,amscd,epsf,verbatim,mathtools,framed}
\usepackage{mathrsfs}
\usepackage{graphicx}
\usepackage{latexsym}
\usepackage{lscape}
\usepackage{extarrows}
\usepackage[colorlinks=true]{hyperref}
\hypersetup{colorlinks, citecolor=blue, filecolor=black, linkcolor=purple, urlcolor=violet}
\usepackage{epstopdf}
\usepackage{tikz}
\usetikzlibrary{calc}
\usetikzlibrary{matrix,arrows,decorations.pathmorphing}
\usepackage{tikz-cd}
\usepackage[all,cmtip]{xy}
\usepackage{color}
\usepackage{multirow}  
\usepackage{scalefnt}
\usepackage{fancyhdr}
\usepackage{enumitem}
\usepackage[margin=1.25in]{geometry}

\setlist[enumerate]{leftmargin=1cm}
\setlist[itemize]{leftmargin=1cm}

\usepackage{relsize} 
\usepackage[bbgreekl]{mathbbol} 
\DeclareSymbolFontAlphabet{\mathbb}{AMSb} 
\DeclareSymbolFontAlphabet{\mathbbl}{bbold} 
\newcommand{\Prism}{{\mathlarger{\mathbbl{\Delta}}}}

\newcommand{\Z}{\mathbb{Z}}

\newcommand{\N}{\mathbb{N}}

\newcommand{\Q}{\mathbb{Q}}
\newcommand{\A}{\mathbb{A}}
\renewcommand{\L}{\mathbb{L}}

\newcommand{\B}{\mathbb{B}}

\newcommand{\mD}{\mathcal{D}}
\newcommand{\mE}{\mathcal{E}}
\newcommand{\mF}{\mathcal{F}}

\newcommand{\mH}{\mathcal{H}}
\newcommand{\mI}{\mathcal{I}}

\newcommand{\mK}{\mathcal{K}}
\newcommand{\mM}{\mathcal{M}}

\newcommand{\mO}{\mathcal{O}}

\newcommand{\mS}{\mathcal{S}}

\newcommand{\fm}{\mathfrak{m}} 
\newcommand{\fU}{\mathfrak{U}}
\newcommand{\fX}{\mathfrak{X}}
\newcommand{\fY}{\mathfrak{Y}}
\newcommand{\fD}{\mathfrak{D}}
\newcommand{\fB}{\mathfrak{B}}
\newcommand{\fS}{\mathfrak{S}}
\newcommand{\fM}{\mathfrak{M}}

\newcommand{\bM}{\mathbb{M}}
\newcommand{\bV}{\mathbb{V}}

\newcommand{\ul}{\underline}

\newcommand{\ra}{\rightarrow}

\newcommand{\minus}{\backslash}

\DeclareMathOperator{\Hom}{Hom}

\newcommand{\BKF}{\mathrm{BKF}}
\newcommand{\DBKF}{\mathrm{DBKF}}

\newcommand{\dR}{\mathrm{dR}}
\newcommand{\CR}{\mathrm{Vect}}

\newcommand{\crys}{\mathrm{crys}}
\newcommand{\Acrys}{A_{\crys}}
\newcommand{\perf}{\mathrm{perf}}
\newcommand{\vp}{\varphi}
\newcommand{\Vect}{\mathrm{Vect}}
\newcommand{\an}{\mathrm{an}}
\newcommand{\Loc}{\mathrm{Loc}}
\newcommand{\st}{{\mathrm{st}}}
\newcommand{\DD}{\mathrm{DD}}
\newcommand{\txi}{{\widetilde{\xi}}}
\newcommand{\conv}{\mathrm{conv}}
\newcommand{\logAinf}{{(\ul\Ainf, \txi)}}
\newcommand{\logqAinf}{{(\ul\Ainf, \xi)}}
\newcommand{\ev}{\mathrm{ev}}
\newcommand{\Strat}{\mathrm{Strat}}
\newcommand{\htimes}{\widehat{\otimes}}
\newcommand{\Kos}{\mathrm{Kos}}
\newcommand{\MIC}{\mathrm{MIC}}
\renewcommand{\le}{\leqslant}
\renewcommand{\ge}{\geqslant}
\newcommand{\simra}{\xrightarrow{\sim}}
\newcommand{\nilp}{\mathrm{nilp}}
\newcommand{\Fil}{\mathrm{Fil}}
\newcommand{\Perfd}{\mathrm{Perfd}}

\newcommand{\MF}{\mathrm{MF}}
\newcommand{\hMF}{\widehat{\MF}}

\newcommand{\hB}{\widehat{B}}

\newcommand{\hD}{{\widehat{D}}}
\newcommand{\rank}{\mathrm{rank}}
\newcommand{\DR}{\mathrm{DR}}

\newcommand{\pre}{\mathrm{pre}}
\newcommand{\Mod}{\mathrm{Mod}}
 
\newcommand{\Ainf}{A_{\mathrm{inf}}}

\newcommand{\Ainfx}{\mathbb{A}_{\mathrm{inf}, X}}
\newcommand{\AinfU}{\mathbb{A}_{\mathrm{inf}, U}}

\newcommand{\logcrys}{\mathrm{logcrys}}

\DeclareMathOperator{\psh}{psh}

\DeclareMathOperator{\lra}{\: \longrightarrow \:}

\DeclareMathOperator{\gp}{{\textup{gp}}}
\newcommand{\cont}{{\mathrm{cont}}}
\newcommand{\ct}{\mathrm{ct}}
\newcommand{\Isoc}{\mathrm{Isoc}}

\newcommand{\gr}[1]{\langle {#1} \rangle} 
\DeclareMathOperator{\Rep}{Rep}
\DeclareMathOperator{\spec}{Spec}
\DeclareMathOperator{\spf}{Spf}
\DeclareMathOperator{\spa}{Spa}

\newcommand{\ett}{\mathrm{\acute{e}t}}
\newcommand{\ket}{{\mathrm{k\acute{e}t}}}

\newcommand{\proet}{{\mathrm{pro\acute{e}t}}}
\newcommand{\proket}{{\mathrm{prok\acute{e}t}}}

\DeclareMathOperator{\Rlim}{\underset{\longleftarrow}{Rlim}}

\newcommand{\bi}{\begin{itemize}}
\newcommand{\ei}{\end{itemize}}
\newcommand{\bt}{\begin{theorem}}
\newcommand{\et}{\end{theorem}}
\newcommand{\bbt}{\begin{theorem*}}
\newcommand{\eet}{\end{theorem*}}
\newcommand{\bp}{\begin{proposition}}
\newcommand{\ep}{\end{proposition}}
\newcommand{\bl}{\begin{lemma}}
\newcommand{\el}{\end{lemma}}
\newcommand{\bbl}{\begin{lemma*}}
\newcommand{\eel}{\end{lemma*}}
\newcommand{\bc}{\begin{corollary}}
\newcommand{\ec}{\end{corollary}}
\newcommand{\beg}{\begin{example}}
\newcommand{\eeg}{\end{example}}
\newcommand{\br}{\begin{remark}}
\newcommand{\er}{\end{remark}}
\newcommand{\bbr}{\begin{remark*}}
\newcommand{\eer}{\end{remark*}}
\newcommand{\bd}{\begin{definition}}
\newcommand{\ed}{\end{definition}}
\newcommand{\be}{\begin{enumerate}}
\newcommand{\ee}{\end{enumerate}}
\newcommand{\bex}{\begin{exercise}}
\newcommand{\eex}{\end{exercise}}
\newcommand{\bproof}{\begin{proof}}
\newcommand{\eproof}{\end{proof}}

\theoremstyle{definition}
\newtheorem{theorem}{Theorem}[section] 
\newtheorem{example}[theorem]{Example}
\newtheorem{definition}[theorem]{Definition}
\newtheorem{proposition}[theorem]{Proposition} 

\newtheorem{lemma}[theorem]{Lemma}
\newtheorem{corollary}[theorem]{Corollary}
\newtheorem{construction}[theorem]{Construction} 

\newtheorem{notation}[theorem]{Notation}
\newtheorem{remark}[theorem]{Remark}

\newtheorem{assumption}[theorem]{Assumption}

\usepackage{etoolbox}
\patchcmd{\section}{\scshape}{\bfseries}{}{}
\makeatletter
\renewcommand{\@secnumfont}{\bfseries}
\makeatother

\numberwithin{equation}{section}

\title[Logarithmic $A_{\mathrm{inf}}$--cohomology: Part II]{  Logarithmic $\scalebox{1.3}{$A_{\mathrm{inf}}$}$--cohomology: Part II \\ \vspace{0.08in} \large $\textup{---  \normalfont  cohomology with coefficients and the}$ $C_{\mathrm{st}}$ $\textup{\normalfont  conjecture}$ }

\author{Hansheng Diao, Zhefan Duan, and Zijian Yao}

\date{}

\begin{document}

\begin{abstract}
We develop a theory of logarithmic $\Ainf$-cohomology with coefficients for a class of $p$-adic log formal schemes that are ``sufficiently log smooth'', where the coefficients are given by relative log BKF modules. Then we establish comparison isomorphisms with \'etale, de Rham, and crystalline cohomology, and also extend these results to the derived setting. As an application, we give a new proof of the $C_{\mathrm{st}}$ conjecture for semistable local systems. The proof also uses the prismatic interpretation of semistable local systems established by Du--Liu--Moon--Shimizu.
\end{abstract}
 
\maketitle
\thispagestyle{empty}

\vspace{-0.2in}

\tableofcontents

\setlength{\parskip}{0.15em}

\section{Introduction} \label{section:intro}
\noindent This article is a continuation of \cite{log1}. It is organized in two parts.
\begin{enumerate}
    \item[(a)] We develop the theory of log $\Ainf$-cohomology with coefficients in so-called relative log Breuil--Kisin--Fargues modules, generalizing the works of \cite{BMS1}, \cite{MT}, and \cite{log1}.
    \item[(b)] As an application, we give a new proof of the semistable comparison theorem for semistable local systems.
\end{enumerate}

In particular, in part (b), we prove the following main theorem. For the setup, we fix a prime number $p$. Let $K$ be a $p$-adic field; i.e., a discretely-valued nonarchimedean field of mixed characteristic $(0,p)$ with perfect residue field. Let $\mO_K$ be its ring of integers and $k$ be the residue field. Let $K_0=W(k)$. Moreover, let $G_K=\mathrm{Gal}(\overline{K}/K)$ be the absolute Galois group.

\begin{theorem}[{Semistable comparison theorem, Theorem~\ref{thm: semistable comparison}}]\label{mainthm}
    Let $\fY$ be a proper semistable formal scheme over $\mO_K$ with adic generic fiber $Y$ and special fiber $\fY_k$.\footnote{Here, $\fY$ is equipped with the divisorial log structure given by the special fiber and $\fY_k$ is equipped with the pullback log structure from $\fY$. The ring $B_{\mathrm{st}}$ stands for the semistable period ring of Fontaine. $H^i_{\ett}$ and $H^i_{\logcrys}$ denote the $p$-adic \'etale cohomology and log-crystalline cohomology, respectively.}
    Let $\L$ be a semistable $\Z_p$-local system on $X$ and let $\mE_{\crys,\Q}$ be the associated $F$-isocrystal on the log crystalline site $(\fY_k/W(\ul k))_{\crys}$. Then there are natural isomorphisms
        \[
           C_{\st} \colon H^i_{\ett}(Y_{\overline{K}}, \L) \otimes_{\Z_p}B_\st \simra H^i_{\logcrys}(\fY_k/W(\ul k), \mE_{\crys,\Q})\otimes_{K_0}B_{\st},
        \]
compatible with $G_K$-actions, Frobenii, filtrations,\footnote{There are natural filtrations on both sides after tensor with the de Rham period ring $B_{\mathrm{dR}}$.} and monodromy operators.
\end{theorem}

Theorem~\ref{mainthm} is also known as the \emph{$C_{\mathrm{st}}$-conjecture with coefficients}. For a precise definition of semistable local systems (as well as their associated $F$-isocrystals), see \cite{Guo-Yang} and \cite{DLMS2}; also see \S \ref{section: semistable local systems} for a quick review. Intuitively, a semistable local system can be viewed as a geometric family of semistable $p$-adic Galois representations. 

Before describing the strategy of the proof, we make some remarks on related works in the literature.
\begin{enumerate}
\item The original $C_{\mathrm{st}}$-conjecture with trivial coefficients was formulated by Fontaine--Jansen, and first proved by Tsuji \cite{tsuji1999}, who used log-syntomic cohomology as a bridge to relate the \'etale and log-crystalline sides.
\item Later, the $C_{\mathrm{st}}$-conjecture with trivial coefficients was revisited and reproved in \cite{Faltings_almost}, \cite{Niziol08}, \cite{andreatta2012semistable}, \cite{Bhatt12}, \cite{Beilinson13}, \cite{CK_semistable}, \cite{Colmez-Niziol17}, and \cite{Colmez-Niziol25}.
\item The $C_{\mathrm{st}}$-conjecture with general coefficients (in the algebraic setting) was already discussed in \cite{Faltings_almost}. In the setting of rigid analytic spaces, the first complete proof of the $C_{\mathrm{st}}$-conjecture with general coefficients (i.e., Theorem~\ref{mainthm}) was given by Du--Moon--Shimizu \cite{DMS}. They used \emph{log prismatic cohomology} (see \cite{Koshikawa, logprism}) as a bridge to relate \'etale and log-crystalline cohomology. More precisely, they proved a comparison theorem between log crystalline cohomology and log prismatic cohomology, while the comparison between \'etale and log prismatic cohomology was earlier handled by Tian \cite{Tian_prismatic_etale_comparison}.
\item Let us mention that Inoue and Koshikawa have also announced Theorem~\ref{mainthm} in a more general setting.
\end{enumerate}

Our proof in this article uses \emph{logarithmic $\Ainf$-cohomology} (with general coefficients) as an intermediate object to relate the \'etale and log-crystalline sides.

\vspace{0.1in}
\subsection{Logarithmic $\Ainf$-cohomology with coeffcients in relative log BKF modules}
\noindent
\vspace{0.1in}

\noindent In \cite{BMS1}, Bhatt, Morrow, and Schloze establish the theory of $\Ainf$-cohomology (with trivial coefficients) for smooth $p$-adic formal schemes. Using $\Ainf$-cohomology as a unifying theory, they are able to prove a sequence of comparison theorems between various \emph{integral} $p$-adic cohomology theories. In a previous work \cite{log1}, the first and third authors of this article generalize the work of Bhatt--Morrow--Scholze to the logarithmic setting; in particular, the $p$-adic formal schemes is allowed to possess singularities as long as they are still \emph{admissibly smooth} (cf. \cite[Definition 6.6]{log1}). In a different direction, Morrow and Tsuji \cite{MT} generalize the work of Bhatt--Morrow--Scholze to $\Ainf$-cohomology with coefficients in so-called \emph{relative Breuil--Kisin--Fargues modules}. 

In this article, we combine the frameworks of \cite{MT} and \cite{log1}, and develop a theory of log $\Ainf$-cohomology with coefficients in \emph{relative log Breuil--Kisin--Fargues modules}.

\vspace{0.05in}

To proceed, let us introduce some notation:

\begin{itemize}
\item Let $C$ be a perfectoid field of characteristic 0 containing all $p$-power roots of unity and let $\mO=\mO_C$ be the ring of integers. Let $\Ainf=W(\mO^{\flat})$, equipped with the natural Frobenius. We fix a compatible system $\{\zeta_{p^n}\}_{n\ge 0}$ of $p$-power roots of unity and let $\varepsilon:=(1,\zeta_p, \zeta_{p^2}, \cdots)\in \mathcal{O}^{\flat}$. Let $\mu:=[\varepsilon]-1$, $\xi:=\mu/\varphi^{-1}(\mu)$, and $\widetilde{\xi}:=\varphi(\xi)$ in $\Ainf$.

\item Suppose that $\mO$ is equipped with a pre-log structure $N_\infty \ra \mO$ where $N_\infty$ is uniquely $n$-divisible for every $n \in \Z_{>0}$. Let $(\mO, N_\infty)^a$ denote the associated log ring.
\item Let $\fX$ be a $p$-adic log formal scheme that is admissibly smooth over $\spf(\mO, N_\infty)^a$. Roughly speaking, this requires $\fX$ to be the saturated base change of some $\fX_0$ that is log smooth over $\spf (\mO, N)^a$ 
\[
\begin{tikzcd}
    \fX \arrow[r] \arrow[d] & \fX_0 \arrow[d] \\
    \spf (\mO, N_\infty)^a \arrow[r, "\kappa"] & \spf (\mO, N)^a
\end{tikzcd}
\]
where $N$ is an fs monoid and $\kappa$ is induced by a monoid homomorphism $N \ra N_\infty$. We refer the reader to \cite[Definition 1.5]{log1} for the precise definition.\footnote{As an example, if $\fY$ is a semistable formal scheme over $\mathcal{O}_K$ (viewed as a $p$-adic log formal scheme with the standard divisorial log structure), then the base change $\fY_{\mO_C}$ is admissibly smooth where $C=\widehat{\overline{K}}$. In fact, we will specialize to this situation in \S \ref{section: Cst comparison} when we prove Theorem~\ref{mainthm}. On the other hand, the notion of admissible smoothness is much more general, allowing both vertical and horizontal log structures.}

\item Let $X$ be the (log) adic generic fiber of $\fX$. Let $X_{\proket}$ be the pro-Kummer \'etale site on $X$ defined in \cite[\S 7.5]{log1}, and let $\widehat{\mO}_X^+$ (resp. $\widehat{\mO}_X^{\flat+}$) be the completed integral structure sheaf (resp. tilted structure sheaf) on $X_{\proket}$. Let $\Ainfx:=W(\widehat{\mO}_X^{\flat+})$ equipped with the natural Frobenius. 
\end{itemize}

The following definition is a straightforward generalization of \cite[Definition 5.1]{MT}.

\begin{definition}[{Definition~\ref{defn: log BKF modules}}]
Let $\nu: X_{\proket}\rightarrow \fX_{\ett}$ be the natural projection of sites.
\begin{enumerate}
\item Let $\bM$ be a sheaf of $\Ainfx$-modules on $X_{\proket}$. For each $r\in \Z_{>0}$, let $\xi_r=\frac{[\varepsilon]-1}{[\varepsilon^{1/p^r}]-1}\in \Ainf$. We say that $\bM$ is \emph{trivial modulo} $\xi_r$ if the sheaf of $\nu_*\big(\Ainfx/\xi_r\big)$-modules $\nu_*(\bM/\xi_r)$ is locally finite free and the counit
\[\nu^{-1}\nu_*(\bM/\xi_r)\otimes_{\nu^{-1}\nu_*(\Ainfx/\xi_r)}\Ainfx/\xi_r\rightarrow \bM/\xi_r\]
is an isomorphism of sheaves on $X_{\proket}$. We say that $\bM$ is \emph{trivial modulo} $<\mu$ if it is trivial modulo $\xi_r$ for all $r\ge 1$.
\item A \emph{relative logarithmic Breuil--Kisin--Fargues module} over $\fX$ is a pair $(\bM, \varphi_{\bM})$ where $\bM$ is a locally finite free $\Ainfx$-module that is trivial modulo $<\mu$, and $\varphi_{\bM}:\varphi^*\bM[\txi^{-1}]\rightarrow \bM[\txi^{-1}]$ is an isomorphism of sheaves of $\Ainfx[\txi^{-1}]$-modules. The category of such objects is denoted by $\BKF^{\log}(\fX, \varphi)$.
\end{enumerate}
\end{definition}

There are natural \'etale, de Rham, and crystalline specialization functors (see Definitions \ref{defn: etale specialization}, \ref{defn: de Rham specialization}, \ref{defn: crystalline specialization})\footnote{For the crystalline specialization, we need to assume that $\fX$ locally admits free charts (see Definition~\ref{defn: locally admits free charts}).}
\begin{align*}
    \sigma^*_{\ett}\colon &\BKF^{\log}(\fX, \varphi)\rightarrow \big\{\textrm{locally free }\widehat{\Z}_p\textrm{-sheaves on }X_{\proket}\big\}\\
    \sigma^*_{\textrm{dR}}\colon & \BKF^{\log}(\fX, \varphi)\rightarrow \big\{\textrm{vector bundles with integrable connections on }\fX\big\}\\
    \sigma^*_{\textrm{crys}}\colon & \BKF^{\log}(\fX, \varphi)\rightarrow \big\{\textrm{locally finite free $F$-crystals on }(\fX_k/W(\ul k))_{\crys}\big\},
\end{align*}
where $\ul k=(k,N_\infty)$ is the residue field of $\ul \mO$, and $\fX_k$ denotes the special fiber of $\fX$ over $\ul k$. The constructions of \'etale and de Rham specialization functors essentially follow the same approach as in \cite[\S 6]{MT}. For crystalline specialization, we first attach to a relative log BKF module an $F$-$q$-crystals on the log $q$-crystalline site $(\fX/(\underline{\Ainf}, \xi))_{q\crys}$ (cf. Definition \ref{defn: qcrys specialization}), then further specialize to an $F$-crystal on $(\fX_k/W(\ul k))_{\crys}$ via base change.

After clarifying the coefficients, we are ready to define log $\Ainf$-cohomology.

\begin{definition} Let $\fX$ be a saturated $p$-adic log formal scheme that is admissibly smooth over $\ul \mO$ and let $X$ be the log adic generic fiber. Let $\bM\in \mathrm{BKF}^{\log}(\fX, \varphi)$. We define
\[
    A\Omega^{\log}_{\fX}(\bM):= L\eta_{\mu}(\widehat{R\nu_*\bM})\in D(\fX_{\ett}, \Ainf)
\]
where the completion is the derived $p$-adic completion, $\nu \colon X_{\proket} \ra \fX_{\ett}$ is the natural projection of sites, and $L\eta$ stands for the d\'ecalage functor (see \cite[\S 6]{BMS1}). The \emph{log $\Ainf$-cohomology of $\bM$} is defined to be
\[R\Gamma_{\Ainf}(\fX, \bM):=R\Gamma\big(\fX_{\ett}, A\Omega^{\log}_{\fX}(\bM)\big)\in D(\Ainf).\]
\end{definition}

As expected, log $\Ainf$-cohomology (with coefficients) can be naturally compared with \'etale, de Rham, and (log-)crystalline cohomology (with coefficients), where the coefficients are matched up through the specialization functors.

\begin{theorem}[Theorem {\ref{thm: etale comparison}}, \ref{thm: de Rham comparison}, Corollary \ref{cor: crystalline comparison}, also see Theorem \ref{thm: perfect complex}] \label{thm: comparisons intro}
There are natural isomorphisms \footnote{For the crystalline comparison, we need to assume that $\fX$ locally admits free charts (see Definition~\ref{defn: locally admits free charts}).}
    \begin{align*}
        \widehat{\Big(A\Omega^{\log}_{\fX}(\bM)[\frac{1}{\mu}]\Big)}^{\varphi_{\bM}=1}&\,\,\cong R\nu_*(\sigma^*_{\ett}\bM),\\
        A\Omega^{\log}_{\fX}(\bM)\otimes^{\L}_{\Ainf, \theta} \mO\,\,&\,\,\cong \sigma^*_{\dR}\bM\otimes_{\mO_{\fX}} \Omega^{\bullet}_{\fX/\ul \mO},\\
        A\Omega^{\log}_{\fX}(\bM)\htimes^{\L}_{\Ainf}W(k)\,\,&\,\,\cong R\upsilon_*(\sigma^*_{\crys}\bM),
    \end{align*}
    where $\upsilon\colon (\fX_k/W(\ul k))_{\crys}\rightarrow \fX_{k, \ett}$ is the natural projection of sites. As a corollary, we obtain isomorphisms on the corresponding cohomologies.
\end{theorem}

For the \'etale and de Rham comparisons, we follow the same approach as in \cite[\S 6]{MT}. For the crystalline comparison, we first construct the comparison on small affine \'etale opens, then consider ``all possible coordinates'' as in \cite{BMS1} to guarantee that the comparison is independent of choices of charts.

\vspace{0.1in}
\subsection{Proof of semistable comparison theorem}
\noindent
\vspace{0.1in}

\noindent Now we briefly describe our strategy for proving Theorem~\ref{mainthm}. Let $K$ be the $p$-adic field at the very beginning of the article and let $\mO_K$ be its ring of integers. Let $\pi$ be a uniformizer in $\mO_K$. We equip $\mO_K$ with the pre-log structure $\N \to \mO_K$ sending $1 \mapsto \pi$. For the rest of this introduction, we take $C=\widehat{\overline{K}}$ and let $\mO=\mO_C$ be its ring of integers. Then the pre-log ring $\underline{\mO_K}=(\mO_K, \mathbb{N})$ extends to a pre-log ring $\underline{\mO}=(\mO, \mathbb{Q}_{\ge 0})$ by choosing a compatible system of $n$-th roots of $\pi$ for all $n\ge 1$. 

Let $\fY$ be a semistable formal scheme over $\mO_K$ equipped with the divisorial log structure given by the special fiber; in particular, $\fY$ is log smooth over $\underline{\mO_K}$. Then the base change $\fY_{\mO}$ is admissibly smooth over $\underline{\mO}$. Let $Y$ be the adic generic fiber of $\fY$ and let $\L$ be a semistable local system on $Y$. Ideally, if $\L$ (viewed as an \'etale $\Z_p$-local system on $Y_C$) is associated with a relative log BKF module via the \'etale specialization functor $\sigma^*_{\ett}$, then the desired semistable comparison would be an immediate consequence of the comparison theorems (Theorem~\ref{thm: comparisons intro}). Unfortunately, this is not true in general.

To fix this issue, we need to work within a larger category --- the category of \emph{derived relative log BKF modules}. To define these objects, let $D_{\perf}\big((\fY_{\mO}/\logAinf)_\Prism\big)$ denote the category of perfect complexes of crystals on the log prismatic site $(\fY_{\mO}/\logAinf)_\Prism$, and let $D_{\perf}(\mathbb{A}_{\mathrm{inf}, Y_C})$ be the category of perfect complexes of $\mathbb{A}_{\mathrm{inf}, Y_C}$-modules on $Y_{C,\proet}$.
Then the fully faithful functor (see \S \ref{subsection: log prismatic crystals})
\[
    \Vect\big((\fY_{\mO}/\logAinf)_\Prism\big) \xrightarrow[]{\sim} \BKF^{\log}(\fY_{\mO}) \hookrightarrow \Vect(\mathbb{A}_{\mathrm{inf}, Y_C})
\]
induces a fully faithful functor
\[
    \beta_{\Prism}\colon D_{\perf}\big((\fY_{\mO}/\logAinf)_\Prism\big) \hookrightarrow D_{\perf}(\mathbb{A}_{\mathrm{inf}, Y_C}).
\]

\begin{definition}[{Definition~\ref{defn: derived BKF modules}}]\label{defn: derived BKF modules_intro}
A \emph{derived relative logarithmic Breuil--Kisin--Fargues module} over $\fY_{\mO}$ is a quadruple $(\bM, \bV, \iota, \varphi_{\bV})$ where 
    \begin{itemize}
    \item $\bM$ is a sheaf of $\mathbb{A}_{\mathrm{inf}, Y_C}$-modules in perfect complexes that lies in the essential image of $\beta_{\Prism}$;
        \item $\bV$ is a sheaf of finite projective $\mathbb{A}_{\mathrm{inf}, Y_C}[\frac 1 \mu]$-module;
        \item $\iota \colon \bM\otimes^{\L}_{\mathbb{A}_{\mathrm{inf}, Y_C}}\mathbb{A}_{\mathrm{inf}, Y_C}[\frac{1}{\mu}] \xrightarrow[]{\sim} \bV$ is an isomorphism in $D_{\perf}(\mathbb{A}_{\mathrm{inf}, Y_C}[\frac 1 \mu])$
        \item $\varphi_{\bV} \colon \varphi^*\bV [\frac 1\txi] \ra \bV[\frac 1\txi]$
        is an isomorphism of $\mathbb{A}_{\mathrm{inf}, Y_C}[\frac 1 \txi]$-modules.
        \end{itemize}
The category of such objects is denoted by $\DBKF^{\log}(\fY_{\mO},\vp)$.
\end{definition}

Since $\Vect\big((\fY_{\mO}/\logAinf)_\Prism\big) \cong \BKF^{\log}(\fY_{\mO})$, we see that the notion of derived relative log BKF modules indeed extends the notion of relative log BKF modules. Moreover, it turns out that the specialization functors and comparison theorems all extend to the context of derived relative log BKF modules (see Propositions~\ref{prop: etale comparison for derived BKF modules} and \ref{cor: crystalline comparison for derived BKF modules}). A key observation is that semistable local systems are associated with derived relative log BKF modules under the \'etale specialization functor. To see this, we consider the absolute log prismatic site $\fY_{\Prism}$ and the category $\Vect^{\an,\vp}(\fY_\Prism)$ of \emph{analytic prismatic $F$-crystals} (cf. Definition \ref{defn: analytic prismatic F-crystals}). By the main theorem of \cite{DLMS2}, there is an equivalence of categories
\begin{equation}\label{eq: st local system as analytic prismatic crystal--intro}
\Vect^{\an,\vp}(\fY_\Prism)\cong \mathrm{Loc}_{\Z_p}^{\mathrm{st}}(Y)
\end{equation}
where $\mathrm{Loc}_{\Z_p}^{\mathrm{st}}(Y)$ is the category of semistable local systems on $Y$. Now, the following logarithmic analogue of \cite[Theorem 5.10]{Guo-Reinecke} is precisely what we need.

\begin{proposition}[Proposition \ref{prop: GR Thm 5.10}]\label{prop: pushforward--intro}
Taking pushforward along the inclusion $\spec(A)\setminus V(p,I) \hookrightarrow \spec(A)$, for all log prisms $(A,I,\mM_A) \in \fY_\Prism$, induces a fully faithful functor 
    \[
        j_*\colon \Vect^{\an, \vp} (\fY_\Prism) \ra D_{\perf}^\vp(\fY_\Prism).
    \]
\end{proposition}

Combining Proposition~\ref{prop: pushforward--intro} and \eqref{eq: st local system as analytic prismatic crystal--intro}, we deduce that every semistable local system is indeed associated with a derived relative log BKF module. Consequently, the desired semistable comparison follows from (the derived version of) Theorem~\ref{thm: comparisons intro}. We refer the reader to \S \ref{subsection: proof of Cst conjecture} for more details.

\begin{remark}
    We expect the $C_{\st}$-conjecture with coefficients to hold for a more general class of admissibly smooth $p$-adic log formal schemes, as long as one can extend the notion of semistable local systems to this context. We leave this to interested readers.
\end{remark}

\vspace{0.1in}
\subsection*{Outline of the article}

\begin{itemize}
    \item In \S \ref{section: log BKF modules}, we define relative log BKF modules. We also prove that the category of relative log BKF modules is equivalent to the category of generalized representations, the category of modules with $q$-connections, and the category of prismatic $F$-crystals.
    \item In \S \ref{section: coefficients}, we construct the \'etale, de Rham, and crystalline specialization functors and prove the comparison theorems with log $\Ainf$-cohomology.
    \item In \S \ref{section: semistable local systems}, we recall the notion of semistable local systems. We also recall the interpretation of semistable local systems in terms of analytic prismatic crystals, established in \cite{DLMS2}. 
    \item In \S \ref{section: Cst comparison}, we introduce the notion of derived relative log BKF modules, and finish the proof of the semistable comparison theorem.
\end{itemize}

\vspace{0.1in}
\subsection*{Notation and convention}
\noindent
\vspace{0.1in}

\noindent We adopt the following notation and convention throughout the article:
\begin{itemize}
\item We follow the convention on adic geometry in \cite{Huber} and \cite{Scholze12}, and log adic geometry in \cite{DLLZ}. Also see \cite[\S 2.3]{log1} for some additional conventions on monoids.
\item All formal schemes and adic spaces in this article are assumed to be qcqs.
\item For a topological group $\Gamma$, we simply use $H^i(\Gamma, -)$ to denote the continuous group cohomology $H^i_{\mathrm{cont}}(\Gamma, -)$, by a slight abuse of notation. 
\item For simplicity, we often write $R\Gamma_{\crys}$ (resp. $H^i_{\crys}$) instead of $R\Gamma_{\logcrys}$ (resp. $H^i_{\logcrys}$) for log-crystalline cohomology. We also use the term \emph{$F$-crystals} (resp. \emph{$F$-isocrystals}) instead of log $F$-crystals (resp. log $F$-isocrystals) whenever the context is clear.
\end{itemize}

\vspace{0.1in}
\subsection*{Acknowledgement}
This article is a long‑overdue sequel to \cite{log1}; in particular, the main results of the present work were already announced in \emph{loc.cit.}. The delay arose as the first and third authors (H.D. and Z.Y.) were occupied with other projects. H.D. and Z.Y. would like to thank Z.D. for agreeing to join this project and for his encouragement to complete the manuscript.
During the project, H.D. was partially supported by the National Natural Science Foundation of China (Grant No. 12422101) and the National Key R\&D Program of China (Grant Nos. 2023YFA1009703 and 2021YFA1000704).

\vspace{0.3in}
\section{Relative logarithmic Breuil--Kisin--Fargues modules}\label{section: log BKF modules}
\noindent In this section, we introduce the category $\BKF^{\log}(\fX, \varphi)$ of \emph{relative logarithmic Breuil--Kisin--Fargues modules} over a $p$-adic log formal scheme $\fX$ that is admissibly smooth over a divisible perfectoid split log point $\underline{\mO}$. (The notions of \emph{admissible smoothness} and \emph{divisible perfectoid split log point} will be recalled in \S \ref{subsection: defn of log BKF modules}.) This notion is a logarithmic analogue of the one introduced in \cite{MT} and will serve as the coefficients for our log $\Ainf$-cohomology theory.

We shall also prove that, locally on $\fX$, the category of relative logarithmic Breuil--Kisin--Fargues modules is equivalent to 
\begin{enumerate}
\item a certain category of generalized representations (cf. Definition \ref{defn: generalized representations});
\item a certain category of log $q$-connections equipped with a Frobenius action (cf. Definition \ref{defn: q-connections});
\item a certain category of log prismatic $F$-crystals. (cf. Definition \ref{def: prismatic F-crystals}).
\end{enumerate}

\vspace{0.1in}
\subsection{Relative logarithmic Breuil--Kisin--Fargues modules}\label{subsection: defn of log BKF modules}
\noindent
\vspace{0.1in}

\noindent We first recall the notions of admissible smoothness and divisible perfectoid split log point introduced in \cite{log1}. 

Throughout the section, let $C$ be a perfectoid field of characteristic 0 and residue characteristic $p$ such that it contains all $p$-power roots of unity. Let $\mO=\mO_C$ be its ring of integers and let $\Ainf=W(\mathcal{O}^{\flat})$. We fix a compatible system $\{\zeta_{p^n}\}_{n\ge 0}$ of $p$-power roots of unity and let $\varepsilon:=(1,\zeta_p, \zeta_{p^2}, \cdots)\in \mathcal{O}^{\flat}$. Consider elements $\mu:=[\varepsilon]-1$, $\xi:=\mu/\varphi^{-1}(\mu)$, and $\widetilde{\xi}:=\varphi(\xi)$ in $\Ainf$. Note that $\xi$ is a generator of $\theta: \Ainf\rightarrow \mathcal{O}$ and $\widetilde{\xi}$ is a generator of $\theta\circ \varphi^{-1}$. Moreover, for each integer $r\ge 1$, we consider $\xi_r:=\mu/\varphi^{-r}(\mu)=\frac{[\varepsilon]-1}{[\varepsilon^{1/p^r}]-1}\in \Ainf$.

Let $N_{\infty}$ be a saturated monoid that is \emph{uniquely $n$-divisible} for all $n\ge 1$; that is, the multiplication-by-$n$ map $[n]\colon N_{\infty}\rightarrow N_{\infty}$ is bijective for all $n\ge 1$. Notice that $N_{\infty}$ is necessarily sharp. Let $\alpha\colon N_{\infty}\rightarrow \mO$ be a homomorphism of monoids where $\mO$ is equipped with the multiplicative monoid structure. In particular, the triple  $\underline{\mO}=(\mO, N_{\infty}, \alpha)$ defines a \emph{divisible perfectoid} pre-log ring in the sense of \cite[Definition 5.1]{log1}. We further require that the pre-log ring $\underline{\mO}$ is \emph{split}. Recall that a pre-log ring $(R, M,\alpha)$ is split if the log structure $\mathcal{M}_X$ on $X=\mathrm{Spec}\,R$ associated with the constant pre-log structure $M_X\rightarrow \mO_{X_{\ett}}$ admits a splitting $\mathcal{M}_X\cong \mO^{\times}_{X_{\ett}}\oplus M_X$. In the rest of the section, we fix such a split divisible perfectoid pre-log ring $\underline{\mO}$.

Let $\spf(\ul \mO)^a= \spf(\mO, N_{\infty})^a$ be the $p$-adic log formal scheme associated with the pre-log ring $\underline{\mO}$ and let $\spa(C, \mO_C)_{N_{\infty}}$ be its adic generic fiber. It turns out $\spa(C, \mO_C)_{N_{\infty}}$ is a \emph{divisible log point} in the sense of \cite[Definition 6.1]{log1}.

Let $\fX$ be a $p$-adic log formal scheme which is \emph{admissibly smooth} over $\spf(\mO, N_{\infty})^a$ in the sense of \cite[Definition 6.6 (2)]{log1}. This means $\fX$ is a saturated $p$-adic log formal scheme (i.e., a $p$-adic log formal scheme which \'etale locally admits charts modeled on saturated monoids) that fits into a Cartesian diagram
\[ 
\begin{tikzcd} 
\fX \arrow[d] \arrow[r] & \fX_0 \arrow[d] 
\\ 
\spf(\mO, N_{\infty})^a \arrow[r] & \spf(\mO, N)^a 
\end{tikzcd}
\]
where
\begin{itemize}
\item $N$ is a toric monoid and $\spf(\mO, N)^a$ is the $p$-adic log formal scheme associated with a split pre-log ring $\alpha_0\colon N\rightarrow \mO$.
\item The map $\spf(\mO, N_{\infty})^a\rightarrow \spf(\mO, N)^a$ is induced from a monoid homomorphism $N\rightarrow N_{\infty}$ and it is identity on the underlying formal schemes.
\item The map $\fX_0\rightarrow \spf(\mO, N)^a$ is an integral, ($p$-completely) log smooth, and saturated morphism of $p$-adic formal fs log schemes. \footnote{According to \cite[Proposition 6.7, Remark 6.9]{log1}, one could replace \emph{saturated} by \emph{pseudo-saturated} without changing the definition. See \cite[Definition 3.8]{log1} for the notion of pseudo-saturated morphisms.}
\end{itemize}
In this case, the map $\fX_0\rightarrow \spf(\mO, N)^a$ is referred to as a \emph{finite model} of $\fX\rightarrow \spf(\mO, N_{\infty})^a$.\footnote{Note that $\fX\rightarrow \fX_0$ is an isomorphism on the underlying formal schemes.} Let $X$ denote the adic generic fiber of $\fX$ which exists by \cite[Proposition 6.3]{log1}. By construction, $X$ is admissibly smooth over $\spa(C, \mO_C)_{N_{\infty}}$ in the sense of \cite[Definition 6.5 (2)]{log1}.

Let $X_{\proket}$ be the pro-Kummer \'etale site on $X$ defined in \cite[Section 7.5]{log1} and let $\widehat{\mO}_X^+$ (resp. $\widehat{\mO}_X^{\flat+}$) be the completed integral structure sheaf (resp. tilted structure sheaf) on $X_{\proket}$. Let $\Ainfx:=W(\widehat{\mO}_X^{\flat+})$. Consider the natural projection of sites $\nu\colon X_{\proket}\rightarrow \fX_{\ett}$. The following definitions are analogues of \cite[Definition 5.1]{MT}.

\begin{definition}\label{defn: log BKF modules}
\begin{enumerate}
\item Let $\bM$ be a sheaf of $\Ainfx$-modules on $X_{\proket}$. We say that $\bM$ is \emph{trivial modulo} $\xi_r$ if the sheaf of $\nu_*\big(\Ainfx/\xi_r\big)$-module $\nu_*(\bM/\xi_r)$ is locally finite free and the counit
\[\nu^{-1}\nu_*(\bM/\xi_r)\otimes_{\nu^{-1}\nu_*(\Ainfx/\xi_r)}\Ainfx/\xi_r\rightarrow \bM/\xi_r\]
is an isomorphism of sheaves on $X_{\proket}$. We say that $\bM$ is \emph{trivial modulo} $<\mu$ if it is trivial modulo $\xi_r$ for all $r\ge 1$.
\item A \emph{relative logarithmic Breuil--Kisin--Fargues module without Frobenius} over $\fX$ is a locally finite free $\Ainfx$-module that becomes trivial modulo $<\mu$. The category of such objects is denoted by $\mathrm{BKF}^{\log}(\fX)$.
\item A \emph{relative logarithmic Breuil--Kisin--Fargues module} over $\fX$ is a pair $(\bM, \varphi_{\bM})$ where 
\begin{itemize}
    \item $\bM\in \mathrm{BKF}^{\log}(\fX)$, and 
    \item $\varphi_{\bM}\colon(\varphi^*\bM)[\txi^{-1}]\xrightarrow[]{\sim} \bM[\txi^{-1}]$ is an isomorphism of sheaves of $\Ainfx[\txi^{-1}]$-modules.\footnote{This is equivalent to giving a Frobenius-semilinear isomorphism $\bM[\xi^{-1}]\xrightarrow[]{\sim} \bM[\txi^{-1}]$.}
\end{itemize} 
The category of such objects is denoted by $\mathrm{BKF}^{\log}(\fX, \varphi)$. For simplicity, we refer to these objects as \emph{relative log BKF modules}.
\end{enumerate}
\end{definition}
 
\begin{remark}
If $\underline{\mO}$ is equipped with the trivial log structure, we recover the notion of relative BKF modules introduced in \cite[\S 5]{MT}.
\end{remark}

\vspace{0.1in}
\subsection{Generalized representations}\label{subsection: generalized repn}
\noindent
\vspace{0.1in}

\noindent As a first attempt to understand relative log BKF modules, we study certain generalized representations \`a la \cite{MT}. In \S \ref{subsection: log BKF vs generalized repn}, we will relate these generalized representations to relative log BKF modules.

We recall the following definition from \cite[Definition 1.1]{MT}.

\begin{definition}\label{defn: MT 1.1}
Let $B$ be a ring equipped with an adic topology and let $G$ be a topological group. Suppose $B$ is equipped with a continuous action of $G$. 
\begin{enumerate}
\item[(1)] A \emph{generalized representation} of $G$ on $B$-modules is a finite projective $B$-module $M$ equipped with a continuous semilinear action of $G$. The category of all such generalized representations is denoted by $\Rep_G(B)$.
\item[(2)] Given an element $b\in B$, we say that a generalized representation $M$ is \emph{trivial modulo $b$} if the $(B/b)^G$-module $(M/b)^G$ is finite projective and the natural map
\[(M/b)^G\otimes_{(B/b)^G}B/b\rightarrow M/b\] is an isomorphism. Let $\Rep^b_G(B)\subset \Rep_G(B)$ denote the full subcategory of generalized representations which are trivial modulo $b$.
\end{enumerate}
\end{definition}

We would like to study generalized representations arising from a ``small'' neighborhood of $\fX$ (cf. \cite[Definition 9.1]{log1}).

\begin{definition}\label{definition: small affine}
Suppose $\fX$ is admissbly smooth over $\underline{\mO}=(\mO, N_{\infty})$ as above. We say that an affine \'etale open  $\fU = \spf R \ra \fX$ is \emph{small} if there exist
\begin{enumerate}
\item an injective homomorphism $u\colon N\hookrightarrow P$ such that
\begin{itemize}
\item $N$ is sharp fs and $P$ is torsion-free fs;
\item $P\cap (-N)=\{0\}$;
\item $u$ is saturated;
\item the cokernel of $u^{\gp}\colon N^{\gp}\rightarrow P^{\gp}$ is torsion-free;
\end{itemize}
\item an injection $\iota\colon N\hookrightarrow N_{\infty}$ (which gives rise to an injection $N_{\Q_{\ge 0}}:= \varinjlim  \frac{1}{m} N \hookrightarrow N_{\infty}$)
\end{enumerate}
such that there is a strictly \'etale morphism $\fU=\spf R \rightarrow \spf(\underline{R^{\square}})^a$ which is a composition of rational localizations and strictly finite \'etale morphisms where $\underline{R^{\square}}$ is the pre-log algebra
\[\underline{R^{\square}}=(P\sqcup_NN_{\infty}\rightarrow R^{\square}:=\mO\langle P\rangle\widehat{\otimes}_{\mO\langle N\rangle}\mO).\]
In particular, $\underline{R^{\square}}$ fits into the following commutative diagram of pre-log rings
\[
\begin{tikzcd}[column sep = 1em, row sep = 2.5em]
(\mO  \gr{P}, P) \arrow[r] 
& (R^{\square}, P)  \arrow[r] 
& \ul{R^{\square}} 
\\
(\mO  \gr{N}, N)   \arrow[u] \arrow[r] 
& (\mO, N) \arrow[r] \arrow[u] 
&  (\mO, N_\infty) \arrow[u] .
\end{tikzcd}
\]
where
\begin{enumerate}
\item both squares are pushout squares in the category of saturated $p$-complete pre-log rings;
\item the left vertical arrow is induced by $u\colon N\hookrightarrow P$;
\item the bottom left horizontal arrow is given by $e^n \mapsto \alpha\circ \iota(n)$ for all $n \in N$;
\item the bottom right horizontal arrow is induced by the injection $\iota\colon N\hookrightarrow N_{\infty}$.
\end{enumerate}
It is clear that $\mathfrak{U}=\spf(\underline{R})^a$ where $\underline{R}$ is the pre-log ring $(P\sqcup_NN_{\infty}\rightarrow R^{\square}\rightarrow R)$. In the situation above, we say that $\fU$ (as well as $\ul{R^{\square}}$  and $\underline{R}$) is modeled on the \emph{small chart} $u\colon N \ra P$. We also refer to the map $R^{\square} \ra R$ as a \emph{small coordinate}. For notational simplicity, sometimes we just write $\fU= \spf R$ or $\fU= \spf \ul R$.
\end{definition}

\begin{remark}
According to \cite[Remark 6.9]{log1}, $\fX$ is \'etale locally small.
\end{remark}

\begin{assumption}
In light of \cite[Remark 8.2]{log1}, it does not harm to assume $N_{\infty}=N_{\Q_{\ge 0}}$. We will make this assumption for the rest of the section.
\end{assumption}

\begin{remark}\label{rmk: N free}
In practice, we are mostly interested in the case when $N$ is free (i.e., $N\cong \mathbb{N}^{\oplus n}$ for some $n$). For example, when we prove the $C_{\mathrm{st}}$-conjecture in \S \ref{section: Cst comparison}, the underlying $p$-adic log formal scheme are required to have semistable reduction; in this case, we can take $N=\mathbb{N}$.
\end{remark}

Let $R$, $R^{\square}$, $N$, $P$ be as in Definition \ref{definition: small affine}. To proceed, we consider some suitable perfectoid cover $R_{\infty}$ (resp. $R_{\infty}^{\square}$) of $R$ (resp. $R^{\square}$), as well as their ``$\Ainf$-deformations'' $A(R)$, $A(R^{\square})$, $\Ainf(R_{\infty})$, and $\Ainf(R^{\square}_{\infty})$ (cf. \cite[Section 9]{log1}). More precisely, 
\begin{itemize}
\item Let $P_{\infty}:= P\sqcup_N N_\infty$.
\item Let $R^{\square}_{\infty}:=\mO\langle P_{\Q\ge 0}\rangle\widehat{\otimes}_{\mO\langle N_{\infty}\rangle}\mO$ and $\underline{R_{\infty}^{\square}}:=(R_{\infty}^{\square}, P_{\Q\ge 0})$. In particular, $R^{\square}_{\infty}$ is a perfectoid $\mO$-algebra and $\underline{R_{\infty}^{\square}}$ is a divisible perfectoid pre-log ring.
\item Let $R_{\infty}:=R^{\square}_{\infty}\widehat{\otimes}_{R^{\square}} R$ and $\underline{R_{\infty}}:=(R_{\infty}, P_{\Q\ge 0})$. In particular, $R_{\infty}$ is a perfectoid $\mO$-algebra and $\underline{R_{\infty}}$ is a divisible perfectoid pre-log ring.
\item Let $\Ainf(R^{\square}_{\infty}):= W((R^{\square}_{\infty})^{\flat}) \cong\Ainf\langle P_{\Q\ge 0}\rangle \widehat{\otimes}_{\Ainf\langle N_{\infty}\rangle} \Ainf$ where the completion is $(p,\mu)$-adic. Let $\ul{\Ainf( R^{\square}_{\infty})}$ be the pre-log ring $(\Ainf(R^{\square}_{\infty}), P_{\Q\ge 0})$.
\item Let $A(R^{\square}):=\Ainf\langle P\rangle\widehat{\otimes}_{\Ainf\langle N\rangle} \Ainf$ where the completion is $(p,\mu)$-adic. Let $\ul{A(R^{\square})}$ be the pre-log ring $(A(R^{\square}), P_{\infty})$.
\item Let $A(R^{\square})\rightarrow A(R)$ be the unique formally \'etale map, where $A(R)$ is $(p,\mu)$-adically complete, that lifts the formally \'etale map $R^{\square}\rightarrow R$. Let $\ul{A(R)}$ be the pre-log ring $(A(R), P_{\infty})$.
\item Let $\Ainf(R_{\infty}):=W(R_{\infty}^{\flat})\cong \Ainf(R^{\square}_{\infty})\widehat{\otimes}_{A(R^{\square})}A(R)$ where the completion is $(p,\mu)$-adic. Let $\ul{\Ainf(R_{\infty})}$ be the pre-log ring $(\Ainf(R_{\infty}), P_{\Q\ge 0})$.
\end{itemize}

Notice that $\spf \underline{R_{\infty}}\rightarrow \spf\underline{R}$ and $\spf \underline{R^{\square}_{\infty}}\rightarrow \spf\underline{R^{\square}}$ are both Galois covers with Galois group $\Gamma=\Hom (P^{\gp}/N^{\gp}, \widehat{\Z}(1))$ (cf. \cite[Section 9]{log1}). By choosing a $\Z$-basis of $P^{\gp}/N^{\gp}$, we have $\Gamma\cong  \widehat{\Z}(1)^d$ where $d=\mathrm{rk}_{\Z}(P^{\gp}/N^{\gp})$. Also recall that all of $\Ainf(R^{\square}_{\infty})$, $\Ainf(R_{\infty})$, $A(R^{\square})$, and $A(R)$ admit natural continuous actions of $\Gamma$.

\begin{definition}\label{defn: generalized representations}
\begin{enumerate}
\item Let $\Rep_{\Gamma}(\Ainf(R_{\infty}))$ (resp. $\Rep^{\mu}_{\Gamma}(\Ainf(R_{\infty}))$) denote the category of generalized representations of $\Gamma$ on $\Ainf(R_{\infty})$-modules (resp. the subcategory of those generalized representations which are trivial modulo $\mu$) in the sense of Definition \ref{defn: MT 1.1}. 

\item Let $\Rep^{<\mu}_{\Gamma}(\Ainf(R_{\infty}))$ denote the subcategory of generalized representations of $\Gamma$ which are ``trivial modulo $<\mu$''; that is trivial modulo $\xi_r$ for all integers $r\ge 1$.

\item Let $\Rep_{\Gamma}(\Ainf(R_{\infty}), \varphi)$, $\Rep^{\mu}_{\Gamma}(\Ainf(R_{\infty}), \varphi)$, and $\Rep^{<\mu}_{\Gamma}(\Ainf(R_{\infty}), \varphi)$ be the corresponding category of generalized representations with Frobenius; that is a pair $(M, \varphi_M)$ where 
\begin{itemize}
    \item $M\in \Rep_{\Gamma}(\Ainf(R_{\infty}))$ (resp. $\Rep^{\mu}_{\Gamma}(\Ainf(R_{\infty}))$; resp. $\Rep^{<\mu}_{\Gamma}(\Ainf(R_{\infty})))$, and
    \item $\varphi_M\colon (\varphi^*M)[\txi^{-1}]\xrightarrow[]{\sim} M[\txi^{-1}]$ is a $\Gamma$-equivariant $\Ainf(R_{\infty})[\txi^{-1}]$-linear isomorphism.
\end{itemize}
\item The categories $\Rep_{\Gamma}(A(R))$, $\Rep^{\mu}_{\Gamma}(A(R))$, $\Rep^{<\mu}_{\Gamma}(A(R))$, $\Rep_{\Gamma}(A(R), \varphi)$, $\Rep^{\mu}_{\Gamma}(A(R), \varphi)$, and $\Rep^{<\mu}_{\Gamma}(A(R), \varphi)$ are defined similarly.
\end{enumerate}
\end{definition}

The goal of this section is to prove the following analogues of \cite[Theorem 1.13]{MT} and \cite[Thm 1.14]{MT}.

\begin{theorem}\label{thm: MT 1.13}
The base change functor 
\[-\otimes_{A(R)}\Ainf(R_{\infty})\colon \Rep_{\Gamma}^{\mu}(A(R))\rightarrow \Rep_{\Gamma}^{\mu}(\Ainf(R_{\infty})))\]
is an equivalence of categories. As a corollary, the base change functor
\[-\otimes_{A(R)}\Ainf(R_{\infty})\colon \Rep_{\Gamma}^{\mu}(A(R), \varphi)\rightarrow \Rep_{\Gamma}^{\mu}(\Ainf(R_{\infty})), \varphi)\]
is also an equivalence of categories.
\end{theorem}

\begin{theorem}\label{thm: MT 1.14}
Let $M\in \Rep_{\Gamma}(\Ainf(R_{\infty}))$. Then $M$ is trivial modulo $\mu$ if and only if it is trivial modulo $<\mu$. Namely, $\Rep_{\Gamma}^{\mu}(\Ainf(R_{\infty}))=\Rep_{\Gamma}^{<\mu}(\Ainf(R_{\infty}))$.
\end{theorem}

The proofs of Theorem \ref{thm: MT 1.13} and Theorem \ref{thm: MT 1.14} are similar to the ones in \cite{MT}. We adopt their overall strategy and mainly focus on where our proof differs from \emph{loc. cit.}

Firstly, let us recall some calculations from \cite[Section 10]{log1}. Fix a set of topological generators $\gamma_1, \ldots, \gamma_d$ of $\Gamma\cong \widehat{\Z}(1)^d$. For a character $\psi\colon \Gamma\rightarrow \Ainf^{\times}$, we have defined
\[\Ainf(R_{\infty})_{\psi}:=\big\{x\in \Ainf(R_{\infty})\,|\, \gamma\cdot x=\psi(\gamma)x,\,\,\,\forall x\in\Gamma \big\}.\]
Then we have
\[\Ainf(R_{\infty})=\widehat{\oplus}_{\psi}\Ainf(R_{\infty})_{\psi}\]
where the completion is $(p,\mu)$-adic and $\psi$ runs through all characters which sends $\gamma_i$ to $[\epsilon^{\alpha_i}]$, for all $i=1, \ldots, d$. Here $\alpha_i=\frac{a_i}{p^{r_i}n_i}$ for some $n_i\in \Z_{>0}$ coprime to $p$, some $r_i\in \Z_{\ge 0}$, and some $a_i\in (\Z/n_i\Z)^{\times}\times \Z_p^{\times}$, and 
\[[\epsilon^{\alpha_i}]:=[(\zeta_{p^{r_i}n_i}, \zeta_{p^{r_i+1}n_i}, \zeta_{p^{r_i+2}n_i}, \cdots)]^{a_i}.\]
Moreover, for every character $\chi\colon\Gamma\rightarrow \mO^{\times}$ of finite order, we write
\[\Ainf(R_{\infty})_{\chi}:=\widehat{\oplus}\Ainf(R_{\infty})_{\psi}\]
where the completion is $(p,\mu)$-adic and the direct sum runs through all $\psi$ that lifts $\chi$ (i.e., the composition of $\psi$ with the mod $\xi$ map $\Ainf^{\times}\rightarrow \mO^{\times}$ is $\chi$). Clearly, we have
\[\Ainf(R_{\infty})=\widehat{\oplus}_{\chi}\Ainf(R_{\infty})_{\chi}\]
where the completion is $(p,\mu)$-adic and $\chi$ runs through all finite order characters $\chi$. If $\chi=1$, we have $\Ainf(R_{\infty})_1=A(R)$. For general $\chi$, we know that $\Ainf(R_{\infty})_{\chi}$ is a finite $A(R)$-module.

We also need the following useful fact: If $M$ is an topological abelian group of the form $M=\varprojlim_{n} M_n$ with surjective transition maps such that each $M_n$ is a discrete $\Gamma$-module killed by some power of $p$ and $M$ is equipped with the inverse limit topology, then we have $R\Gamma_{\cont}(\Gamma, M)\cong R\Gamma(\Z^d, M)$. This allows us to reduce the computation of continuous group cohomology to the non-topological group cohomology.

Consider the isomorphisms $\theta_r\colon \Ainf(R_{\infty})/\xi_r\xrightarrow[]{\sim} W_r(R_{\infty})$. Taking inverse limit as $r\rightarrow \infty$, we obtain a natural injection $\Ainf(R_{\infty})/\mu\hookrightarrow W(R_{\infty})$ of $\Z^d$-modules. Notice that the $\Gamma$-action on $\Ainf(R_{\infty})/\mu$ does not extends to $W(R_{\infty})$ continuously.

\begin{lemma}\label{lemma: MT 1.7}
We have
\begin{enumerate}
\item $H^i(\Z^d, \Ainf(R_{\infty})/\mu)\rightarrow H^i(\Z^d, W(R_{\infty}))$ is injective for all $i\ge 0$.
\item $H^i(\Z^d, \Ainf(R_{\infty})/(\mu, p^N))\rightarrow H^i(\Z^d, W(R_{\infty})/p^N)$ is injective for all $i\ge 0$ and $N\ge 0$.
\item $H^i(\Z^d, \Ainf(R_{\infty})/\varphi^{-s}(\mu))$ is $p$-torsion-free for all $i\ge 0$ and $s\in \Z$.
\item $H^1(\Z^d, W(R_{\infty}))$ is $p$-torsion-free.
\end{enumerate}
\end{lemma}

\begin{proof}
This is an analogue of \cite[Lemma 1.7]{MT}. Extend the injection $\Ainf(R_{\infty})/\mu\hookrightarrow W(R_{\infty})$ to a short exact sequence
\[0\rightarrow \Ainf(R_{\infty})/\mu\rightarrow W(R_{\infty})\rightarrow \Xi(R_{\infty})\rightarrow 0\] where $\Xi(R_{\infty})$ is a $p$-torsion free $\Ainf(R_{\infty})$-module that is killed by $W(\mathfrak{m}^{\flat})$. Note that $\Ainf(R_{\infty})/(\mu, p^N)$ has no nonzero elements killed by $W(\mathfrak{m}^{\flat})$ (see the proof of \cite[Proposition 10.12 (3)]{log1}). Hence, the sequence
\[0\rightarrow \Ainf(R_{\infty})/(\mu, p^N)\rightarrow W(R_{\infty})/p^N\rightarrow \Xi(R_{\infty})/p^N\rightarrow 0\] is exact for all $N$. To finish the proof of (1) and (2), it suffices to show $H^i(\Z^d, \Ainf(R_{\infty})/\mu)$ and $H^i(\Z^d, \Ainf(R_{\infty})/(\mu, p^N))$ have no nonzero elements killed by $W(\mathfrak{m}^{\flat})$, for all $N$. But this is already proved in (the proof of) \cite[Proposition 10.12 (3)]{log1}.

To see (3), it suffices to apply $\varphi^s$ and reduce to the case $s=0$, which is proved in (the proof of) \cite[Proposition 10.12 (2)]{log1}.

Given (1)-(3), the proof of (4) is exactly the same as in \cite[Lemma 1.7]{MT}.
\end{proof}

\begin{corollary}\label{cor: MT 1.11 & 1.12}
Let $M\in \Rep_{\Gamma}(\Ainf(R_{\infty}))$. 
\begin{enumerate}
\item If $M$ is trivial modulo $<\mu$, then it is trivial modulo $\varphi^{-1}(\mu)$.
\item If $M$ is trivial modulo $<\mu$ and trivial modulo $(p,\mu)$, then $M$ is trivial modulo $\mu$.
\end{enumerate}
\end{corollary}

\begin{proof}
Given Lemma \ref{lemma: MT 1.7}, the proof is the same as the proof of \cite[Proposition 1.11 \& Corollary 1.12]{MT}.
\end{proof}

Following the strategy of \cite{MT}, the key to prove Theorem \ref{thm: MT 1.13} and Theorem \ref{thm: MT 1.14} is the following theorem.

\begin{theorem}\label{thm: MT 1.16}
Let the triple $(A, A_{\infty}, A^{\square})$ stand for either
\begin{enumerate}
\item[(i)] (integral case) $A=\Ainf$, $A_{\infty}=\Ainf(R_{\infty})$, $A^{\square}=A(R)$; or
\item[(ii)] (mod $p$ case) $A=\mO^{\flat}=\Ainf/p$, $A_{\infty}=R_{\infty}^{\flat}=\Ainf(R_{\infty})/p$, $A^{\square}=A(R)/p$.
\end{enumerate}
Moreover, let $c\in A$ such that $\mu A\subset cA\subset \varphi^{-1}(\mu)A$ and such that $c/\varphi^{-1}(\mu)\not\in A^{\times}$. Then the base change functor
\[-\otimes_{A^{\square}} A_{\infty}\colon \Rep^{c}_{\Gamma}(A^{\square})\rightarrow \Rep^{c}_{\Gamma}(A_{\infty})\]
is an equivalence of categories.
\end{theorem}

\begin{proof}[Proof of fully-faithful-ness in Theorem \ref{thm: MT 1.16}]
The functor is faithful as $A^{\square}\rightarrow A_{\infty}$ is injective. For any $N_1, N_2\in \Rep^c_{\Gamma}(A^{\square})$ and any $A_{\infty}$-linear, $\Gamma$-equivariant homomorphism $f\colon N_1\otimes_{A^{\square}}A_{\infty}\rightarrow N_2\otimes_{A^{\square}}A_{\infty}$, we have to show that $f(N_1)\subset N_2$.

For every finite order character $\chi\colon \Gamma\rightarrow \mO^{\times}$, let $A_{\infty, \chi}$ either denote $\Ainf(R_{\infty})_{\chi}$ in the integral case, or denote $\Ainf(R_{\infty})_{\chi}/p$ in the mod $p$ case. Then we have $A_{\infty}\cong \widehat{\oplus}_{\chi} A_{\infty, \chi}$ where the completion is $(p,\mu)$-adic. This induces a decomposition
\[N_2\otimes_{A^{\square}}A_{\infty}\cong \widehat{\oplus}_{\chi} \big(N_2\otimes_{A^{\square}}A_{\infty, \chi}\big).\]
For every $n\in N_1$, we write $f(n)=\sum_{\chi} f_{\chi}(n)$ with $f_{\chi}(n)\in N_2\otimes_{A^{\square}}A_{\infty, \chi}$. Suppose $\chi$ corresponds to $(k_1, \ldots, k_d)\in (\Q\cap [0,1))^d$; namely, $\chi$ sends $\gamma_i$ to $\zeta^{k_i}$. We define a generalized representation $N_{2,\chi}\in \Rep_{\Gamma}(A^{\square})$ as follows. As an $A^{\square}$-module, $N_{2,\chi}:=N_2\otimes_{A^{\square}}A_{\infty, \chi}$, while the $\Gamma$-action is twisted by
\[\gamma_i\cdot x:=[\varepsilon^{-k_i}]\cdot \gamma_ix\]
for all $x\in N_{2,\chi}$ and all $i=1, \ldots, d$. One checks that $N_{2,\chi}\in \Rep^c_{\Gamma}(A^{\square})$. Now, $f_{\chi}$ defines an element in $N_{\chi}:=\Hom_{A^{\square}}(N_1, N_{2,\chi})\in \Rep^c_{\Gamma}(A^{\square})$ such that $([\varepsilon^{k_i}]\gamma_i-1)\cdot f_{\chi}=0$. 

The rest of the proof is the same as in \cite[Theorem 1.16]{MT}. We include here for completeness. Write $\gamma_i=1+c\delta_i$ for some $\delta_i\in \mathrm{End}_{A^{\square}}(N_{\chi})$. Suppose $\chi\neq 1$. Then $k_i\neq 0$ for some $i$. In this case, we can write
\[[\varepsilon^{k_i}]\gamma_i-1=([\varepsilon^{k_i}]-1)\Big(1+\frac{[\varepsilon^{1/p}]-1}{[\varepsilon^{k_i}]-1}[\varepsilon^{k_i}](c/\varphi^{-1}(\mu))\delta_i\Big).\]
Notice that $[\varepsilon^{k}]-1$ divides $[\varepsilon^{1/p}]-1$ in $\Ainf$ for any $k\in \Q\cap (0,1)$. Also notice that $1+\frac{[\varepsilon^{1/p}]-1}{[\varepsilon^{k_i}]-1}[\varepsilon^{k_i}](c/\varphi^{-1}(\mu))\delta_i$ is an automorphism of $N_{\chi}$ as $N_{\chi}$ is $(c/\varphi^{-1}(\mu))$-adically complete. Consequently, $[\varepsilon^{k_i}]\gamma_i-1$ is injective on $N_{\chi}$ and hence $f_{\chi}=0$, as desired.
\end{proof}

To prove essential surjectivity, we need the following sequence of lemmas.

\begin{lemma}\label{lemma: MT 1.20}
Let $A_{\infty, \chi\neq 1}:=\widehat{\oplus}_{\chi} A_{\infty, \chi}$ where $\chi$ runs through all non-trivial characters $\chi\colon\Gamma\rightarrow \mO^{\times}$ of finite order and the completion is $(p,\mu)$-adic as usual. Notice that $A_{\infty}=A^{\square}\oplus A_{\infty, \chi\neq 1}$.
\begin{enumerate}
\item $H^i(\Gamma, A_{\infty, \chi\neq 1}/c)$ is killed by $\varphi^{-1}(\mu)$ for all $i\ge 0$.
\item In the integral case, $H^i(\Gamma, A_{\infty, \chi\neq 1}/c)$ is $p$-torsion-free, for all $i\ge 0$.
\item In the integral case, $H^1(\Gamma, A_{\infty}/c)$ and $H^1(\Gamma, A^{\square}/c)$ are both $p$-torsion-free.
\end{enumerate}
\end{lemma}

\begin{proof}
The proof is similar to the computation in the proof of \cite[Proposition 10.12]{log1}.
\end{proof}

\begin{lemma}\label{lemma: MT 1.22}
Let $\alpha\colon \Gamma\rightarrow \mathrm{GL}_n(A_{\infty})$ be a continuous 1-cocycle such that 
$\alpha(\gamma)\in 1+cM_n(A_{\infty})$ for all $\gamma\in \Gamma$. Then there exists $X\in 1+(c/\varphi^{-1}(\mu))M_n(A_{\infty})$ such that the 1-cocycle $\alpha'$ sending $\gamma$ to $X^{-1}\alpha(\gamma)\gamma(X)$ satisfies $\alpha'(\gamma)\in 1+cM_n(A^{\square})$ for all $\gamma\in \Gamma$.
\end{lemma}

\begin{proof}
This is an analogue of \cite[Lemma 1.22]{MT}. The proof of \emph{loc. cit.} applies verbatim as long as we use Lemma \ref{lemma: MT 1.20} in place of \cite[Lemma 1.20]{MT}.
\end{proof}

\begin{lemma}\label{lemma: MT 1.23}
\begin{enumerate}
\item The natural homomorphism $A^{\square}/c=(A^{\square}/c)^{\Gamma}\rightarrow (A_{\infty}/c)^{\Gamma}$ is a proj-isomorphism; namely, it induces an equivalence between the category of finite projective $A^{\square}/c$-modules and the category of finite projective $(A_{\infty}/c)^{\Gamma}$-modules.
\item The natural homomorphisms $A^{\square}/(c, p)\rightarrow (A_{\infty}/c)^{\Gamma}/p\rightarrow A_{\infty}/(c,p)$ induce homeomorphisms on $\mathrm{Spec}$.
\end{enumerate}
\end{lemma}

\begin{proof}
This is an analogue of \cite[Lemma 1.23]{MT}. By Lemma \ref{lemma: MT 1.20}, we have an isomorphism $(\Ainf(R_{\infty})/c)^{\Gamma}/p\xrightarrow[]{\sim}\big(\Ainf(R_{\infty})/(c,p)\big)^{\Gamma}$. Since $(\Ainf(R_{\infty})/c)^{\Gamma}$ is $p$-adically complete and separated, the homomorphism $(\Ainf(R_{\infty})/c)^{\Gamma}\rightarrow \big(\Ainf(R_{\infty})/(c,p)\big)^{\Gamma}$ is a proj-isomorphism. Similarly, the homomorphism $A(R)/c\rightarrow A(R)/(c,p)$ is a proj-isomorphism. To prove (1), it remains to show the inclusion $A(R)/(c,p)\hookrightarrow \big(\Ainf(R_{\infty})/(c,p)\big)^{\Gamma}$ is a proj-isomorphism. It suffices to check that the inclusion becomes an isomorphism after modulo the nilradicals on both sides.

Every element in $\big(\Ainf(R_{\infty})/(c,p)\big)^{\Gamma}$ can be expressed as a finite sum $\sum_{\chi}f_{\chi}$ with $f_{\chi}\in \Ainf(R_{\infty})_{\chi}$ such that, for every $\chi$ appearing in the sum, $(\varepsilon^{k_i}-1)f_{\chi}=0$ where $(k_1, \ldots, k_d)\in \Q\cap[0,1)$ is the $d$-tuple corresponding to $\chi$. We claim that for any $k\in \Q\cap (0,1)$, the annihilators of $(\varepsilon^k-1)$ in $\Ainf(R_{\infty})/(c,p)$ are precisely multiples of $c/(\varepsilon^k-1)$. Indeed, by definition $\Ainf(R_{\infty})/(c,p)$ is \'etale over $(\mO^{\flat}/c)[P_{\infty}] \otimes_{(\mO^{\flat}/c)[N_{\infty}]} (\mO^{\flat}/c)$ and, by the argument in \cite[Section 8]{log1}, we have an identification of $(\mO^{\flat}/c)$-modules
\[(\mO^{\flat}/c)[P_{\infty}] \otimes_{(\mO^{\flat}/c)[N_{\infty}]} (\mO^{\flat}/c)\cong (\mO^{\flat}/c)[Q]\] where $Q=Q(N, P, u)$ as in \cite[Section 8]{log1}. In particular, $(\mO^{\flat}/c)[P_{\infty}] \otimes_{(\mO^{\flat}/c)[N_{\infty}]} (\mO^{\flat}/c)$ is free as an $(\mO^{\flat}/c)$-module. This yields the claim. 

We obtain that, if $\chi\neq 1$, each $f_{\chi}$ must be a multiple of $c/(\varepsilon^k-1)$ which is nilpotent. This is enough to conclude part (1).

This also takes care of the first arrow in part (2). To finish the proof, it suffices to show that for every element $f\in \Ainf(R_{\infty})/(c,p)$, there exists $N\ge 1$ such that $f^N\in A(R)/(c,p)$. This is clear from the definition.
\end{proof}

\begin{lemma}\label{lemma: MT 1.24}
Let $R'$ be the $p$-adic completion of a localization of $R$ and let $R'_{\infty}:=R'\widehat{\otimes}_R R_{\infty}$. We can define $\Ainf(R'_{\infty})$ and $A(R')$. Then
\begin{enumerate}
\item There is a natural isomorphism
\[\big(\Ainf(R_{\infty})/(c,p)\big)^{\Gamma}\otimes_{A(R)/(c,p)}A(R')/(c,p)\xrightarrow[]{\sim}\big(\Ainf(R'_{\infty})/(c,p)\big)^{\Gamma}.\]
\item The natural homomorphisms 
\[A(R)/(c,p)\rightarrow A(R')/(c,p)\]
and \[\big(\Ainf(R_{\infty})/(c,p)\big)^{\Gamma}\rightarrow \big(\Ainf(R'_{\infty})/(c,p)\big)^{\Gamma}\]
are localizations.
\end{enumerate}
\end{lemma}

\begin{proof} 
This is an analogue of \cite[Lemma 1.24]{MT}. Notice that \[\big(\Ainf(R_{\infty})/(c,p)\big)^{\Gamma}=\oplus_{\chi}\big(\Ainf(R_{\infty})_{\chi}/(c,p)\big)^{\Gamma}.\] From the argument in the proof of Lemma \ref{lemma: MT 1.23}, we have seen that, if $\chi\neq 1$ and $\chi$ corresponds to $(k_1, \ldots, k_d)\in \big(\Q\cap[0,1)\big)^d$, we have
\[\big(\Ainf(R_{\infty})_{\chi}/(c,p)\big)^{\Gamma}=\frac{c}{\varepsilon^k-1}\big(\Ainf(R_{\infty})_{\chi}/(c,p)\big)\]
where $k$ is the element in $\{k_1, \ldots, k_d\}$ with smallest $p$-adic valuation. The same computation applies to $\Ainf(R'_{\infty})$. Notice that $A(R)/(c,p)\rightarrow A(R')/(c,p)$ is \'etale by construction. The assertion (1) then follows from the flatness of $A(R)/(c,p)\rightarrow A(R')/(c,p)$.

After modulo the nilpotent element $\xi$, the map $A(R)/(c,p)\rightarrow A(R')/(c,p)$ becomes $R/(c,p)\rightarrow R'/(c,p)$ which is a localization. Hence, the \'etale map $A(R)/(c,p)\rightarrow A(R')/(c,p)$ is itself a localization. The last statement follows from part (1).
\end{proof}

\begin{proof}[Proof of essential surjectivity in Theorem \ref{thm: MT 1.16}]
The proof of \cite[Theorem 1.16]{MT} applies here verbatim as long as we replace \cite[Lemma 1.22, Lemma 1.23, Lemma 1.24]{MT} with Lemma \ref{lemma: MT 1.22}, Lemma \ref{lemma: MT 1.23}, Lemma \ref{lemma: MT 1.24}, respectively.
\end{proof}

We are ready to finish the proofs of Theorem \ref{thm: MT 1.13} and Theorem \ref{thm: MT 1.14}.

\begin{proof}[Proof of Theorem \ref{thm: MT 1.13}]
Just take $c=\mu$ in the integral case of Theorem \ref{thm: MT 1.16}.
\end{proof}

\begin{proof}[Proof of Theorem \ref{thm: MT 1.14}]
Suppose $M$ is trivial modulo $<\mu$. Then $M/p\in \Rep_{\Gamma}(R_{\infty}^{\flat})$ is trivial modulo $(\varepsilon-1)/(\varepsilon^{1/p^r}-1)$ for all $r\ge 1$. Applying the mod $p$ case of Theorem \ref{thm: MT 1.16} with $c=(\varepsilon-1)/(\varepsilon^{1/p^r}-1)$, we conclude that $M/p\cong N\otimes_{A(R)/p} R_{\infty}^{\flat}$ for some $N\in \Rep^{(\varepsilon-1)/(\varepsilon^{1/p^r}-1)}_{\Gamma}(A(R)/p)$.

We claim that $N$ must be trivial modulo $\varepsilon-1$ (notice that the $\Gamma$-action on $A(R)/p$ is trivial modulo $\varepsilon-1$). Indeed, since $\gamma(n)-n\in \frac{\varepsilon-1}{\varepsilon^{1/p^r}-1}N$ for all $n\in N$ and $\gamma\in \Gamma$. It remains to prove that 
\[\bigcap_{r\ge 1}\frac{\varepsilon-1}{\varepsilon^{1/p^r}-1}N=(\varepsilon-1)N.\] 
It suffices to prove
\begin{equation}\label{eq: MT 1.14}
\bigcap_{r\ge 1}\frac{\varepsilon-1}{\varepsilon^{1/p^r}-1}(A(R)/p)=(\varepsilon-1)(A(R)/p).
\end{equation}
Recall that \[\Ainf(R_{\infty})/p=\widehat{\oplus}_{\chi} \big(\Ainf(R_{\infty})_{\chi}/p\big)\]
where $\chi$ runs through all finite order characters $\chi\colon\Gamma\rightarrow \mO^{\times}$ and the completion is $\mu$-adic. Recall that $\big(\Ainf(R_{\infty})_{\chi=1}\big)/p=A(R)/p$. Hence, it suffices to prove the analogue of (\ref{eq: MT 1.14}) for $\Ainf(R_{\infty})/p$. Indeed, $\Ainf(R_{\infty})/p$ is formally \'etale over $\mO^{\flat}\langle P_{\infty}\rangle \widehat{\otimes}_{\mO^{\flat}\langle N_{\infty}\rangle} \mO^{\flat}$ and, by the argument in \cite[Section 8]{log1}, we have an identification
\[\mO^{\flat}\langle P_{\infty}\rangle \widehat{\otimes}_{\mO^{\flat}\langle N_{\infty}\rangle} \mO^{\flat}\cong \mO^{\flat}\langle Q\rangle\] where $Q=Q(N, P, u)$ as in \cite[Section 8]{log1}. In particular, $\mO^{\flat}\langle P_{\infty}\rangle \widehat{\otimes}_{\mO^{\flat}\langle N_{\infty}\rangle} \mO^{\flat}$ is topologically free as an $\mO^{\flat}$-module. It remains to prove 
\[\bigcap_{r\ge 1}\frac{\varepsilon-1}{\varepsilon^{1/p^r}-1}\mO^{\flat}=(\varepsilon-1)\mO^{\flat}.\]
But this follows from \cite[Lemma 1.5]{MT}.
\end{proof}

\vspace{0.1in}
\subsection{Relative log BKF modules as generalized representations}\label{subsection: log BKF vs generalized repn}
\noindent
\vspace{0.1in}

\noindent Now we show that, on a small affine \'etale open $\mathfrak{U}=\spf R$ in $\fX$, relative log BKF modules (with or without Frobenius) on $\mathfrak{U}$ can be explicitly described in terms of generalized representations of $\Gamma$ on $\Ainf(R_{\infty})$.

\begin{theorem}\label{thm: BKF mod vs generalized representations}
Let $\mathfrak{U}=\spf R$ be a small affine open in $\fX$ as above and let $U$ be its log adic generic fiber. Let $U_{\infty}$ denote the log affinoid perfectoid object in $U_{\proket}$ whose associated perfectoid space is $\spa(R_{\infty}[\frac{1}{p}], R_{\infty})$. Then the global section functors
\[\Gamma(U_{\infty}, -)\colon \mathrm{BKF}^{\log}(\mathfrak{U})\rightarrow \Rep_{\Gamma}^{<\mu}(\Ainf(R_{\infty}))\]
and 
\[\Gamma(U_{\infty}, -)\colon \mathrm{BKF}^{\log}(\mathfrak{U}, \varphi)\rightarrow \Rep_{\Gamma}^{<\mu}(\Ainf(R_{\infty}), \varphi)\]
are equivalences of categories.
\end{theorem}

The proof of Theorem \ref{thm: BKF mod vs generalized representations} is similar to that of \cite[Theorem 5.14]{MT}. For completeness, we include a sketch of the proof.

We need some preparations. Consider a triple $(B, \mathfrak{M}, \mathbb{B})$ where $B$ is a ring, $\mathfrak{M}\subset B$ is an ideal, and $\mathbb{B}$ is a sheaf of $B$-algebras on $U_{\proket}$. We assume that $(B, \mathfrak{M}, \mathbb{B})$ is in one of the following three cases:
\begin{itemize}
\item $B=\mO$, $\mathfrak{M}=\mathfrak{m}$, $\mathbb{B}=\widehat{\mO}^+_U$;
\item $B=W_r(\mO)$, $\mathfrak{M}=W_r(\mathfrak{m})$, $\mathbb{B}=W_r(\widehat{\mO}^+_U)=\AinfU/\xi_r$, for some integer $r\ge 1$;
\item $B=W_r(\mO^{\flat})$, $\mathfrak{M}=W_r(\mathfrak{m}^{\flat})$, $\mathbb{B}=W_r(\widehat{\mO}^{\flat+}_U)=\AinfU/p^r$, for some integer $r\ge 1$.
\end{itemize}

One checks that such triples $(B, \mathfrak{M}, \mathbb{B})$ satisfy (the pro-Kummer \'etale analogues of) the conditions (B1)-(B4) in \cite[\S 5.1]{MT}. We have the following analogue of \cite[Theorem 5.7]{MT}.

\begin{proposition}\label{prop: MT 5.7}
For any $\pi\in \mathfrak{M}$, taking global sections induces an equivalence of categories
\[\Gamma(U_{\infty}, -)\colon \Big\{\begin{array}{l} \textrm{locally finite projective sheaves of } \mathbb{B}\textrm{-modules}\\ \textrm{on }U_{\proket} \textrm{ that are trivial modulo }<\pi\end{array} \Big\}\xrightarrow[]{\sim} \Rep^{<\pi}_{\Gamma}(\Gamma(U_{\infty}, \mathbb{B})).\]
Moreover, for an element $\bM$ in the category on the left-hand side, there is a natural almost isomorphism (with respect to $\mathfrak{M}$) of complexes of $\Gamma(X, \mathbb{B})$-modules
\[\mathrm{R}\Gamma_{\cont}(\Gamma, \Gamma(U_{\infty}, \bM))\rightarrow \mathrm{R}\Gamma_{\proket}(U, \bM).\]
\end{proposition}

\begin{proof}
The same proof of \cite[Theorem 5.7]{MT} applies here.
\end{proof}

In order to prove Theorem \ref{thm: BKF mod vs generalized representations}, we also need the following classification result of relative log BKF modules.

\begin{lemma}\label{lemma: MT 5.13}
Let $\bM$ be a sheaf of $\AinfU$-modules. Then the following are equivalent
\begin{enumerate}
\item[(a)] $\bM$ is locally finite projective and trivial modulo $<\mu$. 
\item[(b)] We can write $\bM=\varprojlim_s\bM_s$ where each $\bM_s$ is a locally finite projective sheaf of $\AinfU/p^s$-modules which is trivial modulo $<\mu$, and such that the transition maps $\bM_s\rightarrow \bM_{s-1}$ induces isomorphisms $\bM_s\xrightarrow[]{\sim} \bM_{s-1}$, for all $s\ge 1$.
\end{enumerate}
Moreover, if $\bM$ satisfies the equivalence condition above and $V\in U_{\proket}$ is any log affinoid perfectoid object, then
\begin{enumerate}
\item $H^i_{\proket}(V, \bM)$ is killed by $[\mathfrak{m}^{\flat}]$ for all $i >0$.
\item $H^1_{\proket}(V, \bM)$ has no non-zero elements killed by either $p$ or $\xi_r$, for any $r\ge 1$.
\item $\Gamma(V, \bM)$ is a finite projective $\Gamma(V, \AinfU)$-module and the canonical map 
\[\Gamma(V, \bM)\otimes_{\Gamma(V, \AinfU)} \AinfU|_V\rightarrow \bM|_V\]
is an isomorphism.
\item If $\nu_*(\bM/\xi)$ is a finite free $\nu_*(\AinfU/\xi)$-module, then $\bM|_V$ is a finite free $\AinfU|_V$-module.
\end{enumerate}
In particular, such an $\bM$ is necessarily locally finite free, and hence $\bM\in \mathrm{BKF}^{\log}(\mathfrak{U})$.
\end{lemma}

\begin{proof}
This is an analogue of \cite[Proposition 5.13]{MT}. The proof in \emph{loc. cit.} applies here verbatim.
\end{proof}

\begin{proof}[Proof of Theorem \ref{thm: BKF mod vs generalized representations}]
Given Proposition \ref{prop: MT 5.7} and Lemma \ref{lemma: MT 5.13}, one checks that the functor is well-defined and fully faithful, following the arguments in the proof of \cite[Theorem 5.4]{MT}. 

To check the essential surjectivity, we construct the inverse functor. Let $M$ be a generalized represetation in $\Rep^{<\mu}_{\Gamma}(\Ainf(R_{\infty}))$. For every $s\ge 1$, the quotient $M_s:=M/p^s$ is a finite projective $\Ainf(R_{\infty})/p^s$-modules equipped with a continuous semilinear $\Gamma$-action which is trivial modulo $<\mu$. Using the equivalence in Proposition \ref{prop: MT 5.7}, it corresponds to a locally finite projective sheaf $\bM_s$ of $\AinfU/p^s$-modules which is trivial modulo $<\mu$. One checks that the inverse system $(\bM_s)_{s \ge 1}$ satisfies the condition in Lemma \ref{lemma: MT 5.13}, and hence $\bM:=\varprojlim_s \bM_s$ is indeed a relative log BKF module.
\end{proof}

\vspace{0.1in}
\subsection{Generalized representations as log $q$-connections}\label{subsection: log q-connections} 
\noindent
\vspace{0.1in}

\noindent Our next goal is to relate relative log BKF modules and generalized representations with certain logarithmic $q$-connections. More precisely, we will establish equivalences of categories
\[\Rep_{\Gamma}^{\mu}(A(R))\cong q\mathrm{MIC}(A(R))\]
and
\[\Rep_{\Gamma}^{\mu}(A(R), \varphi)\cong q\mathrm{MIC}(A(R), \varphi)\]
where $q\mathrm{MIC}(A(R))$ (resp. $q\mathrm{MIC}(A(R), \varphi)$) is the category of the finite projective $A(R)$-modules with integrable log $q$-connections (resp. those equipped with an additional semilinear Frobenius action).

We first recall the notions of log $q$-de Rham complex and log $q$-connections in a general setup (cf. \cite[\S 2.1-2.2]{MT}). For the rest of the section, we keep the following assumptions:

\begin{assumption}\label{assumption: qDR}
Let $A$ be a commutative ring. 
\begin{enumerate}[leftmargin=1.5cm]
\item[($q$DR1)] Suppose $A^{\square}$ is a commutative $A$-algebra equipped with $d$ commuting $A$-algebra automorphisms $\gamma_1, \ldots, \gamma_d$. In particular, $A^{\square}$ admits an action of $\Z^d$. Fix $q\in A$ such that $q-1$ is a nonzero divisor of $A^{\square}$. Assume that $\gamma_i\equiv \mathrm{id} \,(\textrm{mod }q-1)$ for all $i=1,\ldots, d$.
\item[($q$DR2)] Assume there exist elements $U_1, \ldots, U_d\in A^{\square}$ such that $\gamma_i(U_i)=qU_i$ and $\gamma_i(U_j)=U_j$ if $i\neq j$.
\item[($q$DR$\varphi$)] There is an endomorphisms $\varphi\colon A\rightarrow A$ such that $\varphi(q)=q^p$ in $A$. There is also a $\varphi$-semilinear endomorphism on $A^{\square}$, still denoted by $\varphi$, that commutes with $\gamma_1, \ldots, \gamma_d$ on $A^{\square}$. Moreover, we assume that $\varphi(U_i)=U_i^p$ for all $i=1, \ldots, d$. 
\end{enumerate}
\end{assumption}

\begin{definition}\label{defn: q-dR complex}
The \emph{logarithmic $q$-de Rham complex} $q\Omega^{\bullet}_{A^{\square}/A}:=\oplus_{n=0}^d q\Omega^n_{A^{\square}/A}$ is the differential graded algebra satisfying the following properties:
\begin{itemize}
\item $q\Omega^0_{A^{\square}/A}=A^{\square}$.
\item $q\Omega^1_{A^{\square}/A}$ is an $A^{\square}$-bimodule. As a left $A^{\square}$-module, it is the free left $A^{\square}$-module with formal basis ``$d\log (U_1)$'', \ldots, ``$d\log (U_d)$''. The right $A^{\square}$-module structure is determined by
\[d\log(U_i)\cdot f=\gamma_i(f)d\log (U_i)\]
for all $f\in A^{\square}$ and $i=1, \ldots, d$.
\item There is an identification as left $A^{\square}$-modules
\[\bigoplus_{1\le i_1<\cdots<i_n\le d}A^{\square}=q\Omega^n_{A^{\square}/A}\]
sending \[(f_{i_1, \ldots, i_n})\mapsto \sum_{1\le i_1<\cdots<i_n\le d} f_{i_1, \ldots, i_d}\,\, d\log(U_{i_1})\cdots d\log(U_{i_n}).\] 
\item $d\log (U_i) d\log (U_i)=0$ and $d\log (U_i) d\log (U_j)=-d\log (U_j) d\log (U_i)$ if $i\neq j$.
\item The differential $\mD_q\colon A^{\square}\rightarrow q\Omega^1_{A_{\square}/A}$ sends \[f\mapsto \sum_{i=1}^d \frac{\gamma_i(f)-f}{q-1} d\log (U_i).\] The general differential is induced by $\mD_q$ such that $d\log (U_i)$'s satisfy the cocycle conditions.
\end{itemize}
\end{definition}

\begin{remark}
In \cite{MT}, $q\Omega^{\bullet}_{A^{\square}/A}$ is simply called the ``$q$-de Rham complex''.
\end{remark}

Notice that $\mD_q\colon A^{\square}\rightarrow q\Omega^1_{A^{\square}/A}$ is $A$-linear satisfying the Leibniz rule
\[\mD_q(fg)=\mD_q(f)g+f\mD_q(g)\]
for all $f, g\in A^{\square}$. It is also clear that $\mD_q(U_i)=U_i\,d\log (U_i)$ for all $i$.

\begin{definition}\label{defn: q-connections}
A \emph{module with log $q$-connection} over $A^{\square}$ is a right $A^{\square}$-module $N$ equipped with an $A$-linear map $\nabla\colon N\rightarrow N\otimes_{A^{\square}}q\Omega^1_{A^{\square}/A}$ satisfying the Leibniz rule
\[\nabla(nf)=\nabla(n)f+n\otimes \mD_q(f)\]
for all $n\in N$ and $f\in A^{\square}$. Such a log $q$-connection naturally extends to 
\[\nabla\colon N\otimes_{A^{\square}}q\Omega^{\bullet}_{A^{\square}/A}\rightarrow N\otimes_{A^{\square}}q\Omega^{\bullet+1}_{A^{\square}/A}.\]
We say that the log $q$-connection is \emph{integrable} if \[\nabla^2\colon N\rightarrow N\otimes_{A^{\square}}q\Omega^{2}_{A^{\square}/A}\]
equals zero.
\end{definition}

\begin{definition}
Let $q\mathrm{MIC}(A^{\square})$ denote the category of finitely generated $A^{\square}$-modules with integrable log $q$-connections.
\end{definition}

\begin{definition}
For $(N, \nabla)\in q\mathrm{MIC}(A^{\square})$ and each $i=1, \ldots, d$, we define $\nabla_i^{\log}\colon N\rightarrow N$ to be the $A^{\square}$-linear endomorphism given by $\nabla(-)=\sum_{i=1}^d \nabla_i^{\log}(-)\otimes d\log U_i$.
\end{definition}

\begin{theorem}\label{thm: MT 2.6}
\begin{enumerate}
\item Given a generalized representation $N\in \Rep_{\Z^d}^{q-1}(A^{\square})$, the map 
\[\nabla\colon N\rightarrow N\otimes_{A^{\square}}q\Omega^1_{A^{\square}/A}; \ \ n\mapsto \sum_{i=1}^d \frac{\gamma_i(n)-n}{q-1}\,\mathrm{dlog}(U_i)\]
is an integrable log $q$-connection on $N$. Moreover, the resulting log $q$-de Rham complex is naturally quasi-isomorphic to $L\eta_{q-1}R\Gamma(\Z^d, N)$. (Here, $L\eta$ stands for the d\'ecalage functor, see \cite[Section 2.2]{log1} for a brief review.)
\item Suppose $A^{\square}$ is in addition $(p,q-1)$-adically complete and separated. Then the construction in (i) induces an equivalence of categories
\[\Rep_{\Z^d}^{q-1}(A^{\square})\xrightarrow[]{\sim} q\mathrm{MIC}(A^{\square})\]
compatible with tensor products and internal homs. 
\end{enumerate}
\end{theorem}

\begin{proof}
This is basically \cite[Proposition 2.6]{MT}. It is also proved in \cite[Proposition 2.5]{Tian_prismatic_etale_comparison}.
\end{proof}

There is also a version with Frobenius. For a log $q$-connection $(N, \nabla)\in q\mathrm{MIC}(A^{\square})$, the $\varphi$-twist $\varphi^*N:=N\otimes_{A^{\square}, \varphi} A^{\square}$ is equipped with a natural log $q$-connection $\varphi^*\nabla$. More precisely, if we write
\[\nabla(n)=\sum_{i=1}^d \nabla_i^{\log}(n)\otimes d\log (U_i),\]
then, for every $n\otimes f\in \varphi^*N=N\otimes_{A^{\square}, \varphi} A^{\square}$, we put
\[(\varphi^*\nabla)(n\otimes f):=\sum_{i=1}^d (\varphi^*\nabla)^{\log}_i (n\otimes f)\otimes d\log(U_i)\]
where
\[(\varphi^*\nabla)^{\log}_i (n\otimes f)=\nabla_i^{\log}(n)\otimes[p]_q \,\gamma_i(f)+n\otimes \frac{\gamma_i(f)-f}{q-1}.\]
Here, $[p]_q$ stands for $\frac{q^p-1}{q-1}=1+q+\cdots+q^{p-1}\in A$, the usual ``$q$-analogue of $p$''.

\begin{definition}
Keep the notations in Assumption \ref{assumption: qDR}.
Let $q\mathrm{MIC}(A^{\square}, \varphi)$ denote the category of integrable log $q$-connections on finitely generated $A^{\square}$-modules equipped with a semilinear Frobenius automorphism; namely, it is the category of triples $(N, \nabla, \varphi_N)$ where $(N, \nabla)\in q\mathrm{MIC}(A^{\square})$ and \[\varphi_N\colon (\varphi^*N)[\frac{1}{[p]_q}]\xrightarrow[]{\sim} N[\frac{1}{[p]_q}]\]
is an isomorphism of log $q$-connections.
\end{definition}

\begin{corollary}\label{cor: MT 2.14}
Suppose $A^{\square}$ is $(p,q-1)$-adically complete and separated. Then the functor in Theorem \ref{thm: MT 2.6} induces an equivalence of categories
\[\Rep_{\Z^d}^{q-1}(A^{\square}, \varphi)\xrightarrow[]{\sim} q\mathrm{MIC}(A^{\square}, \varphi).\]
\end{corollary}

\begin{proof}
The proof is the same as in \cite[Proposition 2.14]{MT}. To see that the functor is well-defined, it suffices to notice that $\varphi_N$ commutes with $\gamma_i$'s if and only if it is horizontal with respect to $\nabla$. But this is precisely the condition in the definition of $q\mathrm{MIC}(A^{\square}, \varphi)$. 
\end{proof}

Now we retrieve the assumption in \S \ref{subsection: generalized repn} and apply the results above to $A=\Ainf$, $A^{\square}=A(R)$ and $q=[\epsilon]$ (i.e., $q-1=\mu$). Notice that $A(R)$ is $(p,\mu)$-adically complete and separated, and it is equipped with a natural action of $\Gamma$ which is trivial modulo $\mu$. To fulfill the assumptions ($q$DR1), ($q$DR2), ($q$DR$\varphi$), we need the following Lemma 
\begin{lemma}\label{Lem: the assumptions of qDR}
Choose an isomorphism $\Gamma\cong \widehat{\Z}(1)^d=\oplus_{i=1}^d \widehat{\Z}(1)\gamma_i$. Then there are elements $U_1, \ldots, U_d\in A(R)$ such that 
\[
\gamma_i(U_j)=\begin{cases}
  qU_j  & \text{ if } i=j\\
  U_j  & \text{ if } i\neq j
\end{cases}
\]
and $\varphi(U_i)=U_i^p$ for all $i$.
\end{lemma} 

To prove Lemma \ref{Lem: the assumptions of qDR}, we first need the following splitting lemma.

\begin{lemma}\label{lemma: splitting}
The injection $N^{\gp}\hookrightarrow P^{\gp}$ admits a splitting
\[P^{\gp}\cong N^{\gp}\oplus \big(\oplus_{i=1}^d \Z e_i \big)\]
such that $e_i\in P$. 
\end{lemma}

\begin{proof}
Choose any splitting $P^{\gp}\cong N^{\gp}\oplus \big(\oplus_{i=1}^d \Z e'_i \big)$ with $e'_i\in P^{\gp}$. Since $P_{\Q_{\ge 0}}$ is a rational polyhedron cone inside $\Q^d$, there exists $\delta_1, \ldots, \delta_d\in P^{\gp}_{\Q}$ such that $\oplus_{i=1}^d \Q_{\ge 0} \delta_i\subset P_{\Q_{\ge 0}}$. Use $\delta_1, \ldots, \delta_d$ as basis for the $\Q$-vector space $P^{\gp}_{\Q}$ and consider the coordinates for $e'_1, \ldots, e'_d$. By basic linear algebra, there exists a linear transformation in $\mathrm{GL}_d(\Z)$ which translate $e'_1, \ldots, e'_d$ to $e_1, \ldots, e_d$ such that all $e_i$'s have non-negative coordinates. This means \[e_i\in P^{\gp}\cap \big(\oplus_{i=1}^d \Q_{\ge 0} \delta_i\big)\subset P^{\gp}\cap P_{\Q}=P\]
as desired.
\end{proof}

\begin{proof}[Proof of Lemma \ref{Lem: the assumptions of qDR}]

Recall that $A(R)=\Ainf(R_{\infty})_{\chi=1}$. It suffices to construct elements $U_1, \ldots, U_d \in \Ainf(R_{\infty})$ such that \[
\gamma_i(U_j)=\begin{cases}
  qU_j  & \text{ if } i=j\\
  U_j  & \text{ if } i\neq j
\end{cases}
\]
and $\varphi(U_i)=U_i^p$ for all $i$. Since $\Ainf(R_{\infty})$ is formally \'etale over $\Ainf(R_{\infty}^{\square})$, it suffices to find such elements $U_1, \ldots, U_d$ inside $\Ainf(R_{\infty}^{\square})$.

Recall that there is an identification 
\[\Ainf(R_{\infty}^{\square})=\Ainf\langle P_{\Q\ge 0}\rangle\widehat{\otimes}_{\Ainf\langle N_{\infty}\rangle}\Ainf \cong \Ainf\langle Q\rangle\]
as $\Ainf$-modules, where $Q=Q(N, P, u)$ as in \cite[Section 8]{log1}. Let $e_1, \ldots, e_d\in P$ as in Lemma \ref{lemma: splitting} and let $q_i:=f_N(e_i)$ (cf. \cite[Section 3.2]{log1}) for all $i=1, \ldots, d$. We claim that $U_i:=e^{q_i}\in \Ainf\langle Q\rangle$ satisfies the desired properties.

To see this, first consider the natural maps of sets $\iota\colon Q\hookrightarrow P_{\Q\ge 0}\rightarrow P_{\Q\ge 0}^{\gp}/N_{\infty}^{\gp}$. By the definition of $Q$, one sees that this composition $\iota$ is injective. For any 
\[ 
    \gamma\in \Gamma=\mathrm{Hom}\big((P_{\Q\ge 0}^{\gp}/P^{\gp})/(N_{\Q}^{\gp}/N^{\gp}), \mu_\infty(\mO)\big)
\] 
and any $\alpha\in Q$, the action of $\gamma$ on $e^{\alpha}\in \Ainf\langle Q\rangle$ is given by the formula
\[
    \gamma\cdot e^{\alpha}=[\gamma(\iota(e^{\alpha}))^{\flat}]e^{\alpha},
\] 
where $\gamma(\iota(e^{\alpha}))^{\flat}$ stands for
\[
    \big(\gamma(\iota(e^{\alpha})), \gamma(\iota(e^{\alpha})/p), \gamma(\iota(e^{\alpha})/p^2), \cdots \big),
\]
and $[\cdot]$ is the Techm\"uller lift. Now, it follows from the construction of $q_i$'s that $\gamma_i\cdot e^{q_i}= [\epsilon] e^{q_i}$ and $\gamma_i\cdot e^{q_j}= e^{q_j}$ if $i\neq j$. It is also clear that $\varphi(e^{q_i})=e^{pq_i}=(e^{q_i})^p$ (notice that the subset $Q\subset P_{\infty}$ is closed under multiplication by $p$).
\end{proof}

\begin{theorem}\label{representations vs q-connections}
Fix $\gamma_1,\ldots, \gamma_d\in \Gamma$ and $U_1, \ldots, U_d\in A(R)$ as above. Given a generalized representation $M\in \Rep_{\Gamma}^{\mu}(A(R))$, the map 
\[
    \nabla\colon M\rightarrow M\otimes_{A(R)}q\Omega^1_{A(R)/\Ainf}; \ \ 
    m\mapsto \sum_{i=1}^d \frac{1}{\mu}(\gamma_i(m)-m)\otimes \mathrm{dlog}(U_i)
\]
is an integrable log $q$-connection. Moreover, the resulting $q$-de Rham complex is naturally quasi-isomorphic to $L\eta_{\mu}R\Gamma_{\cont}(\Gamma, M)$. This induces equivalences of categories
\[\Rep_{\Gamma}^{\mu}(A(R))\xrightarrow[]{\sim} q\mathrm{MIC}(A(R))\]
and 
\[\Rep_{\Gamma}^{\mu}(A(R), \varphi)\xrightarrow[]{\sim} q\mathrm{MIC}(A(R), \varphi).\]
\end{theorem}

\begin{proof}
By fixing a compatible system of roots of unity, we may identify $\widehat{\Z}(1)$ with $\widehat{\Z}$. Using Theorem \ref{thm: MT 2.6} and Corollary \ref{cor: MT 2.14}, it suffices to show that the functor $\Rep_{\Gamma}^{\mu}(A(R))\rightarrow \Rep_{\Z^d}^{\mu}(A(R))$ induced from the natural inclusion \[\Z^d=\oplus_{i=1}^d\Z\gamma_i\hookrightarrow\Gamma= \oplus_{i=1}^d\widehat{\Z}\gamma_i\] is an equivalence of categories. 

Consider the inclusion $\Z^d=\oplus_{i=1}^d\Z\gamma_i\hookrightarrow \Z_p^d=\oplus_{i=1}^d\Z_p\gamma_i$, which factors as $\Z^d\hookrightarrow \Gamma=\widehat{\Z}^d\twoheadrightarrow \Z_p^d$. We first show that the inclusion $\Z^d\hookrightarrow \Z_p^d$ induces an equivalence
\[\Rep_{\Z_p^d}^{\mu}(A(R))\xrightarrow[]{\sim} \Rep_{\Z^d}^{\mu}(A(R)).\]
Indeed, by assumption, we have $\gamma\equiv \mathrm{id} \,(\mathrm{mod}\,\, \mu)$ for any $\gamma\in \Z^d$. An easy computation of binomial expansion shows that $\gamma^{p^m}\equiv \mathrm{id} \,(\mathrm{mod}\,\, (p,\mu)^m)$ for every integer $m>0$. In particular, the $\Z^d$-action automatically extends to a continuous action of $\Z_p^d$.

By the construction of $A(R)$, the $\widehat{\Z}^d$-action on $A(R)$ factors through $\Z_p^d$. Hence, the surjection $\widehat{\Z}^d\rightarrow \Z_p^d$ induces a functor $\Rep_{\Z_p^d}^{\mu}(A(R))\rightarrow \Rep_{\widehat{\Z}^d}^{\mu}(A(R))$ whose composition with $\Rep_{\widehat{\Z}^d}^{\mu}(A(R))\rightarrow \Rep_{\Z^d}^{\mu}(A(R))$ is the equivalence $\Rep_{\Z_p^d}^{\mu}(A(R))\xrightarrow[]{\sim} \Rep_{\Z^d}^{\mu}(A(R))$. It remains to show that every generalized representation in $\Rep_{\Z_p^d}^{\mu}(A(R))$ has a unique way to extend to an element in $\Rep_{\widehat{\Z}^d}^{\mu}(A(R))$. Indeed, for any $t$ in the kernel of $\widehat{\Z}^d\rightarrow \Z_p^d$, we have $t\in p^m\widehat{\Z}^d$ for every $m> 0$. Hence, $t\equiv \mathrm{id}\,(\mathrm{mod } (p,\mu)^m)$ for all $m>0$ by the computation above. This means $t=\mathrm{id}$ by the $(p,\mu)$-adic completeness.

\end{proof}

For later use, we introduce the notion of convergent log $q$-connections. We resume Assumption \ref{assumption: qDR}.

\begin{definition}\label{defn: convergent log q-connection}
An element $(N, \nabla)\in q\mathrm{MIC}(A^{\square})$ is called \emph{$(p,[p]_q)$-adically convergent} (resp., \emph{$[p]_q$-adically convergent}) if for every $n\in N$ and $1\le i\le d$, there exists $m\ge 1$ such that 
\[\nabla_i^{\log}(\nabla_i^{\log}-[1]_q)(\nabla_i^{\log}-[2]_q)\cdots(\nabla_i^{\log}-[m]_q)(n)\in (p, [p]_q)N\]
(resp., $\in [p]_qN$) where $[a]_q$ stands for $\frac{q^a-1}{q-1}=1+q+\cdots+q^{a-1}$.
For simplicity, we just say \emph{convergent} when we mean $(p,[p]_q)$-adically convergent. We use $q\mathrm{MIC}_{\mathrm{conv}}(A^{\square})$ to denote the full subcategory of convergent log $q$-connections.
\end{definition}

\begin{definition}
A generalized representation in $\Rep^{\mu}_{\Gamma}(A^{\square})$ is called \emph{convergent} if the corresponding log $q$-connection is convergent. We use $\Rep^{\mu}_{\Gamma, \mathrm{conv}}(A^{\square})$ to denote the full subcategory of convergent generalized representations.
\end{definition}

It turns out every log $q$-connection equipped with Frobenius is automatically convergent.

\begin{proposition}\label{prop: automatically convergent}
Every element in $q\mathrm{MIC}(A^{\square}, \varphi)$ is $[p]_q$-adically convergent (hence also $(p,[p]_q)$-adically convergent).
\end{proposition}

\begin{proof}
The proof is similar to \cite[Lemma 2.24]{MT}. Recall that $\varphi^*N=N\otimes_{A^{\square}, \varphi}A^{\square}$ is equipped with a log $q$-connection $\varphi^*\nabla$. Using the logarithmic coordinates \[\nabla (-)=\sum_{i=1}^d \nabla_i^{\log}(-)\otimes d\log U_i\] and 
\[\varphi^*\nabla (-)=\sum_{i=1}^d (\varphi^*\nabla)_i^{\log}(-)\otimes d\log U_i,\] 
we have the description
\[
    (\varphi^*\nabla)_i^{\log}\colon\varphi^*N\rightarrow \varphi^*N; \ \  n\otimes f\mapsto \nabla_i^{\log}(n)\otimes [p_q]\gamma_i(f)+n\otimes d_{q,i}^{\log}(f),
\]
where $d_{q,i}^{\log}(f)$ stands for $\frac{\gamma_i(f)-f}{q-1}$.

We claim that for any integer $m\ge 0$ and any $f\in A^{\square}$, we have 
\[(d_{q,i}^{\log}-[m]_q)(d_{q,i}^{\log}-[m+1]_q)\cdots (d_{q,i}^{\log}-[m+p-1]_q)(f)\in [p]_qA^{\square}.\]
Notice that $d_{q,i}^{\log}(\varphi(f))\in [p]_q A^{\square}$ for all $f\in A^{\square}$. It follows that
\[
(d_{q,i}^{\log}-[m]_q)(d_{q,i}^{\log}-[m+1]_q)\cdots (d_{q,i}^{\log}-[m+p-1]_q)(\varphi(f))\in [p]_qA^{\square},
\]
because exactly one of $[m]_q, [m+1]_q, \ldots, [m+p-1]+q$ is divisible by $[p]_q$. Since $A^{\square}$ is generated by $\varphi(A^{\square})$ and $U_1^{k_1}\cdots U_d^{k_d}$ for $0\le k_1, \ldots, k_d\le p-1$, it suffices to show that 
\[(d_{q,i}^{\log}-[m]_q)(d_{q,i}^{\log}-[m+1]_q)\cdots (d_{q,i}^{\log}-[m+p-1]_q)(U_1^{k_1}\cdots U_d^{k_d})\in [p]_qA^{\square}\]
for such $k_1, \ldots, k_d$. Straightforward calculation shows that
\begin{align*}
& (d_{q,i}^{\log}-[m]_q)(d_{q,i}^{\log}-[m+1]_q)\cdots (d_{q,i}^{\log}-[m+p-1]_q)(U_1^{k_1}\cdots U_d^{k_d})\\
=& ([k_i]_q-[m]_q)([k_i]_q-[m+1]_q)\cdots ([k_i]_q-[m+p-1]_q)U_1^{k_1}\cdots U_d^{k_d}.
\end{align*}
Notice that $k_i\equiv m+t\,\,(\textrm{mod}\,\,p)$ for some $0\le t\le p-1$. It follows that $[k_i]_q-[m+t]_q\in [p]_q A^{\square}$, as desired.

Back to the proof. Given the claim, we conclude that 
\[((\varphi^*\nabla)_i^{\log}-[m]_q)((\varphi^*\nabla)_i^{\log}-[m+1]_q)\cdots ((\varphi^*\nabla)_i^{\log}-[m+p-1]_q)(\varphi^*N)\subset [p]_q \varphi^*N\]
for every $1\le i\le d$ and any $m\ge 0$. Hence, 
\[(\varphi^*\nabla)_i^{\log}((\varphi^*\nabla)_i^{\log}-[1]_q)\cdots ((\varphi^*\nabla)_i^{\log}-[pt-1]_q)(\varphi^*N)\subset [p]_q^t \varphi^*N\]
for every $1\le i \le d$ and any $t\ge 1$.

By assumption, there exists $m_1, m_2\in \Z$ such that $[p]_q^{m_1}N\subset \varphi_N(\varphi^*N)\subset [p]_q^{m_2}N$. Using the identity $(\nabla_i^{\log}-[m]_q)\varphi_N=\varphi_N((\varphi^*\nabla)_i^{\log}-[m]_q)$ for every $m\ge 0$, we conclude that 
\[\nabla_i^{\log}(\nabla_i^{\log}-[1]_q)(\nabla_i^{\log}-[2]_q)\cdots(\nabla_i^{\log}-[p(m_1-m_2+1)-1]_q)(n)\in [p]_qN\]
as desired.
\end{proof}

\vspace{0.1in}
\subsection{Relative log BKF modules as log prismatic crystals}\label{subsection: log prismatic crystals}
\noindent
\vspace{0.1in}

\noindent Our final task of this section is to relate relative log BKF modules with log prismatic crystals introduced in \cite{Koshikawa}. In the rest of the section, we are mainly interested in the case where $N$ is free; i.e., $N \cong \N^{\oplus n}$ for some $n$. (cf. Remark \ref{rmk: N free}.)

First of all, we briefly review the theory of log prismatic cohomology from \cite{Koshikawa}. Recall that a \emph{$\delta_{\log}$-ring} is a quadruple $(A, \alpha\colon M \ra A, \delta, \delta_{\log})$ where
\begin{enumerate}
\item[(1)] $(A, \alpha\colon M \ra A)$ is a pre-log ring;
\item[(2)] $\delta \colon A \ra A$ is map making $(A, \delta)$ a $\delta$-ring; and 
\item[(3)] $\delta_{\log} \colon M \ra A$ is a map satisfying the following conditions:
\begin{itemize}
    \item $\delta_{\log}(0) =0$, 
    \item $\delta_{\log}(m_1+m_2)= \delta_{\log}(m_1)+\delta_{\log}(m_2)+p\delta_{\log}(m_1)\delta_{\log}(m_2)$,
    \item $\alpha(m)^p\cdot \delta_{\log}(m)= \delta(\alpha(m))$.
\end{itemize}
\end{enumerate}
When the context is clear, we simply write $(A, M)$ for the $\delta_{\log}$-ring.

A \emph{$\delta_{\log}$-triple} is a triple $(A,I,M)$ where $(A,M)$ is a $\delta_{\log}$-ring and $I\subset A$ is an ideal.

A \emph{pre-log prism} is a triple $(A, I, M)$ where $(A, M)$ is a $\delta_{log}$-ring and $(A, I)$ is a prism in the sense of \cite{BS}. A pre-log prism $(A, I, M)$ is called \emph{bounded} if $(A,I)$ is a bounded prism. A \emph{log prism} is $(A, I, \mM)= (A, I, M)^a$ where $(A, I, M)$ is a bounded pre-log prism and ``$a$'' stands for taking the associated log structure.

\begin{definition}
Let $(A,I,M_A)$ be a bounded pre-log prism such that $M_A$ is integral, and let $(A,I,\mM_A)$ be the associated log prism. Let $\fX= (\fX, \mM_{\fX})$ be a $p$-adic log formal scheme over $(A/I, M_A)$. The category $(\fX/(A,I,M_A))_{\Prism}$ is defined to be the opposite of the category of integral log prisms $(B,IB,\mM_B)= (B,IB,M_B)^a$ equipped with 
\begin{itemize}
    \item a morphism $(A,I,\mM_A) \ra (B,IB,\mM_B)$ of log prisms,
    \item a map $f\colon \spf(B/IB) \ra \fX$ of $p$-adic formal schemes, and
    \item an exact closed immersion $(\spf(B/IB), f^*\mM_{\fX}) \ra (\spf B ,\mM_B)$ of log $(p,I)$-adic formal schemes.
\end{itemize}
There is an obvious notion of morphisms between such objects. When the context is clear, we simply write $(\spf B \leftarrow \spf(B/IB)\rightarrow \fX)$, or $(B,IB,\mM_B)$ (or even $(B,IB,M_B)$) for an object in this category. 

Moreover, a morphism $(B,IB,\mM_B) \ra (C,IC,\mM_C)$ is called an \emph{\'etale cover} if the map $(\spf C,\mM_C)\ra (\spf B, \mM_B)$ of log $(p,I)$-adic formal schemes is strict and the ring map $B\rightarrow C$ is $(p,I)$-completely \'etale, and faithfully flat. Equipped with the \'etale topology, we obtain the \emph{(relative) log prismatic site} $(\fX/(A,I,M_A))_{\Prism}$. When $\fX=\spf \ul R$ is affine,\footnote{Recall that, for a pre-log ring $\ul R$, we use $\spf \ul R$ to denote $\spf (\ul R)^a$.} we also use the notation $(\ul R/ (A,I,M_A))_{\Prism}$. For an object in $(\ul R/ (A,I,M_A))_{\Prism}$, we write $(B\rightarrow B/IB \leftarrow R)$ instead of $(\spf B \leftarrow \spf(B/IB)\rightarrow \spf R)$.

There is a structure sheaf $\mO_{\Prism}$ (resp. reduced one $\overline{\mO}_{\Prism}$) sending $(B,IB,\mM_B)\mapsto B$ (resp., $(B,IB,\mM_B)\mapsto B/IB$). 
\end{definition}

\begin{notation}
\begin{enumerate}
\item For any $\Ainf$-algebra $B$, we write $B^{(1)}$ for the Frobenius twist $B\otimes_{\Ainf,\varphi}\Ainf$ along $\vp\colon \Ainf \xrightarrow[]{\cong} \Ainf$. 
\item Any $\mO$-algebra $C$ may be regarded as an $\Ainf$-algebra via $\theta \colon \Ainf \ra \Ainf/\xi \cong \mO$, and thus we have $C^{(1)} \cong C\otimes_{\Ainf/\xi,\vp}\Ainf/\txi$.
\item Let $R, R_{\infty}, A(R), \Ainf(R_{\infty})$, and $\Gamma$ be as in \S \ref{subsection: generalized repn}. Let $R^{(1)}$ and $A(R)^{(1)}$ be the Frobenius twists of $R$ and $A(R)$ respectively. We regard $R^{(1)}$ as a quotient of $A(R)$ through the $\Ainf$-linear homomorphism
\[
    A(R) \xrightarrow[]{\widetilde\theta} A(R)/\txi \xrightarrow[\cong]{\vp^{-1}} (A(R)/\xi)^{(1)} = R^{(1)}.
\]
Notice that we also have $R^{(1)} = A(R)^{(1)}/\txi$.
\item Recall that we have fixed a split divisible perfectoid pre-log ring $\underline{\mO}=(\mO, N_{\infty}, \alpha)$. Consider the pre-log ring $\underline{\Ainf}:=(\Ainf, N_{\infty}, \alpha')$ given by $\alpha'\colon N_{\infty}\rightarrow \Ainf$ sending $n\mapsto [\alpha(n)^{\flat}]=[(\alpha(n), \alpha(n/p), \ldots)]$.
\item Let $\underline{R^{(1)}}$ be the pre-log ring with the pre-log structure 
\[
    P_{\infty}=P\sqcup_N N_\infty \cong (P\sqcup_N N_\infty)\sqcup_{N_\infty,p}N_\infty \xrightarrow[]{\alpha_R\sqcup \alpha'} R\otimes_{\Ainf,\vp}\Ainf = R^{(1)},
\]
where $\alpha_R \colon P_{\infty} \ra R$ denotes the pre-log structure on $R$.
\item Consider the log prismatic site $\big(\underline{R^{(1)}}/(\Ainf, \txi, N_\infty)\big)_{\Prism}$. For simplicity, we sometimes just write $\big(R^{(1)}/(\ul\Ainf, \txi)\big)_{\Prism}$. Notice that $\big(A(R) \ra A(R)/\txi \xleftarrow[]{\vp} R^{(1)}\big)$ is a prism in $\big(R^{(1)}/(\ul\Ainf, \txi)\big)_{\Prism}$, where $A(R)$ is equipped with the usual pre-log structure $P_{\infty}\rightarrow A(R)$.
\end{enumerate}
\end{notation}

\begin{definition}\label{def: prismatic F-crystals}
    Let $(A,I,M_A)$ be a bounded pre-log prism with $M_A$ being integral. Let $\fX$ be a $p$-adic log formal scheme over $(A/I, M_A)$. A \emph{locally finite free log crystal} on the log prismatic site $(\fX/(A,I,M_A))_{\Prism}$ is a sheaf of $\mO_{\Prism}$-module $\mF$ such that 
    \begin{itemize}
        \item for every object $\fB=(\spf(B)\leftarrow \spf(B/IB)\ra\fX)$ of $(\fX/(A,I,M_A))_{\Prism}$, $\mF(\fB)$ is a finite projective $B$-module;
        \item for every morphism \[\fB^\prime=(\spf(B^\prime)\leftarrow \spf(B^\prime/IB^\prime)\ra\fX) \ra \fB=(\spf(B)\leftarrow \spf(B/IB)\ra\fX)\] in $(\fX/(A,M_A))_{\Prism}$, the pullback homomorphism $\mF(\fB)\otimes_{B}B^\prime \ra \mF(\fB^\prime)$ is an isomorphism.
    \end{itemize}
    Let $\CR\big((\fX/(A,I,M_A))_{\Prism}\big)$ denote the category of locally finite free log crystals on $(\fX/(A,I,M_A))_{\Prism}$. When $\fX= \spf \underline{R}$ is affine, we also write $\CR\big((R/(A,I,M_A))_{\Prism}\big)$.
\end{definition} 

Consider the log prismatic site $\big(R^{(1)}/(\underline{\Ainf}, \txi)\big)_{\Prism}$. Taking specialization along $\big(A(R) \ra A(R)/\txi \xleftarrow[]{\vp} R^{(1)}\big)$, we obtain a functor
\[
    \mathrm{ev}_{A(R)}\colon\CR\big(\big(R^{(1)}/(\ul\Ainf, \txi)\big)_{\Prism}\big)\rightarrow \Rep^{\mu}_{\Gamma}(A(R)).
\] 
Indeed, for any log prismatic crystal $\mF\in \CR(R^{(1)}/(\ul{\Ainf}, \txi)_{\Prism})$, the $\Gamma$-action on $\mathrm{ev}_{A(R)}(\mF)$ is induced from the action of $\Gamma$ on $A(R)$. To see the triviality modulo $\mu$, notice that $\big(A(R)/\mu\rightarrow R/(\zeta_p-1)\leftarrow R^{(1)}\big)$ is also a prism in $\big(R^{(1)}/(\ul\Ainf, \txi)\big)_{\Prism}$ with trivial $\Gamma$-action and that $\mathrm{ev}_{A(R)/\mu}=\mathrm{ev}_{A(R)}/\mu$. Note that $\mathrm{ev}_{A(R)}$ fits into the following commutative diagram
\[
 \begin{tikzcd}
 \CR\big(\big(R^{(1)}/(\ul\Ainf, \txi)\big)_{\Prism}\big)  \arrow[r, "\mathrm{ev}_{A(R)}"] \arrow[dr, "\mathrm{ev}_{\Ainf(R_{\infty})}"'] & \Rep^{\mu}_{\Gamma}(A(R)) \arrow[d, "\cong"]\\
& \Rep^{\mu}_{\Gamma}(\Ainf(R_{\infty}))
\end{tikzcd}
\]
where $\mathrm{ev}_{\Ainf(R_{\infty})}$ is the specialization map at the prism $(\Ainf(R_{\infty}) \rightarrow \Ainf(R_{\infty})/\tilde{\xi} \xleftarrow[]{\varphi} R^{(1)})$.

Recall that $\Rep_{\Gamma,\conv}^{\mu}(A(R))$ stands for the full subcategory of $\Rep_{\Gamma}^{\mu}(A(R))$ consisting of those generalized representations that are convergent (i.e., $(p,\mu)$-adically convergent). The main goal of this section is to prove the following theorem.

\begin{theorem}\label{thm: fully faithful}
Let $\fU=\spf\underline{R}$ be a small affine \'etale open (cf. Definition \ref{definition: small affine}) modeled on the small chart $u:N\rightarrow P$. We further assume that both $N$ and $P$ are free. Then $\mathrm{ev}_{A(R)}$ is fully faithful with essential image $\Rep_{\Gamma,\conv}^{\mu}(A(R))$. 
\end{theorem}

This is an analogue of \cite[Theorem 3.2]{MT} and the proof is similar to the one of \emph{loc. cit.} The idea is to pass through an intermediate category of \emph{stratifications}. 

Firstly, we recall the notion of log $q$-PD triples (cf. \cite[Definition 7.1]{Koshikawa}). Let $A= \Z[\![q-1]\!]$ with $\delta$-structure given by $\delta(q)=1$. Let 
\[
    [p]_q= \dfrac{q^p-1}{q-1}= 1+q+ \cdots + q^{p-1}
\]
which is a $q$-analogue of ``$p$'' in $A$. The resulting prism $(A,[p]_q)$ is called the \emph{$q$-crystalline prism}. A \emph{$q$-PD pair} (cf. \cite[Definition 16.2]{BS}) is a $(p, [p]_q)$-complete $\delta$-ring $(D,I)$ over $(A,(q-1))$ such that
\begin{itemize}
    \item $(D,[p]_q)$ is a bounded prism over $(A,[p]_q)$,
    \item $\varphi(I) \subseteq [p]_q D$ and $\gamma(I) \subseteq I$, where 
    \[
        \gamma(x):= \frac{\varphi(x)}{[p]_q}-\delta(x),
    \]
    \item $D/(q-1)$ is $p$-torsion free with finite $(p,[p]_q)$-complete Tor-amplitude over $D$, and 
    \item $D/I$ is classically $p$-complete.
\end{itemize}

\begin{definition}\label{def: log q-PD triple}
\begin{enumerate}
\item[(1)] A \emph{pre-log $q$-PD triple} is a triple $(D,I,M)$ such that $(D,M)$ is a $\delta_{\log}$-ring and $(D,I)$ is a $q$-PD pair.
\item[(2)] A \emph{log $q$-PD triple} is a triple $(D,I, \mathcal{M})$ such that $(D,I)$ is a $q$-PD pair and 
\[(D,[p]_q,\mathcal{M})=(D,[p]_q, M)^a\]
for some pre-log $q$-PD triple $(D,I,M)$.
\end{enumerate}
\end{definition}

\begin{example}
    If we set $q= [\epsilon]$, then $\xi= \varphi^{-1}([p]_q)$. In this case, both $(\Ainf, \xi, N_\infty)$ and $(\Ainf, \txi, N_\infty)$ are pre-log $q$-PD triples. As the pre-log structure on $\Ainf$ is clear, in what follows, we simply write $(\ul\Ainf, \xi)$ and $(\ul\Ainf, \txi)$. 
\end{example}

For every $m\ge 0$, let $A(R)(m)$ be the $(p,\mu)$-adic completion of the $(m+1)$-fold tensor product $A(R)\otimes_{\Ainf} \cdots \otimes_{\Ainf} A(R)$, equipped with the natural log structure. We consider the log $q$-PD-envelope of \[A(R)(m)\xrightarrow[]{\Delta} A(R)\rightarrow A(R)/\xi=R\] over the pre-log $q$-PD pair $(\ul \Ainf, \xi)$, where $\Delta$ is the multiplication map. The existence of this log $q$-PD-envelope is guaranteed by the following lemma.

\begin{lemma}\label{lemma: existence of log q-PD-envelope}
\begin{enumerate}[leftmargin=*, itemindent=0pt]
\item Let $E(m)\rightarrow A(R)$ be the exactification of the multiplication map $\Delta\colon A(R)(m)\rightarrow A(R)$. Then the kernel of the composition $E(m)\rightarrow A(R)\rightarrow A(R)/\xi=R$ is $(p,\mu)$-adically generated by a regular sequence relative to $\Ainf$.
\item The log $q$-PD envelope of $A(R)(m)\rightarrow R$ (relative to $(\ul\Ainf, \xi)$) exists and coincides with the $q$-PD envelope of $E(m)\rightarrow R$.
\end{enumerate}
\end{lemma}

\begin{proof}
It suffices to prove the corresponding statements for $R^{\square}$ and $A(R^{\square})$ in place of $R$ and $A(R)$. More precisely, for every $m\ge 0$, let $A(R^{\square})(m)$ denote the $(p,\mu)$-adic completion of the $(m+1)$-fold tensor product of $A(R^{\square})$ over $\Ainf$, equipped with the natural log structure. It suffices to prove (1) and (2) for the maps $A(R^{\square})(m)\rightarrow A(R^{\square})$ and $A(R^{\square})(m)\rightarrow R^{\square}$.
\begin{enumerate}[leftmargin=*, itemindent=0pt]
\item Notice that 
\begin{align*}
A(R^{\square})(m) & = \big(\Ainf\langle P\rangle \widehat{\otimes}_{\Ainf\langle N\rangle} \Ainf\big)\widehat{\otimes}_{\Ainf} \cdots \widehat{\otimes}_{\Ainf}\big(\Ainf\langle P\rangle \widehat{\otimes}_{\Ainf\langle N\rangle} \Ainf\big)\\
& \cong \Ainf\langle P^{\oplus (m+1)}\rangle \widehat{\otimes}_{\Ainf\langle N^{\oplus (m+1)}\rangle}\Ainf
\end{align*}
equipped a log structure modeled on the chart $P^{(m)}:= P^{\oplus (m+1)}\sqcup_{N^{\oplus (m+1)}} N_{\infty}$ where $N^{\oplus (m+1)}\rightarrow N_{\infty}$ is the composition of $N^{\oplus (m+1)}\xrightarrow[]{\Delta} N\xrightarrow[]{\iota} N_{\infty}$. We need to compute the exactification $P^{(m)}_E\rightarrow P_{\infty}$ of \[P^{(m)}=P^{\oplus (m+1)}\sqcup_{N^{\oplus (m+1)}} N_{\infty}\rightarrow P_{\infty}=P\sqcup_NN_{\infty}\]
induced by the summation map $P^{\oplus (m+1)}\rightarrow P$. By definition, $P^{(m)}_E$ is the preimage of \[P_{\infty}=P\sqcup_NN_{\infty}\subset P^{\gp}\sqcup_{N^{\gp}}N_{\infty}^{\gp}\]
under the surjection
\[(P^{(m)})^{\gp}=(P^{\gp})^{\oplus (m+1)}\sqcup_{(N^{\gp})^{\oplus (m+1)}} N_{\infty}^{\gp}\rightarrow P_{\infty}^{\gp}=P^{\gp}\sqcup_{N^{\gp}}N_{\infty}^{\gp}.\]
Consider monoids
\[P'_{(m)}:=\{(x_1, \ldots, x_{m+1})\in (P^{\gp})^{\oplus (m+1)}\,|\, x_1+\cdots+x_{m+1}\in P\}\]
and \[N'_{(m)}:=\{(x_1, \ldots, x_{m+1})\in (N^{\gp})^{\oplus (m+1)}\,|\, x_1+\cdots+x_{m+1}\in N\}.\]
Notice that the map $(N^{\gp})^{\oplus (m+1)}\rightarrow (P^{\gp})^{\oplus (m+1)}$ restricts to $N'_{(m)}\rightarrow P'_{(m)}$ and the map $(N^{\gp})^{\oplus (m+1)}\rightarrow N_{\infty}^{\gp}$ restricts to $N'_{(m)}\rightarrow N_{\infty}$. 

We claim that $P_E^{(m)}=P'_{(m)}\sqcup_{N'_{(m)}}N_{\infty}$. To prove the claim, we observe the identifications
\[N'_{(m)}\xrightarrow[]{\sim} (N^{\gp})^{\oplus m}\oplus N;\ \ (x_1,\ldots, x_{m+1})\mapsto (x_1, \ldots, x_m; x_1+\cdots +x_{m+1})\]
and 
\[P'_{(m)}\xrightarrow[]{\sim} (P^{\gp})^{\oplus m}\oplus P;\ \ (x_1,\ldots, x_{m+1})\mapsto (x_1, \ldots, x_m; x_1+\cdots +x_{m+1}).\]
Under these identifications, the map $N'_{(m)}\rightarrow P'_{(m)}$ is simply $(N^{\gp})^{\oplus m}\oplus N\rightarrow (P^{\gp})^{\oplus m}\oplus P$ induced by the inclusions $N\rightarrow P$ and $N^{\gp}\rightarrow P^{\gp}$. In particular, \[P'_{(m)}\sqcup_{N'_{(m)}}N_{\infty}\cong (P^{\gp}/N^{\gp})^{\oplus m}\oplus (P\sqcup_NN_{\infty})\] is a saturated monoid and the map \[P'_{(m)}\sqcup_{N'_{(m)}}N_{\infty}\rightarrow (P^{\gp})^{\oplus (m+1)}\sqcup_{(N^{\gp})^{\oplus (m+1)}} N_{\infty}^{\gp}\] can be identified with the map \[(P^{\gp}/N^{\gp})^{\oplus m}\oplus (P\sqcup_NN_{\infty})\rightarrow (P^{\gp}/N^{\gp})^{\oplus m}\oplus (P^{\gp}\sqcup_{N^{\gp}}N^{\gp}_{\infty})\]
induced by the identity on the first factor and the inclusion $P\sqcup_NN_{\infty}\rightarrow P^{\gp}\sqcup_{N^{\gp}}N^{\gp}_{\infty}$ on the second factor. This is exactly what we desired.

Back to the proof. By definition, the exactification $E^{\square}(m)\rightarrow A(R^{\square})$ is given by 
\begin{align*}
E^{\square}(m)& = A(R^{\square})\widehat{\otimes}_{\Ainf\langle P^{(m)}\rangle}\Ainf\langle P^{(m)}_E\rangle\\
& \cong \Ainf\langle P^{(m)}_E\rangle \widehat{\otimes}_{\Ainf\langle N_{\infty}\rangle} \Ainf\\
&\cong \left(\Ainf\langle (P^{\gp}/N^{\gp})^{\oplus m}\rangle\widehat{\otimes}_{\Ainf}\Ainf\langle P\sqcup_N N_{\infty}\rangle\right)\widehat{\otimes}_{\Ainf\langle N_{\infty}\rangle} N\\
& = \left(\Ainf\langle (P^{\gp}/N^{\gp})^{\oplus m}\rangle\right) \widehat{\otimes}_{\Ainf} A(R^{\square})
\end{align*}
It is clear that the kernel of $E^{\square}(m)\rightarrow A(R^{\square})$ is $(p,\mu)$-adically generated by a regular sequence relative to $\Ainf$.

\item By \cite[Lemma 16.10]{BS}, $E^{\square}(m)\rightarrow R^{\square}$ admits a $q$-PD envelope relative to $\Ainf$. More precisely, if $x_1, \ldots, x_r$ $(r=2m\cdot \mathrm{rk}_{\Z}(P^{\gp}/N^{\gp}))$ is a regular sequence generating the kernel of $E^{\square}(m)\rightarrow A(R^{\square})$, then the $q$-PD envelope of $E^{\square}(m)\rightarrow R^{\square}$ is given by
\[E^{\square}(m)\Big\{\frac{\varphi(x_1)}{[p]_q}, \cdots,\frac{\varphi(x_r)}{[p]_q}\Big\}^{\wedge}_\delta\]
obtained by freely adjoining $\frac{\varphi(x_i)}{[p]_q}$'s in the category of $(p,\mu)$-adically complete $\delta$-$\Ainf$-algebras. (Note that $[p]_q=\txi$ here.)

Finally, by the construction in \cite[Lemma 7.4]{Koshikawa}, the log $q$-PD envelope of the surjection $A(R^{\square})(m)\rightarrow R^{\square}$ is precisely the $q$-PD envelope of $E^{\square}(m)\rightarrow R$.
\end{enumerate}
\end{proof}

Let $\mD(R)(m)$ denote the log $q$-PD envelope of $A(R)(m) \xrightarrow[]{\Delta} A(R) \ra R$. Let $J(m)$ be the kernel of $\mD(R)(m) \ra R$, then $\vp(J(m))\subseteq (\txi)$ by the definition of log $q$-PD triples. Therefore, $\mD(R)(m)/\txi$ admits a morphism 
\[
    R^{(1)}= (\mD(R)(m)/J(m))^{(1)} \xrightarrow[]{\vp} \mD(R)(m)/\txi.
\]
Consider the pre-log prism
\[\fD(R)(m):= \big(\mD(R)(m), \txi, P^{(m)}_E\big),\]
viewed as an object in the log prismatic site $\big(R^{(1)}/(\ul\Ainf, \txi)\big)_{\Prism}$.

\begin{lemma}
For every $m\ge 0$, $\fD(R)(m)$ is equal to the coproduct of $m+1$ copies of $\fD(R)(0)=\big(A(R), \txi, P_\infty \big)$ in $(R^{(1)}/(\ul\Ainf, \txi))_{\Prism}^{\mathrm{op}}$.
\end{lemma}

\begin{proof}
The proof essentially follows that of \cite[Lemma 3.13]{MT}.
For every $m\ge 0$, let $(A(R)/\txi)(m)$ denote the $p$-adic completion of the $(m+1)$-fold tensor product of $A(R)/\txi$ over $R^{(1)}$. The surjection $A(R)\rightarrow A(R)/\txi$ induces a surjection $A(R)(m)\rightarrow (A(R)/\txi)(m)$. Let $J_A'(m)$ be the kernel of $A(R)(m)\rightarrow (A(R)/\txi)(m)$. Since the coproducts in $\big(R^{(1)}/(\ul\Ainf, \txi)\big)_{\Prism}^{\mathrm{op}}$ are computed by log prismatic envelopes, the coproduct of $m+1$ copies of $\fD(R)(0)$ is given by 
\[
    \big( \mD'(R)(m) \ra \mD'(R)(m)/\txi \leftarrow R^{(1)}\big),
\]
where $\mD'(R)(m)$ is the log prismatic envelope of $A(R)(m) \ra (A(R)/\txi)(m)$. Here, the structure map $R^{(1)} \ra \mD'(R)(m)/\txi$ is given by the composition
\[
    R^{(1)} \ra (A(R)/\txi)(m) \ra \mD'(R)(m)/\txi
\]
where the second map is induced by the inclusion $\vp(J'_A(m)) \subseteq (\txi)$ in $\mD'(R)(m)$. 
Let $J'(m)$ be the kernel of $\mD(R)(m)\rightarrow (A(R)/\txi)(m)$.
Since $\txi$ lies in the kernel of $A(R)(m) \ra (A(R)/\txi)(m)$, we have a commutative diagram of pre-log rings
\begin{equation}\label{diagram: MT Lem3.13}
    \begin{tikzcd}
    A(R)(m)\arrow[r] & (A(R)/\txi)(m)\\
    A(R)(m)\arrow[u, "\varphi"] \arrow[r]& R \arrow[u]
    \end{tikzcd}
\end{equation}
where $\varphi$ is induced from the Frobenius map. Taking log $q$-PD envelopes of the bottom and top arrows respectively, we obtain a diagram
\begin{equation}\label{diagram: MT Lem3.13(2)}
    \begin{tikzcd}
    A(R)(m)\arrow[r] & \mD'(R)(m)\\
    A(R)(m)\arrow[u, "\varphi"] \arrow[r] & \mD(R)(m) \arrow[u, "\iota"']
    \end{tikzcd}
\end{equation}
We claim that $\iota \circ \vp^{-1}$ induces an isomorphism
\[
    \big( \mD(R)(m) \ra \mD(R)(m)/\txi \leftarrow R^{(1)}\big) \cong \big( \mD'(R)(m) \ra \mD'(R)(m)/\txi \leftarrow R^{(1)}\big)
\]
between log prisms. To prove this, we need to check
\begin{enumerate}
    \item the map $\iota \colon \mD(R)(m) \ra \mD'(R)(m)$ is an isomorphism;
    \item their structure maps from $R^{(1)}$ coincide, i.e. the following diagram commutes
    \[
    \begin{tikzcd}
        & (A(R)/\txi)(m) \arrow[r] & \mD'(R)(m)/\txi \\
        R^{(1)} \arrow[ru] \arrow[r,equal] & (\mD(R)(m) /J(m))^{(1)} \arrow[r,"\vp"] & \mD(R)(m)/\txi \arrow[u, "\iota\circ \vp^{-1}"']
    \end{tikzcd}
    \]
\end{enumerate}

For (1), it is equivalent to showing that $(\ul\Ainf, \txi)\rightarrow (\ul{\mD(R)(m)}, \txi)$ is the log prismatic envelope of the map of $\delta_{\log}$-triples $(\ul\Ainf, \txi)\rightarrow (\ul{A(R)(m)}, J_A'(m))$. By the explicit description of log prismatic envelope (cf. \cite[Proposition 3.9]{Koshikawa}), the log prismatic envelope of $(\ul\Ainf, \txi)\rightarrow (\ul{A(R)(m)}, J_A'(m))$ is obtained by taking the exactification $E(m)$ then freely adjoining $\frac{\varphi(x_1)}{[p]_q},\ldots, \frac{\varphi(x_1)}{[p]_q}$ as $\delta$-$\Ainf$-algebra. (Notice that $[p]_q= \txi$). This is precisely $\mD(R)(m)$.

For (2), consider diagram (\ref{diagram: MT Lem3.13(2)}). Let $J_A(m)$ denote the kernel of the $A(R)(m) \ra R$. Then $\vp(J_A(m)) \subseteq (\txi)$ in $\mD(R)(m)$ and $\vp(J_A'(m))\subseteq (\txi)$ in $\mD'(R)(m)$ respectively. Further notice that $\vp\colon A(R)(m) \ra A(R)(m)$ sends $J_A(m)$ into $J_A'(m)$ by diagram (\ref{diagram: MT Lem3.13}), then diagram (\ref{diagram: MT Lem3.13(2)}) leads to the following commutative diagram of $\delta_{\log}$-$\Ainf$-rings:
\[
\begin{tikzcd}
    (A(R)/\txi)(m) \arrow[r, equal] & A(R)(m)/J_A'(m) \arrow[r, "\vp"] & \mD'(R)(m)/\txi \\
    R \arrow[r,equal] \arrow[u] & A(R)(m)/J_A(m) \arrow[u, "\vp"] \arrow[r, "\vp"] & \mD(R)(m)/\xi \arrow[u, "\iota"].
\end{tikzcd}
\]
It fits into the commutative diagram
    \[
    \begin{tikzcd}
        & (A(R)/\txi)(m) \arrow[r] & \mD'(R)(m)/\txi \\
        & R \arrow[r] \arrow[u] & \mD(R)(m)/\xi \arrow[u, "\iota"]\\
        R^{(1)} \arrow[ru] \arrow[ruu] \arrow[r,equal] & (\mD(R)(m) /J(m))^{(1)} \arrow[r,"\vp"] & \mD(R)(m)/\txi \arrow[u, "\vp^{-1}"'],
    \end{tikzcd}
    \]
    which completes the proof.
\end{proof}

Now, $\fD(R)(\bullet)$ defines a cosimplicial object of $\big(R^{(1)}/(\ul\Ainf, \txi)\big)_{\Prism}^{\mathrm{op}}$. 
\begin{itemize}
\item Let $p_0, p_1\colon\fD(R)(0)\rightarrow \fD(R)(1)$ be the canonical morphisms corresponding to the map $[0]=\{0\}\rightarrow [1]=\{0,1\}$ with image $\{0\}$ and $\{1\}$, respectively.
\item Let $p_{01}, p_{12}, p_{02}\colon\fD(R)(1)\rightarrow \fD(R)(2)$ be the canonical morphisms corresponding to the map $[1]=\{0,1\}\rightarrow [2]=\{0,1,2\}$ with image $\{0,1\}$, $\{1,2\}$, and $\{0,2\}$, respectively.
\item Let $\Delta\colon\fD(R)(1)\rightarrow \fD(R)(0)$ be the morphism corresponding to the unique map $[1]=\{0,1\}\rightarrow [0]=\{0\}$.
\item Let $q_0, q_1, q_2\colon\fD(R)(0)\rightarrow \fD(R)(2)$ be the canonical morphisms corresponding to the map $[0]=\{0\}\rightarrow [2]=\{0,1,2\}$ with image $\{0\}$, $\{1\}$, and $\{2\}$, respectively.
\end{itemize}

\begin{definition}\label{defn: stratification}
A \emph{stratification} on an $A(R)$-module $N$ with respect to $\mD(R)(\bullet)$ is a $\mD(R)(1)$-linear isomorphism $\varepsilon\colon N\otimes_{A(R), p_2}\mD(R)(1)\xrightarrow[]{\sim} N\otimes_{A(R), p_1}\mD(R)(1)$ satisfying
\begin{enumerate}
\item The scalar extension $\Delta^*(\varepsilon)$ of $\varepsilon$ by $\Delta\colon \mD(R)(1)\rightarrow A(R)$ is the identity map on $N$.
\item There is an identification 
\[
    p_{01}^*(\varepsilon)\circ p_{12}^*(\varepsilon)=p_{02}^*(\varepsilon)\colon N\otimes_{A(R), q_2}\mD(R)(2)\xrightarrow[]{\sim}N\otimes_{A(R), q_0}\mD(R)(2).
\]
\end{enumerate}
\end{definition}

Let $\mathrm{Strat}(\mD(R)(\bullet))$ denote the category of finite projective $A(R)$-modules equipped with a stratification with respect to $\mD(R)(\bullet)$. There is a natural functor
\[
    \mathrm{ev}_{\fD(R)(\bullet)}\colon \CR\big(\big(R^{(1)}/(\ul\Ainf, \txi)\big)_\Prism\big) \rightarrow \mathrm{Strat}(\mD(R)(\bullet))
\]
sending a log prismatic crystal $\mF$ to the $A(R)$-module $N:=\mF(\fD(R)(0))$ equipped with the stratification defined as the composition
\[\varepsilon\colon N\otimes_{A(R), p_2}\mD(R)(1)\xrightarrow[p_2]{\sim} \mF(\fD(R)(1))\xrightarrow[p_1^{-1}]{\sim} N\otimes_{A(R), p_1} \mD(R)(1).\]

\begin{lemma} \label{Lem: Tian Lemma 1.9}
     Assume that $N$ and $P$ are free. Let $(B, \txi, \mM_B)$ be an object in $\big((R^{(1)},P_\infty)/(\ul\Ainf,\txi)\big)_\Prism$. Then the chart $P_\infty \ra \Gamma(\spf R^{(1)}, \mM_{\spf R^{(1)}}) \ra \Gamma\big(\spf(B/\txi), \mM_{\spf(B/\txi)}\big)$ lifts to a chart $\alpha_B \colon P_\infty \ra \Gamma(\spf B, \mM_{\spf B})$ such that the diagram of monoids
     \[
         \begin{tikzcd}
             P_\infty \arrow[rr] \arrow[d,equal] & & R^{(1)} \arrow[d] \\
             P_\infty \arrow[r, "\alpha_B"] & B \arrow[r] & B/\txi 
         \end{tikzcd}
     \]
     commutes, and $(B, P_\infty)$ has a natural induced $\delta_{\log}$-structure such that $(B,\txi, P_\infty)^a = (B, \txi, \mM_B)$.
\end{lemma}

\begin{proof}
    The proof is essentially the same as that of \cite[Lemma 1.9]{Tian_prismatic_etale_comparison}, but we need an additional combinatorial argument to deal with the more general situation. By definition, the log structure $\mM_{\spf (B/\txi)}$ on $\spf(B/\txi)$ coincides with the pullback log structure from $\spf R$, hence it admits a chart $P_\infty \ra \Gamma(\spf (B/\txi),\mM_{\spf (B/\txi)})$.  By \cite[Lemma 2.1]{DLMS2}, the log prism $(B, \txi, \mM_B)$ comes from a pre-log prism $(B, \txi, M_B)$ with 
    \[
        M_B = P_\infty \times_{\Gamma\big(\spf(B/\txi), \mM_{\spf (B/\txi)}\big)} \Gamma(\spf B, \mM_{\spf B})
    \]
    and the natural projection $\beta\colon M_B \ra P_\infty$ is a torsor under the group $1+(\txi) \subseteq B^\times$. The same argument in \cite[Lemma 1.9(2)]{Tian_prismatic_etale_comparison} shows that the structure map $(\Ainf,\txi) \ra (B,\txi)$ extends uniquely to a morphism between pre-log rings
    \[
        \iota\colon (\Ainf,\txi,N_\infty) \ra (B,\txi,M_B).
    \]
By assumption, both $N$ and $P$ are free. Assume that $P\cong \oplus_{i=1}^m\N e_i$ and $N \cong \oplus_{j=1}^n \N e'_j$. Since $u\colon N \ra P$ is quasi-saturated and the cokernel of $u^{\gp}$ is torsion free, by \cite[Proposition 3.7]{log1}, there are nonempty finite sets $A_j \subseteq \{1,2,\cdots,m\}$ for any $1\le j\le n$ such that $u(e_j')= \prod_{i\in A_j} e_i$.
Let $m_i\in M_B$ be a lifting of $e_i$ under the projection $\beta \colon M_B \ra P_\infty$, for all $1\le i \le m$. Then for any $1\le j \le n$, $\prod_{i\in A_j} m_i$ and $\iota(e'_j)$ have the same image in $P_\infty$ under $\beta$. This means that there exist $u_j \in 1+(\txi) \subseteq B^\times$ such that $\iota(e'_j)= u_j \prod_{i\in A_j} m_i$ for $1\le j \le n$. 

Since $u \colon N \ra P$ is injective, we must have an inequality
\begin{equation} \label{eq: inequality of cardinality}
|\bigcup_{j=1}^n A_j| \ge n
\end{equation}
on cardinality. Inspired by the inequality (\ref{eq: inequality of cardinality}), we claim that one can choose $a_j \in A_j$ for each $1\le j \le n$ such that $a_j \neq a_k$ whenever $j \neq k$. We prove this by an induction on $n$. The claim is trivial when $n=1$. In general, we may assume that $|A_n|= \min_{1\le j \le n} |A_j|$. Since $u \colon N \ra P$ is injective, we still have $|\bigcup_{j=1}^{n-1} A_j| \ge n-1$. Consider the following two situations:
\begin{enumerate}
\item $|\bigcup_{j=1}^{n-1} A_j| = n-1$. In this case, the set $(\bigcup_{j=1}^{n} A_j)\minus (\bigcup_{j=1}^{n-1} A_j)$ is non-empty. By induction, we can choose $a_j\in A_j$ for $1\le j\le n-1$ such that $a_j \neq a_k$ whenever $j \neq k$. Then we can choose $a_n \in (\bigcup_{j=1}^{n} A_j)\minus (\bigcup_{j=1}^{n-1} A_j) \subseteq A_n$.
\item $|\bigcup_{j=1}^{n-1} A_j| \ge n$. Choose any $a_n \in A_n$. Let $A_j'= (A_j\cup \{a_n\})\minus \{a_n\} \subseteq A_j$. Then $|\bigcup_{j=1}^{n-1} A_j'| \ge n-1$. Moreover, the assumption $|A_n|= \min_{1\le j \le n} |A_j|$ implies that $A_j' \neq \emptyset$ for all $1\le j\le n-1$. Applying the induction on $N'= \oplus_{j=1}^{n-1} \N e'_j$, we can choose $a_j\in A_j'\subseteq A_j$ such that $a_j \neq a_k$ whenever $j \neq k$. Moreover, $a_j \neq a_n$ for any $j\neq n$.
\end{enumerate}
This finishes the proof of the claim. 
    
Given this, we can proceed as in the proof of \cite[Lemma 1.9]{Tian_prismatic_etale_comparison}. By replacing $m_{a_j}$ by $u_im_{a_j}$, we may assume that $\iota(e'_j)=  \prod_{i\in A_j} m_i$ for $1\le j \le n$. Consider $s\colon P_\infty \ra M_B$ given by
    \[
        s\left(x\sqcup\sum_{i=1}^{m} b_ie_i\right) := \iota(x) + \sum_{i=1}^{m} b_im_i, 
    \]
which is a section of $\beta\colon M_B \ra P_\infty$. We then define $\alpha_B$ as the composition
    \[
        \alpha_B \colon P_\infty \xrightarrow[]{s} M_B \ra \Gamma\big(\spf B, \mM_{\spf B}\big).
    \]
Since $\beta\colon M_B \ra P_\infty$ is a torsor under $1+(\txi) \subseteq B^\times$, we know that $\alpha_B$ is indeed a chart for $(\spf B, \mM_{\spf B})$. The commutativity of the diagram follows from the construction.
\end{proof}

\begin{proposition}\label{prop: weakly final obj: free chart}
The log prism $\fD(R)= (A(R), \txi, P_\infty)^a$ is a cover of the final object of the log prismatic site $\big((R^{(1)},P_\infty)/(\ul\Ainf, \txi)\big)_\Prism$. 
\end{proposition}

\begin{proof}
    Notice that there is a morphism of pre-log prisms 
    \[
        (A(R), \txi, P_\infty) \ra (\Ainf(R_\infty), \txi, P_\infty),
    \]
    where the pre-log structure on $\Ainf(R_\infty)$ is given by 
    \[
        P_\infty \ra A(R) \ra \Ainf(R_\infty).
    \]
    We are left to check that $(\Ainf(R_\infty), \txi, P_\infty)^a$ covers the final object of $((R^{(1)},P_\infty)/(\ul\Ainf, \txi))_\Prism$. 
    
    For any $(B, \txi, \mM_B)\in ((R^{(1)},P_\infty)/(\ul\Ainf, \txi))_\Prism$, we have to find a cover of $(B, \txi, \mM_B)$ that admits a morphism (of log prisms) from $(\Ainf(R_\infty), \txi, P_\infty)^a$. To this end, we first construct a non-log prism $(C,\txi)$ that covers $(B,\txi)$ and then equip $(C,\txi)$ with a suitable log structure. Consider the quasi-syntomic cover $B/\txi \ra (B/\txi)\htimes_{R^{(1)}} R_\infty^{(1)}$. By \cite[Proposition 7.11]{BS}, there exists a (non-log) prism $(C, \txi)$ over $\spf R^{(1)}$ that covers $(B, \txi)$ such that the structure map $R^{(1)} \ra C/\txi$ factors through $R^{(1)} \ra (B/\txi)\htimes_{R^{(1)}} R_\infty^{(1)}$. 
    In particular, $C/\txi$ admits a map from $R_\infty^{(1)}$, and \cite[Lemma 4.8]{BS} implies that the map lifts to a map $(\Ainf(R_\infty), \txi) \ra (C,\txi)$ of prisms. By Lemma \ref{Lem: Tian Lemma 1.9}, there is a chart $M_B= P_\infty$ on $B$ such that $(B,\txi,P_\infty)^a = (B, \txi, \mM_B)$. Let $(C, \txi, P_\infty)$ be the pre-log prism with
    the pullback pre-log structure $P_\infty$ from $(B, \txi, P_\infty)$. Then $(B, \txi, P_\infty)^a \ra (C, \txi, P_\infty)^a$ is a cover by definition. The morphism $(\Ainf(R_\infty), \txi) \ra (C,\txi)$ naturally extends to 
    \[
        (\Ainf(R_\infty), \txi, P_\infty)^a \ra (C,\txi,P_\infty)^a,
    \]
    which completes the proof.
    \end{proof}

Now, from a module with stratification, one can construct a generalized representation. More precisely, suppose that $(N, \varepsilon)\in \mathrm{Strat}(\mD(R)(\bullet))$ is a module with stratification with respect to $\mD(R)(\bullet)$. For every $\gamma\in \Gamma$, the base change of \[\varepsilon\colon N\otimes_{A(R), p_2}\mD(R)(1)\xrightarrow[]{\sim} N\otimes_{A(R), p_1}\mD(R)(1)\] along the map $\mD(R)(1)\xrightarrow[]{(1, \gamma)} \mD(R)(1)\xrightarrow[]{\Delta} A(R)$ yields an isomorphism \[N\otimes_{A(R), \gamma} A(R)\xrightarrow[]{\sim} N\] which defines an action of $\gamma$ on $N$. Using conditions (1) and (2) in Definition \ref{defn: stratification}, one checks that the actions of different $\gamma$'s are compatible, and hence defines a semilinear action of $\Gamma$ on $N$. Moreover, since the action of $\Gamma$ on $\mD(1)/\mu$ is the identity, the induced action of $\Gamma$ on $N/\mu$ is also the identity. Thus, we obtain a functor
\[
   \mathrm{ev}^{\mathrm{Strat}}_{A(R)}\colon \mathrm{Strat}(\mD(R)(\bullet))\rightarrow \Rep^{\mu}_{\Gamma}(A(R)).
\] 
The following diagram is commutative up to canonical isomorphism (cf. \cite[Lemma 3.16]{MT}):
\[
 \begin{tikzcd}
 \CR\big(\big(R^{(1)}/(\Ainf, \txi)\big)_{\Prism}\big)  \arrow[rr, "\mathrm{ev}_{\fD(R)(\bullet)}"] \arrow[drr, "\mathrm{ev}_{A(R)}"'] & & \mathrm{Strat}(\mD(R)(\bullet)) \arrow[d, "\mathrm{ev}^{\mathrm{Strat}}_{A(R)}"]\\
& &\Rep^{\mu}_{\Gamma}(A(R)).
\end{tikzcd}
\]

\begin{proposition}\label{prop: stratifications are convergent}
An element in the image of $\mathrm{ev}^{\mathrm{Strat}}_{A(R)}$ must be convergent.
\end{proposition}

\begin{proof}
Let $(N, \varepsilon)\in \mathrm{Strat}(\mD(R)(\bullet))$ and let $\nabla_N$ be the log $q$-connection associated with the generalized representation $\mathrm{ev}^{\mathrm{Strat}}_{A(R)}(N, \varepsilon)$. We want to show that $\nabla_N$ is $(p,\mu)$-adically convergent. Let $\overline{N}:=N/\mu$ and let $\nabla_{\overline{N}}$ be the induced log connection on $\overline{N}$. It suffices to show that $\nabla_{\overline{N}}$ is $p$-adically convergent.

To simplify the notation, let $L:=N\otimes_{A(R), p_1} \mD(R)(1)$ and let $\overline{L}:=L/\mu$. Equip $L$ with a $\Gamma$-action where $\Gamma$ acts on $\mD(R)(1)$ via $\Gamma=\Gamma\times\{1\}\subset \Gamma^2$. Then $L$ is equipped with a natural log $q$-connection $\nabla_L\colon L\rightarrow L\otimes_{A(R)} q\Omega^1_{A(R)/\Ainf}$ given by
\[l\mapsto \sum_{i=1}^d\frac{1}{\mu}(\gamma_i(l)-l)\otimes d\log U_i.\]
We have $L^{\Gamma}=L^{\nabla_L}$. Identify $N$ as an $A(R)$-submodule via $N\hookrightarrow N\otimes_{A(R), p_2} \mD(R)(1)$ sending $n\mapsto n\otimes 1$. By the construction of $\mathrm{ev}^{\mathrm{Strat}}_{A(R)}$, we see that $\varepsilon (N)\subset L^{\Gamma}=L^{\nabla_L}$. Modulo $\mu$, we write $\overline{A(R)}:=A(R)/\mu$ and $\overline{\Ainf}:=\Ainf/\mu$ and we obtain $\bar{\varepsilon}(\overline{N})\subset \overline{L}^{\nabla_{\overline{L}}}$ where $\nabla_{\overline{L}}: \overline{L}\rightarrow\overline{L}\otimes_{\overline{A(R)}}\Omega^1_{\overline{A(R)}/\overline{\Ainf}}$ is the induced log connection on $\overline{L}$.

On the other hand, for each $i=1, \ldots, d$, let $\tau_i:= p_2(U_i)-p_1(U_i)\in A(R)(1)$ and let $\bar{\tau}_i\in A(R)(1)/\mu$. Then $\mD(R)(1)/\mu$ is precisely the $p$-adic completion of the PD polynomial over $\overline{A(R)}$ in variables $\bar{\tau}_1, \ldots, \bar{\tau}_d$. In particular, each $l\in \overline{L}$ can be uniquely written as
\[
    l=\sum_{\underline{j}\in \Z_{\ge 0}^d} n_{\underline{j}}\otimes \bar{\tau}^{[\underline{j}]},
\]
where $n_{\underline{j}}\in \overline{N}$, $p$-adically converging to $0$ as $|\underline{j}|\rightarrow \infty$, and $\bar{\tau}^{[\underline{j}]}$ stands for the usual divided power $\bar{\tau}_1^{[j_1]}\cdots \bar{\tau}_d^{[j_d]}$. Notice that 
\[(\gamma_i, 1)(\tau_i)=p_2(U_i)-[\varepsilon] p_1(U_i)=\tau_i-\mu p_1(U_i)\]
and $(\gamma_i, 1)(\tau_j)=\tau_j$ if $i\neq j$. 
Let $\nabla_{\overline L,i}$ and $\nabla_{\overline N,i}$ for $1\le i \le d$, be the logarithmic coordinates of the connections $\nabla_{\overline L}$
and $\nabla_{\overline N}$ respectively. We obtain
\[\nabla^{\log}_{\overline{L},i}(l) =\sum_{\underline{j}\in \Z_{\ge 0}^d} \nabla_{\overline{N}, i}^{\log} (n_{\underline{j}})\otimes \bar{\tau}^{[\underline{j}]}-\sum_{\underline{j}\in \Z_{\ge 0}^d} p_1(U_i) n_{\underline{j}}\otimes \bar{\tau}^{[\underline{j}-1_i]}.\]
where $1_i=(0,\ldots,0,1,0,\ldots,0)$ denotes the multi-index with a single $1$ in the $i$th-place, Now, if $n\in \overline{N}$ and $l=\bar{\varepsilon}(n)$, we have $\nabla^{\log}_{\overline{L},i}(l)=0$ for all $i$. Direct computation shows that \[n_{\underline{j}}= \prod_{i=1}^d \nabla_{\overline{N},i}^{\log}(\nabla_{\overline{N},i}^{\log}-1)\cdots (\nabla_{\overline{N},i}^{\log}-(j_i-1))(n)\]
for all $\underline{j}=(j_1, \ldots, j_d)$. Since $n_{\underline{j}}$ converge $p$-adically as $|\underline{j}|\rightarrow \infty$, we conclude that $\nabla_{\overline{N}}$ is $p$-adically convergent, as desired.

\end{proof}

Conversely, from a convergent generalized representation in $\Rep^{\mu}_{\Gamma, \mathrm{conv}}(A(R))$, we are able to construct a stratification. We need the following lemma.

\begin{lemma}\label{lemma: MT 3.17}
Let $N\in \Rep^{\mu}_{\Gamma, \mathrm{conv}}(A(R))$. Then the composition
\[(N\otimes_{A(R), p_1}\mD(R)(1))^{\Gamma}\hookrightarrow N\otimes_{A(R), p_1}\mD(R)(1)\rightarrow N\]
is an isomorphism, where the second map is induced by $\Delta\colon \mD(R)(1)\rightarrow A(R)$.
\end{lemma}

\begin{proof}
The proof is similar to \cite[Lemma 3.17]{MT}. Let $\overline{A(R)}, \overline{N}, \nabla_N, \nabla_{\overline{N}}, L, \overline{L}, \nabla_L$, and $\nabla_{\overline{L}}$ be as above. We want to show that $L^{\nabla_L}\rightarrow N$ is an isomorphism. Modulo $\mu$, we obtain a map $\Delta_N\colon \overline{L}^{\nabla_{\overline{L}}=0}\rightarrow \overline{N}$. By the same argument as in \emph{loc. cit.}, it suffices to prove that $\Delta_N$ is an isomorphism and that the log de Rham complex $\overline{L}\otimes_{\overline{A(R)}}\Omega^{\bullet}_{\overline{A(R)}/\overline{\Ainf}}$ is acyclic in the positive degrees.

As in the proof of Proposition \ref{prop: stratifications are convergent}, for every $l\in \overline{L}$, there is a unique way to write
\[l=\sum_{\underline{j}\in \Z_{\ge 0}^d} n_{\underline{j}}\otimes \bar{\tau}^{[\underline{j}]}\]
where $n_{\underline{j}}\in \overline{N}$, $p$-adically converging to $0$ as $|\underline{j}|\rightarrow \infty$. Also recall that
\[\nabla^{\log}_{\overline{L},i}(l) =\sum_{\underline{j}\in \Z_{\ge 0}^d} \nabla_{\overline{N}, i}^{\log} (n_{\underline{j}})\otimes \bar{\tau}^{[\underline{j}]}-\sum_{\underline{j}\in \Z_{\ge 0}^d} p_1(U_i) n_{\underline{j}}\otimes \bar{\tau}^{[\underline{j}-1_i]}\]
for every $i=1, \ldots, d$. If $l\in \overline{L}^{\Gamma}=\overline{L}^{\nabla_{\overline{L}}=0}$, we have 
\[n_{\underline{j}}= \prod_{i=1}^d \nabla_{\overline{N},i}^{\log}(\nabla_{\overline{N},i}^{\log}-1)\cdots (\nabla_{\overline{N},i}^{\log}-(j_i-1))(n_{\underline{0}})\]
for all $\underline{j}=(j_1, \ldots, j_d)$. The map $\Delta_N\colon \overline{L}^{\nabla_{\overline{L}}=0}\rightarrow \overline{N}$ sends $l=\sum_{\underline{j}\in \Z_{\ge 0}^d} n_{\underline{j}}\otimes \bar{\tau}^{[\underline{j}]}$ to $n_{\underline{0}}$ with an inverse given by $n\mapsto \sum_{\underline{j}\in \Z_{\ge 0}^d} n_{\underline{j}}\otimes \bar{\tau}^{[\underline{j}]}$
where \[n_{\underline{j}}= \prod_{i=1}^d \nabla_{\overline{N},i}^{\log}(\nabla_{\overline{N},i}^{\log}-1)\cdots (\nabla_{\overline{N},i}^{\log}-(j_i-1))(n).\]
The last term is well-defined because $\nabla_N$ is convergent. 

The proof of acyclicity of $\overline{L}\otimes_{\overline{A(R)}}\Omega^{\bullet}_{\overline{A(R)}/\overline{\Ainf}}$ is the same as in \emph{loc. cit.}
\end{proof}

\begin{proposition}\label{representation vs stractification}
Let $N\in \Rep^{\mu}_{\Gamma, \mathrm{conv}}(A(R))$. The inverse of the isomorphism in Lemma \ref{lemma: MT 3.17} induces an isomorphism
\[\varepsilon\colon N\otimes_{A(R), p_2} \mD(R)(1)\xrightarrow[]{\sim} N\otimes_{A(R), p_1}\mD(R)(1).\]
Then $\varepsilon$ is a stratification. Consequently, we arrive at a functor 
\[
    \mathrm{Strat}_{\mD(R)(\bullet)}\colon \Rep^{\mu}_{\Gamma, \mathrm{conv}}(A(R))\rightarrow \mathrm{Strat}(\mD(R)(\bullet))
\]
sending $N\mapsto (N, \varepsilon)$.
\end{proposition}

\begin{proof}
The proof of \cite[Proposition 3.18 (i)]{MT} applies here verbatim. More precisely, to see the isomorphism $\varepsilon$ is a stratification, we need to show the homomorphism $N\otimes_{A(R),q_0}\mD(R)(2) \ra N$ induced by $\Delta\colon \mD(R)(2) \ra A(R)$ restricts to an injection
\[
    \big(N\otimes_{A(R),q_0}\mD(R)(2)\big)^{\Gamma\times\Gamma\times\{1\}} \hookrightarrow N.
\]
Then we are in the same situation as \cite[Proposition 3.20]{MT} and the same proof applies here.
\end{proof}

We can now finish the proof of Theorem \ref{thm: fully faithful}.

\begin{proof}[Proof of Theorem \ref{thm: fully faithful}]
    With the preparations above, the proofs of \cite[Proposition 3.18 (ii)]{MT} and \cite[Theorem 3.2]{MT} apply here verbatim.
\end{proof}

Finally, we take the Frobenius actions into consideration. 
\begin{definition}
    Let $(A,I,M_A)$ be a bounded pre-log prism and $\fX$ be a saturated $p$-adic log formal scheme over $(A/I, M_A)$. A \emph{prismatic $F$-crystal} $(\mF, \vp_{\mF})$ on the log prismatic site $(\fX/(A,I,M_A))_{\Prism}$ is a locally free log crystal $\mF \in \CR\big((\fX/(A,I,M_A))_{\Prism}\big)$ equipped with an isomorphism of sheaves of $\mO_\Prism[I^{-1}]$-modules
    \[
        \vp_{\mF} \colon \varphi^*\mF [\frac 1 I] \simra \mF [\frac 1 I].
    \]
Let $\CR^{\varphi}\big((\fX/(A,I,M_A))_{\Prism}\big)$ denote the category of prismatic $F$-crystals on $(\fX/(A,I,M_A))_{\Prism}$. When $\fX= \spf \ul R$ is affine, we also use the notation $\CR^{\varphi}\big((R/(A,I,M_A))_{\Prism}\big)$.
\end{definition}

Let $\fU=\spf \ul R$ be a small affine \'etale open as before. There is a commutative diagram
\[
 \begin{tikzcd}
 \Vect^{\vp}\big(\big(R^{(1)}/(\ul\Ainf, \txi)\big)_{\Prism}\big)  \arrow[r, "\mathrm{ev}_{A(R)}"] \arrow[dr, "\mathrm{ev}_{\Ainf(R_{\infty})}"'] & \Rep^{\mu}_{\Gamma}(A(R), \varphi) \arrow[d, "\cong"]\\
& \Rep^{\mu}_{\Gamma}(\Ainf(R_{\infty}), \varphi).
\end{tikzcd}
\]

\begin{corollary}\label{cor: equivalence between generalized representations and prismatic $F$-crystals}
Assume that $P$ is free, then there are equivalences of categories
\[
    \Vect^{\vp}\big(\big(R^{(1)}/(\ul\Ainf, \txi)\big)_{\Prism}\big) \xrightarrow[]{\sim}\Rep^{\mu}_{\Gamma}(A(R), \varphi) \xrightarrow[]{\sim}\Rep^{\mu}_{\Gamma}(\Ainf(R_{\infty}), \varphi)\xrightarrow[]{\sim}  \mathrm{BKF}^{\log}(\mathfrak{U}, \varphi).
\]
\end{corollary}

\begin{proof}
This follows immediately from Theorem \ref{thm: fully faithful} and Proposition \ref{prop: automatically convergent}.
\end{proof}

Finally, we consider the global situation. Let $\fX$ be an admissibly smooth $p$-adic log formal scheme over $\underline{\mO}$. Consider the morphism
\[\varphi: \ul{\mO}=(\mO, N_{\infty})\rightarrow \ul{\Ainf/\txi}=(\Ainf/\txi, N_{\infty})\]
of pre-log rings induced by the isomorphism $\mO=\Ainf/\xi\xrightarrow[]{\varphi} \Ainf/\txi$ on the rings and the multiplication-by-$p$ map on the monoids. Define
\[
    \fX^{(1)}= \fX\times_{\spf(\ul\mO)^a,\vp}\spf(\ul{\Ainf/\txi})^a
\]
to be the base change of $\fX$ along $\varphi$. If $\fX=\spf \ul R$ is affine small, then $\fX^{(1)}$ is nothing but $\spf \ul{R^{(1)}}$.

\begin{definition}\label{defn: locally admits free charts}
We say that $\fX$ \emph{locally admits free charts} if $\fX$ is covered by small affine \'etale open neighborhoods $\{\fU_\alpha= \spf R_\alpha\}$ modeled on small charts $u_{\alpha}: N_{\alpha}\rightarrow P_{\alpha}$ (as in Definition \ref{definition: small affine}) such that both $N_{\alpha}$ and $P_{\alpha}$ are free.
\end{definition}

Using the same approach as in \cite[Theorem 5.15]{MT}, we have the following:
\begin{theorem}\label{thm: MT thm 5.15}
Let $\fX$ be an admissibly smooth $p$-adic log formal scheme over $\underline{\mathcal{O}}$ that locally admits free charts. Then there is a natural fully faithful functor
    \[
        \Vect\big(\big(\fX^{(1)}/(\ul\Ainf, \txi)\big)_{\Prism}\big)\ra\mathrm{BKF}^{\log}(\fX)
    \]
whose restriction on any small affine \'etale open $\fU=\spf \ul R \hookrightarrow \fX$ (modeled on a small chart $u:N\rightarrow P$ with both $N$ and $P$ free) coincides with  
    \begin{equation}\label{eq: prismatic crystal to BKF: local}
        \Vect\big(\big(R^{(1)}/(\ul\Ainf, \txi)\big)_{\Prism}\big)\xrightarrow[]{\ev_{A(R)}} \Rep_{\Gamma}^\mu(A(R)) \cong \mathrm{BKF}^{\log}(\fU).
    \end{equation}
\end{theorem}
\begin{proof}
The proof is similar to that of \cite[Theorem 5.15]{MT}.\footnote{Alternatively, one may use the stackiness of log prismatic crystals to glue the local functors as in the proof of \cite[Theorem 5.17(ii)]{MT}. We leave this to the interested reader.} Let $X$ denote the adic generic fiber of $\fX$. Let $\fB$ be the collection of log affinoid perfectoid objects $V\in X_{\proket}$ such that $V\ra X$ factors through $\spa(R[p^{-1}],R)$ for some small affine \'etale open $\spf R \subseteq \fX$ modeled on small charts $N\rightarrow P$ with $N$ and $P$ free. (For simplicity, in what follows, we will refer to such $\spf R$ as a \emph{small affine \'etale open with a free chart}.) We claim that $\fB$ forms a basis of $X_{\proket}$. Firstly, log affinoid perfectoid objects form a basis of $X_{\proket}$ according to \cite[Proposition 7.35]{log1}. Secondly, by assumption, $\fX$ is covered by small affine \'etale open neighborhoods $\{\fU_\alpha= \spf R_\alpha\}$ with free charts. Hence, each log affinoid perfectoid object $V\in X_{\proket}$ admits a covering by $V\times_X\spa(R_\alpha[p^{-1}],R_\alpha)$; these are elements in $\fB$ by \cite[Corollary 7.33]{log1}.

For any $\mF \in \Vect\big(\big(\fX^{(1)}/(\ul\Ainf, \txi)\big)_{\Prism}\big)$, consider a presheaf $\mF_{\BKF}^{\pre}$ on $\fB$ given by 
\[
\mF_{\BKF}^{\pre}(V):= \Gamma\big((\spf(\Ainf(A)) \leftarrow \spf A^{(1)} \ra\fX^{(1)}),\mF\big).
\]
where $V\in \fB$ and $ A= \Gamma(V,\widehat{\mO}_X^+)$. Let $\mF_{\BKF}$ be the sheaf of $\Ainfx$-modules on $X_{\proket}$ associated with the presheaf $\mF_{\BKF}^{\pre}$. The rule $\mF \mapsto \mF_{\BKF}$ defines a functor 
\[
\Vect\big(\big(\fX^{(1)}/(\ul\Ainf, \txi)\big)_{\Prism}\big)\ra \Vect(\Ainfx),
\]
where $\Vect(\Ainfx)$ stands for the category of finite projective $\Ainfx$-modules on $X_{\proket}$.
We are left to check
\begin{enumerate}
\item $\mF_{\BKF} \in \mathrm{BKF}^{\log}(\fX)$;
\item for any small affine \'etale open $\spf R \subseteq \fX$ with a free chart, the diagram
\[
\begin{tikzcd}[column sep= 10ex]
\Vect\big(\big(\fX^{(1)}/(\ul\Ainf, \txi)\big)_{\Prism}\big) \arrow[r, "\mF \mapsto \mF_{\BKF}"] \arrow[d] & \BKF^{\log}(\fX) \arrow[d] \\
\Vect\big(\big(R^{(1)}/(\ul\Ainf, \txi)\big)_{\Prism}\big) \arrow[r, "(\ref{eq: prismatic crystal to BKF: local})"] &\mathrm{BKF}^{\log}(\spf R),
\end{tikzcd}
\]
commutes, where the vertical arrows are restriction functors.
\end{enumerate}
In fact, it suffices to check
     \begin{enumerate}
         \item[(2$'$)] for any small affine \'etale open $\spf R \subseteq \fX$ with a free chart, the diagram
         \[
         \begin{tikzcd}[column sep= 10ex]
             \Vect\big(\big(\fX^{(1)}/(\ul\Ainf, \txi)\big)_{\Prism}\big) \arrow[rr, "\mF \mapsto \mF_{\BKF}"] \arrow[d] & &\Vect(\Ainfx) \arrow[d] \\
             \Vect\big(\big(R^{(1)}/(\ul\Ainf, \txi)\big)_{\Prism}\big) \arrow[r, "(\ref{eq: prismatic crystal to BKF: local})"] &\mathrm{BKF}^{\log}(\spf R) \ar[r] & \Vect(\A_{\inf,\spf R}),
         \end{tikzcd}
         \]
        commutes, where the vertical arrows are the restriction functors.
     \end{enumerate}
Note that (1) follows from (2$'$) as the condition is \'etale local on $\fX$, and (2) follows from (1) and (2$'$). To prove (2$'$), we may assume that $\fX= \fU= \spf R$ is itself small affine with a free chart $u:N\rightarrow P$. Let $U=\spa(R[\frac{1}{p}], R)$ be the adic generic fiber of $\fU$ and let $U_{\infty}=\spa(R_{\infty}[\frac{1}{p}], R_{\infty})$ be the standard perfectoid Galois cover (cf. \S \ref{subsection: generalized repn}) with Galois group $\Gamma=\Hom(P^{\mathrm{gp}}/N^{\mathrm{gp}}, \widehat{\mathbb{Z}}(1))$. 

Let $\bM$ be the image of $\mF$ in $\BKF^{\log}(\fU)$ via (\ref{eq: prismatic crystal to BKF: local}). We shall construct a morphism $\mF_{\BKF} \ra \bM$. Let $M= \Gamma(U_{\infty},\bM)= \ev_{\Ainf(R_\infty)}(\mF)$. Notice that $M=\mF_{\BKF}^{\pre}(U_\infty)$ by construction. For any log affinoid perfectoid object $V\in U_{\proket}$, the fiber product $V_\infty= V\times_U U_\infty$ is also log affinoid perfectoid by \cite[Corollary 7.33]{log1} and it is a Galois cover of $V$ with Galois group $\Gamma$. On the one hand, by the crystal property of $\mF$, we have
     \[
         \mF_{\BKF}^{\pre}(V_\infty)= \mF_{\BKF}^{\pre}(U_\infty)\otimes_{\Ainf(R_\infty)} \Gamma(V_\infty, \AinfU)= M\otimes_{\Ainf(R_\infty)} \Gamma(V_\infty, \AinfU).
     \]
On the other hand, for any $s\ge 1$ we have 
     \[
         \Gamma(V,\bM/p^s)= (\Gamma(V_\infty, \bM/p^s))^\Gamma= \big(M\otimes_{\Ainf(R_\infty)}\Gamma(V_\infty, \AinfU)/p^s)\big)^\Gamma,
     \]
where the second equality is due to Lemma \ref{lemma: MT 5.13}(3). These give rise to a natural map
     \[
         f_V \colon \mF_{\BKF}^{\pre}(V) \ra \big(\mF_{\BKF}^{\pre}(V_\infty)\big)^\Gamma \ra \bM(V)
     \]
by passing to limit $s\to \infty$. Since this applies to all log affinoid perfectoid object $V$, we obtain the desired map $\mF_{\BKF}\rightarrow \bM$. 

Notice that $f_{U_\infty}$ is precisely the canonical isomorphism $\mF_{\BKF}^{\pre}(U_\infty)= M$. For any log affinoid perfectoid object $V$ lying over $U_{\infty}$, there is an isomorphism
    \[
        \mF_{\BKF}^{\pre}(U_\infty)\otimes_{\Gamma(U_{\infty},\AinfU)}\Gamma(V,\AinfU) \xrightarrow[]{\sim}\mF_{\BKF}^{\pre}(V) 
    \]
by the crystal property of $\mF$. We also have an isomorphism
    \[
        \bM(U_\infty)\otimes_{\Gamma(U_{\infty},\AinfU)}\Gamma(V,\AinfU) \xrightarrow[]{\sim}\bM(V)
    \]
by Lemma \ref{lemma: MT 5.13}(3). Therefore, $f_V$ is actually an isomorphism for any log affinoid perfectoid object $V$ lying over $U_{\infty}$; consequently, since such $V$'s form a basis, the map $\mF_{\BKF}\rightarrow \bM$ must be an isomorphism, as desired.

Finally, full faithful-ness can be checked locally, and hence follows from Theorem \ref{thm: fully faithful}.
\end{proof}

Taking Frobenius actions into account, we arrive at an equivalence of categories between relative log BKF modules and log prismatic $F$-crystals.

\begin{theorem}\label{thm: equivalence between BKF modules and prismatic $F$-crystals}
Let $\fX$ be an admissibly smooth $p$-adic log formal scheme over $\underline{\mO}$. Assume that $\fX$ locally admits free charts. Then there is a natural equivalence of categories
\[
    \Vect^{\vp}\big(\big(\fX^{(1)}/(\ul\Ainf, \txi)\big)_{\Prism}\big)\xrightarrow[]{\sim} \mathrm{BKF}^{\log}(\fX, \varphi),
\]
which coincides with the equivalence in Corollary \ref{cor: equivalence between generalized representations and prismatic $F$-crystals} when restricted to small affine \'etale opens.
\end{theorem}

\begin{proof}
This follows from Theorem \ref{thm: MT thm 5.15} and Corollary \ref{cor: equivalence between generalized representations and prismatic $F$-crystals}.
\end{proof}

\vspace{0.3in}
\section{Logarithmic $\Ainf$-cohomology with coefficients}\label{section: coefficients}
\noindent In this section, we establish the theory of logarithmic $\Ainf$-cohomology with coefficients (i.e., with coefficients in relative log BKF modules). This generalizes the logarithmic $\Ainf$-cohomology theory (with trivial coefficient) established in \cite{log1}. Our treatment here is parallel to the one in \cite[\S 6]{MT}.

Throughout \S \ref{section: coefficients}, let $\ul\mO= (\mO, \alpha\colon N_\infty \ra \mO)$ be a divisible perfectoid split log point (cf. \S \ref{subsection: defn of log BKF modules}). 

\begin{definition}\label{defn: log Ainf cohomology with coefficients}
Let $\fX$ be an admissibly smooth $p$-adic log formal scheme over $\ul \mO$ with adic generic fiber $X$, and let $\bM\in \mathrm{BKF}^{\log}(\fX, \varphi)$. We define
\[
A\Omega^{\log}_{\fX}(\bM):= L\eta_{\mu}(\widehat{R\nu_*\bM})\in D(\fX_{\ett}, \Ainf)
\]
where the completion is the derived $p$-adic completion, $\nu\colon X_{\proket} \ra \fX_{\ett}$ is the natural projection of sites, and $L\eta$ stands for the d\'ecalage functor. The \emph{log $\Ainf$-cohomology of $\bM$} is defined to be the complex
\[R\Gamma_{\Ainf}(\fX, \bM):=R\Gamma\big(\fX_{\ett}, A\Omega^{\log}_{\fX}(\bM)\big)\in D(\Ainf).\]
\end{definition}

Parallel to \cite[Theorem 6.2]{MT}, we shall construct specialization functors (see Definition \ref{defn: etale specialization}, \ref{defn: de Rham specialization}, \ref{defn: crystalline specialization})\footnote{In the rest of the paper, locally finite free $\widehat{\Z}_p$-sheaves are also referred to as \emph{$\widehat{\Z}_p$-local systems}.} \footnote{For the crystalline specialization functor, we need to assume that $\fX$ locally admits free charts (cf. Definition \ref{defn: locally admits free charts}).}
\begin{align*}
    \sigma^*_{\ett}\colon &\BKF^{\log}(\fX, \varphi)\rightarrow \big\{\textrm{locally finite free }\widehat{\Z}_p\textrm{-sheaves on }X_{\proket}\big\},\\
    \sigma^*_{\dR}\colon & \BKF^{\log}(\fX, \varphi)\rightarrow \big\{\textrm{vector bundles with integrable log connections on }\fX\big\},\\
    \sigma^*_{\crys}\colon & \BKF^{\log}(\fX, \varphi)\rightarrow \big\{\textrm{locally finite free $F$-crystals on }(\fX_k/W(\ul k))_{\crys}\big\},
\end{align*}
then we shall prove the following comparison isomorphisms.

\begin{theorem}[Theorem \ref{thm: etale comparison}, Theorem \ref{thm: de Rham comparison}, Corollary \ref{cor: crystalline comparison}] \label{thm: comparisons}
Let $\fX$ be an admissibly smooth $p$-adic log formal scheme over $\ul \mO$. For any $\bM\in \mathrm{BKF}^{\log}(\fX, \varphi)$, these specialization functors induce natural isomorphisms
\begin{align*}
\widehat{\Big(A\Omega^{\log}_{\fX}(\bM)[\frac{1}{\mu}]\Big)}^{\varphi_{\bM}=1}&\cong R\nu_*(\sigma^*_{\ett}\bM),\\
A\Omega^{\log}_{\fX}(\bM)\otimes^{\L}_{\Ainf, \theta} \mO&\cong \sigma^*_{\dR}\bM\otimes_{\mO_{\fX}} \Omega^{\bullet}_{\fX/\mO},\\
A\Omega^{\log}_{\fX}(\bM)\htimes^{\L}_{\Ainf}W(k) &\cong R\upsilon_*(\sigma^*_{\crys}\bM).
\end{align*}
of complexes on the \'etale site $\fX_{\ett}$, where $\upsilon: (\fX_k/W(\ul k))_{\crys}\rightarrow \fX_{k, \ett} \cong \fX_{\ett}$ is the natural projection of sites.\footnote{Again, in the last isomorphism, we need assume that $\fX$ locally admits free charts. See Corollary \ref{cor: crystalline comparison} for the precise statement.}
\end{theorem}

When $\fX$ is proper, by taking the global sections, we shall deduce the following comparison isomorphisms between various cohomology theories (cf. \S \ref{subsection: classical BKF modules}).

\begin{theorem}\label{thm: perfect complex}
Suppose $\fX$ is in addition proper, and locally admits free charts. Then $R\Gamma_{\Ainf}(\fX, \bM)$ is a perfect complex of $\Ainf$-modules equipped with a $\varphi$-semilinear map $\varphi: R\Gamma_{\Ainf}(\fX, \bM)\rightarrow R\Gamma_{\Ainf}(\fX, \bM)$ which induces an isomorphism
\[R\Gamma_{\Ainf}(\fX, \bM)[\frac{1}{\xi}]\cong R\Gamma_{\Ainf}(\fX, \bM)[\frac{1}{\txi}].\]
All cohomology groups of $R\Gamma_{\Ainf}(\fX, \bM)$ are Breuil--Kisin--Fargues modules in the sense of \cite[Definition 1.5]{BMS1}. Moreover, Theorem \ref{thm: comparisons} yields isomorphisms
    \begin{align*}
        R\Gamma_{\Ainf}(\fX, \bM)\otimes^{\L}_{\Ainf}\Ainf[\frac 1\mu] &\cong R\Gamma_{\ett}(X, \sigma^*_{\ett}\bM),\\
        R\Gamma_{\Ainf}(\fX, \bM)\otimes^{\L}_{\Ainf, \theta} \mO&\cong R\Gamma_{\dR}(\fX, \sigma^*_{\dR}\bM),\\
        R\Gamma_{\Ainf}(\fX, \bM)\otimes^{\L}_{\Ainf}W(k) &\cong R\Gamma_{\crys}\big(\fX_k/W(\ul k), \sigma^*_{\crys}\bM\big),
    \end{align*}
assuming for the first comparison that $C$ is algebraically closed.
\end{theorem}

\vspace{0.1in}
\subsection{\'Etale specialization}\label{subsection: etale specialization}
\noindent
\vspace{0.1in}

\noindent We first generalize \cite[Proposition 6.15]{MT} to our situation.

\begin{proposition}\label{Frobenius invariant}
Let $X$ be a locally noetherian saturated log adic space that is admissibly smooth over $\spa(C,C^+)_{N_\infty}$ (cf. \cite[Definition 6.5(2)]{log1}). Let $\bV$ be a sheaf of $\Ainfx$-modules on $X_{\proket}$ such that $\bV[\frac 1\xi]$ and $\bV[\frac 1\txi]$ are locally finite free sheaves of $\Ainf[\frac 1\xi]$-modules and $\Ainf[\frac 1\txi]$-modules, respectively.\footnote{Notice that $\bV[\frac 1\mu]$ is also locally free.} Let $\vp_{\bV} \colon \bV[\frac 1\xi] \ra \bV[\frac 1\txi]$ be a Frobenius semilinear isomorphism. Notice that $\vp_{\bV}$ extends to $\vp_{\bV} \colon \bV[\frac 1\mu] \ra \bV[\frac 1{\vp(\mu)}]$. We have 
    \begin{enumerate} 
        \item $\L :=  \big(\bV \otimes_{\Ainfx}  W(\widehat\mO_{X^\flat})\big)^{\vp_{\bV}=1} $ is a locally finite free $\widehat{\Z}_p$-sheaf on $X_{\proket}$ of the same rank as $\bV[\frac{1}{\xi}]$;
        \item $\L \subseteq \bV[\frac 1\mu]$, and the sequence 
        \[
            0 \lra \L \lra \bV[\frac 1\mu] \xrightarrow{1-\varphi_{\bV}^{-1}} \bV[\frac 1\mu] \lra 0 
        \]
        is exact (here, note that $\bV[\frac 1\mu] \subseteq \bV[\frac{1}{\vp(\mu)}]$, so $\vp_{\bV}^{-1}$ is well-defined on $\bV[\frac 1\mu]$); 
        \item the inclusion $\L \subseteq \bV[\frac 1\mu]$ induces an identification $\L \otimes_{\widehat{\Z}_p}\Ainfx[\frac 1\mu]= \bV[\frac 1\mu]$. 
    \end{enumerate}
\end{proposition}
\begin{proof}
    Let $\fB$ be the collection of log affinoid perfectoid objects $U\in X_{\proket}$.
    Take any $U= \varprojlim (\spa(A_i,A_i^+),\mM_i)\in \fB$, let $\widehat U= (\spa(A,A^+),\mM_{\widehat U})$ be the log perfectoid space associated with $U$. Let $\Vect(W(A^\flat))^{\vp=1}$ denote the category of finite projective modules $M$ over $W(A^\flat)$ equipped with a Frobenius isomorphism $\vp_M \colon \vp^* M \ra M$. 
    By \cite[Proposition 3.6, Example 3.4]{BS_crystal} and an induction argument, for any $n\ge 1$, we obtain an equivalence of categories
    \[
        \Vect(W(A^\flat)/p^n)^{\vp=1} \xrightarrow[]{\sim} \Loc_{\Z_p/p^n}(\spa (A,A^+)).
    \]
    By passing to the limits, we arrive at
    \[
        \Vect(W(A^\flat))^{\vp=1} \xrightarrow[]{\sim} \Loc_{\Z_p}(\spa (A,A^+))
    \]
    (also see \cite[Theorem 8.5.3]{Kedlaya-LiuI}).
    On the other hand, the equivalence of topoi $\widehat U_{\proet}^{\sim}  \cong X_{\proket/U}^{\sim}$ from \cite[Lemma 5.38]{DLLZ} induces an equivalence between categories of $\mathbb{Z}_p$-local systems
    \[
        \Loc_{\Z_p}(X_{\proket/U}) \cong \Loc_{\Z_p}(\widehat U_{\proet}) = \Loc_{\Z_p}(\spa (A,A^+)).
    \]
    Putting everything together, (1) follows from the equivalence of categories
    \[
        \Loc_{\Z_p}(X) \cong \lim_{U\in \fB} \Loc_{\Z_p}(X_{\proket/U}).
    \]
    
    The proofs of (2) and (3) follow the strategy in the proof of \cite[Proposition 6.15]{MT}; we present the proof in detail for completeness. Since the assertions are pro-Kummer-\'etale local, we can choose a log affinoid perfectoid object $U \in X_{\proket}$ with associated log perfectoid space $\widehat U = \spa(A,A^+)$ such that $\L$ is constant on $U$, and such that both $\bV[\frac 1\xi]$ and $\bV[\frac 1\txi]$ are free on $U$. Let $L= \Gamma(U,\L)$ and $V= \Gamma(U,\bV)$. Note that $V$ is equipped with a Frobenius semilinear isomorphism $\vp_V \colon V[\frac 1\xi] \ra V[\frac 1\txi]$. Let $M$ be a finite free $W(A^{\flat,+})$-lattice of $V[\frac 1\xi]$. 
    To prove (2), we need to show that $1-\vp_{V}^{-1} \colon M[\frac 1\mu] \ra M[\frac 1\mu]$ is surjective with kernel precisely $L$. Choose a sufficiently large integer $r$ such that $\txi^r M \subseteq \vp_{V}(M)$, then 
    \[
        \vp_{V}^{-1}\big(\frac 1{\mu^r} M\big) = \vp_{V}^{-1}\big(\frac {\txi^r}{\vp(\mu)^r} M\big) = \frac 1{\mu^r} \vp_{V}^{-1}(\txi^r M) \subseteq \frac 1{\mu^r} M,
    \]
    Thus, $\vp_{V}$ restricts to $\mu^{-r}M$. Replacing $M$ by $\mu^{-r}M$, we are reduced to showing that  $1-\vp_{V}^{-1} \colon M \ra M$ is surjective with kernel $L$. This follows from the proof of \cite[Proposition 6.15(ii)]{MT} using an almost-mathematical argument after mod $p$. 
    
    To prove (3), we keep the notation from the previous paragraph. Consider the finitely generated $W(A^{\flat+})$-module
    \[
        H= \Hom_{W(A^{\flat+})}\big(V, L \otimes_{\Z_p}W(A^{\flat+})\big).
    \]
    Then both $H[\frac 1\xi]$ and $H[\frac 1\txi]$ are finite free, and there is a Frobenius semilinear isomorphism $\vp_H \colon H[\frac 1\xi] \ra H[\frac 1\txi]$ induced by that on $V$. It extends to $\vp_H \colon H[\frac 1\mu] \ra H[\frac 1{\vp(\mu)}] \supseteq H[\frac 1\mu]$. Since the canonical map $\beta\colon L \otimes_{\Z_p}W(A^{\flat+})[\frac 1\mu] \ra V[\frac 1\mu]$ is invertible after base change to $W(A^\flat)$, we deduce that $\beta^{-1} \in \widehat H= H\otimes_{W(A^{\flat+})}W(A^{\flat})$. Since $\beta$ is Frobenius-invariant, so is $\beta^{-1}$. Therefore, $\beta\in \widehat{H}^{\vp_H^{-1} =1} \subseteq H[\frac 1\mu]$ by (2). Consequently, $\beta^{-1}$ induces a map from $V[\frac 1\mu]$ to $ L \otimes_{\Z_p}W(A^{\flat+})[\frac 1\mu]$, which completes the proof.
\end{proof}

\begin{corollary}
    \label{Frobenius invariant 2}
    Let $X$ be as above. Let $\bM$ be a locally finite free sheaf of $\Ainfx$-modules on $X_{\proket}$, equipped with an isomorphism 
    $\varphi_{\bM}\colon \varphi^*\bM[\txi^{-1}]\xrightarrow{\sim} \bM[\txi^{-1}]$. Then
    \begin{enumerate}
        \item $\L := \big(\bM\otimes_{\Ainfx}W(\widehat{\mO}_{X^\flat})\big)^{\varphi_\bM=1} $ is a locally finite free $\widehat{\Z}_p$-sheaf of the same rank as $\bM$;
        \item $\L \subseteq \bM[\frac 1\mu]$ and the sequence
        \[
            0 \lra \L \lra \bM[\frac 1\mu] \xrightarrow{1-\varphi_\bM^{-1}} \bM[\frac 1\mu] \lra 0 
        \]
        is exact;
        \item the inclusion $\L \subseteq \bM[\frac 1\mu]$ induces an identification $\L \otimes_{\widehat{\Z}_p}\Ainfx[\frac 1\mu]= \bM[\frac 1\mu]$. 
    \end{enumerate}
\end{corollary}

\begin{proof}
   Apply Proposition \ref{Frobenius invariant} to $\bV= \bM$.
\end{proof}

Motivated by this result, we can associate to any relative log Breuil--Kisin--Fargues module $(\bM, \varphi_{\bM}) \in \BKF^{\log}(\fX,\varphi)$ a $\widehat{\Z}_p$-local system defined by
\[
    \L = \big(\bM\otimes_{\Ainfx}W(\widehat{\mO}_{X^\flat})\big)^{\varphi_\bM=1}.
\]
\begin{definition}\label{defn: etale specialization}
   Let $\fX$ be an admissibly smooth $p$-adic log formal scheme over $\ul \mO$ with adic generic fiber $X$. Let $\Loc_{\Z_p}(X)$ denote the category of $\widehat{\mathbb{Z}}_p$-local systems (i.e., locally finite free $\widehat{\Z}_p$-sheaves) on $X_{\proket}$. The \emph{\'etale specialization functor}
    \[
       \sigma_{\ett}^*\colon \BKF^{\log}(\fX, \vp) \ra \Loc_{\Z_p}(X)
    \]
    is defined by 
    \[
       \sigma^*_{\ett}\bM:=  \big(\bM\otimes_{\Ainfx}W(\widehat{\mO}_{X^\flat})\big)^{\varphi_\bM=1}.
    \]
\end{definition}

Now we prove the comparison isomorphisms between \'etale cohomology and logarithmic $\Ainf$-cohomology with coefficients.

\begin{theorem}\label{thm: etale comparison}
    Let $\fX$ be an admissibly smooth $p$-adic log formal scheme over $\ul \mO$ with adic generic fiber $X$, and let $(\bM, \vp_\bM)\in \BKF^{\log}(\fX,\vp)$. There is a natural isomorphism
\[
    \widehat{\Big(A\Omega^{\log}_{\fX}(\bM)[\frac{1}{\mu}]\Big)}^{\varphi_{\bM}=1}\cong R\nu_*(\sigma^*_{\ett}\bM)
\]
of complexes of sheaves on $\fX_{\ett}$, where $\nu\colon X_{\proket} \ra X_{\ett}$ is the natural projection of sites.
\end{theorem}

\begin{proof}
Using the identity $\widehat{\Ainfx[\frac 1\mu]} =W(\widehat\mO_{X^\flat})$, we obtain isomorphisms
    \[
        A\Omega^{\log}_{\fX}(\bM)\htimes_{\Ainf}\Ainf[\frac{1}{\mu}] \cong (R\nu_*\bM)\htimes_{\Ainf}\Ainf[\frac{1}{\mu}] \cong R\nu_*\big(\bM \otimes_{\Ainfx}W(\widehat{\mO}_{X^\flat})\big).
    \]
which yields isomorphisms
    \[
         \widehat{\Big(A\Omega^{\log}_{\fX}(\bM)[\frac{1}{\mu}]\Big)}^{\varphi_{\bM}=1}= \Big(R\nu_*\big(\bM \otimes_{\Ainfx}W(\widehat{\mO}_{X^\flat})\big)\Big)^{\varphi_\bM=1}= R\nu_*(\sigma^*_{\ett}\bM),
    \]
as desired.
\end{proof}

Applying Corollary \ref{Frobenius invariant 2} (3) and the primitive comparison theorem (\cite[Theorem 8.3]{log1}), we deduce the main result:

\begin{theorem}\label{thm: strong etale comparison}
    Assume $C$ is algebraically closed. Let $\fX$ be a proper admissibly smooth $p$-adic log formal scheme over $\ul \mO= \ul{\mO_C}$ with adic generic fiber $X$, and let $(\bM, \vp_\bM)\in \BKF^{\log}(\fX,\vp)$. Then the identity $\L \otimes_{\widehat{\Z}_p}\Ainfx[\frac 1\mu]= \bM[\frac 1\mu]$ induces a natural isomorphism 
    \[
        R\Gamma_{\ket}(\sigma^*_{\ett}\bM) \otimes^{\L}_{\Z_p} \Ainf[\frac{1}{\mu}] \cong R\Gamma_{\Ainf}(\bM)\otimes^{\L}_{\Ainf}\Ainf[\frac{1}{\mu}].
    \]
\end{theorem}

\begin{proof}
    Firstly, we claim that the natural maps
    \[
        R\Gamma_{\ket}(\L/p^n)\otimes^{\L}_{\Z/p^n} W_n(\mO^\flat) \ra R\Gamma_{\proket}\big(X, \L \otimes_{\Z/p^n} W_n(\widehat{\mO}_{X^{\flat}}^+)\big)
    \]
    are almost isomorphisms for all $n\ge 1$. The claim follows by an induction on $n$, while the case $n = 1$ follows
    from the primitive comparison theorem (\cite[Theorem 8.3]{log1}). Taking derived inverse limits, we arrive at an almost isomorphism
    \begin{align*}
        R\Gamma_{\ket}(\L)\otimes^{\L}_{\Z_p} \Ainf & \cong^a R\Gamma_{\proket}\big(X, R\varprojlim_n\L \otimes_{\Z/p^n} W_n(\widehat{\mO}_{X^{\flat}}^+)\big) \\
        & = R\Gamma_{\proket}\big(X, \L\otimes_{\widehat{\Z}_p}\widehat{\Ainfx}\big) \\
        & = R\Gamma_{\proket}(X, \widehat{\bM}),
    \end{align*}
    where the last equality follows from Proposition \ref{Frobenius invariant} (3).
    Inverting $\mu$, we obtain
    \[
        R\Gamma_{\ket}(\L) \otimes^{\L}_{\Z_p} \Ainf[\frac{1}{\mu}] \cong R\Gamma(\fX_{\ett}, R\nu_*\widehat{\bM}[\frac{1}{\mu}]) \cong R\Gamma(\fX_{\ett}, A\Omega_{\fX}^{\log}(\bM))\otimes^{\L}_{\Ainf}\Ainf[\frac{1}{\mu}].
    \]
\end{proof}

\subsection{De Rham specialization}\label{subsection: de Rham specialization}

We begin by recalling the logarithmic $p$-adic Cartier isomorphism, which plays a fundamental role in our construction. Let $\fX$ be an admissibly smooth $p$-adic log formal scheme over $\ul \mO$ with adic generic fiber $X$. Let $\nu \colon X_{\proket} \ra X_{\ett}$ be the natural projection as above and we define
\[
    \widetilde{\Omega}_{\fX}^{\log} := L\eta_{(\xi_p-1)}R\nu_*\widehat{\mO}_{\fX}^+,
\]
which is a complex in $D(\fX_{\ett})$.
\begin{proposition}[{\cite[Theorem 9.5]{log1}}]\label{p adic Cartier}
Let $\fX$ be a saturated $p$-adic log formal scheme that is admissibly smooth over $\underline{\mO}$ .
Then for each $i \ge 0$, there is a canonical isomorphism of sheaves
\[
    \Omega_{\fX/\underline{\mO}}^{i, \ct} \xrightarrow[]{\sim} \mH^i(\widetilde{\Omega}^{\log}_\fX)\{i\},
\]
where $\Omega_{\fX/\underline{\mO}}^{i, \ct}$ denotes the sheaf of continuous log differentials, i.e., the $p$-adic completion of $\Omega_{\fX/\underline{\mO}}^i$, and $\{i\}$ denotes the Breuil–Kisin–Fargues twist.
\end{proposition}

Recall the site $\fX_{\ett,\mathrm{small}}$ (cf. \cite[Construction 10.3]{log1}) consisting of small $p$-adic formal affine \'etale opens $\fU = \spf R \rightarrow \fX$ (cf. Definition \ref{definition: small affine}), equipped with the \'etale topology. Let $\fX_{\ett}^{\psh}$ denote the presheaf topos on this site. There is a natural morphism of topoi
\[
    j= (j^{-1}, Rj_*)\colon \fX_{\ett} \ra \fX_{\ett}^{\psh},
\]
where $Rj_*$ is the forgetful functor and $j^{-1}$ is given by sheafification. We refer the reader to \cite[\S 10]{log1} for more details. For any relative log BKF module $\bM\in \BKF^{\log}(\fX)$, consider presheaves
\[
        A\Omega_{\fX}^{\log,\psh}(\bM):= L\eta_\mu {R\nu^{\psh}_*\widehat{\bM}}
\]
and
\[
     \widetilde{\Omega}^{\log,\psh}_{\fX}(\bM/\xi):= L\eta_{(\zeta_p-1)}R\nu^{\psh}_*(\bM/\xi),
\]
where $\nu^{\psh}= j\circ \nu \colon X_{\proket} \ra \fX^{\psh}_{\ett}$. Now consider a small affine $\fU = \spf(\underline{R})^a \in \fX_{\ett,\mathrm{small}}$ with its (log) adic generic fiber $U$. Then we have
\[
    R\Gamma\big(\fU, A\Omega_{\fX}^{\log,\psh}(\bM)\big) = A\Omega_{\underline{R}}^{\proket}(\bM):= L\eta_\mu R\Gamma_{\proket}(U, \widehat{\bM})
\]
and
\[
    R\Gamma\big(\fU, \widetilde{\Omega}^{\log, \psh}_{\fX}(\bM)\big) = \widetilde{\Omega}^{\proket}_{\underline{R}}(\bM/\xi) := L\eta_{(\zeta_p-1)}R\Gamma_{\proket}(U, \bM/\xi).
\]
Recall that $\fU$ comes equipped with small coordinates
\[
    \square \colon \underline{R^{\square}}= \big(P\sqcup_NN_{\infty}\rightarrow R^{\square}=\mO\langle P\rangle\widehat{\otimes}_{\mO\langle N\rangle}\mO\big) \ra \underline{R}= (P\sqcup_NN_{\infty}\rightarrow R).
\]
These coordinates yield a perfectoid pre-log ring
\[
    \underline{R_\infty}= \big(P_{\Q\ge 0} \ra R_\infty= R_\infty^\square\widehat{\otimes}_{R^\square}R\big).
\]
It is a Galois cover of $\underline{R}$ with Galois group $\Gamma= \Hom(P^{\gp}/N^{\gp}, \widehat{\Z}(1)) \cong \widehat{\Z}^d$, where $d= \mathrm{rk}_{\Z}(P^{\gp}/N^{\gp})$. Let $U_\infty$ be the associated log affinoid perfectoid object in $X_{\proket}$. By Proposition \ref{prop: MT 5.7}, the coordinates $\square$ induce almost isomorphisms
\[
    R\Gamma_{\cont}(\Gamma, \Gamma(U_\infty, \bM)) \simra R\Gamma_{\proket}(U, \widehat{\bM}),
\]
and
\[
    R\Gamma_{\cont}(\Gamma, \Gamma(U_\infty, \bM/\xi)) \simra R\Gamma_{\proket}(U, \bM/\xi).
\]
Applying the d\'ecalage functors $L\eta_\mu$ and $L\eta_{(\zeta_p-1)}$, respectively, we obtain maps
\[
    A\Omega^{\square,\gp}_{\underline{R}}(\bM) := L\eta_{\mu}R\Gamma_{\cont}(\Gamma, \Gamma(U_\infty, \bM)) \ra A\Omega_{\underline{R}}^{\proket}(\bM)
\]
and
\[
    \widetilde{\Omega}^{\square, \gp}_{\underline{R}}(\bM/\xi) := L\eta_{(\zeta_p-1)}R\Gamma_{\cont}(\Gamma, \Gamma(U_\infty, \bM/\xi)) \ra \widetilde{\Omega}^{\proket}_{\underline{R}}(\bM/\xi).
\]

\begin{lemma} \label{group coh vs presheaf coh}
The two natural maps $A\Omega^{\square,\gp}_{\underline{R}}(\bM) \ra A\Omega_{\underline{R}}^{\proket}(\bM)$ and $\widetilde{\Omega}^{\square, \gp}_{\underline{R}}(\bM/\xi) \ra \widetilde{\Omega}^{\proket}_{\underline{R}}(\bM/\xi)$ are isomorphisms.
\end{lemma}

\begin{proof}
By \cite[Lemma 10.10]{log1} and \cite[Lemma 8.11]{BMS1}, it suffices to prove the following two claims:
\begin{enumerate}
        \item The cohomology groups $H^i_{\cont}(\Gamma, \Gamma(U_\infty, \bM/\xi))$ and $H^i_{\cont}(\Gamma, \Gamma(U_\infty, \bM/\xi))/(\zeta_p-1)$ have no nontrivial almost zero element for each $i\ge 0$.
        \item For each $i\ge 0$, $H^i_{\cont}(\Gamma, \Gamma(U_\infty, \bM)/\mu)$ has no nontrivial $W(\fm^\flat)$-torsion.
\end{enumerate}
Since $\bM\in \BKF^{\log}(\fX)$ is trivial modulo $<\mu$, the module $\bM/\xi$ is trivial over $\Ainfx/\xi$. Moreover, by Theorem \ref{thm: MT 1.14}, $\Gamma(U_\infty, \bM)/\mu$ is trivial over $\Ainf(R_\infty)/\mu$. Therefore, we reduce to the case of the trivial BKF module $\bM= \Ainfx$ and then apply \cite[Lemma 9.21]{log1} and \cite[proposition 10.12]{log1}.
\end{proof}

\begin{lemma}\label{group coh base change}
The $\Gamma$-equivariant surjection $\widetilde{\theta} = \theta\circ \varphi^{-1} \colon \Ainf(R_\infty) \ra R_\infty$ induces an isomorphism
\[
    A\Omega^{\square,\gp}_{\underline{R}}(\bM)/\txi \lra \widetilde{\Omega}^{\square, \gp}_{\underline{R}}(\bM/\xi).
\]
\end{lemma}

\begin{proof}
    According to \cite[Proposition 10.10]{log1}, it is enough to demonstrate that for each $i \ge 0$, the cohomology group
    \[
        H^i\big(R\Gamma_{\cont}(\Gamma, \Gamma(U_\infty, \bM))\otimes^{\L}\Ainf(R_\infty)/\mu \big)=  H^i_{\cont}(\Gamma, \Gamma(U_\infty, \bM)/\mu)
    \]
    has no nontrivial $p$-torsion for each $i\ge 0$. By Theorem \ref{thm: MT 1.14}, $\Gamma(U_\infty, \bM)/\mu$ is trivial over $\Ainf(R_\infty)/\mu$. Thus, we again reduce to the trivial case $\bM = \Ainfx$ and invoke \cite[Proposition 10.12]{log1}.
\end{proof}

Combining Lemma \ref{group coh vs presheaf coh} and Lemma \ref{group coh base change}, we obtain the following immediate corollary.

\begin{corollary}\label{presheaf coh base change} 
Let $\fX$ be an admissibly smooth $p$-adic log formal scheme over $\underline{\mO}$ and let $\bM\in \BKF^{\log}(\fX)$.  The natural map 
    \[
        A\Omega_{\fX}^{\log,\psh}(\bM)/\txi \lra \widetilde{\Omega}^{\log,\psh}_{\fX}(\bM/\xi)
    \]
induced by $\widetilde{\theta}$ is an isomorphism.
\end{corollary}

\begin{proposition}\label{Hodge-Tate comparison}
    Let $\fX$ and $\bM$ be as above. The canonical base change map
    \[
        A\Omega^{\log}_{\fX}(\bM)/\txi \lra \widetilde{\Omega}^{\log}_{\fX}(\bM/\xi)= L\eta_{(\zeta_p-1)}R\nu_*(\bM/\xi)
    \]
    is an isomorphism.
\end{proposition}

\begin{proof}
By Corollary \ref{presheaf coh base change}, we only need to verify the following two isomorphisms:
    \begin{enumerate}
        \item $\widetilde{\Omega}^{\log,\psh}_{\fX}(\bM/\xi) \xrightarrow[]{\sim} Rj_* \widetilde{\Omega}^{\log}_{\fX}(\bM/\xi)$, and
        \item $A\Omega_{\fX}^{\log,\psh}(\bM) \xrightarrow[]{\sim} Rj_*A\Omega_{\fX}(\bM)$.
    \end{enumerate}
In other words, it suffices to check that the presheaves $A\Omega_{\fX}^{\log,\psh}(\bM)$ and $\widetilde{\Omega}^{\log,\psh}_{\fX}(\bM/\xi)$ are already sheaves.

The first isomorphism follows from a stalk computation argument identical to \cite[Proposition 9.16]{log1}. For the second, note that $L\eta_{\mu}$ preserves derived completeness in $D(\Ainf)$. Thus, $A\Omega_{\fX}^{\log,\psh}(\bM)$ is derived $\txi$-complete. Consequently, we are reduced to show
    \[
        A\Omega_{\fX}^{\log,\psh}(\bM)\otimes^{\L}\Ainfx/\txi \simra Rj_*j^{-1}\big(A\Omega_{\fX}^{\log,\psh}(\bM)\otimes^{\L}\Ainfx/\txi\big),
    \]
which follows from Corollary \ref{presheaf coh base change} and (1).
\end{proof}

Now, we are ready to construct the de Rham specialization functor and state the de Rham comparison isomorphism. By the $\mu$-smallness again, we have an isomorphism
\[
    L\eta_{(\zeta_p-1)}R\nu_*(\bM/\xi) \xrightarrow[]{\sim} L\eta_{(\zeta_p-1)}R\nu_*\big(\nu^{-1}\nu_*(\bM/\xi)\otimes_{\nu^{-1}\mO_\fX}\widehat{\mO}^+_X\big) = \nu_*(\bM/\xi)\otimes_{\mO_\fX}\widetilde{\Omega}^{\log}_\fX.
\]
Combining this with Proposition \ref{p adic Cartier} and Proposition \ref{Hodge-Tate comparison}, we obtain:

\begin{corollary}\label{cor: padic Cartier for BKF mod}
Let $\fX$ be an admissibly smooth $p$-adic log formal scheme over $\underline{\mO}$ and let $\bM\in \BKF^{\log}(\fX)$. For any $i\ge 0$, there is a natural isomorphism
    \[
        \mH^i\big(A\Omega_{\fX}^{\log}(\bM)/\txi\big)\{i\} \xrightarrow[]{\sim} \nu_*(\bM/\xi)\otimes_{\mO_\fX}\Omega^{i,\ct}_{\fX/\underline{\mO}}. 
    \]
\end{corollary}

The complex $A\Omega_{\fX}^{\log}(\bM)/\txi$ is a module over $A\Omega_{\fX}^{\log}/\txi$, where $A\Omega_{\fX}^{\log}= L\eta_{\mu}(\widehat{R\nu_*\bM})$. Consequently, its cohomology $\mH^i\big(A\Omega_{\fX}^{\log}(\bM)/\txi\big)$ acquires the structure of a graded module over
\[
    \mH^i(A\Omega_{\fX}^{\log}/\txi) \cong \Omega^{i,\ct}_{\fX/\underline{\mO}}.
\]
The Bockstein maps $\mathrm{Bock}_\xi: \mH^\bullet \rightarrow \mH^{\bullet +1}$ endow this with the structure of a log differential graded module. This yields an integrable log connection
\[
    \nabla_\xi \colon \nu_*(\bM/\xi) \ra \nu_*(\bM/\xi)\otimes_{\mO_\fX}\Omega^{1, \ct}_{\fX/\underline{\mO}}.
\]
\begin{definition}\label{defn: de Rham specialization}
     Let $\fX$ be an admissibly smooth $p$-adic log formal scheme over $\underline{\mO}$. Let $\MIC^{\log}(\fX)$ denote the category of vector bundles with integrable log connections on $\fX$.
    The \emph{de Rham specialization functor} is defined to be 
\begin{align*}
    \sigma^*_{\dR} \colon \BKF^{\log}(\fX) & \lra \MIC^{\log}(\fX)\\
    \bM & \longmapsto (\nu_*(\bM/\xi), \nabla_\xi).
\end{align*}
Combinging with the forgetful functor on Frobenius, we also obtain
\[
    \sigma^*_{\dR} \colon \BKF^{\log}(\fX,\vp)  \ra \MIC^{\log}(\fX)
\]
by a slight abuse of notation.
\end{definition}

\begin{theorem}\label{thm: de Rham comparison}
   Let $\fX$ be an admissibly smooth $p$-adic log formal scheme over $\underline{\mO}$. For any $\bM \in \BKF^{\log}(\fX)$, there is a natural isomorphism
    \[
        A\Omega^{\log}_{\fX}(\bM)\otimes_{\Ainf, \theta}^\L \mO \cong \sigma^*_\dR \bM \otimes_{\mO_\fX}\Omega^{\bullet, \ct}_{\fX/\underline{\mO}}.
    \]
Taking cohomology, we obtain an isomorphism
    \[
        R\Gamma_{\Ainf}(\fX,\bM)\otimes_{\Ainf}^{\L}\mO \cong R\Gamma_{\mathrm{logdR}}(\fX, \sigma_{\dR}^*\bM),
    \]
where $R\Gamma_{\mathrm{logdR}}(\fX, \sigma_{\dR}^*\bM)$ stands for the log de Rham cohomology of the log connection $\sigma_{\dR}^*\bM$.
\end{theorem}

\begin{proof} We have a chain of isomorphisms
    \begin{align*}
        A\Omega^{\log}_{\fX}(\bM)\otimes_{\Ainf, \theta}^{\L}\mO &= \big(L\eta_\mu \widehat{R\nu_*\bM}\big)\otimes^{\L}_{\Ainf} \Ainf/\xi \\
        & \xrightarrow[\sim]{\varphi} \big(L\eta_{\varphi(\mu)} \widehat{R\nu_*\bM}\big)\otimes^{\L}_{\Ainf} \Ainf/\txi\\  
        & \xrightarrow[]{\sim} \big(L\eta_{\txi} L\eta_\mu \widehat{R\nu_*\bM}\big)\otimes^{\L}_{\Ainf} \Ainf/\txi\\
        & \xrightarrow[]{\sim} \mH^{\bullet}\big(A\Omega^{\log}_{\fX}(\bM)\otimes_{\Ainf}^{\L}\Ainf/\txi\big)\\
        & \xrightarrow[]{\sim} \mH^{\bullet}\big(\widetilde{\Omega}^{\log}_{\fX}(\bM/\xi)\big)(\bullet)\\
        & \xrightarrow[]{\sim}\nu_*(\bM/\xi)\otimes_{\mO_{\fX}}\Omega_{\fX/\underline{\mO}}^{\bullet,\ct}= \sigma_\dR^*\bM \otimes_{\mO_{\fX}}\Omega_{\fX/\underline{\mO}}^{\bullet,\ct}.
    \end{align*}
    The third last isomorphism follows from \cite[Proposition 6.15]{BMS1}, the second last isomorphism follows from Proposition \ref{Hodge-Tate comparison}, and the last identification is the $p$-adic Cartier isomorphism (Proposition \ref{p adic Cartier}).
\end{proof}

\begin{corollary}\label{cor: perfect complex}
Let $\fX$ and $\bM$ be as above, and assume that $\fX$ is additionally proper. The log $\Ainf$-cohomology
    \[
        R\Gamma_{\Ainf}(\fX,\bM)= R\Gamma\big(\fX, \A\Omega_{\fX}^{\log}(\bM)\big)
    \]
is a perfect complex in $D(\Ainf)$.
\end{corollary}

\begin{proof}
Firstly, note that $A\Omega_{\fX}^{\log}(\bM)$ is derived $(p,\xi)$-adically complete. Indeed, since $Rj_*$ commutes with $\Rlim$, it suffices to check the completeness of $A\Omega_{\fX}^{\log,\psh}(\bM)$, which is due to the completeness of $\bM$ and the fact that $L\eta_{\mu}$ preserves completeness in $D(\Ainf)$ (cf. \cite[Lemma 6.20]{BMS1}).

By derived $(p,\xi)$-adic completeness, to show the perfectness of $R\Gamma_{\Ainf}(\fX,\bM)$, it suffices to show that $R\Gamma_{\Ainf}(\fX,\bM) \otimes^\mathbb{L}_{\Ainf} \mO$ is perfect. When $\fX$ is proper, the de Rham cohomology $R\Gamma_{\mathrm{logdR}}(\fX/\ul{\mO}, \sigma_{\dR}^*\bM)$ is a perfect complex in $D(\mO)$, and hence the desried statement follows from Theorem \ref{thm: de Rham comparison}.
\end{proof}

\vspace{0.1in}
\subsection{Crystalline specialization}\label{subsection: crys specailzation}

\begin{assumption}
In \S \ref{subsection: crys specailzation}-\ref{subsection: classical BKF modules}, we work under the assumption that $\fX$ locally admits free charts (cf. Definition \ref{defn: locally admits free charts}), namely, $\fX$ is locally modeled on small charts $u:N\rightarrow P$ where both $N$ and $P$ are free (cf. Definition \ref{definition: small affine}).
\end{assumption}

We now proceed to the construction of the crystalline specialization. Our approach employs \emph{log $q$-crystals} as intermediary objects. Firstly, let us recall the notion of the log $q$-crystalline site as introduced in \cite[Definition 7.5]{Koshikawa}.
Recall that a $\delta_{\log}$-triple $(D,I,M_D)$ is called a \emph{pre-log $q$-PD triple} if $(D,I)$ is a $q$-PD pair (cf. Definition \ref{def: log q-PD triple}). For each pre-log $q$-PD triple $(D,I,M_D)$, we can associate a \emph{log $q$-PD triple} $(D,I,\mM_D)=(D,I,M_D)^a$. In the extreme case $q=1$, such triples are referred to as \emph{pre-$\delta_{\log}$-PD triple} and \emph{$\delta_{\log}$-PD triple} (cf. \cite[Definition 6.4]{Koshikawa}).

\begin{definition}
    Fix a pre-log $q$-PD triple $(D,I,M_D)$ with integral $M_D$. Let $\fX=(\fX, \mM_{\fX})$ be a saturated $p$-adic log formal scheme over $(D/I, M_D)$. The \emph{log $q$-crystalline site}, which is denoted by $(\fX/(D,M_D))_{q\crys}$, is the opposite category of log $q$-PD triples $(E,J, \mM_{E})$ such that
    \begin{itemize}
        \item $(E, J, \mM_E)$ is the log $q$-PD triple associated with a pre-log $q$-PD triple $(E, J, M_E)$ over $(D,I,M_D)$ with $M_E$ being integral,
        \item it is equipped with a morphism $f\colon \spf(E/J) \ra \fX$ of $p$-adic formal schemes over $D/I$, and 
        \item $f$ induces an exact closed immersion 
        \[
            (\spf(E/J),f^*\mM_{\fX}) \ra (\spf E, \mM_E).
        \]
    \end{itemize}
    This site is endowed with the \'etale topology. 
    
    In the extreme case $q=1$, the resulting site is called the \emph{$\delta_{\log}$-crystalline site} (consisting of $\delta_{\log}$-PD triples satisfying a similar set of conditions as above), which is denoted by $(\fX/(D,M_D))_{\delta\crys}$ (cf. \cite[Definition 6.4, Remark 7.7]{Koshikawa}).
\end{definition}
     Let $\mO_{q\crys}$ denote the structure sheaf on $(\fX/(D,M_D))_{q\crys}$ sending $(E,J,\mM_E)$ to $E$. In analogy with \cite[\S 1]{chatzistamatiou}, we introduce the concept of log $q$-crystals.
     
\begin{definition}\label{defn: F-q-crystals}
    Let $(D,I,M_D)$ be a pre-log $q$-PD triple with integral $M_D$. Let $\fX$ be a saturated $p$-adic log formal scheme over $(D/I, M_D)$. 
    
\begin{enumerate}    
\item A \emph{(locally finite free) log $q$-crystal} on the log $q$-crystalline site $(\fX/(D,I,M_D))_{q\crys}$ is a sheaf of $\mO_{q\crys}$-module $\mF$ such that
    \begin{itemize}
        \item for every object $\mathfrak{E}= (E,J,\mM_E)$ of $(\fX/(D,I,M_D))_{q\crys}$, $\mF(\mathfrak{E})$ is a finite projective $E$-module, and 
        \item for every morphism $\mathfrak{E}= (E,J,\mM_{E})\rightarrow \mathfrak{E}'= (E',J,\mM_{E'})$ of $(\fX/(D,I,M_D))_{q\crys}$, the pullback homomorphism $\mF(\mathfrak{E})\otimes_{E}E' \ra \mF(\mathfrak{E}')$ is an isomorphism.
    \end{itemize}

\item Let $\mF$ be a log $q$-crystal as above. A \emph{Frobenius structure} $\vp_{\mF}$ on $\mF$ is an isomorphism of sheaves of $\mO_{q\crys}[I^{-1}]$-modules
    \[
        \vp_{\mF} \colon \mF \otimes_{\mO_{q\crys}, \vp}\mO_{q\crys}[\frac 1 I] \simra \mF[\frac 1 I].
    \]
A \emph{log $F$-{q}-crystal} is a log $q$-crystal $\mF$ equipped with a Frobenius structure $\vp_{\mF}$. 

\item The category of log $q$-crystals (resp. log $F$-$q$-crystals) is denoted by $\CR\big((\fX/(D,I,M_D))_{q\crys}\big)$ (resp. $\CR^\vp\big((\fX/(D,I,M_D))_{q\crys}\big)$). For simplicity, we also write $\CR\big((\fX/(\ul D,I))_{q\crys}\big)$ and $\CR^\vp\big(\fX/(\ul D,I))_{q\crys}\big)$.
When $\fX = \spf R$, we adopt the notation $\CR\big((R/(\ul D,I))_{q\crys}\big)$ and $\Vect^\vp\big(R/(\ul D,I))_{q\crys}\big)$. 

\item In the extreme case $q=1$, we can similarly define the categories of \emph{$\delta_{\log}$-crystals} and \emph{$\delta_{\log}$-$F$-crystals}, and we employ a similar set of notation by replacing ``$q$crys'' with ``$\delta$crys''.
\end{enumerate}
\end{definition}

If $(E, I, M_E)$ is a pre-log $q$-PD triple, then $(E/(q-1), I/(q-1), M_E)$ is naturally a pre-$\delta_{\log}$-PD triple. This defines a natural functor between sites
\[
    (\fX/(\ul D/(q-1)))_{\delta\crys} \ra (\fX/(\ul D,I))_{q\crys},
\]
where $\ul D/(q-1)= (D/(q-1),M_D)$. Pulling back along this functor, we obtain functors
\[
        \Vect\big((\fX/(\ul D,I))_{q\crys}\big) \ra \Vect\big((\fX/(\ul D/(q-1)))_{\delta\crys}\big)
\]
and
\[
        \Vect^{\varphi}\big((\fX/(\ul D,I))_{q\crys}\big) \ra \Vect^{\varphi}\big((\fX/(\ul D/(q-1)))_{\delta\crys}\big).
\]
For any $\mE\in \Vect\big((\fX/(\ul D,I))_{q\crys}\big)$, denote its image in $\Vect\big((\fX/(\ul D/(q-1)))_{\delta\crys}\big)$ by $\mE_{q=1}$.

\begin{assumption}\label{assumption: locally lifting}
   Let $(D,I,M_D)$ be a pre-log $q$-PD triple with integral $M_D$. Let $\fX$ be a saturated $p$-adic log formal scheme over $(D/I, M_D)$. In the rest of this section,  we assume that \'etale locally, $\fX$ admits a log smooth lifting over $(D,M_D)$. This holds when $(D,I,M_D)= (\Ainf, \xi, N_\infty)$ and $\fX$ is admissibly smooth over $\underline{\mO}$. Indeed, when $\fX= \spf R$ is small affine, the lift is given by $\spf \ul{A(R)}$.
\end{assumption}
        
\begin{proposition}\label{qcrys vs crys}
    Let $(D,I,M_D)$ be a pre-log $q$-PD triple with integral $M_D$. Let $\fX$ be a saturated $p$-adic log formal scheme over $(D/I, M_D)$ satisfying Assumption \ref{assumption: locally lifting}. 
    For any $\mE\in  \Vect\big((\fX/(\ul D,I))_{q\crys}\big)$, there is a natural isomorphism
    \[
        R\Gamma\big((\fX/(\ul D, I))_{q\crys},\mE\big)\htimes_{D}^{\L} D/(q-1) \xrightarrow[]{\sim} R\Gamma\big((\fX/(\ul D/(q-1)))_{\delta\crys},\mE_{q=1}\big).
    \]
    The same statement holds if we replace $\Vect$ by $\Vect^{\varphi}$.
\end{proposition}
\begin{proof}
Since the assertion is local for the \'etale topology, we may assume $\fX = \spf \ul{R}$ is affine, and admits a log smooth lift $\tilde R$ over $\underline{D}=(D,M_D)$. Let 
    \[
        \ul{P_0}= (P_0,\N^T\oplus M_D)= (D\langle (X_s)_{s\in S},\ul\N^T\rangle, \N^T\oplus M_D)
    \]
    be a $(p,[p]_q)$-completed free $\delta_{\log}$-$D$-algebra with a surjection $\ul{P_0} \ra \ul{\tilde R}$, where  $T$ is a finite set. 
    Let $\ul P$ be the $(p, I)$-completed $\delta_{\log}$-ring over $(\ul D, I)$ generated by $P_0$; namely,
    \[
        \ul{P}= \big(D\langle \{(X_s)_{s\in S}\}_\delta,\{\ul\N^T\}_{\delta}^{\log}\rangle_{(p,I)}^{\wedge}, \N^T\oplus M_D\big)
    \]
    Let $\ul{P}(\bullet)$ be the $p$-completed \v{C}ech nerve of $(D,M_D) \ra \ul{P}$. Here $\{-\}_{\delta}$ (resp. $\{-\}_{\delta}^{\log}$) denote the process of freely joining the set in the category of $(p,I)$-complete $\delta$-$D$-algebras (resp. $\delta_{\log}$-$D$-algebras). According to \cite[Lemma 7.4]{Koshikawa}, the log $q$-PD envelope of $\ul{P}(\bullet) \ra \ul{R}$ exists, which we denote by $\mD_q(P(\bullet))$. Note that $\mD_q(P(0)) = \mD_q(P)$ is a weakly final object in $(\fX/(\ul D,I))_{q\crys}$. Consequently, for any log $q$-crystal $\mE \in \Vect\big((\fX/(\ul D,I))_{q\crys}\big)$, its cohomology $R\Gamma((\fX/(\ul D,I))_{q\crys},\mE)$ is computed by the \v{C}ech--Alexander complex
    \[
        0 \ra \mE(\mD_q(P(0))) \ra \mE(\mD_q(P(1))) \ra \mE(\mD_q(P(2))) \ra \cdots.
    \]
    On the other hand, let $\mD\big(P(\bullet)/(q-1)\big)$ be the $p$-completed log PD envelope of $\ul{P}(\bullet)/(q-1) \ra \ul{R}$. Then, by \cite[Lemma 7.4(5)]{Koshikawa}, there is a natural isomorphism of simplicial rings
    \[
        \mD_q(P(\bullet))\htimes_D D/(q-1) \cong \mD\big(P(\bullet)/(q-1)\big).
    \]
    Therefore, by the crystal property combined with the flatness of $\mE$, the derived object 
    \[
        R\Gamma\big((\fX/(\ul D,I))_{q\crys},\mE\big)\htimes_{D}^{\L} D/(q-1)
    \]
    is computed by the \v{C}ech--Alexander complex
    \[
       \mE(\mD_q(P(\bullet)))\htimes_D D/(q-1) \cong  \mE\left(\mD_q(P(\bullet))\htimes_D D/(q-1)\right) \cong \mE\big(\mD\big(P(\bullet)/(q-1)\big)\big).
   \]
   This leads to a natural isomorphism
   \[
       R\Gamma\big((\fX/(\ul D, I))_{q\crys},\mE\big)\htimes_{D}^{\L} D/(q-1) \xrightarrow[]{\sim} R\Gamma\big((\fX/(\ul D/(q-1)))_{\delta\crys},\mE_{q=1}\big).
   \]
\end{proof}

As in \cite[Remark 6.6]{Koshikawa}, we define the \emph{the mixed characteristic log crystalline site} $(X/(\ul D, I))_{\crys}$ (with \'etale topology) by dropping the $\delta$-structure and the $\delta_{\log}$-structure on $\mathfrak{E}$ in the definition of the $\delta_{\log}$-crystalline site. Its structure sheaf is denoted by $\mO_{\crys}$. There is a natural cocontinuous functor of sites
\[
    ((\fX/(\ul D,I))_{\delta\crys} \ra ((\fX/(\ul D,I))_{\crys}
\]
by forgetting $\delta$-structures and $\delta_{\log}$-structures. 
Pulling back along this functor, We obtain functors
\[
    \Vect\big((\fX/(\ul D,I))_{\delta\crys}\big) \ra \Vect\big((\fX/(\ul D,I))_{\crys}\big)
\]
and
\[
    \Vect^{\varphi}\big((\fX/(\ul D,I))_{\delta\crys}\big) \ra \Vect^{\varphi}\big((\fX/(\ul D,I))_{\crys}\big).
\]
    
\begin{proposition}\label{prop: delta crys vs crys}
    Fix a pre-$\delta_{\log}$ triple $(D,I,M_D)$ with integral $M_D$. Let $\fX$ be a saturated $p$-adic log formal scheme over $(D/I, M_D)$ satisfying Assumption \ref{assumption: locally lifting}. 
    For any $\mE\in \Vect((\fX/(\ul D,I))_{\delta\crys}$, we still use $\mE$ to denote its image in $\Vect\big((\fX/(\ul D,I))_{\crys}\big)$. Then there is a natural isomorphism
    \[
        R\Gamma\big((\fX/(\ul D, I))_{\delta\crys},\mE\big)\xrightarrow[]{\sim} R\Gamma_{\crys}\big(\fX/(\ul D, I),\mE\big).
    \]
    The same statement holds if we replace $\Vect$ by $\Vect^{\varphi}$.
\end{proposition}
\begin{proof}
The proof essentially follows that of \cite[Proposition 6.8]{Koshikawa}. Since the assertion is local for the \'etale topology, we may assume $\fX = \spf \ul{R}$ is affine and admits a log smooth lift $\tilde R$ over $(D,M_D)$. Let 
    \[
        \ul{P_0}= (P_0,\N^T\oplus M_D)= (D\langle (X_s)_{s\in S},\ul\N^T\rangle, \N^T\oplus M_D)
    \]
    be a $(p,[p]_q)$-completed free $\delta_{\log}$-$D$ algebra with a surjection $\ul{P_0} \ra \ul{\tilde R}$, where  $T$ is a finite set. 
    Let $\ul{P_0}(\bullet)$ be the $p$-completed \v{C}ech nerve of $(D,M_D) \ra \ul{P_0}$. Let $\ul {C_0}(\bullet)$ be the $p$-complete log-PD envelope of $\ul{P_0}(\bullet) \ra \ul R$. Write $J$ for the kernel of $C_0(0) \ra R$, then $(\ul{C_0}(0),J)^a$ is a weakly final object in $((\fX/(\ul D,I))_{\crys}$. Consequently, the crystalline cohomology $R\Gamma_{\crys}\big(\fX/(\ul D, I),\mE\big)$ is computed by $\mE(C_0(\bullet))$.

    On the other hand,  let $\ul P$ be the $(p, I)$-completed $\delta_{\log}$-ring over $(\ul D, I)$ generated by $P_0$; namely,
    \[
        \ul{P}= \big(D\langle \{(X_s)_{s\in S}\}_\delta,\{\ul\N^T\}_{\delta}^{\log}\rangle_{(p,I)}^{\wedge}, \N^T\oplus M_D\big).
    \]
    Let $\ul{P}(\bullet)$ be the $p$-completed \v{C}ech nerve of $(D,M_D) \ra \ul{P}$. By \cite[Proposition 6.8]{Koshikawa}, the $\delta$-crystalline cohomology $R\Gamma\big((\fX/(\ul D, I))_{\delta\crys},\mE\big)$ is computed by $\mE(C(\bullet))$.
    Moreover, the \v{C}ech nerve has the following description:
     \[
         C(\bullet) \cong (P(\bullet)\otimes_{P_0(\bullet)} C_0(\bullet))_p^{\wedge}.
     \]
     By \cite[Proposition 6.8]{Koshikawa} again, $P_0(\bullet) \ra P(\bullet)$ is a homotopy equivalence. Therefore, by the crystal properties, we obtain quasi-isomorphisms
     \[
          \mE(C(\bullet)) \cong \mE(C_0(\bullet))\otimes_{C_0(\bullet)}C(\bullet) \cong \mE(C_0(\bullet))\htimes_{P_0(\bullet)}P(\bullet) \cong \mE(C_0(\bullet)).
     \]
     This finishes the proof.
     
\end{proof}

\begin{corollary}
    Let $(D,I,M_D)$ and $\fX$ be as above. The reduction modulo $q-1$ gives a natural functor
    \[
        \Vect\big((\fX/(\ul D,I))_{q\crys}\big) \ra \Vect\big(\big(\fX/(\ul D/(q-1)))_{\crys}\big),
    \]
    sending $\mE$ to $\mE_{q=1}$. For any $\mE\in \Vect\big((\fX/(\ul D,I))_{q\crys}\big)$, there is a natural isomorphism
    \[
        R\Gamma\big((\fX/(\ul D, I))_{q\crys},\mE\big)\otimes_{D}^{\L} D/(q-1) \xrightarrow[]{\sim} R\Gamma_{\crys}\big(\fX/(\ul D/(q-1)),\mE_{q=1}\big).
    \]
    The same statement holds if we replace $\Vect$ by $\Vect^{\varphi}$.
\end{corollary}

Now we specialize to the situation $(D,I,M_D)= (\ul \Ainf, \xi)= (\Ainf, \xi, N_\infty)$. Let $\fX$ be a saturated $p$-adic log formal scheme that is admissibly smooth over $\underline{\mO}$. Assume that $\fX$ locally admits free charts. By \cite[theorem 7.13]{Koshikawa}, there is a canonical functor 
    \begin{equation}\label{q-crystalline site to prismatic site}
        \gamma_q \colon (\fX/\logqAinf)_{q\crys} \ra \big(\fX^{(1)}/\logAinf\big)_{\Prism}.
    \end{equation}
    It sends an object $(E,J,M_E)$ of $\big(\fX/\logqAinf\big)_{q\crys}$ to the log prism 
    \[
        (\varphi_*E, ([p]_q= \txi), M_E) \in \big(\fX^{(1)}/\logAinf\big)_{\Prism}.
    \]
Pulling back along $\gamma_q$ yields a natural functor
    \[
        \gamma^*_q \colon \Vect^{\vp}\big(\big(\fX^{(1)}/\logAinf\big)_{\Prism}\big) \ra \Vect^{\vp}\big((\fX/\logqAinf)_{q\crys}\big).
    \]
For any $\bM\in \BKF^{\log}(\fX,\vp)$, let $\bM_\Prism$ denote the corresponding prismatic $F$-crystal via the equivalence in Theorem \ref{thm: equivalence between BKF modules and prismatic $F$-crystals}.

\begin{definition}\label{defn: qcrys specialization}
Let $\fX$ be a saturated $p$-adic log formal scheme that is admissibly smooth over $\underline{\mO}$. Assume that $\fX$ locally admits free charts. We define the \emph{$q$-crystalline specialization functor} as
    \begin{align*}
        \sigma_{q\crys}^*\colon \BKF^{\log}(\fX,\vp) &\lra \Vect^{\vp}((\fX/\logqAinf)_{q\crys}\big) \\
        \bM &\longmapsto \bM_{q\crys} := \gamma_{q}^*\bM_\Prism.
    \end{align*}
\end{definition}

For any finite sets $S,T$, we fix the following notation (cf. \cite[Construction 11.1]{log1}):
\begin{equation}\label{notation Sigma(S,T)}
    \Sigma(S, \ul T) := (\mO\langle(X_s^{\pm 1})_{s\in S}, \ul{T}\rangle, \N^T\oplus N_\infty).
\end{equation}
Choose a surjection $\Sigma= \Sigma(S, \ul T) \ra \ul{R}$ of pre-log rings. Let 
\[
    A(\Sigma):= ((\Ainf\{(X_s^{\pm 1})_{s\in S}\}_\delta, \{\ul{T}\}_\delta^{\log})_{p,\mu}^{\wedge}, \N^T\oplus N_\infty)
\]
be the free $(p,\mu)$-completed $\delta_{\log}$-$\Ainf$-algebra generated by $\Sigma(S, \ul T)$. Here $\{-\}_{\delta}$ (resp. $\{-\}_{\delta}^{\log}$) denotes the process of freely joining the set in the category of $(p,\mu)$-complete $\delta$-$\Ainf$-algebras (resp. $\delta_{\log}$-$\Ainf$-algebras). Let 
\begin{equation}\label{defn of A_0(Sigma)}
    A_0(\Sigma):= (\Ainf\langle(X_s^{\pm 1})_{s\in S}, \ul{T}\rangle, \N^T\oplus N_\infty)
\end{equation}
be the smooth $\delta_{\log}$-$\Ainf$-algebra with the $\delta_{\log}$ structure given by $\delta_{\log}(\N^T)=0$. Let $\mD(\Sigma)$ be the log $q$-PD envelope of $A(\Sigma)$ with respect to the surjection $A(\Sigma) \ra \ul{R}$.
Let $A(\Sigma)(i)$ be the $p$-completed $(i+1)$-fold product of $A(\Sigma)$ over $\Ainf$. Then $A(\Sigma)(\bullet)$ is the \v{C}ech nerve of $(\Ainf, N_\infty) \ra A(\Sigma)$. Let $\mD_{q,\Sigma}(\bullet)= \mD_q(A(\Sigma)(\bullet))$ be the $p$-completed log $q$-PD envelope of the surjection
\[
     A(\Sigma)(\bullet) \ra \Sigma(S, \ul T) \ra \ul R.
\]
Let $\mE\in \Vect\big((R/\logqAinf)_{q\crys}\big)$. Then its cohomology $R\Gamma\big((\fX/\logqAinf)_{q\crys}, \mE\big)$ is computed by the \v{C}ech--Alexander complex $\mE(\mD_{q,\Sigma}(\bullet)))$. 
Parallelly, let $A(\Sigma)(\bullet)$ be the \v{C}ech nerve of $(\Ainf, N_\infty) \ra A_0(\Sigma)$. Let $\mD(A_0(\Sigma)(\bullet))$ be the $p$-completed log $q$-PD envelope of the surjection
\begin{equation}\label{defn of D_Sigma}
    A_0(\Sigma)(\bullet) \ra \Sigma(S, \ul T) \ra \ul R.
\end{equation}
Denote $\mD_{\Sigma}= \mD(A_0(\Sigma))$ and $\mD_{\Sigma}(\bullet)= \mD(A_0(\Sigma)(\bullet))$. 

\begin{lemma}
    Keep the notation as above. Let $\mE\in \Vect\big((R/\logqAinf)_{q\crys}\big)$ be a log $q$-crystal. Then, its cohomology $R\Gamma\big((\fX/\logqAinf)_{q\crys}, \mE\big)$ is also computed by $\mE(\mD_{\Sigma}(\bullet))= \mE\big(\mD(A_0(\Sigma) (\bullet))\big)$. 
\end{lemma}

\begin{proof}
We only need to show that the inclusion $\mD_\Sigma(\bullet) \ra \mD_{q,\Sigma}(\bullet)$ induces a quasi-isomorphism
\[
    \mE(\mD_\Sigma(\bullet)) \ra \mE(\mD_{q,\Sigma}(\bullet))
\]
between \v{C}ech--Alexander complexes. By $(p,q-1)$-adic completeness, we are left to check this after modulo $(q-1)$. Namely, it suffices to show that the morphism
\[
    \mE(\mD_\Sigma(\bullet))\htimes^{\L}_{\Ainf}\Ainf/(q-1) \ra \mE(\mD_{q,\Sigma}(\bullet))\htimes^{\L}_{\Ainf}\Ainf/(q-1)
\]
is an isomorphism. By \cite[Lemma 7.4(5)]{Koshikawa}, there is a natural isomorphism 
\[
    \mD_{q,\Sigma}(\bullet)\htimes_D D/(q-1) \cong \mD\big(A(\Sigma)(\bullet)/(q-1)\big),
\]
of simplicial rings, where $\mD\big(A(\Sigma)(\bullet)/(q-1)\big)$ is the $p$-completed log PD envelope of $A(\Sigma)(\bullet)/(q-1) \ra R$. Similarly, by \cite[Lemma 7.4(5)]{Koshikawa} again, there is a quasi-isomorphism
\[
    \mD_\Sigma(\bullet)\htimes_D D/(q-1) \cong \mD\big(A_0(\Sigma)(\bullet)/(q-1)\big).
\]
So it remains to verify the quasi-isomorphism
\[
    \mE_{q=1}\big(\mD\big(A(\Sigma)(\bullet)/(q-1)\big)\big) \cong \mE_{q=1}\big(\mD\big(A_0(\Sigma)(\bullet)/(q-1)\big)\big).
\]
Indeed, the two sides compute the $\delta$-crystalline cohomology $R\Gamma(\fX/(\ul\Ainf, \xi)_{q\crys},\mE_{q=1})$ and the crystalline cohomology $R\Gamma_{\crys}(\fX/(\ul \Ainf/\mu),\mE_{q=1})$, respectively. The desired isomorphism follows from Proposition \ref{prop: delta crys vs crys}.
\end{proof}

Let $\fX= \spf R$ be an affine admissibly smooth $p$-adic log formal scheme over $\ul \mO$. 
Notice that there is a natural functor \footnote{This functor is actually an equivalence of categories. The proof is essentially the same as that of \cite[Theorem 6.2]{Kato}.}
\begin{equation}\label{eq: q-crystals vs stratifications}
    \Vect\big((\fX/\logqAinf)_{q\crys}\big)\ra \Strat(\mD_{\Sigma}(\bullet))
\end{equation}
sending a log $q$-crystal to a module with stratification with respect to $\mD_{\Sigma}(\bullet)$. More precisely, let $\mE\in \Vect\big((\fX/\logqAinf)_{q\crys}\big)$ be a log $q$-crystal. It sends $\mE$ to $(E, \varepsilon)$ where $E= \mE(\mD_{\Sigma})$ and
\[
    \varepsilon: E\otimes_{\mD_{\Sigma}, p_2}\mD_{\Sigma}(1)\xrightarrow[]{\cong} \mE(\mD_\Sigma(1)) \xrightarrow[]{\cong} E\otimes_{\mD_{\Sigma}, p_1}\mD_{\Sigma}(1).
\]
For every $\gamma\in \Gamma_{\Sigma}$, the base change of $\varepsilon: E\otimes_{\mD_{\Sigma}, p_2}\mD_{\Sigma}(1)\xrightarrow[]{\cong} E\otimes_{\mD_{\Sigma}, p_1}\mD_{\Sigma}(1)$ along the map $\mD_{\Sigma}(1)\xrightarrow[]{(1, \gamma)} \mD_{\Sigma}(1)\xrightarrow[]{\Delta} \mD_{\Sigma}$ induces an isomorphism $E\otimes_{\mD_{\Sigma}, \gamma} \mD_{\Sigma} \xrightarrow[]{\sim} E$; this defines an action of $\gamma$ on $E$. Using the conditions (1) and (2) in Definition \ref{defn: stratification}, one checks that the actions of different $\gamma$'s are compatible and hence define a semilinear action of $\Gamma$ on $E$. Moreover, since the action of $\Gamma$ on $\mD_\Sigma(1)/\mu$ is the identity, the induced action of $\Gamma$ on $E/\mu$ is also the identity.
Thus, we obtain a functor
\[
    \ev^{\Strat}_{\mD_{\Sigma}}: \mathrm{Strat}(\mD_{\Sigma}(\bullet))\rightarrow \Rep^{\mu}_{\Gamma_{\Sigma}}(\mD_{\Sigma}).
\] 
Notice that there is a functor $\Rep^{\mu}_{\Gamma_{\Sigma}}(\mD_{\Sigma}) \ra q\MIC(\mD_{\Sigma})$ defined in Theorem \ref{representations vs q-connections}, which associates $E \in \Rep^{\mu}_{\Gamma_{\Sigma}}(\mD_{\Sigma})$ with a log $q$-connection $\nabla \colon E \ra E\otimes_{\mD_{\Sigma}} q\Omega^1_{\mD_{\Sigma}/\Ainf}$ defined by
\[
    e \mapsto \sum_{s\in S} \frac{\gamma_s(e)-e}{q-1}d\log X_s+ \sum_{t\in T} \frac{\gamma_t(e)-e}{q-1}d\log X_t, \ \forall e\in E.
\]
Composing this with $\ev_{\mD_\Sigma}^{\Strat}$, we arrive at a functor
\begin{align*}
     \Vect\big((\fX/\logAinf)_{q\crys}\big) &\lra q\MIC(\mD_{\Sigma})\\
     \mE &\longmapsto (\mE_{q\dR}, \nabla):= (E,\nabla).
\end{align*}
Let $\Delta_{\Sigma}= \{\gamma_s\mid s\in S\}\cup \{\gamma_t\mid t\in T\}$ and let
\[
    \Kos\big(E, \frac{\Delta_\Sigma-1}{q-1}\big)= \Kos\big(E, \big\{\frac{\gamma_s-1}{q-1}\big\}_{s\in S}, \big\{\frac{\gamma_t-1}{q-1}\big\}_{t\in T}\big)
\]
denote the corresponding Koszul complex. By construction, there is a quasi-isomorphism 
\[
    E\otimes_{\mD_\Sigma} q\Omega^{\bullet}_{\mD_{\Sigma}/\Ainf} \cong \Kos\big(E, \frac{\Delta_\Sigma-1}{q-1}\big)
\]
between complexes in $D(\fX_{\ett},\Ainf)$.

\begin{proposition}\label{qcrys vs qdR}
Let $\fX= \spf R$ be an affine admissibly smooth $p$-adic log formal scheme over $\ul \mO$. Let $\mE\in \Vect\big((\fX/\logqAinf)_{q\crys}\big)$ be a log $q$-crystal and let $(E, \nabla)$ be the associated $\mD_{\Sigma}$-module with $q$-connection. Then there is a natural isomorphism
\[
    E\otimes_{\mD_\Sigma} q\Omega^{\bullet}_{\mD_{\Sigma}/\Ainf}\cong  R\Gamma\big((R/\logqAinf)_{q\crys},\mE\big)
\]
in $D(\fX_{\ett},\Ainf)$.
\end{proposition}
 
\begin{proof}
    Recall that $\mD_\Sigma(\bullet)$ is defined to be the log PD-envelope of (\ref{defn of D_Sigma}). For any integer $m \ge 0$, put $E^m= \mE(\mD_{\Sigma}(m)) \cong E\otimes_{\mD_{\Sigma}}\mD_{\Sigma}(m)$. Here, the isomorphism is induced by the  face map $[0]\hookrightarrow [m]$ sending $0$ to $0$. For any $m,n \ge 0$, let
    \[
        E^{m,n}= E^m\otimes_{\mD_{\Sigma}(m)}q\Omega^n_{\mD_{\Sigma}(m)/\Ainf} \cong E\otimes_{\mD_{\Sigma}}\Omega^n_{\mD_{\Sigma}(m)/\Ainf},
    \]
    where \[q\Omega^\bullet_{\mD_{\Sigma}(m)/\Ainf}:= \mD_{\Sigma}(m)\htimes_{A_0(\Sigma)(m)} q\Omega^\bullet_{A_0(\Sigma)(m)/\Ainf}\]
    is the $q$-de Rham complex extending $q\Omega^\bullet_{A_0(\Sigma)(m)/\Ainf}$ by \cite[Construction 16.20]{BS}.
    On the one hand, for any fixed $n \ge 0$, the complex $E^{\bullet,n}= \mE(\mD_{\Sigma}(\bullet))\otimes_{\mD_{\Sigma}(m)}\Omega^n_{\mD_{\Sigma}(m)/\Ainf}$ has a cosimlplicial structure induced from that on $\mD_{\Sigma}(\bullet)$. So we equip it with the differential $d_1^\bullet$ by the associated \v{C}ech--Alexander complex structure. On the other hand, let $d_2^{m,n} \colon E^{m,n} \ra E^{m,n+1}$ be the $n$-th differential of the $q$-de Rham complex
    \[
        E^m\otimes_{\mD_{\Sigma}(m)} q\Omega^\bullet_{\mD_{\Sigma}(m)/\Ainf}= E^m\htimes_{A_0(\Sigma)(m)} q\Omega^\bullet_{A_0(\Sigma)(m)/\Ainf}.
    \]
    Putting these together, we obtain a double complex $\big(E^{\bullet, \bullet}, d_1^{\bullet, \bullet},d_2^{\bullet,\bullet} \big)$.
    One observes that
    \begin{itemize}
        \item the $0$-th row $E^{\bullet,0}$ computes the log $q$-crystalline cohomology $R\Gamma\big((R/\logqAinf)_{q\crys},\mE\big)$, 
        \item the $0$-th column $E^{0,\bullet}$ is the log $q$-de Rham complex $E\otimes_{\mD_{\Sigma}}q\Omega^{\bullet}_{\mD_{\Sigma}/\Ainf}$.
   \end{itemize}
    For $n>0$, by \cite[Theorem 7.17]{Koshikawa}, we know $q\Omega^n_{\mD_{\Sigma}(\bullet)/\Ainf}$ is cosimplically homotopy equivalent to $0$. Therefore, the same holds for $E^{\bullet, n}$ because $E^{\bullet, n}\cong E\otimes_{\mD_{\Sigma}} q\Omega^n_{\mD_{\Sigma}(\bullet)/\Ainf}$. Consequently, we obtain an isomorphism
    \[
        R\Gamma\big((R/\logqAinf)_{q\crys},\mE\big) \cong \mathrm{Tot}(E^{\bullet, \bullet}).
    \]
    
    Now for any $0\le i < j$, consider any face map $\mD_{\Sigma}(i)\ra \mD_{\Sigma}(j)$. It induces maps between complexes $E^{i,\bullet} \ra E^{j,\bullet}$ as well as $q\Omega^\bullet_{\mD_{\Sigma}(i)/\Ainf} \ra q\Omega^\bullet_{\mD_{\Sigma}(j)/\Ainf}$. 
    Notice that by the crystal property, $E^j \cong E^i\otimes_{\mD_{\Sigma}(i)} \mD_{\Sigma}(j)$. So there is a Cartesian diagram of complexes 
    \[
    \begin{tikzcd}
        {q\Omega^\bullet_{\mD_{\Sigma}(i)/\Ainf}} & {q\Omega^\bullet_{\mD_{\Sigma}(j)/\Ainf}}\\
        {E^{i,\bullet}} & {E^{j,\bullet}}
        \arrow[from=1-1, to=2-1]
        \arrow[from=1-1, to=1-2]
	\arrow[from=2-1, to=2-2]
	\arrow[from=1-2, to=2-2]
    \end{tikzcd}
    \]
    in the derived category of $(p,q-1)$-completed $\mD_{\Sigma}(i)$-modules. The top horizontal map $q\Omega^\bullet_{\mD_{\Sigma}(i)/\Ainf} \ra q\Omega^\bullet_{\mD_{\Sigma}(j)/\Ainf}$ is a homotopy equivalence due to \cite[Theorem 7.17]{Koshikawa}; then so does the bottom horizontal map $E^{i,\bullet} \ra E^{j,\bullet}$. Consequently, we obtain a quasi-isomorphism
    \[
        E\otimes_{\mD_\Sigma} q\Omega^{\bullet}_{\mD_{\Sigma}/\Ainf}\cong \mathrm{Tot}(E^{\bullet, \bullet}).
    \]
    Therefore, the complexes $E\otimes_{\mD_\Sigma} q\Omega^{\bullet}_{\mD_{\Sigma}/\Ainf}$ and $ R\Gamma\big((R/\Ainf)_{q\crys},\mE\big)$ are quasi-isomorphic, as desired.
\end{proof}

Let $\fX$ be an admissibly smooth $p$-adic log formal scheme over $\ul \mO$. Assume that $\fX$ locally admits free charts. Let $\fX_{\mO/p}$ be the fiber of $\fX$ over $\mO/p$. 
For any small affine \'etale open $\spf R$ of $\fX$, we pick any surjection $\Sigma= \Sigma(S, \ul T) \ra \ul{R}$ of pre-log rings, so that we can consider $\mD_{\Sigma}(\bullet)$ as above. Since $([p]_q) = (p)$ in $\Acrys$,\footnote{Let $t= \log ([\varepsilon])\in \Acrys$, then $\mu= [\varepsilon]-1= tv$ for some unit $v$. The fact that $\vp(t)= pt$ induces an equality $[p]_q= \vp(\mu)/\mu= p \vp(v)/v$.} 
there is a natural isomorphism between simplicial rings
\[
    \mD_{\Sigma}(\bullet)\htimes_{\Ainf}\Acrys \cong \mD_{\crys,\Sigma}(\bullet):= \mD(\Acrys(\Sigma)(\bullet)),
\]
where $\Acrys(\Sigma)(\bullet)= A_0(\Sigma)(\bullet) \htimes_{\Ainf} \Acrys$ and $\mD_{\crys,\Sigma}(\bullet)= \mD(\Acrys(\Sigma)(\bullet))$ is the $p$-completed log PD envelope of the surjection $\Acrys(\Sigma)(\bullet) \ra R$ in the crystalline site $(\fX_{\mO/p}/\Acrys)_{\crys}$. Notice that $\mD_{\crys,\Sigma}(\bullet)$ is a weakly final object in $(\fX_{\mO/p}/\Acrys)_{\crys}$. 
Let $\ul\Acrys= (\Acrys, N_\infty)$, equipped with the pre-log structure induced from $\ul \Ainf$. For any log $F$-$q$-crystal $(\mE,\vp_{\mE})$ on $(\fX/\logqAinf)_{q\crys}$, the \emph{associated absolute crystalline $F$-crystal} $\mE_{\Acrys} \in \Vect^{\vp}\big((\fX_{\mO/p}/\ul \Acrys)_{\crys}\big)$ is defined by
\begin{equation}\label{qcrystal to crystal over Acrys}
    \mE_{\Acrys}(\mD_{\crys,\Sigma}(\bullet)) := \mE(\mD_{\Sigma}(\bullet))\htimes_{\Ainf}\Acrys,
\end{equation}
equipped with the natural Frobenius structure.

The following result is an immediate consequence of the construction.

\begin{proposition}\label{prop: q-crys vs abs crys}
Let $\fX$, $\mE$, and $\mE_{\Acrys}$ be as above. Then there is a natural isomorphism
    \[
        R\Gamma\big((\fX/\logqAinf)_{q\crys}, \mE\big)\htimes^{\L}_{\Ainf}\Acrys \cong R\Gamma_{\crys}\big(\fX_{\mO/p}/\ul \Acrys, \mE_{\Acrys}\big).
    \]
\end{proposition}

\begin{definition}\label{defn: crystalline specialization}
    Let $\fX$ be an admissibly smooth $p$-adic log formal scheme over $\ul \mO$ and assume that $\fX$ locally admits free charts. We define the \emph{absolute crystalline specialization functor} 
    \[
         \sigma_{\Acrys}^* \colon \BKF^{\log}(\fX,\vp) \ra \Vect^{\vp}\big((\fX_{\mO/p}/\ul\Acrys)_{\crys}\big)
    \]
    by $\sigma_{\Acrys}^*\bM := \mE_{\Acrys}$, where $\mE= \bM_{q\crys} \in \Vect^{\vp}\big((\fX/\logqAinf)_{q\crys}\big)$. 

    Moreover,  we define the \emph{crystalline specialization functor}
    \[
         \sigma_{\crys}^* \colon \BKF^{\log}(\fX,\vp) \ra \Vect^{\vp}((\fX_k/W(\ul k))_{\crys})
    \]
    as the composition of the absolute crystalline specialization functor $\sigma_{\Acrys}^*$ and the base change functor $\Vect^{\vp}\big((\fX_{\mO/p}/\ul\Acrys)_{\crys}\big) \ra \Vect^{\vp}((\fX_k/W(\ul k))_{\crys})$ induced by the natural map $\ul{\Acrys} \ra W(\ul k)$.
\end{definition}

\begin{remark}
    For a more general admissibly smooth $p$-adic log formal scheme $\fX$ over $\ul \mO$ (i.e., not necessarily locally admit free charts), we can still define the absolute crystalline specialization functor using the methods in \cite[\S 4]{MT}. More precisely, one can make sense of the category of ``log crystalline Breuil--Kisin--Fargues modules'' $\BKF_{\crys}^{\log}(\fX,\vp)$, which locally can be identified with generalized representations in $\Rep_{\Gamma}^{<\mu}(\Acrys(R_\infty))$. It comes with a natural functor 
    \[
        \Rep^{<\mu}_\Gamma(\Ainf(R_\infty)) \ra \Rep^{<\mu}_\Gamma(\Acrys(R_\infty))
    \]
    via $-\otimes_{\Ainf}\Acrys$. Here, $\Acrys(R_\infty):= \Ainf(R_\infty) \widehat{\otimes}_{\Ainf}\Acrys$. Following the strategy in \cite[\S 4]{MT}, one can construct a fully faithful functor
    \[
        \MIC(\Acrys(R)) \xrightarrow[]{\sim} \Rep_{\Gamma}^{\mu}(\Acrys(R)) \hookrightarrow \Rep_{\Gamma}^{<\mu}(\Acrys(R_\infty))
    \]
    whose essential image contains the image of $\Rep_{\Gamma}^{<\mu}(\Ainf(R_\infty)) \cong \Rep_{\Gamma}^{\mu}(A(R))$. Globalizing, we obtain functors
    \[
    \begin{tikzcd}
        & \Vect(\fX_{\mO/p}/\Acrys)_{\crys} \ar[d, hook]\\
        \BKF^{\log}(\fX,\vp) \ar[r] \ar[ru,"\sigma_{\Acrys}^*"] & \BKF_{\crys}^{\log}(\fX,\vp).
    \end{tikzcd}
    \]
    However, as we shall see, our proof of crystalline comparison does not generalize to this setting.
\end{remark}

\vspace{0.1in}
\subsection{Crystalline comparison}\label{subsection: crystalline comparison} 
\noindent
\vspace{0.1in}

\noindent To prove the crystalline comparison isomorphism, we follow \cite[\S 12]{BMS1} and \cite[\S 5]{CK_semistable} and use the method of ``all possible coordinates'' to obtain a functorial morphism between log $q$-crystalline cohomology and log $\Ainf$-cohomology. 

Let $\fX$ be an admissibly smooth $p$-adic log formal scheme over $\underline{O}$. Assume that $\fX$ locally admits free charts. Let $\bM\in \BKF^{\log}(\fX,\vp)$ and let $\bM_{q\crys}$ be the associated log $F$-$q$-crystal. Let $\upsilon\colon (\fX/(\ul\Ainf, \xi))_{q\crys} \ra \fX_{\ett}$ be the natural projection of sites. We first construct a functorial map from $R\upsilon_*(\bM_{q\crys})$ to $A\Omega^{\log}_{\fX}(\bM)$. Since the question is \'etale local, we may assume that $\fX= \spf \ul{R}$ is small affine, modeled on a small chart $u:N\rightarrow P$ such that both $N$ and $P$ are free. 

Consider triples $\Phi= (S,P_\Lambda, \iota)$ consisting of
\begin{enumerate}
    \item a finite set $S$ that indexes the coordinates of the formal $\mO$-torus (equipped with the trivial log structure)
    \[
        R^\square_S := \mO\langle X_s^{\pm 1}\rangle_{s\in S};
    \]
    \item a nonempty finite set of monoids $P_\Lambda= \{P_\lambda \mid \lambda \in \Lambda\}$ together with injective homomorphisms $u_\lambda \colon N \hookrightarrow P_\lambda$ satisfying the conditions in Definition \ref{definition: small affine}, which induces a pre-log ring
    \[
        \ul{R^\square_\lambda} := (\mO\langle P_\lambda \rangle\htimes_{\mO\langle N \rangle} \mO, P_\lambda\sqcup_N N_\infty)
    \]
    over $\ul \mO$;
    \item an exact closed immersion of pre-log rings
    \[
        \iota \colon \fX= \spf \underline{R} \ra \spf (R^\square_S) \times \prod_{\lambda\in \Lambda} \spf(\ul{R^\square_\lambda})^a
    \]
    where the products are formed over $\spf(\ul\mO)^a$, satisfying
    \begin{itemize}
        \item the map $\iota_S \colon \fX \ra \spf (R^\square_S)$ is already a closed immersion, and
        \item the map $\iota_\lambda \colon \fX \ra \spf(\ul{R^\square_\lambda})^a$ is strictly \'etale for every $\lambda\in \Lambda$.
    \end{itemize}
\end{enumerate}

On the one hand, such a triple $\Phi= (S,P_\Lambda, \iota)$ provides a way to compute log $\Ainf$-cohomology. To explain the computation, we start with certain perfectoid covers and ``$\Ainf$-deformations'' of $R^\square_S$ and $\ul{R^\square_\lambda}$.
\begin{itemize}
    \item For each $\lambda \in \Lambda$, consider the perfectoid ring $R_{\infty,\lambda}^\square = \mO\langle P_{\lambda,\Q\ge 0}\rangle\htimes_{\mO\langle N\rangle} \mO$. Let $\ul{R_{\infty,\lambda}^\square}= (R_{\infty,\lambda}^\square, P_{\lambda,\Q\ge 0})$, which is a pro-Kummer \'etale cover of $\ul{R_\lambda^\square}$. Let $R_{\infty,\lambda}= R_{\infty,\lambda}^\square\htimes_{R_\lambda^\square} R$, and $\ul{R_{\infty,\lambda}}= (R_{\infty,\lambda}, P_{\lambda, \Q\ge 0})$. Then consider $\Ainf$-algebras $\Ainf(R_{\infty,\lambda})= W((R_{\infty,\lambda})^\flat)$ and $A(R_\lambda^\square)= \Ainf\langle P_\lambda\rangle \htimes_{\Ainf\langle N\rangle}\Ainf$. By \'etaleness, $A(R_\lambda^\square)$ lifts to $A(R_\lambda)$ over $R$. Let $\ul{A(R_\lambda)}= (A(R_\lambda), P_\lambda)$ and $\ul{\Ainf(R_{\infty,\lambda})}= (\Ainf(R_{\infty,\lambda}), P_{\lambda, \Q\ge 0})$. Then the natural inclusion $\ul{A(R_\lambda)} \ra \ul{\Ainf(R_{\infty,\lambda})}$ is a perfectoid Galois cover with Galois group
\[
    \Gamma_{\lambda}= \Hom(P_{\lambda}^{\gp}/N^{\gp}, \widehat{\Z}(1)) \cong \widehat{\Z}(1)^{d_\lambda},
\]
where $d_\lambda= \mathrm{rk}_{\Z}P_\lambda^{\gp}$. 
\item Let $\ul{R^\square_\Lambda}= \widehat{\bigotimes}_{\lambda\in \Lambda} \ul{R^\square_\lambda}$, and $\ul{ R_{\infty,\Lambda}^\square}= \widehat{\bigotimes}_{\lambda\in \Lambda} \ul {R^\square_{\infty,\lambda}}$. Then $\ul{R^\square_{\infty,\Lambda}}$ is a perfectoid cover of $\ul{R^\square_{\Lambda}}$ with Galois group 
\[
    \Gamma_{\Lambda} \cong \prod_{\lambda\in \Lambda}\Gamma_\lambda \cong \widehat{\Z}(1)^{d_\Lambda},
\]
where $d_\Lambda= \sum_{\lambda\in \Lambda}d_{\lambda}$.
\item Similarly, for the log-free part, let $R^\square_{\infty,S}= \mO\langle T_s^{\pm \frac{1}{\infty}}\rangle_{s\in S}$. It is a perfectoid Galois cover of $R^\square_S$ with Galois group
\[
    \Gamma_S= \Hom(\Q^{S}, \mu_\infty(\mO))\cong \widehat{\Z}(1)^{S}.
\]
\item Put $\ul{R^\square_{\Phi}}= R^\square_S\htimes_\mO \ul{R^\square_\Lambda}$, $\ul{R^\square_{\infty,\Phi}}= R^\square_{\infty,S}\htimes_\mO \ul{R^\square_{\infty,\Lambda}}$ and $\Gamma_{\Phi}= \Gamma_S \times \Gamma_{\Lambda}$.
Let $\ul{R_{\infty,\Phi}}$ be the normalization of $R$ in $ R\htimes_{R_{\Phi}^\square}R_{\infty,\Phi}^\square[\frac{1}{p}]$ equipped with the pullback log structure. Then $\ul{R_{\infty,\Phi}}$ is a perfectoid Galois cover of $\ul R$ with Galois group $\Gamma_{\Phi}$. Let $\Ainf(R_{\infty,\Phi})= W((R_{\infty, \Phi})^\flat)$, which is also equipped with a natural $\Gamma_\Phi$-action.
\end{itemize}
For each $\lambda\in \Lambda$, by Lemma \ref{lemma: splitting}, the injection $N^{\gp}\hookrightarrow P_\lambda^{\gp}$ admits a splitting
\[
    P_\lambda^{\gp}\cong N^{\gp}\oplus \big(\oplus_{i=1}^{d_\lambda} \Z e_{\lambda,i} \big)
\]
such that $e_{\lambda,i} \in P_\lambda$. As a basis of $P_\lambda^{\gp}/N^{\gp}$, $\{e_{\lambda,i}\mid 1\le i \le d_\lambda\}$ corresponds to a set of topological generators $\Delta_\lambda= \{\gamma_{\lambda,1}, \gamma_{\lambda,2},\cdots,\gamma_{\lambda,d_\lambda}\}$ of $\Gamma_\lambda$ as in Lemma \ref{lemma: splitting}. Moreover, by Lemma \ref{Lem: the assumptions of qDR}, we can construct elements $U_{\lambda,1}, U_{\lambda,2}\cdots, U_{\lambda,d_\lambda}$ in $A(R_\lambda)$ such that 
\[
\gamma_{\lambda,i}(U_{\lambda,j})=\begin{cases}
  qU_{\lambda,j}  & \text{ if } i=j\\
  U_{\lambda,j} & \text{ if } i\neq j,
\end{cases}
\]
and $\varphi(U_{\lambda,i})=U_{\lambda ,i}^p$ for all $1\le i\le d_\lambda$. Let $\Delta_\Lambda= \bigsqcup_{\lambda\in \Lambda}\Delta_\lambda \subseteq \Gamma_\Lambda$ be a set of topological generators in $\Gamma_\Lambda$. Let $V_{\lambda,i}$ be the image of $U_{\lambda, i}$ under the map $A(R_\lambda) \ra \Ainf(R_{\infty,\Phi})$ induced by the inclusion $R_\lambda \ra R_{\infty,\Phi}$. Then
\[
\gamma_{\lambda_1,i}(V_{\lambda_2,j})=\begin{cases}
  qV_{\lambda_2,j}  & \text{ if } (\lambda_1,i)=(\lambda_2,j)\\
  V_{\lambda_2,j} & \text{ if } (\lambda_1,i) \neq (\lambda_2,j),
\end{cases}
\]
and $\varphi(V_{\lambda,i})=V_{\lambda ,i}^p$ for all $\lambda \in \Lambda$, $1\le i\le d_\lambda$. Similarly, $\Gamma_S$ has a natural set of topological generators $\Delta_S := \{\gamma_s \mid s\in S\}$, corresponding to elements in $S$. The set $\Delta_{S}$ satisfies 
\begin{equation}\label{action of Delta_S}
   \gamma_s(X_{s'})=\begin{cases}
  qX_s  & \text{ if } s=s'\\
  X_{s'} & \text{ if } s \neq s'.
  \end{cases} 
\end{equation}
Denote $\Delta_{\Phi}= \Delta_S \sqcup \Delta_{\Lambda}$. 

Given these, let us demonstrate how to compute log $\Ainf$-cohomology in terms of Koszul complexes. Let $U$ denote the (log) adic generic fiber of $\spf R$. For any $\bM\in \BKF^{\log}(\spf R,\vp)$, let $M_{\infty,\Phi}= \Gamma(U,\bM)\in \Rep_{\Gamma_\Phi}^\mu(\Ainf(R_{\infty,\Phi}))$. Then $A\Omega_{\fX}^{\log}(\bM)$ is computed by the Koszul complex
\[
    \eta_\mu R\Gamma_{\cont}(\Gamma_{\Phi},M_{\infty,\Phi}) \cong \Kos\Big(M_{\infty,\Phi}; \frac{\Delta_{\Phi}-1}{[\varepsilon]-1}\Big).
\]
Here, $\Kos\big(M_{\infty,\Phi}; \frac{\Delta_{\Phi}-1}{[\varepsilon]-1}\big)$ denotes the Koszul complex $\Kos\big(M_{\infty,\Phi}; \big\{\frac{\gamma-1}{[\varepsilon]-1}\big\}_{\gamma\in \Delta_{\Phi}}\big)$. 
Moreover, for each $\lambda \in \Lambda$, let $M_{\infty,\lambda}= \Gamma(\spa (R_{\infty,\lambda}[\frac 1p], R_{\infty,\lambda}),\bM)\in \Rep_{\Gamma_{\lambda}}^\mu(\Ainf(R_{\infty,\lambda}))$ be the corresponding generalized representation (cf. Theorem \ref{thm: BKF mod vs generalized representations}). By Theorem \ref{thm: MT 1.13}, there is an equivalence of categories
\[
    \Rep_{\Gamma_{\lambda}}^\mu(A(R_{\lambda})) \cong \Rep_{\Gamma_{\lambda}}^\mu(\Ainf(R_{\infty,\lambda}));
\]
let $M_\lambda \in \Rep_{\Gamma_{\lambda}}^\mu(A(R_{\lambda}))$ be the object on the left-hand side corresponding to $M_{\infty,\lambda}$. Then $A\Omega_{\fX}^{\log}(\bM)$ is also computed by 
\[
    \eta_\mu R\Gamma_{\cont}(\Gamma_\lambda,M_\lambda)\cong \Kos\Big(M_{\lambda}; \frac{\Delta_\lambda-1}{[\varepsilon]-1}\Big) \cong M_\lambda \otimes q\Omega^\bullet_{A(R_\lambda)/\Ainf}.
\]
where the second quasi-isomorphism follows from Theorem \ref{representations vs q-connections}. Note that the natural map $A(R_\lambda) \ra \Ainf(R_{\infty,\Phi})$ induces a quasi-isomorphism
\[
   \Kos\Big(M_{\lambda}; \frac{\Delta_\lambda-1}{[\varepsilon]-1}\Big) \simra \Kos\Big(M_{\infty,\Phi}; \frac{\Delta_{\Phi}-1}{[\varepsilon]-1}\Big).
\]

On the other hand, we can also compute the $q$-crystalline cohomology using the triple $\Phi= (S,P_\Lambda, \iota)$. For each $\lambda\in \lambda$, let $T_\lambda= \{e_{\lambda,i}\mid 1\le i \le d_{\lambda}\} \subseteq P_\lambda$. This induces a monoid homomorphism $\N^{T_\lambda} \ra P_\lambda$. Let $T= \bigsqcup_{\lambda\in \Lambda} T_\lambda$, then we obtain a surjection of pre-log rings
\[
    \Sigma= \Sigma(S, \ul T) \ra R_S^\square \otimes_{\mO}\Big(\bigotimes_{\lambda\in \Lambda} \ul{R^\square_\lambda}\Big) \ra \ul R.
\]
Consider the log perfectoid ring
\[
    \Sigma_\infty= \Sigma_\infty(S,T) :=  \big(\mO\langle(X_s^{\pm\frac{1}{\infty}})_{s\in S}, \Q_{\ge 0}^T\rangle, \Q_{\ge 0}^T\oplus N_\infty\big),
\]
which is a Kummer pro-\'etale Galois cover of $\Sigma(S, \ul T)$ with Galois group 
\[
    \Gamma_{\Sigma} = \Gamma_S\times \prod_{\lambda\in \Lambda}\Gamma_{T_{\lambda}} \cong \widehat{\Z}(1)^S \times \widehat{\Z}(1)^T.
\]
Here, $\Gamma_S$ and $\Gamma_{T_\lambda}$ denote the Galois group of $\Sigma_\infty (S,\ul\emptyset) \ra R$ and $\Sigma_\infty (\emptyset, \ul{T_\lambda}) \ra \ul{R_\lambda^\square} \ra R$, respectively. The Galois group $\Gamma_\Sigma$ naturally acts on $\Ainf(\Sigma_\infty)$. Note that the map $\{\N^{T_\lambda} \ra P_\lambda\}_{\lambda\in \Lambda}$ induces an isomorphism $\Gamma_{\Sigma} \cong \Gamma_\Phi$. Recall that $\Delta_S$ is the set of generators of $\Gamma_S \cong \widehat{\Z}(1)^S$ corresponding to elements in $S$ (cf. (\ref{action of Delta_S})). Similarly, let $\Delta_{T,\lambda}$ be the set of generators of $\Gamma_{T_\lambda} \cong \widehat{\Z}(1)^{T_\lambda}$ corresponding to elements in $T_\lambda$. Then, by the construction of $T_{\lambda}$, $\Delta_{T,\lambda}$ maps to $\Delta_{\lambda}$ under the isomorphism $\Gamma_\Sigma \cong \Gamma_\Phi$. Put
\[
    \Delta_{\Sigma}= \Delta_S\sqcup\Big(\bigsqcup_{\lambda\in \Lambda} \Delta_{T,\lambda}\Big)= \{\gamma_s \mid s\in S\}\sqcup\{\gamma_t \mid t\in T\}.
\]
Recall $A_0(\Sigma)= (\Ainf\langle (X_s)_{s\in S}, \ul T\rangle, \N^T\oplus N_\infty)$ defined in (\ref{defn of A_0(Sigma)}), and recall the ring $\mD(\Sigma)$ defined as the $(p,\mu)$-completed log $q$-PD envelope of the surjection $A_0(\Sigma) \ra R$. 
One checks that $A_0(\Sigma)$ satisfies ($q$DR1), ($q$DR2), and ($q$DR$\vp$) in Assumption \ref{assumption: qDR}, with respect to variables $(X_s)_{s\in S}\sqcup(X_t)_{t\in T}$ and $\Delta_{\Sigma}$. Consequently, we can define the log $q$-de Rham complex $q\Omega^\bullet_{A_0(\Sigma)/\Ainf}$ as in Definition \ref{defn: q-dR complex}. By \cite[Construction 16.20]{BS}, this extends to a log $q$-de Rham complex 
\[
    q\Omega^\bullet_{\mD_{\Sigma}/\Ainf}:= \mD_{\Sigma} \ra \mD_{\Sigma}\htimes q\Omega^1_{A_0(\Sigma)/\Ainf} \ra \mD_{\Sigma}\htimes q\Omega^2_{A_0(\Sigma)/\Ainf} \ra \cdots.
\]
Note that the reduction $q\Omega^\bullet_{\mD_{\Sigma}/\Ainf}/(q-1)$ coincides with the classical de Rham complex of the $p$-completed log PD envelope as in \cite[1.7]{Beilinson13}.

There is a unique $\Gamma_\Sigma$-action on $A_0(\Sigma)$ making $A_0(\Sigma) \ra \Ainf(\Sigma_\infty)$ into a $\Gamma_\Sigma$-equivariant morphism. Moreover, since $[\varepsilon]-1$ is mapped to $0$ in $R$, the kernel of the surjection $A_0(\Sigma) \ra R$ is stable under the $\Gamma_\Sigma$-action. Consequently, the $\Gamma_\Sigma$-action on $A_0(\Sigma)$ extends to a $\Gamma_\Sigma$-action on $\mD(\Sigma)$.

We are ready to construct the comparison map between log $q$-crystalline cohomology and log $\Ainf$-cohomology.
\begin{construction}\label{construction of map from qcrys to Ainf}
    Suppose $\fX= \spf \underline{R}$ is small affine, modeled on a small chart $u:N\rightarrow P$ with both $N$ and $P$ free. Let $\bM \in \BKF^{\log}(\fX, \varphi)$. Let $\bM_{\Prism}\in \Vect^{\vp}((\fX^{(1)}/(\ul{\Ainf},\txi))_{\Prism})$ be the corresponding prismatic $F$-crystal (cf. Corollary \ref{cor: equivalence between generalized representations and prismatic $F$-crystals}) and $\bM_{q\crys}=\sigma_{q\crys}^*\bM$ be the corresponding log $F$-$q$-crystal (Definition \ref{defn: qcrys specialization}). Consider a triple $\Phi=(S, P_\Lambda, \iota)$. 
    By \cite[Construction 4.7]{Koshikawa}, $(\mD_{\Sigma}, \txi)$ is a weakly final object in $(\fX^{(1)}/(\ul{\Ainf},\txi))_{\Prism}$. Let $M_{\infty,\Phi}= \bM(\Ainf(R_{\infty,\Phi}))\in \Rep_{\Gamma_\Phi}^\mu(\Ainf(R_{\infty,\Phi}))$, $M_{\Sigma}= \bM_{\Prism}(\mD_{\Sigma})= \bM_{q\crys}(\mD_{\Sigma})$, and
\[
M_\lambda= \bM(\Ainf(R_{\infty,\lambda})) \in \Rep_{\Gamma_\lambda}^\mu(\Ainf(R_{\infty,\lambda})) \cong\Rep_{\Gamma_\lambda}^\mu(A(R_{\lambda})), \ \forall \lambda \in \Lambda.
\]
By Proposition \ref{qcrys vs qdR}, $R\Gamma\big((R/\logqAinf)_{q\crys},\bM_{q\crys}\big)$ is computed by  
    \[
        M_{\Sigma}\otimes q\Omega^{\bullet}_{\mD_{\Sigma}/\Ainf} \cong \Kos\big(M_{\Sigma}, \frac{\Delta_\Sigma-1}{[\varepsilon]-1}\big).
    \]
    By the proof of \cite[Theorem 8.1]{Koshikawa}, there is a unique $\Gamma_{\Sigma}$-equivariant map $\mD_{\Sigma} \ra \Ainf(R_{\infty, \Phi})$ between $\delta_{\log}$-$\Ainf$-algebras, induced by the following commutative diagram
    \[
    \begin{tikzcd}
        {A_0(\Sigma)} & {\ul{\Ainf(R_{\infty,\Phi})}} \\
        {\Sigma(S, \ul T)} & {\ul{R_{\infty,\Phi}}.}
        \arrow[from=1-1, to=1-2]
        \arrow[from=1-1, to=2-1]
        \arrow[from=1-2, to=2-2]
        \arrow[from=2-1, to=2-2]
    \end{tikzcd}
    \]
    Here, the objects on the right-hand side are equipped with $\Gamma_{\Sigma}$-actions via the identification $\Gamma_\Sigma \simeq \Gamma_\Phi$. 
    Consider the evaluation map (with respect to the fixed free chart $P$)
    \[
        \ev_{A(R)} \colon \Vect^{\vp}((\fX^{(1)}/(\ul{\Ainf},\txi))_{\Prism}) \ra \BKF^{\log}(\fX,\vp).
    \]
    By Theorem \ref{thm: MT thm 5.15}, this map is independent of the choice of the chart $P$. On the one hand, by definition, we have $\ev_{A(R)}(\bM_{\Prism}) \cong \bM$. On the other hand, consider the log prism $(\Ainf(R_{\infty,\Phi}), \txi)\in (\fX^{(1)}/(\ul{\Ainf},\txi))_{\Prism}$. The unique map $\mD_{\Sigma} \ra \Ainf(R_{\infty,\Phi})$ between $\delta_{\log}$-algebras induces a map between log prisms $(\mD_{\Sigma},\txi) \ra (\Ainf(R_{\infty,\Phi}),\txi)$.
    By the proof of Theorem \ref{thm: MT thm 5.15}, we have 
    \begin{align*}
        M_{\infty,\Phi} & =  R\Gamma(\Ainf(R_{\infty,\Phi}), \ev_{A(R)}(\bM_{\Prism})) = \bM_{\Prism}((\Ainf(R_{\infty,\Phi}),\txi)) \\
        & \cong \bM_{\Prism}((\mD_{\Sigma},\txi)) \otimes_{\mD_{\Sigma}} \Ainf(R_{\infty,\Phi}) = M_{\Sigma} \otimes_{\mD_{\Sigma}} \Ainf(R_{\infty,\Phi}).
    \end{align*} 
    That is, the map $\mD_{\Sigma} \ra \Ainf(R_{\infty, \Phi})$ induces a canonical $\Gamma_{\Sigma}$-equivariant isomorphism 
    \begin{equation}\label{qcrystal BKFmodule base change}
        M_{\infty,\Phi} \cong M_{\Sigma}\otimes_{\mD_{\Sigma}}\Ainf(R_{\infty, \Phi}).
    \end{equation}
This then induces a natural morphism 
    \[
        M_{\Sigma} \otimes q\Omega^\bullet_{\mD_{\Sigma}/\Ainf} \cong \Kos\big(M_{\Sigma}; \frac{\Delta_\Sigma-1}{[\varepsilon]-1}\big) \ra  \Kos\big(M_{\infty,\Phi}; \frac{\Delta_{\Phi}-1}{q-1}\big).
    \]
    Finally, we define the comparison map $R\upsilon_*(\bM_{q\crys}) \ra A\Omega^{\log}_{\fX}(\bM)$ to be
    \[
        \varinjlim_{\Phi= (S,P_{\lambda}, \iota)}\Big(M_{\Sigma} \otimes q\Omega^\bullet_{\mD_{\Sigma}/\Ainf} \ra \Kos\Big(M_{\infty,\Phi}; \frac{\Delta_{\Phi}-1}{q-1}\Big)\Big).
    \]
    This is independent of the choice of coordinates, hence can be globalized to any admissibly smooth $p$-adic log formal scheme $\fX$ over $\ul \mO$.
    Consequently, we obtain the desired comparison map 
    \[
        R\upsilon_*(\bM_{q\crys}) \ra A\Omega^{\log}_{\fX}(\bM)
    \]
    in the global setting.
    Moreover, by construction, the morphism is compatible with Frobenius on both sides.
\end{construction}

\begin{proposition}\label{prop: q-crystalline comparison}
    Let $\fX$ be a saturated $p$-adic log formal scheme that is admissibly smooth over $\ul\mO$ and let $\bM\in \BKF^{\log}(\fX,\vp)$. Assume that $\fX$ locally admits free charts. Then there is a natural isomorphism 
    \[
        R\upsilon_*(\bM_{q\crys}) \cong A\Omega^{\log}_{\fX}(\bM).
    \]
    in $D(\fX_{\ett})$ compatible wit Frobenius.
\end{proposition}

\begin{proof}
    By construction \ref{construction of map from qcrys to Ainf}, there is at least a natural map
    \[
        R\upsilon_*(\bM_{q\crys}) \rightarrow A\Omega^{\log}_{\fX}(\bM).
    \]
    To check that it is an isomorphism, it suffices to restrict to a small affine \'etale open $\spf R \ra \fX$ modeled on a small chart $u:N\rightarrow P$ with both $N$ and $P$ free. By the same argument as the proof of Proposition \ref{prop: weakly final obj: free chart}, $(\ul {A(R)}, \xi)$ is a weakly final object in $(\spf R/(\ul \Ainf,\xi))_{q\crys}$. Let $\Gamma= \Hom(P^{\gp}/N^{\gp}, \widehat\Z(1))$. Let $M$ be the image of $\bM|_{\spf R}$ in $\Rep_{\Gamma}^\mu(A(R))$ under the equivalence in Theorem \ref{thm: MT 1.13} and Theorem \ref{thm: BKF mod vs generalized representations}. Then 
    \[
        M= \Gamma\big((\ul {A(R)}, \xi), \sigma_{q\crys}^*\bM\big).
    \]
    The same argument as Proposition \ref{qcrys vs qdR}, using the total complex connecting the $q$-dR complex and the \v{C}ech--Alexander complex, shows that the $q$-crystalline cohomology of $\sigma_{q\crys}^*\bM$ is computed by 
    \[
        R\Gamma((\spf R/(\ul\Ainf,\xi))_{q\crys},\sigma_{q\crys}^*\bM)\cong  M\otimes q\Omega_{A(R)/\Ainf}^{\bullet},
    \]
    which is naturally quasi-isomorphic to $L\eta_{\mu}R\Gamma_{\cont}(\Gamma, M)$ by Theorem \ref{representations vs q-connections}. This completes the proof.
\end{proof}

\begin{theorem}\label{thm: abs crystalline comparison}
    Let $\fX$ be a saturated $p$-adic log formal scheme that is admissibly smooth over $\ul\mO$ and let $\bM\in \BKF^{\log}(\fX,\vp)$. Assume that $\fX$ locally admits free charts. Then there is a natural isomorphism 
    \[
    A\Omega^{\log}_{\fX}(\bM) \htimes^{\L}_{\Ainf}\Acrys \cong R\upsilon_*(\sigma_{\Acrys}^* \bM)
    \]
    of complexes of sheaves on $\fX_{\ett}$ compatible wit Frobenius, where $\upsilon\colon (\fX_{\mO/p}/\ul\Acrys)_{\crys} \ra \fX_{\ett}$ denotes the natural projection of sites.
\end{theorem}

\begin{proof}
    Combine Proposition \ref{prop: q-crystalline comparison} and Proposition \ref{prop: q-crys vs abs crys}.
\end{proof}

\begin{corollary}\label{cor: crystalline comparison}
     Let $\fX$ be a saturated $p$-adic log formal scheme that is admissibly smooth over $\ul\mO$ and let $\bM\in \BKF^{\log}(\fX,\vp)$. Assume that $\fX$ locally admits free charts. Then there is a natural isomorphism 
    \[
        A\Omega^{\log}_{\fX}(\bM) \htimes^{\L}_{\Ainf}W(k) \cong R\upsilon_*(\sigma_{\crys}^* \bM)
    \]
    of complexes of sheaves on $\fX_{\ett}$ compatible wit Frobenius, where $\upsilon\colon (\fX_k/W(\underline{k}))_{\crys} \ra \fX_{\ett}$ denotes the natural projection of sites. In particular, suppose that $\fX$ is in additional proper, there is an isomorphism
    \[
        R\Gamma_{\Ainf}(\fX, \bM)\otimes^{\L}_{\Ainf}W(k) \cong R\Gamma_{\crys}\big(\fX_k/W(\ul k), \sigma^*_{\crys}\bM\big),
    \]
\end{corollary}

\begin{proof}
This follows from Theorem \ref{thm: abs crystalline comparison} and the base change theorem for log crystalline cohomology.
\end{proof}

\vspace{0.1in}
\subsection{Breuil--Kisin--Fargues modules}\label{subsection: classical BKF modules}
Finally, we finish the proof of Theorem \ref{thm: perfect complex}. The only thing left to show is that, when $\fX$ is proper, the cohomology groups $H^i_{\Ainf}(\fX, \bM)$ of $R\Gamma_{\Ainf}(\fX, \bM)$ are Breuil--Kisin--Fargues modules, whose definition we recall below.

\begin{definition}
A \emph{Breuil–Kisin–Fargues module} is a finitely presented $\Ainf$-module $M$ equipped with a Frobenius semilinear map $\varphi_M \colon M \ra M$, such that $M[\frac 1p]$ is a finite free $\Ainf[\frac 1p]$-module and that $\varphi_M$ induces an isomorphism
     \[
         \varphi_M \colon M[\frac {1}{\xi}] \xrightarrow[]{\sim} M[\frac{1}{\txi}].
     \]
\end{definition}

To prove this, we need some preparation. Let $\fX$ be a saturated $p$-adic log formal scheme that is admissibly smooth over $\ul\mO$. We further assume that $\fX$ is proper and locally admits free charts. By definition, $\fX$ admits a finite model $\fX_0$ (cf. \S \ref{subsection: defn of log BKF modules}) which fits into a Cartesian diagram
\begin{equation}\label{diagram: admissibly smooth 2}
\begin{tikzcd}
    \fX \arrow[r] \arrow[d] & \fX_0 \arrow[d] \\
    \spf(\mO, N_\infty)^a \arrow[r] & \spf(\mO, N)^a,
\end{tikzcd} 
\end{equation}
where $N\subset N_{\infty}$ is a toric submonoid, $\fX_0 \ra \spf(\mO, N)^a$ is a log smooth and saturated morphism of fs log formal schemes, and $\fX\rightarrow \fX_0$ is an isomorphism on the underlying $p$-adic formal schemes. Note that (\ref{diagram: admissibly smooth 2}) induces an equivalence of categories
\begin{equation}\label{eq: admissibly smooth 2}
    \Vect^{\vp}\big((\fX_{0,\mO/p}/(\Acrys, N))_{\crys}\big) \cong \Vect^{\vp}\big((\fX_{\mO/p}/(\Acrys, N_\infty))_{\crys}\big).
\end{equation}
By a slight abuse of notation, for any $F$-crystal $\mE\in \Vect^{\vp}\big((\fX_{\mO/p}/(\Acrys, N_\infty))_{\crys}\big)$, we still use $\mE$ to denote the corresponding object in $\Vect^{\vp}\big((\fX_{0,\mO/p}/(\Acrys, N))_{\crys}\big)$. Then the base change along $(\mO, N) \ra (\mO, N_\infty)$ induces an isomorphism
\begin{equation}\label{crys cohom N vs N_infty}
     R\Gamma_{\crys}\big(\fX_{0,\mO/p}/(\Acrys,N), \mE\big) \cong R\Gamma_{\crys}\big(\fX_{\mO/p}/(\Acrys, N_\infty), \mE\big),
\end{equation}
compatible with Frobenius actions on both sides.

Let $k$ denote the residue field of $\mO$ and let $\fX_{0,k}$ denote the special fiber of $\fX_0$ over $(k, N)$. Here, the pre-log structure on $k$ is induced from that on $\mO$. For any $\mE\in \Vect^{\vp}\big((\fX_{0,\mO/p}/(\Acrys, N))_{\crys}\big)$, taking the base change along $(\mO/p,N) \ra (k,N)$ yields a log $F$-crystal $\mE_k\in \Vect^{\vp}\big((\fX_{0,k}/(W(k), N))_{\crys}\big)$. 

For any $\bM \in \BKF(\fX,\vp)$, let $\mE= \sigma_{\Acrys}^*\bM$. According to \cite[Lemma 4.20]{BMS1}, to show that $H^i_{\Ainf}(\fX, \bM)$ is a Breuil--Kisin--Fargues module, we need to check that $H_{\Ainf}^i(\fX,\bM)\otimes_{\Ainf}B_\crys^+$ is finite projective over $B_\crys^+$. Using the crystalline comparison (Theorem \ref{thm: abs crystalline comparison}), it suffices to construct an isomorphism
\begin{equation}\label{eq: finiteness of crystalline cohomology}
    R\Gamma_{\crys}\big(\fX_{0,k}/(W(k), N), \mE_k\big) \otimes_{W(k)} B_\crys^+ \cong  R\Gamma_{\crys}\big(\fX_{\mO/p}/(\Acrys,N_\infty),\mE\big)[\frac 1p],
\end{equation}
then the desired result follows from properties of classical crystalline cohomology.

To construct \eqref{eq: finiteness of crystalline cohomology}, we fix a section $k\ra\mO/p$. We would like to apply the base change theorem of log crystalline cohomology along the section $k \ra \mO/p$. However, the section does not automatically extend to a morphism $(k,N)^a \ra (\mO/p,N)^a$ between log rings. To deal with this, we introduce new log structures on $\mO/p$, $\mO/p^{1/p^n}$, and $\Acrys$, for some integer $n\ge 0$.  We fix the following notation:
\begin{enumerate}
    \item Let $\alpha\colon N \ra \Acrys$ denote the original pre-log structure on $\Acrys$ induced from that on $\Ainf$. It induces pre-log structures $\alpha\colon N \ra \mO/p$, $\alpha\colon N \ra \mO/p^{1/p^n}$, and $\alpha\colon N \ra k$ through natural projections.
    \item Let $\beta\colon N \ra \Acrys$ denote the pre-log structure on $\Acrys$ such that $\beta(N\minus\{0\})=0$. It also induces pre-log structures $\beta\colon N \ra \mO/p$, $\beta\colon N \ra \mO/p^{1/p^n}$, and $\beta\colon N \ra k$ via projections. Let $\beta\colon N \ra W(k)$ denote the pre-log structure on $W(k)$ such that $\beta(N\minus\{0\})=0$.
\end{enumerate}
Notice that $(k,N,\alpha)^a$ coincides with $(k,N,\beta)^a$. 

\begin{proposition} \label{crys vs absolute crys}
    For any $F$-crystal $\mE$ on the log crystalline site $(\fX_{\mO/p}/(\Acrys, N_\infty))_{\crys}$, there is a (non-canonical) isomorphism\footnote{
    This isomorphism is not compatible with Galois actions and monodromy since the monodromy on the right hand side is trivial.}
    \[
        R\Gamma_{\crys}\big(\fX_{0,k}/(W(k), N, \beta), \mE_k\big) \otimes_{W(k)} B_\crys^+ \cong  R\Gamma_{\crys}\big(\fX_{\mO/p}/(\Acrys,N_\infty),\mE\big)[\frac 1p],
    \]
    which depends on the choice of section $k \ra \mO/p$. 
\end{proposition}
\begin{proof}
    The proof is essentially the same as in \cite[Proposition 8.9]{logprism}; we recall the proof for completeness. Notice that for sufficiently large $n$, the image of $N \minus \{0\}$ is $0$ in $\mO/p^{1/p^n}$ under $\alpha$; hence $\alpha$ coincides with $\beta$ on $\mO/p^{1/p^n}$. Therefore, the section $k \ra \mO/p$ can be upgraded to a morphism of log rings $(k, N, \beta)^a\rightarrow (\mO/p, N, \alpha)^a$ via the composition
    \[
        (k, N, \beta)^a \ra (\mO/p, N, \beta)^a \ra (\mO/p^{1/p^n}, N, \beta)^a\xrightarrow[]{\sim} (\mO/p^{1/p^n}, N, \alpha)^a \xrightarrow[]{\vp^n} (\mO/p, N, \alpha)^a.
    \]
    This further induces morphisms of log rings 
    \begin{equation}\label{eq: KY25 prop 8.9}
        (W(k), N, \beta)^a \ra (\Acrys, N, \beta)^a \xrightarrow[]{\vp^n} (\Acrys, N, \alpha)^a.
    \end{equation}    
    Since $\fX_0$ is fine and qcqs, for sufficiently large $n$, there is an isomorphism of $p$-adic log formal schemes
    \begin{equation}\label{eq: ad hoc isom}
        \fX_{0,k}\times_{\spf(k, N,\beta)^a} \spf(\mO/p^{1/p^n}, N,\beta)^a \cong \fX_0\times_{\spf(\mO, N,\alpha)^a} \spf(\mO/p^{1/p^n}, N, \alpha)^a
    \end{equation}
    over $(\mO/p^{1/p^n}, N, \alpha)^a = (\mO/p^{1/p^n}, N, \beta)^a$. We denote the left-hand side and the right-hand side by $(\fX_{0,k})_{\mO/p^{1/p^n}}$ and $\fX_{0,\mO/p^{1/p^n}}$, respectively. Recall that we still use $\mE$ to denote the $F$-crystal $\Vect^{\vp}\big((\fX_{0,\mO/p}/(\Acrys, N,\alpha))_{\crys}\big)$ corresponding to $\mE\in \Vect^{\vp}\big((\fX_{\mO/p}/(\Acrys, N_\infty))_{\crys}\big)$ via the equivalence (\ref{eq: admissibly smooth 2}). Pulling back along $(\mO/p,N,\alpha) \ra (\mO/p^{1/p^n}, N, \alpha)$, $\mE$ further induces a $F$-crystal in $\Vect^{\vp}\big((\fX_{0,\mO/p^{1/p^n}}/(\Acrys, N,\alpha))_{\crys}\big)$, which we still denote by $\mE$, by a slight abuse of notation. The isomorphism \eqref{eq: ad hoc isom} then induces an isomorphism
    \[
        R\Gamma_{\crys}\big((\fX_{0,k})_{\mO/p^{1/p^n}}/(\Acrys,N,\alpha),\mE\big)
        \xrightarrow[]{\sim} 
        R\Gamma_{\crys}\big(\fX_{0,\mO/p^{1/p^n}}/(\Acrys,N,\alpha),\mE\big).
    \]
    After taking Frobenius twists and inverting $p$, we arrive at
    \[
        R\Gamma_{\crys}\big((\fX_{0,k})_{\mO/p}/(\Acrys,N,\alpha),\mE\big)[\frac 1p]\xrightarrow[]{\sim} 
        R\Gamma_{\crys}\big(\fX_{0,\mO/p}/(\Acrys,N,\alpha),\mE\big)[\frac 1p].
    \]
    Finally, taking base change along (\ref{eq: KY25 prop 8.9}) and then combing with the isomorphism (\ref{crys cohom N vs N_infty}), we obtain
    \begin{align*}
         R\Gamma_{\crys}\big(\fX_{0,k}/(W(k), N,\beta), \mE_k\big) \otimes_{W(k)} B_\crys^+ & \xrightarrow[]{\sim} 
         R\Gamma_{\crys}\big((\fX_{0,k})_{\mO/p}/(\Acrys,N,\alpha),\mE\big)[\frac 1p]\\
         & \xrightarrow[]{\sim} R\Gamma_{\crys}\big(\fX_{0,\mO/p}/(\Acrys,N,\alpha),\mE\big)[\frac 1p]\\
         & \xrightarrow[]{\sim} R\Gamma_{\crys}\big(\fX_{\mO/p}/(\Acrys,N_\infty),\mE\big)[\frac 1p].
    \end{align*}
\end{proof}

\begin{proof}[Proof of Theorem \ref{thm: perfect complex}]
    By Corollary \ref{cor: perfect complex} , $R\Gamma_{\Ainf}(\fX, \bM)$ is a perfect complex. So we are left to show that all cohomology groups of $R\Gamma_{\Ainf}(\fX, \bM)$ are Breuil--Kisin--Fargues modules. Let $M= H^i_{\Ainf}(\fX,\bM)$. By \cite[Lemma 4.20]{BMS1}, it suffices to check
    \begin{itemize}
        \item $M[\frac{1}{p\mu}]$ is finite projective over $\Ainf[\frac{1}{p\mu}]$, and
        \item $M\otimes_{\Ainf}B_{\mathrm{crys}}^+$ is finite projective over $B_{\mathrm{crys}}^+$.
    \end{itemize}
    The first statement follows from the \'etale comparison (Theorem \ref{thm: strong etale comparison}), while the second statement follows from the absolute crystalline comparison (Theorem \ref{thm: abs crystalline comparison}), Proposition \ref{crys vs absolute crys}, and the finiteness of log crystalline cohomology.
\end{proof}

\vspace{0.3in}
\section{Semistable local systems}\label{section: semistable local systems}
\noindent The goal of this section is to recall the notion of semistable local systems on semistable $p$-adic formal schemes, as well as the prismatic interpretations of semistable local systems established by Du--Liu--Moon--Shimizu \cite{DLMS2}. 

In this section and in the next, let $K$ be a complete discrete valuation field of mixed characteristic $(0,p)$ with perfect residue field $k$. Let $\mO_K$ be the ring of integers and $\pi$ be a uniformizer. Let $E(u)\in W(k)[u]$ be the minimal polynomial of $\pi$. We equip $\mO_K$ with the pre-log structure $\N \to \mO_K$ sending $1 \mapsto \pi$. Let $\underline{\mO_K}= (\mO_K, \N \to \mO_K)$ be the corresponding pre-log ring. 

\begin{remark}
In this section and in the next, we will use $\fY$ to denote a semistable $p$-adic formal schemes over $\mO_K$. The base change $\fY_{\mathcal{O}}$ is then admissibly smooth over $\underline{\mathcal{O}}$ and locally admits free charts, where $\underline{\mathcal{O}}=(\mathcal{O}, N_{\infty})$ is any divisible perfectoid split log point extending $\underline{\mathcal{O}_K}$. Theorem \ref{thm: semistable comparison+ filtration} is formulated in this setting, but we expect the theorem to hold for a broader class of log smooth $p$-adic formal schemes, provided one generalizes the notion of semistable local systems appropriately to these situations. We leave such generalizations to the interested reader.
\end{remark}

\vspace{0.1in}
\subsection{Semistable local systems}\label{subsection: semistable local systems}
\noindent
\vspace{0.1in}

\noindent We begin by briefly recalling the notion of semistable local systems, following \cite{DLMS2} and \cite{Guo-Yang}. For integers $1\le m \le d$, consider
\[
    R^0 := \mO_K\langle T_1, \ldots, T_m, T^{\pm 1}_{m+1}, \ldots, T^{\pm 1}_d \rangle/(T_1\cdots T_m-\pi),
\]
equipped with the pre-log structure $\N^d \ra R^0$ sending the standard $i$-th basis $e_i$ to $T_i$ for $1\le i \le d$. Let $R$ be a connected $\mO_K$-algebra equipped with a ($p$-adically completed) \'etale map
\[
    \square \colon R^0 \ra R.
\]
In particular, $\spf \underline{R}=(\spf R, \,\N^d \to R^0 \xrightarrow[]{\square} R)^a$ is a $p$-adic (fs) log formal scheme that is log smooth over $\underline{\mathcal{O}_K}$.

\begin{definition}
A $p$-adic (fs) log formal scheme $\fY$ over $\mathcal{O}_K$ is called \emph{semistable small affine} if it is of the form $\fY= (\spf R, \,\N^d \to R^0 \xrightarrow[]{\square} R)^a$ as above. A $p$-adic (fs) log formal scheme $\fY$ over $\mathcal{O}_K$ is called \emph{semistable} if \'etale locally it is semistable small affine.
\end{definition}

\begin{remark}
If $\fY$ is a semistable $p$-adic formal scheme over $\mO_K$, then its log structure coincides with the divisorial log structure given by the special fiber $\fY_k$. In particular, it only possesses vertical log structures.
\end{remark}

Let $\fY$ be a semistable $p$-adic log formal scheme over $\mathcal{O}_K$ and let $Y$ be the adic generic fiber. Note that $Y$ is actually smooth over $K$. Let $Y_{\proet}$ denote the pro-\'etale site and let $\A_{\crys}$, $\B^+_{\crys}$, and $\B_{\crys}$ denote the crystalline period sheaves on $Y_{\proet}$ introduced in \cite[\S 4.1]{Guo-Yang} (also see \cite[\S 2A]{tantong}). 

We also consider the special fiber $\fY_k$, equipped with the pullback log structure from $\fY$. Let $\fY_{k,\crys}$ denote the absolute log crystalline site. Let $\Vect(\fY_{k,\crys})$ (resp. $\Vect^\vp(\fY_{k,\crys})$) denote the category of \emph{locally finite free crystals} (resp. \emph{locally finite free $F$-crystals}) on $\fY_{k,\crys}$. Let $\Isoc^\vp(\fY_{k,\crys})$ denote the category of locally finite free $F$-isocrystals on $\fY_{k,\crys}$, defined as the isogeny category of $\Vect^\vp(\fY_{k,\crys})$ (cf. \cite[Definition B.10]{DLMS2}). For simplicity, we will just say \emph{$F$-crystals} and \emph{$F$-isocrystals}. For any such $F$-isocrystal $\mE$ on $Y_{k,\crys}$, one can define a sheaf $\B^+_{\crys}(\mE)$ of finite projective $\B^+_{\crys}$-modules on $Y_{\proet}$ in terms of evaluation (for details, we refer the reader to \cite[Construction 4.1]{Guo-Yang}), then we put $\B_{\crys}(\mE)= \B^+_{\crys}(\mE) \otimes_{\B^+_{\crys}}\B_{\crys}$. 

\begin{definition}[{\cite[Definition 3.39]{DLMS2}}]\label{defn: association}
    Let $\fY$ be a semistable $p$-adic (log) formal scheme with adic generic fiber $Y$. Let $\L$ be an \'etale $\Z_p$-local system on $Y$ and let $\mE$ be an $F$-isocrystal on $\fY_{k,\crys}$.
    We say $\L$ and $\mE$ are \emph{associated} if there is a Frobenius equivariant isomorphism of $\B_\crys$-vector bundles
    \[
        \alpha_{\L, \mE} \colon \B_{\crys}(\mE) \simra \L\otimes_{\Z_p}\B_\crys.
    \]
\end{definition}

\begin{definition}[{\cite[Definition 4.4]{Guo-Yang}}]\label{defn: semistable local systems}
    Let $\fY$ and $Y$ be as above, and let $\L$ be an \'etale $\Z_p$-local system on $Y$. We say $\L$ is a \emph{semistable local system} if there exists a locally finite free $F$-isocrystal $\mE \in \Isoc^\vp(\fY_{k,\crys})$ that is associated with $\L$.
\end{definition}

\begin{remark}\label{remark: filtration on F-isocrystals}
Note that in Definition \ref{defn: association}, we do not require $\mE$ to be filtered. In fact, there is a unique way to put a filtered $F$-isocrystal structure on $\mE$ so that the ismorphism $\alpha_{\L, \mE}$ also respects filtrations. (Also see Proposition \ref{prop: F-isocrystal structure}.)
\end{remark}

\vspace{0.1in}
\subsection{Analytic prismatic $F$-crystals}\label{subsection: analytic prismatic F-crystals}
\noindent
\vspace{0.1in}

\noindent One key ingredient of our proof of semistable comparison is the prismatic interpretation of semistable local systems established in \cite{DLMS2}; namely, semistable local systems can be viewed as \emph{analytic $F$-crystals} on the corresponding (log) prismatic site. For the reader's convenience, we review the relevant material in the rest of this section.

\begin{definition}[{\cite[Definition 7.23]{logprism}}]
    Let $\fY$ be a bounded fs log formal scheme over $\underline{\mO_K}$. 
\begin{enumerate}
    \item An object in the \emph{absolute log prismatic site} $\fY_{\Prism}$ consists of a $\delta_{\log}$-triple $(A,I, \mM_A)$ together with a strict map $(\spf(A/I), \mM_{A/I}) \to \fY$ of log schemes, where $\mM_{A/I}$ denote the pullback log structure from $\mM_A$. There is an obvious notion of morphisms. We equipped $\fY_\Prism$ with the flat topology; i.e., a morphism 
    \[
        f\colon (\spf B, J, \mM_B) \to (\spf A, I, \mM_A)
    \]
    is a cover if $(\spf B, \mM_B) \to (\spf A, \mM_A)$ is a $(p,I)$-completely faithfully flat map between log formal schemes.
    \item Let $\mO_{\Prism}$ (resp. $\mI_{\Prism}$) be the structure sheaf (resp. the ideal sheaf of the Hodge--Tate divisor) on $\fY_{\Prism}$ defined by associating to an object $(\spf A, I, \mM_A)$ the ring $A$ (resp. the ideal $I$). 
\end{enumerate}
\end{definition}

\begin{remark}
This notion of absolute log prismatic site $\fY_{\Prism}$ is actually called the \emph{strict absolute log prismatic site} in \cite[Definition 2.3]{DLMS2}. It is the full subcategory of the \emph{absolute log prismatic site} in \emph{loc. cit.} consisting of strict objects. As noted in \cite[Remark 3.6]{DLMS2}, when $\fY$ is semistable, the categories of analytic prismatic $F$-crystals (cf. Definition \ref{defn: analytic prismatic F-crystals} below) on these two sites are equivalent. 
\end{remark}

\begin{definition}[{\cite[Definition 3.3]{DLMS2}}]
For any object $(\spf A, I, \mM_A)$ of $\fY_{\Prism}$, let $\vp_A$ be the Frobenius on $A$ induced by the $\delta$-structure. Notice that $\vp_A$ acts on $\spec(A)\minus V(p,I)$.
\begin{enumerate}
    \item Let $\Vect^\vp(A,I)$ denote the category of pairs $(\fM, \vp_{\fM})$, where
    $\fM$ is a finite projective $A$-module and
    \[
        \vp_{\fM}\colon \vp^*\fM[I^{-1}]\xrightarrow[]{\sim} \fM[I^{-1}]
    \]
     is an isomorphism of $A$-modules.
    \item Let $\Vect^{\an,\vp}(A,I)$ denote the category of pairs $(\mE, \vp_{\mE})$, where $\mE$ is a vector bundle on $\spec(A) \minus V(p,I)$ and
    \[
        \vp_{\mE}\colon \vp^*\mE[I^{-1}]\xrightarrow[]{\sim} \mE[I^{-1}]
    \]
    is an isomorphism of vector bundles.
    \item Let $\Vect(A[I^{-1}]_p^{\wedge})^{\vp=1}$ denote the category of pairs $(\mM, \vp_{\mM})$, where $\mM$ is a finite projective $A[I^{-1}]_p^{\wedge}$-module and $\vp_{\mM} \colon \vp^*\mM \to \mM$
    is an isomorphism of $A[I^{-1}]_p^{\wedge}$-modules.
    \item Let $D_{\perf}^\vp(A,I)$ be the $\infty$-category of pairs $(\mE,\vp_E)$, where $\mE$ is a perfect complex over $A$, and
    \[
        \vp_{\mE} \colon \vp_A^*\mE\otimes^\L_A A[I^{-1}] \ra \mE\otimes^\L_A A[I^{-1}]
    \]
    is an isomorphism of perfect complexes.
    \end{enumerate}
\end{definition}

\begin{definition}[{\cite[Definition 3.3, Definition 3.12]{DLMS2}}]\label{defn: analytic prismatic F-crystals}
Let $\fY$ be a bounded fs $p$-adic log formal scheme over $\ul{\mO_K}$.
    \begin{enumerate}
    \item A \emph{prismatic $F$-crystal} over $\fY$ is a vector bundle $\mE$ on the ringed site $(\fY_{\Prism},\mO_{\Prism})$ equipped with an isomorphism
    \[
        \vp_{\mE}\colon \vp^*\mE[I^{-1}]\xrightarrow[]{\sim} \mE[I^{-1}].
    \]
    Let $\Vect^\vp(\fY_{\Prism})$ denote the category of prismatic $F$-crystals. According to \cite[Proposition 3.2]{DLMS2}, there is an equivalence of categories
    \[
        \Vect^\vp(\fY_{\Prism}) \cong \lim\limits_{(\spf A,  I, \mM_A)\in \fY_\Prism} \Vect^\vp(A,I).
    \]
    \item The category of \emph{analytic prismatic $F$-crystals} over $\fY$ is defined to be 
    \[
        \Vect^{\an,\vp}(\fY_{\Prism}) := \lim\limits_{(\spf A,  I, \mM_A)\in \fY_\Prism} \Vect^{\an,\vp}(A,I).
    \]
    \item A \emph{Laurent $F$-crystal} on $\fY_{\Prism}$ is a locally finite free crystal $\mE$ on the ringed site $\big(\fY_{\Prism},\mO_{\Prism}[\mI_{\Prism}^{-1}]_p^{\wedge}\big)$ equipped with an isomorphism $\vp_{\mE} \colon \vp^* \mE \xrightarrow[]{\sim} \mE$.
    We denote the category of Laurent $F$-crystals on $\fY_{\Prism}$ by $\Vect\big(\fY_{\Prism},\mO_{\Prism}[\mI_{\Prism}^{-1}]_p^{\wedge}\big)^{\vp=1}$. According to \cite[Lemma 3.13]{DLMS2}, there is an equivalence of categories
    \[
        \Vect\big(\fY_{\Prism},\mO_{\Prism}[\mI_{\Prism}^{-1}]_p^{\wedge}\big)^{\vp=1} \cong \lim\limits_{(\spf A,  I, \mM_A)\in \fY_\Prism} \Vect(A[I^{-1}]_p^{\wedge})^{\vp=1}.
    \]
    \item The $\infty$-category of \emph{prismatic $F$-crystals in perfect complexes} over $\fY$ is defined by 
    \[
        D_{\perf}^\vp(\fY_{\Prism}) := \lim\limits_{(\spf A,  I, \mM_A)\in \fY_\Prism} D_{\perf}^\vp(A,I).
    \]
\end{enumerate}
\end{definition}

\vspace{0.1in}
\subsection{Breuil--Kisin log prism and Kisin descent data}\label{subsection: Breuil--Kisin log prisms}
\noindent
\vspace{0.1in}

\noindent Recall that we have fixed a uniformizer $\pi$ of $\mO_K$ with $E(u)\in W(k)[u]$ being its minimal polynomial. Consider the \emph{Breuil--Kisin ring} $\fS= W(k)[\![ u]\!]$, equipped with the pre-log structure $\N \ra \fS$ sending $1$ to $u$, together with a Frobenius structure $\vp \colon \fS \ra \fS$ sending $u$ to $u^p$. The natural map $\fS \ra \mO_K$ sending $u$ to $\pi$ induces a morphism between pre-log rings $\ul\fS=(\fS, \N) \ra \ul{\mO_K}$.

Consider a semistable small affine $\spf \underline{R}= (\spf R, \N^d \to R^0 \xrightarrow[]{\square} R)^a$ as in \S \ref{subsection: semistable local systems}, where 
\[
    R^0 = \mO_K\langle T_1, \ldots, T_m, T^{\pm 1}_{m+1}, \ldots, T^{\pm 1}_d \rangle/(T_1\cdots T_m-\pi),
\]
equipped with the pre-log structure $\N^d \ra R^0$ sending the standard $i$-th basis $e_i$ to $T_i$ for $1\le i \le d$ and $R$ is a connected $\mO_K$-algebra equipped with a ($p$-adically completed) \'etale map $\square \colon R^0 \ra R$. Consider the ``$\fS$-deformation''
\[
     \fS_{R^0}= W(k)\langle T_1, \ldots, T_m, T^{\pm 1}_{m+1}, \ldots, T^{\pm 1}_d \rangle[\![u ]\!]/(T_1\cdots T_m-u),
\]
equipped with the pre-log structure $\N^d \to \fS_{R^0}$ sending $e_i$ to $T_i$, which lifts the pre-log structure on $R^0$. There is a Frobenius structure $\vp_{\fS_{R^0}} \colon \fS_{R^0} \ra \fS_{R^0}$ sending $T_i$ to $T_i^p$, which is compatible with the Frobenius structure on $\fS$. Since $\square$ is $p$-adically completed \'etale, it uniquely lifts to \footnote{The ring $\fS_R$ is also denoted by $\fS_{\square, R}$ in \cite[\S 2.2]{DLMS2}.} 
\[
    \square_{\fS} \colon \fS_{R^0} \ra \fS_R.
\]
The Frobenius $\vp_{\fS_{R^0}}$ also lifts to $\vp_{\fS_R} \colon \fS_R \ra \fS_R$.

\begin{definition}
    The \emph{Breuil--Kisin log prism} $(\fS_R, E(u), \mM_{\fS})$ is the log prism that associates to $(\fS_R, E(u),\N^d)$ with $\delta_{\log}(e_i)=0$ and $\delta(T_j)=0$. It is an object of $(\spf\underline{R})_{\Prism}$ via the canonical isomorphism $R \xrightarrow[]{\sim} \fS_R/(E(u))$.
\end{definition}

\begin{proposition}[{\cite[Lemma 2.8]{DLMS2}}]
Let $\spf\underline{R}$ be semistable small affine as above. Then $(\fS_R, E(u), \mM_{\fS})$ is a cover of the final object of the topos associated to $(\spf\underline{R})_{\Prism}$. 
\end{proposition}

Let $(\fS_R(n), E(u), \N^d)^a$ be the $(n+1)$-st self-product of $(\fS_R, E(u), \N^d)^a$ in $(\spf\underline{R})_{\Prism}^{\mathrm{op}}$, where the pre-log structure $\N^d \to \fS_R(n)$ is given by the pre-log structure on the first coordinate. As in \cite[\S 2]{DLMS2}, $\fS_R(1)$ can be explicitly described as
\[
    \fS_R(1)= B(1)\Bigg\{\frac{1-\frac{T_{1,2}}{T_{1,1}}}{E(u)}, \ldots, \frac{1-\frac{T_{d,2}}{T_{d,1}}}{E(u)}\Bigg\}_{\delta}^{\wedge},
\]
where $B(1)= \fS_R[\![1- \frac{T_{i,2}}{T_{i,1}}]\!]_{1\le i \le d}$.
Let $p_1, p_2 \colon \fS_R \ra \fS_R(1)$ denote the two projection maps.
By \cite[Lemma 3.5]{BS}), we must have $p_1(E(u))\fS_R(1)= p_2(E(u))\fS_R(1)$. Therefore, the natural projections restrict to 
\[
    p_1, p_2 \colon \spec(\fS_R(1))\minus V(p,E) \lra \spec(\fS_R)\minus V(p,E).
\]

\begin{definition}[{\cite[Definition 3.7]{DLMS2}}]
    Let $\DD_{\fS_R}$ be the category of \emph{Kisin descent data}. An object of $\DD_{\fS_R}$ is a triple $(\fM, \vp_{\fM}, \varepsilon)$, where
    \begin{itemize}
        \item $\fM$ is a torsion free finite $\fS_R$-module such that $\fM[p^{-1}]$ is projective over $\fS_R[p^{-1}]$, $\fM[E^{-1}]$ is projective over $\fS_R[E^{-1}]$, and $\fM= \fM[p^{-1}]\cap \fM[E^{-1}]$,
        \item $\vp_{\fM}\colon \vp_{\fS_R}^*\fM[E^{-1}] \xrightarrow[]{\sim} \fM[E^{-1}]$ is an isomorphism of $\fS_R[E^{-1}]$-modules,
        \item $\varepsilon \colon \fM\otimes_{\fS_R,p_1}\fS_R(1) \xrightarrow[]{\simeq} \fM\otimes_{\fS_R,p_2}\fS_R(1)$ is an isomorphism of $\fS_R(1)$-modules that satisfies the cocycle condition over $\fS_R(2)$ and is compatible with Frobenii after inverting $E$.
    \end{itemize}
\end{definition}

The following proposition reveals that, \'etale locally on a semistable $p$-adic formal scheme, we can use the category of Kisin descent data to describe analytic prismatic $F$-crystals.

\begin{proposition}[{\cite[Lemma 3.8]{DLMS2}}]\label{prop: DLMS lem 3.8}  Let $\spf \underline{R}$ be a semistable small affine $p$-adic log formal scheme over $\mO_K$ as above. Let $U := \spec(\fS_R)\minus V(p,E)$. Then restriction along the open immersion $j\colon U \hookrightarrow \spec(\fS_R)$ induces an equivalence of categories
    \[
        j_*\colon \DD_{\fS_R} \xrightarrow{\sim} \Vect^{\an, \vp}((\spf\underline{R})_{\Prism}).
    \]
\end{proposition}

For later use, we include the following analogue of \cite[Theorem 5.10]{Guo-Reinecke}, which will allow us to relate analytic prismatic $F$-crystals with relative log BKF-modules in \S \ref{section: Cst comparison}.

\begin{proposition}\label{prop: GR Thm 5.10}
    Let $\fY$ be a semistable $p$-adic (log) formal scheme over $\mO_K$. Then taking pushforward along the inclusion $j_{A} \colon \spec(A)\minus V(p,I) \hookrightarrow \spec(A)$ for all log prisms $(A,I,\mM_A) \in \fY_\Prism$ induces a fully faithful functor 
    \[
        j_* \colon \Vect^{\an, \vp} (\fY_{\Prism}) \ra D_{\perf}^{\vp}(\fY_{\Prism})
    \]
\end{proposition}

\begin{proof} The proof is similar to that of \cite[Theorem 5.10]{Guo-Reinecke}. Firstly, by \'etale descent, we can assume $\fY= \spf \underline{R}$ is semistable small affine. Recall that, by fixing a framing, the Breuil--Kisin prism $(\fS_R, E(u), \N^d)^a$ gives rise to a covering of the final object of $\fY_{\Prism}$. Consider the $(n+1)$-th self-product $(\fS_R(n),E(u), \N^d)^a$ of $(\fS_R, E(u), \N^d)^a$ in the log prismatic site $\fY_\Prism$. By \cite[Corollary 2.12]{DLMS2}, all coface maps $\fS_R \ra \fS_R(n)$ are classically faithfully flat. By $(p,I)$-completely faithfully flat descent for vector bundles and perfect complexes, we have
    \begin{align*}
        \Vect^{\an,\vp}(\fY_{\Prism}) & \cong \lim\limits_{[n]\in \Delta[2]} \Vect^{\an,\vp}(\fS_R(n)), \\
        D_{\perf}^\vp(\fY_{\Prism}) & \cong \lim\limits_{[n]\in \Delta[2]} D_{\perf}^\vp(\fS_R(n)).
    \end{align*}
    By the same argument as in \cite[Theorem 5.10]{Guo-Reinecke}, we can define compatible functors of $\infty$-categories
    \[
        F_n \colon \Vect^{\an,\vp}(\fS_R(n)) \ra D_{\perf}^\vp(\fS_R(n))
    \]
    for all $n$. More precisely, let $j_n \colon \spec(\fS_R(n))\minus V(p,I) \hookrightarrow \spec(\fS_R(n))$ be the open immersion. Then for $(\mE_n, \vp_{\mE_n})\in \Vect^{\an,\vp}(\fS_R(n))$, the same proof of \cite[Lemma 5.8]{Guo-Reinecke} implies that $j_{n,*}\mE_n$ is a coherent sheaf on $\spec(\fS_R(n))$. We put
    \[
        F_n(\mE_n, \vp_{\mE_n}) := (j_{n,*}\mE_n, \vp_{\mE_n}) \in D_{\perf}^\vp(\fS_R(n)).
    \]
    The proof of fully faithfulness in \emph{loc. cit.} also applies here, so we are left to show the construction is independent of choices of framing. Let $\square$ and $\square^\prime$ be two framings, and let $(\fS_\square, E(u), \mM_{\square})$ and  $(\fS_{\square^\prime}, E(u), \mM_{\square^\prime})$ be the corresponding Breuil--Kisin prisms. Let $(\fS_{\square, \square^\prime},E(u), \mM_{\square, \square^\prime})$ be the coproduct of $(\fS_\square, E(u), \mM_{\square})$ and $(\fS_{\square^\prime}, E(u), \mM_{\square^\prime})$ in $\fY_{\Prism}$, provided by \cite[Lemma 2.9]{DLMS2}. Given two maps of prisms
    \[
        (\fS_{\square}, E(u), \mM_{\square}) \ra (B,I,\mM_B) \leftarrow (\fS_{\square^\prime}, E(u), \mM_{\square^\prime}),
    \]
    we obtain a unique map $(\fS_{\square, \square^\prime},E(u), \mM_{\square, \square^\prime}) \ra (B,I,\mM_B)$. For any $(\mE, \vp_{\mE})\in \Vect^{\an, \vp}(\fY_{\Prism})$, let $\mE_{\square}$ and $\mE_{\square^\prime}$ be the restriction of $\mE$ on $\spec(\fS_{\square})\minus V(p,I)$ and $\spec(\fS_{\square^\prime})\minus V(p,I)$, respectively. Let $j_\square$, $j_{\square^\prime}$, and $j_{\square,\square^{\prime}}$ denote the corresponding open immersions. By the crystal property and flat base change, we obtain
\[
        (j_{\square,*}\mE_{\square}) \otimes^\L_{\fS_{\square}}\fS_{\square, \square^\prime} \cong j_{\square,\square^\prime,*}(\mE_{\square}\otimes^\L_{\fS_{\square}}\fS_{\square,\square^\prime})
        \cong j_{\square,\square^\prime,*}(\mE_{\square^\prime}\otimes^\L_{\fS_{\square^\prime}}\fS_{\square,\square^\prime}) \cong (j_{\square^\prime,*}\mE_{\square^\prime}) \otimes^\L_{\fS_{\square^\prime}}\fS_{\square, \square^\prime},
\]
as desired.
\end{proof}

\vspace{0.1in}
\subsection{Semistable local systems as analytic prismatic $F$-crystals} 
\noindent
\vspace{0.1in}

\noindent Let $\fY$ be a semistable $p$-adic log formal scheme over $\mO_K$ and let $Y$ be its adic generic fiber. Let $\Loc_{\Z_p}(Y)$ be the category \'etale $\Z_p$-local systems on $Y$. The following result asserts that \'etale $\Z_p$-local systems can be viewed as prismatic Laurent $F$-crystals.

\begin{theorem}[{\cite[Theorem 6]{logprism},\cite[Theorem 3.14]{DLMS2}}]\label{KY thm6} 
There is a natural equivalence of categories
    \[
        \Vect\big(\fY_{\Prism},\mO_{\Prism}[\mI_{\Prism}^{-1}]_p^{\wedge}\big)^{\vp=1} \cong \Loc_{\Z_p}(Y)
    \]
where $\Vect\big(\fY_{\Prism},\mO_{\Prism}[\mI_{\Prism}^{-1}]_p^{\wedge}\big)^{\vp=1}$ denotes the category of prismatic Laurent $F$-crystals (see \cite[Definition 7.34]{logprism}).
\end{theorem}

\begin{definition}
    For any log prism $(A, I, \mM_A)$ in $\fY_\Prism$, there is a natural functor $\Vect^{\an,\vp}(A,I) \to \Vect(A[I^{-1}]_p^{\wedge})^{\vp=1}$ induced by base change, which then induces a functor
     \begin{equation}\label{eq: from analytic crystal to Laurent crystal}
         \Vect^{\an,\vp}(\fY_{\Prism}) \ra \Vect\big(\fY_{\Prism},\mO_{\Prism}[\mI_{\Prism}^{-1}]_p^{\wedge}\big)^{\vp=1}.
     \end{equation}
     Via the equivalence in Theorem \ref{KY thm6}, we define the \emph{\'etale realization} functor to be the composition
     \[
         T_{\ett}\colon \Vect^{\an,\vp}(\fY_{\Prism}) \ra \Vect\big(\fY_{\Prism},\mO_{\Prism}[\mI_{\Prism}^{-1}]_p^{\wedge}\big)^{\vp=1} \xrightarrow{\sim} \Loc_{\Z_p}(Y).
     \]
\end{definition}

\begin{proposition}[{\cite[Proposition 3.20]{DLMS2}}] 
Let $\fY$ be a semistable $p$-adic log formal scheme over $\mO_K$. The functor \eqref{eq: from analytic crystal to Laurent crystal} is fully faithful.
\end{proposition}

The following theorem is one of the main results of \cite{DLMS2}, which classifies semistable local systems in terms of analytic prismatic $F$-crystals.

\begin{theorem}[{\cite[Corollaty 5.2]{DLMS2}}]\label{thm: prismatic description of semistable local systems}
An \'etale $\Z_p$-local system on $Y$ is semistable if and only if it lies in the essential image of $T_{\ett}$. In particular, $T_{\ett}$ induces an equivalence of categories
    \[
    T_{\ett}\colon \Vect^{\an,\vp}(\fY_{\Prism})\cong \Loc^{\mathrm{st}}_{\Z_p}(Y)
    \]
where $\Loc^{\mathrm{st}}_{\Z_p}(Y)$ denote the category of semistable local systems on $Y$.
\end{theorem}

Let $\ul{\mO}= (\mO_C, N_\infty)$ be a divisible perfectiod split log point extending $\ul{\mO_K}= (\mO_K, \N)$. Let $\fY_{\mO}$ be the base change of $\fY$ along $\ul{\mO_K} \ra \ul{\mO}$ with adic generic fiber $Y_C$. Then we have equivalences of categories (cf. Theorem \ref{thm: MT thm 5.15}) 
    \begin{equation}\label{equiv of BKF mod and prismatic crystal after base change}
    \Vect^{\vp}(\fY^{(1)}_{\mO, \Prism})\cong  \Vect^{\vp}\big(\big(\fY^{(1)}_{\mO}/(\ul\Ainf, \txi)\big)_{\Prism}\big) \cong   \BKF^{\log}(\fY_{\mO},\vp).
    \end{equation}
Composing \eqref{equiv of BKF mod and prismatic crystal after base change} with the base change functor $\Vect^{\vp}(\fY_{\Prism}) \ra \Vect^{\vp}(\fY^{(1)}_{\mO, \Prism})$, we obtain a natural functor 
\begin{equation}\label{eq: prismatic over K to BKF}
        \Vect^{\vp}(\fY_{\Prism}) \ra \BKF^{\log}(\fY_{\mO},\vp)
\end{equation}
from the category of prismatic $F$-crystals on $\fY$ to that of relative log BKF modules over $\fY_{\mO}$. \'Etale locally, when $\fY= \spf R$ is semistable small affine, let $R_{\mO}= R\otimes_{\mO_K}\mO$. Then $\spf R_{\mO}$ is also small affine in the sense of Definition \ref{definition: small affine}. There is a natural morphism 
\begin{equation}\label{eq: fS_R to A(R)}
        (\fS_R, E(u), \N) \ra (A(R_{\mO}), \txi, N_\infty)
\end{equation} 
between log prisms, which sends $u$ to $\vp([\pi^\flat])$. Notice that $(\fS_R, E(u), \N)$  and $(A(R_{\mO}), \txi, N_\infty)$ are weakly final objects of $\fY_{\Prism}$ and $(\fY_{\mO}^{(1)}/(\ul{\Ainf},\txi)))_{\Prism}$, respectively. Hence, by Theorem \ref{thm: fully faithful}, the functor (\ref{eq: prismatic over K to BKF}) is locally induced by (\ref{eq: fS_R to A(R)}). 

Recall the \'etale specialization functor 
\[
    \sigma_{\ett}^*\colon \BKF^{\log}(\fY_{\mO},\vp) \ra \Loc_{\Z_p}(Y_C)
\]
from Definition \ref{defn: etale specialization}. The following proposition shows that $\sigma_{\ett}^*$ is compatible with the \'etale realization functor $T_{\ett}$.

\begin{proposition}\label{compatibility of etale realization}
The functors $\sigma_{\ett}^*$ and $T_{\ett}$ fit into the following commutative diagram
\begin{equation}\label{diagram: etale specialization vs etale realization}
\begin{tikzcd}
    \Vect^{\vp}(\fY_{\Prism}) \arrow[r] \arrow[d, "{\eqref{eq: prismatic over K to BKF}}"] &  \Vect^{\mathrm{an},\vp}(\fY_{\Prism})\arrow[r, "T_{\ett}"] &  \Loc_{\Z_p}(Y) \arrow[d]\\
    \BKF^{\log}(\fY_{\mO},\vp) \arrow[rr, "\sigma^*_{\ett}"] & & \Loc_{\Z_p}(Y_C).
\end{tikzcd} 
\end{equation}
\end{proposition}

\begin{proof}
This follows from a careflul unwinding of definitions. Let $\Perfd_{Y_C}$ denote the full subcategory of $Y_{C,\proet}$ consisting of affinoid perfectoid objects that contain all $p$-power roots of unity. By \cite[Theorem 7.35]{logprism} and its proof, there are equivalences of categories
    \begin{align*} 
        \Vect\big(\fY_{\mO, \Prism},\mO_{\Prism}[\mI_{\Prism}^{-1}]_p^{\wedge}\big)^{\vp=1} & \cong \lim\limits_{(A,I,\mM_A) \in \fY^{\perf}_{\mO,\Prism}} \Vect(A[I^{-1}]_p^{\wedge})^{\vp=1} \\
        & \cong \lim\limits_{{\spa(S,S^+)}\in \Perfd_{Y_C}} \Vect(W(S^\flat))^{\vp=1}.
    \end{align*} 
On the side of local systems, we have
    \[
        \Loc_{\Z_p}(Y_C) \cong \lim\limits_{\spa(S,S^+) \in \Perfd_{Y_C}} \Loc_{\Z_p}(\spa(S,S^+)).
    \]
Then the equivalence of categories (cf. Theorem \ref{KY thm6})
\[
\Vect\big(\fY_{\mO, \Prism},\mO_{\Prism}[\mI_{\Prism}^{-1}]_p^{\wedge}\big)^{\vp=1}\cong \Loc_{\Z_p}(Y_C)
\]
is induced by
    \begin{align*}
             \Vect(W(S^\flat))^{\vp=1} &\simra\Loc_{\Z_p}(\spa(S,S^+)))\\
             (\mE, \vp_{\mE}) &\longmapsto \mE^{\vp_{\mE}=1}. 
    \end{align*} 
as in the proof of Proposition \ref{Frobenius invariant}. By construction (cf. Definition \ref{defn: etale specialization}), the \'etale specialization functor
    \[
        \sigma_{\ett}^*\colon \BKF^{\log}(\fY_{\mO}, \vp) \ra \Loc_{\Z_p}(Y_C)
    \]
factors through the category of vector bundles of $\widehat{\A_{\inf,Y_C}[\frac 1\mu]}$-modules with Frobenius structures, which is denoted by $\Vect(W(\widehat\mO_{Y_C^\flat}))^{\vp=1}$. We have
    \[
        \Vect(W(\widehat\mO_{Y_C^\flat}))^{\vp=1} \cong \lim\limits_{\spa(S,S^+) \in \Perfd_{Y_C}} \Vect(W(S^\flat))^{\vp=1}.
    \]
Consequently, the desired commutativity follows from the following commutative diagram
    \small{\[
    \begin{tikzcd}
        \BKF^{\log}(\fY_{\mO},\vp) \arrow[r, "\cong"] \arrow[d] & \Vect^{\vp}(\fY^{(1)}_{\mO,\Prism}) \arrow[d]& \Vect^{\vp}(\fY_{\Prism}) \arrow[d] \arrow[l]\\
         \Vect\big(W(\widehat\mO_{Y_C^\flat})\big)^{\vp=1} \arrow[d, "\cong"] & \Vect\big(\fY^{(1)}_{\mO,\Prism},\mO_{\Prism}[\mI_{\Prism}^{-1}]_p^{\wedge}\big)^{\vp=1} \arrow[d, "\cong"] & 
         \Vect\big(\fY_{\Prism},\mO_{\Prism}[\mI_{\Prism}^{-1}]_p^{\wedge}\big)^{\vp=1} \arrow[l] \arrow[d]\\
         \lim\limits_{\Perfd_{Y_C}} \Vect(W(S^\flat))^{\vp=1}  \arrow[r, "\vp^*"] \arrow[rd] & \lim\limits_{\Perfd_{Y_C}} \Vect(W(S^\flat)^{(1)})^{\vp=1} \arrow[d, "\cong"] & \Loc_{\Z_p}(Y) \arrow[ld]\\
         &  \Loc_{\Z_p}(Y_C).
    \end{tikzcd}
    \]}
\end{proof}

\vspace{0.1in}
\subsection{Crystalline realization}
\noindent
\vspace{0.1in}

\noindent Let $\ul\mS= (\mS,\N)$ denote the $p$-adic completion of the log PD-envelope of $\ul\fS= (W(k)[\![u]\!], \N)$ with respect to $(E(u))$. Let $\fY= \spf R$ be a semistable small affine $p$-adic log formal scheme over $\mO_K$. Consider $\ul{\fS_R}$ associated with a framing $\square$. Define $\ul{\mS_R}$ as the $p$-adic completion of the log PD-envelope of $\ul{\fS_R}$ with respect to the kernel of the exact surjection $\fS_R \ra R/p$. It is equipped with a pre-log structure $\N^d$ sending $e_i \mapsto T_i$. Let $\fY_{\mO_K/p}$ be the base change of $\fY$ along $\mO_K \ra \mO_K/p$.

\begin{proposition}[{\cite[Lemma 3.28, Example B.22]{DLMS2}}]
The triple $(R/p, \mS_R, \N^d)^a$ is a weakly final ind-object in $\fY_{\mO_K/p,\crys}$. Consequently, there is an equivalence between $\Vect(\fY_{\mO_K/p,\crys})$ and the category of finite projective $\fS_R$-modules together with HPD-stratifications (cf. \cite[Definition B.23]{DLMS2}). 
\end{proposition} 

\begin{remark}\label{absolute crys vs relative crys}
\begin{enumerate}
    \item By \cite[Remark B.20]{DLMS2}, after passing to the isogeny categories, the natural projection $\mO_K/p \ra k$ induces an equivalence between the category of ($F$-)isocrystals on $\fY_{\mO_K/p,\crys}$ and that on $\fY_{k,\crys}$.
    \item Note that $\ul{\mS_R}$ admits a natural map from $\ul\mS$, and hence the category of ($F$-)crystals on the absolute log crystalline site $\fY_{\mO_K/p,\crys}$ is equivalent to that of the relative log crystalline site $(\fY_{\mO_K/p}/\ul\mS)_{\crys}$.
\end{enumerate}
\end{remark}
For each $n\in \N$, let $(R/p, \mS_R(n),\N^d)^a$ be the $(n+1)$-st self-product of $(R/p, \mS_R, \N^d)^a$ in $\fY_{\mO_K/p,\crys}$. By \cite[Lemma 2.39]{BS}, we have an isomorphism
\[
    \mS_R \cong \fS_R\left\{\frac{\vp(E(u))}{p}\right\}_{\delta},
\]
where $\{-\}_\delta$ denotes the process of adjoining elements as $\delta$-ring.
Following the convention in \cite{DLMS2}, we consider the log prism $(\fS_R, \vp(E(u)), \vp^*\N^d)^a \in \fY_\Prism$ where the unfortunate notation ``$\vp^*\N^d$'' stands for the pre-log structure given by the composition $\N^d\rightarrow \fS_R\xrightarrow[]{\varphi} \fS_R$ which sends $e_i \mapsto T_i^p$. The Frobenius $\vp$ on $\fS_R$ induces a morphism 
\[
    \vp \colon (\fS_R, E(u), \N^d)^a \ra (\fS_R, \vp(E(u)), \vp^*\N^d)^a
\]
of log prisms in $\fY_\Prism$. Since $\vp(E(u)) \in (p)$ in $\mS_R$, the natural inclusion $\fS_R \ra \mS_R$ induces a morphism  $(\fS_R, \vp(E(u))) \ra (\mS_R, (p))$ between prisms,\footnote{Note that $\vp(E(u))/p$ is a unit in $\mS_R$ by \cite[Lemma 2.24]{BS}).} which further extends to a map of log prisms
\begin{equation}\label{eq: fS vs mS}
\vp \colon (\fS_R, E(u), \N^d)^a \xrightarrow[]{\varphi} (\fS_R, \vp(E(u)), \vp^*\N^d)^a\ra (\mS_R, (p), \vp^*\N^d)^a
\end{equation}
where ``$\vp^*\N^d$'' in the last term stands for the pre-log structure given by the composition $\N^d\rightarrow \mS_R\xrightarrow[]{\varphi} \mS_R$.
Moreover, for each $n\ge 1$, we still use $\varphi^*\N^d$ to denote the pullback pre-log structure of $\varphi^*\N^d$ along $\mS_R\rightarrow \mS_R(n)$. Then \eqref{eq: fS vs mS} induces a map
\[
    \vp \colon (\fS_R(n), E(u), \N^d)^a \ra (\mS_R(n), (p), \vp^*\N^d)^a
\]
of log prisms.

Let $(\mE, \vp_{\mE})$ be any analytic prismatic $F$-crystal on $\fY_{\Prism}$ and let $(\fM, \vp_{\fM}, \varepsilon) \in \DD_{\fS_R}$ be the Kisin descent datum associated with $(\mE, \vp_{\mE})$ as in Proposition \ref{prop: DLMS lem 3.8}. For each $n\ge 1$, put
\[
    \mM_n = \fM/p^n\fM \otimes_{\fS_R,\vp} \mS_R.
\]
Via the map $\vp \colon \fS_R(n)\ra \mS_R(n)$, $\varepsilon$ induces an isomorphism
\[
    \eta_n \colon \mM_n \otimes_{\mS_R, p_1} \mS_R(1) \xrightarrow[]{\sim} \mM_n \otimes_{\mS_R, p_2} \mS_R(1),
\]
which satisfies the cocycle condition over $\mS_R(2)$. Therefore, 
\[
    (\mM, \eta ) := \varprojlim_n (\mM_n, \eta_n)
\]
defines a crystal in $\CR(\fY_{\mO_K/p,\crys})$. Moreover, the Frobenius structure $\vp_{\fM}$ induces an isomorphism $\vp_{\mM} \colon \vp^*\mM[p^{-1}] \xrightarrow[]{\sim} \mM[p^{-1}]$, and hence $(\mM[p^{-1}], \varphi_{\mM})$ defines an $F$-isocrystal on $\fY_{\mO_K/p,\crys}$. By \cite[Lemma 3.30]{DLMS2}, this construction is independent of the choice of framing $\square$. 

The construction above can be globalized as follows.

\begin{proposition}[{\cite[Corollary 3.31]{DLMS2}}]\label{prop: crystalline realization}
    Let $\fY$ be a semistable $p$-adic formal scheme over $\mO_K$. For any $(\mE, \vp_{\mE}) \in \Vect^{\an, \vp}(\fY_\Prism)$, there exists a unique crystal $\mE_{\crys}$ on $\fY_{k,\crys}$ such that for each affine small open $\spf R$ of $\fY$, the restriction of $\mE_{\crys}$ on $(R/p, \fS_R, \mM_\fS)$ coincides with $(\mM, \eta )$ constructed above. Moreover, passing to the isogeny category, we obtain an $F$-isocrystal $(\mE_{\crys,\Q}, \vp_{\mE_{\crys,\Q}})$ on $\fY_{\mO_K/p,\crys}$.
\end{proposition}
    
\begin{definition}\label{defn: crystalline realization}
    Let $\fY$ be a semistable $p$-adic formal scheme over $\mO_K$. We define the \emph{crystalline realization functor} as
    \begin{align*}
        D_{\crys} \colon \Vect^{\an, \vp}(\fY_{\Prism}) &\lra \Isoc^{\vp}(\fY_{\mO_K/p,\crys}) \simeq  \Isoc^{\vp}(\fY_{k,\crys})\\
        (\mE,\vp_\mE) &\longmapsto (\mE_{\crys,\Q},\vp_{\mE_{\crys,\Q}}).
    \end{align*}
\end{definition}

\begin{proposition}[{\cite[Proposition 3.42]{DLMS2}}]
    Let $\fY$ be as above. For any analytic prismatic $F$-crystal $\mE$ over $\fY$, its \'etale realization $T_{\ett}(\mE)$ and its crystalline realization $\mE_{\crys, \Q}$ are associated in the sense of Definition \ref{defn: association}. 
\end{proposition}

To wrap up the section, we discuss filtrations on the $F$-isocrystals $\mE_{\crys, \Q}$ following \cite{DLMS2} (cf. Remark \ref{remark: filtration on F-isocrystals}). Let $\fY$ be a semistable $p$-adic log formal scheme over $\mathcal{O}_K$ as above, with adic generic fiber $Y$. By \cite[Proposition B.30, B.32]{DLMS2}, there is a canonical functor 
\begin{equation}\label{eq: associate integrable connections to F-isocrystals}
         \Isoc^{\vp}(\fY_{k,\crys}) \ra \MIC(Y),
\end{equation}
where $\MIC(Y)$ stands for the category of vector bundles with integrable connection on $Y$.

\begin{definition}
Let $\fY$ be a semistable $p$-adic log formal scheme over $\mO_K$, with adic generic fiber $Y$. A \emph{filtered $F$-isocrystal} on $\fY$ is a pair $\big((\mE, \vp_\mE),(E, \nabla_E, \Fil^\bullet E)\big)$ where 
    \begin{itemize}
        \item $(\mE, \vp_\mE)$ is an $F$-isocrystal on $\fY_{k,\crys}$;
        \item $(E, \nabla_E)$ is the vector bundle with integral connection on $Y$ attached to $(\mE, \vp_\mE)$ via \eqref{eq: associate integrable connections to F-isocrystals};
        \item $\Fil^\bullet E$ is a $\Z$-indexed separated and exhaustive decreasing filtration of $\mO_Y$-submodules of $E$ satisfying Griffiths transversality such that $\Fil^i E/\Fil^{i+1} E$ is locally free over $\mO_Y$ for each $i$.
    \end{itemize}
\end{definition}

Now, let $\mE$ be an analytic prismatic $F$-crystal on $\fY_{\Prism}$. Let $\L=T_{\ett}(\mE)$ be the associated semistable local system (see Theorem \ref{thm: prismatic description of semistable local systems}) and let $D_{\crys}(\mE)=(\mE_{\crys,\Q},\vp_{\mE_{\crys,\Q}})$ be the associated $F$-isocrystal on $\fY_{k,\crys}$ given by the crystalline realization functor (see Definition \ref{defn: crystalline realization}). We can describe the canonical filtration on $D_{\crys}(\mE)$ in terms of $\L$. For this, consider the \emph{de Rham period sheaf} $\mO\B_{\dR}$ on $Y_{\proet}$ defined in in \cite[definition 6.8]{Scholze12}, equipped with a natural integrable connection and a natural filtration. Then $p$-adic Riemann-Hilbert functor $D_{\mathrm{dR}}$ of Liu--Zhu in \cite{LiuZhu} yields a vector bundle 
\[D_{\dR}(\L)\coloneqq \nu_*\left(\L\otimes_{\widehat{\Z}_p}\mO\B_{\dR}\right)\]
on $Y$ of the same rank as $\L$, where $\nu\colon Y_{\proet} \ra Y_{\mathrm{an}}$ is the natural projection. Moreover, $D_{\dR}(\L)$ is equipped with an integrable connection $\nabla_{D_{\dR}(\L)}$ and a filtration $\Fil^{\bullet}_{D_{\dR}(\L)}$ inherited from those of $\mO\B_{\dR}$, which satisfies Griffiths transversality.

\begin{proposition}[{\cite[Corollary 3.46]{DLMS2}}]\label{prop: F-isocrystal structure}
     Let $\fY$ be a semistable $p$-adic log formal scheme over $\mO_K$ with adic generic fiber $Y$. Let $\mE$ be an analytic prismatic $F$-crystal on $\fY_\Prism$. Let $\L=T_{\ett}(\mE)$ be the associated semistable local system and let $D_{\crys}(\mE)$ be the associated $F$-isocrystal on $\fY_{k, \crys}$.
\begin{enumerate}
         \item Let $(E,\nabla_E)$ be the vector bundle with integral connection on $Y$ attached to $D_{\crys}(\mE)$ via \eqref{eq: associate integrable connections to F-isocrystals}. Then
         \[
             (E,\nabla_E) \cong \left(D_{\dR}(\L), \nabla_{D_{\dR}(\L)}\right).
         \]
         \item The pair $\left(D_{\crys}(\mE), \left(D_{\dR}(\L), \nabla_{D_{\dR}(\L)}, \Fil^{\bullet}_{D_{\dR}(\L)}\right)\right)$ is a filtered $F$-isocrystal on $\fY$.
     \end{enumerate} 
\end{proposition}

\vspace{0.3in}
\section{Semistable comparison theorem with coefficients}\label{section: Cst comparison}
\noindent In this section, we finally prove the semistable comparison theorem (Theorem \ref{mainthm}); namely, the $C_{\mathrm{st}}$ conjecture with coefficients. Intuitively, the specialization functors and comparison theorems of log $\Ainf$-cohomology (cf. Theorem \ref{thm: comparisons}) should serve as a bridge connecting \'etale and log crystalline cohomology. In particular, if an \'etale $\Z_p$-local system $\L$ is precisely the \'etale specialization of a relative log BKF module $\bM$, we should already obtain a comparison result between the \'etale cohomology of $\L$ and the log crystalline cohomology of $\sigma_{A_\crys}^*\bM$. However, in general, not every semistable local system arises in this way. We therefore need to work within a larger category that contains relative log BKF modules. To deal with this issue, we introduce the category of \emph{derived relative Breuil--Kisin--Fargues modules}.

\vspace{0.1in}
\subsection{Derived relative Breuil--Kisin--Fargues modules}
\noindent
\vspace{0.1in}

\noindent Let $\fX$ be an admissibly smooth $p$-adic log formal scheme over $\ul\mO=(\mO_C, N_\infty)$ and let $X$ be its adic generic fiber.
In this section, we work under the assumption that $\fX$ locally admits free charts (cf. Definition \ref{defn: locally admits free charts}), namely, $\fX$ is locally modeled on small charts $u:N\rightarrow P$ where both $N$ and $P$ are free (cf. Definition \ref{definition: small affine}).
We use $D_{\perf}\big((\fX^{(1)}/\logAinf)_\Prism\big)$ to denote the category of perfect complexes of prismatic crystals on $(\fX^{(1)}/\logAinf)_\Prism$. Let $D_{\perf}(\Ainfx)$ be the category of perfect complexes of $\Ainfx$-modules on $X_{\proet}$.
By Theorem \ref{thm: MT thm 5.15}, there is a fully faithful functor
\[
    \Vect\big((\fX^{(1)}/\logAinf)_\Prism\big) \hookrightarrow \BKF^{\log}(\fX) \hookrightarrow \Vect(\Ainfx),
\]
which induces a fully faithful functor between derived categories:
\begin{equation}\label{eq: beta Prism}
    \beta_{\Prism}\colon D_{\perf}\big((\fX^{(1)}/\logAinf)_\Prism\big) \hookrightarrow D_{\perf}(\Ainfx).
\end{equation}

\begin{definition}\label{defn: derived BKF modules}
    Let $\fX$ be an admissibly smooth $p$-adic log formal scheme over $\ul \mO=(\mO, N_\infty)$ with adic generic fiber $X$. Assume that $\fX$ locally admits free charts.
    \begin{enumerate}
    \item Let $\bM$ be a perfect complex of $\Ainfx$-modules on $X_{\proket}$. We say that $\bM$ is \emph{trivial modulo $\xi_r$} if there is a perfect complex of $\nu_*(\Ainfx/\xi_r)$-modules $\overline{\bM}_r$ together with an isomorphism
    \[
        u_r \colon \nu^{-1}\overline{\bM}_r \otimes^{\L}_{\nu^{-1}\nu_*(\Ainfx/\xi_r)} \Ainfx/\xi_r \simra \bM/\xi_r
    \]
    We say that $\bM$ is \emph{trivial modulo $<\mu$} if it is trivial modulo $\xi_r$ for all $r \ge 1$, and the modules $\{\overline{\bM}_r\}_{r\ge 1}$ can be chosen to be compatible in the sense that
    \begin{itemize}
        \item for each $r$, there is a restriction map $\rho_r \colon \overline{\bM}_{r+1} \ra \overline{\bM}_r$ that fits into a commutative diagram
        \[
        \begin{tikzcd}
            \nu^{-1}\overline{\bM}_{r+1} \otimes^{\L}_{\nu^{-1}\nu_*(\Ainfx/\xi_{r+1})} \Ainfx/\xi_{r+1} \arrow[r,"u_{r+1}"] \arrow[d,"\rho_r"'] &\bM/\xi_{r+1} \arrow[d] \\
            \nu^{-1}\overline{\bM}_r \otimes^{\L}_{\nu^{-1}\nu_*(\Ainfx/\xi_r)} \Ainfx/\xi_r \arrow[r,"u_r"] & \bM/\xi_r,
        \end{tikzcd} 
        \]
        \item the map $\rho_r$ induces an isomorphism
        \[
            \overline{\bM}_{r+1} \otimes^{\L}_{\nu_*(\Ainfx/\xi_{r+1})} \nu_*(\Ainfx/\xi_r) \xrightarrow[]{\sim} \overline{\bM}_r.
        \]
    \end{itemize}
    \item A \emph{derived relative logarithmic Breuil--Kisin--Fargues module without Frobenius} over $\fX$ is a triple $\mM = (\bM, \bV, \iota)$ such that
    \begin{itemize}
        \item $\bM$ is a sheaf of $\Ainfx$-modules in perfect complexes lying in the essential image of $\beta_{\Prism}$;
        \item $\bV$ is a sheaf of $\Ainfx$-modules such that both $\bV[\frac 1\xi]$ and $\bV[\frac 1\txi]$ are locally finite free;
        \item $\iota \colon \bM\otimes^{\L}_{\Ainfx}\Ainfx[\frac 1\txi] \xrightarrow[]{\sim} \bV[\frac 1\txi]$ is an isomorphism in $D_{\perf}(\Ainfx[\txi^{-1}])$.
    \end{itemize} 
    The category of such objects is denoted by $\DBKF^{\log}(\fX)$.

    \item A \emph{derived relative logarithmic Breuil--Kisin--Fargues module} over $\fX$ is defined to be a pair $(\mM, \vp_{\mM})$, where 
    \begin{itemize}
        \item $\mM = (\bM, \bV, \iota) \in \DBKF^{\log}(\fX)$ is a derived relative logarithmic Breuil--Kisin--Fargues module without Frobenius;
        \item $\vp_{\mM}= \varphi_{\bV} \colon \varphi^* \bV [\txi^{-1}]\ra \bV[\txi^{-1}]$
        is an isomorphism of $\Ainfx[\txi^{-1}]$-modules.\footnote{Composing with $\iota$, $\vp_\bV$ induces an isomorphism $\vp_\bM \colon \vp^*\bM\otimes_{\Ainfx}^\L\Ainfx[\frac 1\txi]\xrightarrow[]{\sim} \bM\otimes_{\Ainfx}^\L\Ainfx[\frac 1\txi]$.}
    \end{itemize} 
    \end{enumerate}
    The category of such objects is denoted by $\DBKF^{\log}(\fX,\vp)$. When $\iota$ and the Frobenius structures are clear from the context, we often denote an object of $\DBKF^{\log}(\fX,\vp)$ as $(\mM,\vp_\mM)= (\bM,\bV, \vp_\bV)$ or even $\mM= (\bM, \bV)$.
\end{definition}

\begin{remark}
    For any locally finite free $\Ainfx$-module $\bM$, it is trivial modulo $<\mu$ as a sheaf of $\Ainfx$-modules if and only if it is trivial modulo $<\mu$ as an object of $D_{\perf}(\Ainfx)$.
\end{remark}

\begin{lemma}
    Let $\fX$ be an admissibly smooth $p$-adic log formal scheme over $\ul\mO$ and assume that $\fX$ locally admits free charts. If $\mM=(\bM, \bV, \iota)$ is an object of $\DBKF^{\log}(\fX)$, then $\bM$ is trivial modulo $<\mu$.
\end{lemma}

\begin{proof}
   By definition and Theorem \ref{thm: MT thm 5.15}, $\bM$ is quasi-isomorphic to a bounded complex
    \[
        0\ra \bM_1\ra \bM_2 \ra \cdots \ra \bM_n \ra 0
    \]
    where each $\bM_i$ is a relative log BKF module. Note that $\bM/\xi_r$ is computed by 
    \[
        0\ra \bM_1/\xi_r\ra \bM_2/\xi_r \ra \cdots \ra \bM_r/\xi_r \ra 0.
    \]
    Since each $\bM_i$ is trivial modulo $\xi_r$, we can simply take 
    \[
    \overline{\bM}_{r}:= 0 \ra \nu_*(\bM_1/\xi_r) \ra \cdots \ra \nu_*(\bM_n/\xi_r) \ra 0.
    \]
    \end{proof}

Now, we can define log $\Ainf$-cohomology attached to a derived relative log BKF module.
\begin{definition}
Let $\fX$ be as above with adic generic fiber $X$. For any $\mM= (\bM,\bV) \in \DBKF^{\log}(\fX, \varphi)$, we define
\[
A\Omega^{\log}_{\fX}(\mM):= L\eta_{\mu}(\widehat{R\nu_*\bM})\in D(\fX_{\ett}, \Ainf)
\]
where the completion is the derived $p$-adic completion, $\nu\colon X_{\proet} \ra X_{\ett}$ is the natural projection of sites, and $L\eta$ stands for the d\'ecalage functor. The \emph{log $\Ainf$-cohomology of $\mM$} is defined to be the complex
\[
    R\Gamma_{\Ainf}(\fX, \mM):=R\Gamma\big(\fX_{\ett}, A\Omega^{\log}_{\fX}(\mM)\big)\in D(\Ainf).
\]
\end{definition}

Now, let us consider the derived \'etale and crystalline specialization functors. 

\begin{corollary}
     Let $\fX$ be as above with adic generic fiber $X$ and let $(\mM, \vp_{\mM}) =(\bM, \bV, \vp_{\bV})$ be an object of $\DBKF^{\log}(\fX,\vp)$. Then 
     \[
         \L := (\bV\otimes_{\Ainfx}W(\widehat\mO_{X^{\flat}}))^{\vp_{\bV}=1}
     \]
    is a locally finite free $\widehat{\Z}_p$-sheaf on $X_{\proket}$ of the same rank of $\bV[\frac{1}{\xi}]$. Moreover, we have $\L \subseteq \bV[\frac 1\mu]$ and $\L\otimes_{\widehat{\Z}_p}\Ainfx[\frac 1\mu]= \bV[\frac 1\mu]$.
\end{corollary}

\begin{proof}
   This is a direct corollary of Proposition \ref{Frobenius invariant}.    
\end{proof}

\begin{definition}
    Let $\fX$ be an admissibly smooth $p$-adic log formal scheme over $\ul\mO$ with adic generic fiber $X$. Assume that $\fX$ locally admits free charts. The \emph{derived \'etale specialization functor} 
    \[
        \sigma^*_{\ett}: \DBKF^{\log}(\fX, \varphi)\rightarrow \Loc_{\Z_p}(X)\\
    \]
    is defined by \[\sigma^*_{\ett}(\mM, \vp_{\mM})= \sigma^*_{\ett}(\bM, \bV, \vp_{\bV}) \coloneqq (\bV\otimes_{\Ainfx}W(\widehat\mO_{X^{\flat}}))^{\vp_{\bV}=1}.\]
\end{definition}

For crystalline specialization, note that the functor $\gamma_q$ in (\ref{q-crystalline site to prismatic site}) also induces a functor 
\[
    \gamma_q^*\colon D_{\perf}^\vp\big((\fX^{(1)}/\logAinf)_\Prism\big) \ra D_{\perf}^\vp\big((\fX/\logqAinf)_{q\crys}\big),
\]
where the right-hand side stands for the category of $F$-$q$-crystals in perfect complexes on the $q$-crystalline site $(\fX/\logqAinf)_{q\crys}$. Then the $q$-crystalline specialization functor extends to the derived setting as
\begin{align*}
        \DBKF^{\log}(\fX,\vp) &\lra D_{\perf}^\vp\big((\fX/\logqAinf)_{q\crys}\big)\\
        \mM=(\bM, \bV) &\longmapsto \bM_{q\crys}= \gamma_q^*\bM.
\end{align*}

\begin{definition}\label{defn: derived crystalline specialization }
     Let $\fX$ be an admissibly smooth $p$-adic log formal scheme over $\ul\mO$ and let $\fX_{\mO/p}$ denote its base change over $\mO/p$. Assume that $\fX$ locally admits free charts.
     \begin{enumerate}
         \item The \emph{derived absolute crystalline specialization functor}
         \[
              \sigma_{A_\crys}^*\colon \DBKF^{\log}(\fX,\vp) \ra D_{\perf}^{\vp}\big((\fX_{\mO/p}/\ul{\Acrys})_{\crys}\big)
         \]
         is defined as follows: for any $(\mM, \vp_{\mM}) = (\bM, \bV, \vp_{\bV})\in \DBKF^{\log}(\fX,\vp)$, we adopt the notation as in (\ref{qcrystal to crystal over Acrys}) and define $\sigma^*_{\Acrys}(\mM)$ locally as 
         \[\sigma^*_{\Acrys}(\mM)(\mD_{\crys,\Sigma}(\bullet)) := \bM_{q\crys}(\mD_{\Sigma}(\bullet))\htimes^{\L}_{\Ainf}\Acrys\]
         
         \item The \emph{derived crystalline specialization functor} 
            \[
                \sigma_{\crys}^*\colon \DBKF^{\log}(\fX,\vp) \ra D_{\perf}^{\vp}\big((\fX_k/W(\ul k))_{\crys}\big)
            \]
         is defined as follows: for any $(\mM, \vp_{\mM}) = (\bM, \bV, \vp_{\bV})\in \DBKF^{\log}(\fX,\vp)$, we define $\sigma_{\crys}^*(\mM)$ to be the image of $\sigma^*_{\Acrys}(\mM)$ under the functor $D_{\perf}^{\vp}\big((\fX_{\mO/p}/\ul{\Acrys})_{\crys}\big)\rightarrow D_{\perf}^{\vp}(\fX_k/W(\ul k))_{\crys}$ induced by the natural map $\ul\Acrys \to W(\ul k)$.
         
     \end{enumerate}  
\end{definition}

We also have \'etale and crystalline comparison theorems in the derived setting.

\begin{proposition}\label{prop: etale comparison for derived BKF modules}
    Let $\fX$ be an admissibly smooth $p$-adic log formal scheme over $\ul\mO$ with adic generic fiber $X$. Let $(\mM, \vp_{\mM})= (\bM, \bV, \vp_{\bV}, \iota) \in \DBKF^{\log}(\fX,\vp)$. 
    \begin{enumerate}
        \item There are natural isomorphism
        \[
        \widehat{\Big(A\Omega^{\log}_{\fX}(\mM)[\frac{1}{\mu}]\Big)}^{\varphi_{\bV}=1} \cong R\nu_*(\sigma^*_{\ett}\mM).
        \]
        \item Assume that $C$ is algebraically closed and $\fX$ is proper. The identification \[\sigma_{\ett}^*\mM \otimes_{\widehat{\Z}_p} \Ainfx[\frac 1\mu]\cong \bV[\frac 1\mu ]\] induces a natural isomorphism
    \[
        R\Gamma_{\ket}(\sigma^*_{\ett}\mM) \otimes^{\L}_{\Z_p} \Ainf[\frac{1}{\mu}] \cong R\Gamma_{\Ainf}(\mM)\otimes^{\L}_{\Ainf}\Ainf[\frac{1}{\mu}].
    \]
    \end{enumerate}
\end{proposition}

\begin{proof}
    Since taking Frobenius invariants commutes with the derived pushforward,   
    \begin{align*}
       \widehat{\Big(A\Omega^{\log}_{\fX}(\mM)[\frac{1}{\mu}]\Big)}^{\varphi_{\bV}=1} &= \Big(R\nu_*\big(\bM \otimes_{\Ainfx}W(\widehat{\mO}_{X^\flat})\big)\Big)^{\varphi_\bV=1}\\
       & \cong \Big(R\nu_*\big(\bV\otimes_{\Ainfx}W(\widehat{\mO}_{X^\flat}\big)\Big)^{\varphi_{\bV}=1} \\
       & = R\nu_*(\sigma^*_{\ett}\mM).  
    \end{align*}
    This proves (1), where the second isomorphism is induced by $\iota$. The proof of (2) is the same as that of Theorem \ref{thm: strong etale comparison}.
\end{proof}

\begin{proposition}\label{prop: derived q-crys vs Ainf}
    Let $\fX$ be an admissibly smooth $p$-adic log formal scheme over $\ul\mO$. Assume that $\fX$ locally admits free charts.
    Let $\mM= (\bM, \bV) \in \DBKF^{\log}(\fX,\vp)$. There is a natural isomorphism 
    \[
        A\Omega_{\fX}^{\log}(\mM) \cong R\upsilon_*(\bM_{q\crys})
    \]
    in $D(\fX_\ett)$ compatible with Frobenius, where $\upsilon \colon (\fX/(\ul{\Ainf},\xi))_{q\crys} \ra \fX_{\ett}$ is the natural projection of sites.
\end{proposition}

\begin{proof}
    The question is \'etale local on $\fX$, so we can assume that $\fX= \spf \underline{R}$ is small affine. We adopt the notation in \S \ref{subsection: crystalline comparison} and choose a triple $\Phi= (S,P_\Lambda, \iota)$. For each $\lambda\in \lambda$, let $T_\lambda= \{e_{\lambda,i}\mid 1\le i \le d_{\lambda}\} \subseteq P_\lambda$. This induces a map $\N^{T_\lambda} \ra P_\lambda$. Let $T= \bigsqcup_{\lambda\in \Lambda} T_\lambda$, then we obtain a surjection of pre-log rings
    \[
        \Sigma= \Sigma(S, \ul T) \ra R_S^\square \otimes_{\mO}\Big(\bigotimes_{\lambda\in \Lambda} \ul{R^\square_\lambda}\Big) \ra \ul R.
    \]
    Let $A_0(\Sigma)= (\Ainf\langle (X_s)_{s\in S}, \ul T\rangle, \N^T\oplus N_\infty)$ defined in (\ref{defn of A_0(Sigma)}), and let $\mD(\Sigma)$ be the log $q$-PD envelope of the surjection $A_0(\Sigma) \ra R$.
    Let $M_{\infty,\Phi}= \bM(\Ainf(R_{\infty,\Phi}))\in \Rep_{\Gamma_\Phi}^\mu(\Ainf(R_{\infty,\Phi}))$.
    Choose a bounded complex 
    \[
        0 \ra \bM_1 \ra \bM_2 \ra \cdots \ra \bM_n \ra 0
    \]
    of log Breuil--Kisin--Fargues modules that is quasi-isomorphic to $\bM$. Let $M_i= \bM_{i,q\crys}(\mD_\Sigma)$. As in Construction \ref{construction of map from qcrys to Ainf}, there is a natural functorial quasi-isomorphisms
    \[
        M_{i,\Sigma} \otimes q\Omega^\bullet_{\mD_{\Sigma}/\Ainf} \xrightarrow[]{\sim} \Kos\Big(M_{i,\infty,\Phi}; \frac{\Delta_{\Phi}-1}{q-1}\Big), 
    \]
    for each $i=1,\ldots, n$. This leads to a natural quasi-isomorphism
    \[
        \varinjlim_{\Phi= (S,P_{\lambda}, \iota)}\mathrm{Tot}\big(M_{\bullet,\Sigma} \otimes q\Omega^\bullet_{\mD_{\Sigma}/\Ainf}\big) \ra \varinjlim_{\Phi= (S,P_{\lambda}, \iota)}\mathrm{Tot}\Big(\Kos\Big(M_{\bullet, \infty,\Phi}; \frac{\Delta_{\Phi}-1}{q-1}\Big)\Big),
    \]
    where the left-hand side and the right-hand side compute $R\upsilon_*(\bM_{q\crys})$ and $A\Omega_{\fX}^{\log}(\mM)$, respectively. 
\end{proof}

\begin{corollary}\label{cor: crystalline comparison for derived BKF modules}
    Let $\fX$ be as above and let $\mM= (\bM,\bV) \in \DBKF^{\log}(\fX,\vp)$. 
    \begin{enumerate}
        \item There is a natural isomorphism
        \[
         A\Omega^{\log}_{\fX}(\mM)\htimes^{\L}_{\Ainf}\Acrys \cong R\upsilon_*(\sigma^*_{A_\crys}\mM)
        \]
        compatible with Frobenius, where $\upsilon\colon (\fX_{\mO/p}/\ul\Acrys)_{\crys} \ra \fX_{\ett}$ denotes the natural projection of sites.
        \item There is a natural isomorphism
        \[
         A\Omega^{\log}_{\fX}(\mM)\htimes^{\L}_{\Ainf}W(k) \cong R\upsilon_*(\sigma^*_{\crys}\mM)
         \]
        compatible with Frobenius, where $\upsilon \colon (\fX_k/W(\ul k))_{\crys}\rightarrow \fX_{k, \ett}\cong \fX_{\ett}$ is the natural projection of sites.
    \end{enumerate}
\end{corollary}

\begin{proof}
    By the definition of the crystalline specialization functor, we have 
    \[
        R\upsilon_*(\sigma^*_{A_\crys}\mM) \cong R\upsilon_*(\bM_{q\crys}) \htimes_{\Ainf}^{\L}\Acrys.
    \]
    Therefore, (1) follows from Proposition \ref{prop: derived q-crys vs Ainf}, and (2) follows from (1) using the base change property of crystalline cohomology. 
\end{proof}

\vspace{0.1in}
\subsection{The derived BKF module associated with a semistable local system}
\noindent
\vspace{0.1in}

\noindent Now we retain the setup from \S \ref{section: semistable local systems}. In the rest of this article, we take $C= \widehat{\overline K}$ and let $\mO=\mO_C$ be its ring of integers. Let $\underline{\mO}=(\mO, N_{\infty})$ be an extension of $\underline{\mO_K}=(\mathcal{O}_K, \mathbb{N})$ with $N_{\infty}=\mathbb{Q}_{\ge 0}$, which amounts to choosing a compatible system of $n$-th roots of $\pi$ for all $n\ge 1$. 

Let $\fY$ be a semistable formal scheme over $\mO_K$. Without loss of generality, we assume that the residue field $k$ of $\mO_K$ is algebraically closed.\footnote{
If $k$ is not algebraically closed, let $\overline k$ be its algebraic closure. Let $K'= K \cdot W(\overline k) \subseteq C$ with the ring of integers $\mO_{K'}$. Let $\fY_{\mO_{K'}}$ denote the base change of $\fY$ along the extension $\mO_K \ra \mO_{K'}$, and $\fY_{\overline k}$ denote the base change of $\fY$ along the projection $\mO_{K'} \ra \overline k$. Write $K_0= W(k)[1/p]$ and $K_0'= W(\overline{k})[1/p]$. Then, by the base change property of log crystalline cohomology, for any $F$-crystal $\mE$ on $\fY_{k,\crys}$, there is a natural isomorphism 
\[
    R\Gamma_{\crys}(\fY_k/W(\ul k), \mE_{\Q}) \otimes_{K_0} K_0' \cong R\Gamma_{\crys}(\fY_{\overline k}/W(\ul{\overline k}), \mE_{\overline k ,\Q})
\]
compatible with Galois actions, Frobenii, filtrations, and monodromy operators.} In partciular, $\mO_K$ and $\mO$ have the same residue field $k$.
Let $Y$ be the adic generic fiber of $\fY$. Let $\fY_{\mO}$ be its base change along $\ul{\mO_K} \ra \ul{\mO}$ and let $Y_C$ be the adic generic fiber of $\fY_{\mO}$. 

\begin{definition}
Let $\Mod(\A_{\inf, Y_C},\vp)$ denote the category of pairs $(\bV, \vp_\bV)$ where
\begin{enumerate}
    \item $\bV$ is a sheaf of $\A_{\inf, Y_C}$-modules on $Y_{C,\proet}$ such that $\bV[\frac 1\xi]$ (resp. $\bV[\frac 1{\txi}]$) is a locally free $\A_{\inf, Y_C}[\frac 1\xi]$-module (resp. $\A_{\inf, Y_C}[\frac 1{\txi}]$-module);
    \item $\vp_{\bV}\colon \varphi^*\bV[\frac 1\txi] \ra \bV[\frac 1\txi]$ is an isomorphism of $\A_{\inf, Y_C}[\frac 1{\txi}]$-modules.
\end{enumerate}
\end{definition}
By Proposition \ref{prop: GR Thm 5.10}, there is a fully faithful functor 
    \[
        j_* \colon \Vect^{\an, \vp} (\fY_{\Prism}) \ra D_{\perf}^{\vp}(\fY_{\Prism})
    \]
Consider the composition 
    \[
        \widetilde{j}_* \colon \Vect^{\an,\vp}(\fY_{\Prism}) \xrightarrow[]{j_*} D_{\perf}^{\vp}(\fY_\Prism) \xrightarrow[]{\vp^*} D_{\perf}^{\vp}((\fY_{\mO}^{(1)}/(\ul{\Ainf},\txi)_{\Prism}),
    \]
    where the second map is induced from the Frobenius-twisted base change along $\fY_{\mO}^{(1)} \ra \fY$. 
    For any $\mE\in \Vect^{\an,\vp}(\fY_{\Prism})$, let $\mE_{\mO} \in \Vect^{\an,\vp}(\fY_{\mO,\Prism}^{(1)})$ denote the Frobenius-twisted base of $\mE$ along $\fY_{\mO}^{(1)} \ra \fY$. 
    
Let $\fB$ be the collection of log affinoid perfectoid objects $V\in Y_{C,\proet}$ such that the image of $V\ra Y_C$ is contained in $\spa(R[p^{-1}],R)$ for some small affine $\spf R \subseteq \fY_{\mO}$. Then $\fB$ forms a basis of $Y_{C,\proet}$. Let $\mE \in \Vect^{\an, \vp} (\fY_{\mO,\Prism}^{(1)})$. For any $V \in \fB$, if we write $ A= \Gamma(V,\widehat{\mO}_Y^+)$, then $(\Ainf(A),\txi)$ is an object in $\fY_{\mO,\Prism}^{(1)}$. Notice that the map $\spf A\ra \spf R \ra \fY_{\mO}$
    is independent of the choice of the small affine $\spf R$. We define $\ev_j(\mF)$ to be the sheafification of 
     \[
         V \mapsto \Gamma\big(\spec(\Ainf(A)), j_{\Ainf(A),*}\mE|_{\spec(\Ainf(A))\minus V(p,\txi)}\big),
     \]
     where $j_{\Ainf(A)}\colon \spec(\Ainf(A))\minus V(p,\txi) \ra \spec(\Ainf(A))$ is the inclusion.
\begin{proposition}
    Let $\fY$ be a semistable formal scheme over $\mO_K$. Then $\ev_j$ induces a functor
    \[
        \ev_j\colon \Vect^{\an, \vp} (\fY_{\mO,\Prism}^{(1)}) \ra \Mod(\A_{\inf,Y_C},\vp).
    \]
    Moreover, for any $\mE\in \Vect^{\an,\vp}(\fY_{\Prism})$, there is a natural isomorphism
\[
    \iota_{\mE}\colon  \widetilde{j}_*\mE\otimes^{\L}_{\Ainf}\Ainf[\frac 1 \txi] \xrightarrow[]{\sim} \ev_j(\mE_{\mO})[\frac 1 \txi]
\]
compatible with Frobenius.
\end{proposition}

\begin{proof}
     Since $\xi^p \in (p,\txi)$, for any $\mE\in \Vect^{\an,\vp}(\fY_{\Prism})$, we know that  both $\ev_j(\mE_{\mO})[\frac 1\xi]$ and $\ev_j(\mE_{\mO})[\frac 1\txi]$ are vector bundles. Furthermore, the Frobenius structure on $\mE$ induces an isomorphism 
     \[
        \vp_{\ev_j(\mE_{\mO})}\colon \varphi^*\ev_j(\mE_{\mO})[\frac 1 \txi] \xrightarrow[]{\sim} \ev_j(\mE_{\mO})[\frac 1 \txi].
     \]
     This proves the first statement. For the second assertion, it suffices to show that the diagram
    \[
    \begin{tikzcd}[column sep=3em]
        \Vect^{\an,\vp}(\fY_{\Prism}) \ar[r,"\vp^*"]\ar[d, "\widetilde{j}_*"] & \Vect^{\an,\vp}(\fY_{\mO,\Prism}^{(1)}) \ar[rr,"\ev_j"] & & \Mod(\A_{\inf,Y_C},\vp) \ar[d, "{\bV \mapsto \bV[\txi^{-1}]}"] \\
        D_{\perf}^{\vp}((\fY_{\mO}^{(1)}/(\ul\Ainf,\txi))_{\Prism}) \ar[r,"\beta_{\Prism}"] & D_{\perf}(\A_{\inf, Y_C}) \ar[rr, "-{\otimes_{\Ainf}^{\L}\Ainf[\frac 1\txi]}"] & &  D_{\perf}(\A_{\inf, Y_C}[\frac 1\txi]) \\
    \end{tikzcd}
    \]
    is commutative. Since the question is \'etale local, we may assume that $\fY= \spf R$ is semistable small affine. Then the commutativity of the diagram follows from the commutativity of 
     \[
     \begin{tikzcd}[column sep=3em]
        \spec A(R_{\mO}) \minus V(p, \txi) \ar[r, "j_{A(R_{\mO})}"] \ar[d, "i_1"] & \spec A(R_{\mO}) \ar[d, "i_2"] \\
         \spec \fS_R \minus V(p,E(u)) \ar[r, "j_{\fS_R}"] & \spec \fS_R,
     \end{tikzcd}
     \]
     where $i_1$ and $i_2$ are induced from the natural morphism $\fS_R \ra A(R_\mO)$. More precisely, the functor
     \[
         \beta_\Prism\circ \widetilde{j}_* \colon \Vect^{\an,\vp}(\fY_{\Prism}) \ra D_{\perf}(\A_{\inf, Y_C})
     \]
     is induced from $j_{\fS_R,*}\circ i_2^*$, while the functor
     \[
         \ev_j\circ \vp^* \colon \Vect^{\an,\vp}(\fY_{\Prism}) \ra \Mod(\A_{\inf,Y_C},\vp)
     \]
     is induced from $i_1^* \circ j_{A(R_\mO),*}$. This completes the proof.
\end{proof}

The crystalline realization functor in Proposition \ref{prop: crystalline realization} extends to the derived setting as
\[
     D_{\crys} \colon D_{\perf}^{\vp}(\fY_\Prism) \ra D_{\perf}^{\vp}(\fY_{k,\crys}).
\]
Note that $(D_{\crys}(j_*\mE))_{\Q}$ coincides with $\mE_{\crys,\Q}$ as objects in the category of $F$-isocrystals on $\fY_{k,\crys}$ in perfect complexes.

\begin{definition}
    Let $\fY$ be a semistable formal schemes over $\mathcal{O}_K$ with adic generic fiber $Y$. Let $\L$ be a semistable $\Z_p$-local system on $Y$ in the sense of Definition \ref{defn: semistable local systems}, and let $\mE$ be the associated analytic log prismatic $F$-crystal in $\Vect^{\an,\vp}(\fY_{\Prism})$. We call
    \[
        (\mM,\vp_\mM):= \Big(\widetilde j_*\mE, \ev_j(\mE_{\mO}), \vp_{\ev_j(\mE_{\mO})},\iota_{\mE}\Big)
    \]
    the \emph{derived relative log BKF module associated with $\L$}. 
\end{definition}

From such a derived relative log BKF module $(\mM,\vp_{\mM})$, we consider $\sigma^*_{A_\crys}\mM$ and $\sigma_{\crys}^* \mM$. In particular, $\sigma_{\crys}^* \mM$ is obtained from $\sigma^*_{A_\crys}\mM$ via pulling back along $(\Acrys, N_\infty) \ra (W(k), N_\infty)$. By the same discussion as in \S \ref{subsection: classical BKF modules}, $\sigma_{\crys}^* \mM$ can be identified as an object in $D_{\perf}^{\vp}\big((\fY_k/W(\underline{k})))_{\crys}\big)$. Similar to (\ref{crys cohom N vs N_infty}), the base change along $W(\underline{k})=(W(k), \mathbb{N}) \ra (W(k), N_\infty)$ induces an isomorphism
\[
     R\Gamma_{\crys}\big(\fY_k/W(\underline{k}), \sigma_{\crys}^*\mM\big) \cong R\Gamma_{\crys}\big(\fY_k/(W(k), N_\infty), \sigma_{\crys}^*\mM\big).
\]
The situation for $\sigma^*_{A_\crys}$ is similar. The following proposition shows that the derived relative log BKF module associated with a semistable local system $\L$ recovers the local system itself and its associated $F$-isocrystal over $W(k)$.

\begin{proposition}\label{crystalline specialization vs crystalline realzation}
     Let $\fY$ be a semistable log formal scheme over $\mO_K$. Let $\L$ be a semitable local system on $\fY$ with the associated analytic log prismatic $F$-crystal $\mE$. Let $(\mM,\vp_{\mM})$ be the derived relative log BKF module associated with $\L$. 
     \begin{enumerate}
         \item There is a natural isomorphism of $\widehat{\Z}_p$-local systems $\sigma^*_{\ett}\mM \cong \L$ on $Y_{C,\proet}$.
         \item Consider the functor
         \[
             D_{\perf}^{\vp}(\fY_{\mO_K/p,\crys})= D_{\perf}^{\vp}\big((\fY_{\mO_K/p}/(\mS,\mathbb{N}))_\crys\big) \ra D_{\perf}^{\vp}\big((\fY_k/(W(k),\mathbb{N}))_{\crys}\big),
         \]
        induced by the quotient map $(\mS,\mathbb{N}) \ra (W(k), \mathbb{N})$. Then the functor sends $D_{\crys}(j_*\mE)$ to $\sigma_{\crys}^*\mM$. It induces an isomorphism $D_{\crys}(j_*\mE)_\Q \cong (\sigma_{\crys}^*\mM)_\Q$ via the equivalence of categories
        \[
            \Isoc^\vp((\fY_{\mO_K/p}/(\mS,\N))_\crys) \cong \Isoc^\vp((\fY_k/(W(k),\N))_\crys)
        \]
        (cf. \cite[Remark B.20]{DLMS2}).
     \end{enumerate}
\end{proposition}

\begin{proof}
    By the same argument as Proposition \ref{compatibility of etale realization}, we have a commutative diagram
    \[
    \begin{tikzcd}
        \Vect^{\an,\vp}(\fY^{(1)}_{\Prism}) \arrow[d]\arrow[r, hook] \arrow[rd, "\cong"] & D_{\perf}^{\vp}(\fY^{(1)}_{\Prism})  \arrow[r] & \DBKF^{\log}(\fY_{\mO},\vp) \arrow[d, "p_2"] \\
        \Vect\big(\fY^{(1)}_{\Prism},\mO_{\Prism}[\mI_{\Prism}^{-1}]_p^{\wedge}\big)^{\vp=1}  \arrow[d] & \Loc_{\Z_p}^{\st}(Y) \arrow[rd]  & \Vect\big(W(\widehat{\mO}_{Y_C^\flat})\big)^{\vp=1} \arrow[d] \\ 
        \Vect\big(\fY_{\mO,\Prism},\mO_{\Prism}[\mI_{\Prism}^{-1}]_p^{\wedge}\big)^{\vp=1} \arrow [rr] &&
        \Loc_{\Z_p}(Y_{C}),
    \end{tikzcd}
    \]
where $p_2$ sends $\mM=(\bM, \bV)$ to $\bV$. In particular, we have $\sigma^*_{\ett}(\mM) \cong T_{\ett}(\mE) \cong \L$ on $Y_{C,\proet}$; this proves (1).

To prove (2), we may assume that $\fY= \spf R$ is semistable small affine.
Firstly, note that the log PD-triple $(A_{\crys}(R_{\mO}), I_{\crys}, P_\infty)^a$ is a weakly final object in $\CR\big(\fY_{\mO/p}/\ul\Acrys)_{\crys}\big)$. 
Indeed, let $(C,I,M_C)^a$ be an object of $\CR\big(\fY_{\mO/p}/\Acrys)_{\crys}\big)$. By the formal smoothness of $\ul{\Acrys(R_{\mO})}$ over $(\Acrys, N_\infty)$, the morphism $(\Acrys(R_{\mO}), P_\infty)^a \ra (C/I, M_C)^a$ lifts to $(C,M_C)$ by the following diagram
    \[
    \begin{tikzcd}
        (\Acrys(R_{\mO}),P_\infty)^a \arrow[r] \arrow[rrd,dashed, start anchor={[xshift= -3pt, yshift= 3pt]south east}, end anchor={[xshift=3pt, yshift= -2pt]north west}] & (R_{\mO},P_\infty)^a\arrow[r] & (C/I,M_C)^a \\
        (\Acrys,\mO^\flat\minus\{0\}) \arrow[u] \arrow[rr] && (C,M_C)^a. \arrow[u]
    \end{tikzcd}
    \]
Let $\mD_{\crys}(R_{\mO})(\bullet)$ be the $p$-completed log PD-envelope of \[\Acrys(R_{\mO})(\bullet)= A(R_{\mO})(\bullet)\htimes_{\Ainf}\Acrys \ra R_\mO.\] Then there are equivalences of categories
\begin{equation}\label{crystal over Acrys vs stratifications}
     \Vect\big((\fU_{\mO/p}/\ul \Acrys)_{\crys}\big) \cong \Strat(\mD_\crys(R_\mO)(\bullet)) \cong \MIC(\Acrys(R_\mO)),
\end{equation}
where $\MIC(\Acrys(R_\mO))$ stands for the category of finite projective modules with flat connection over $\Acrys(R)$.
Since $(E([\pi^\flat]))= (\xi)$ in $\Ainf$, there is a natural map between log prisms $\fS_R \ra A(R_{\mO})$ sending $u\mapsto \vp([\pi^\flat])$. It induces a map $\mS_R(\bullet) \ra \Acrys(R_{\mO})(\bullet)$ between log PD-pairs. For any integer $m\ge 1$, the diagram of maps
\[
\begin{tikzcd}
        \fS_R \otimes_{\fS} \cdots \otimes_{\fS} \fS_R \ar[r] \ar[d]& A(R_{\mO})\otimes_{\Ainf} \cdots \otimes_{\Ainf} A(R_\mO) \ar[d]\\
        \mS_R \otimes_{\mS} \cdots \otimes_{\mS} \mS_R \ar[r] &\Acrys(R_{\mO})\otimes_{\Acrys} \cdots \otimes_{\Ainf} \Acrys(R_\mO) 
\end{tikzcd}
\]
between $(m+1)$-fold self-products induces the following commutative diagram after taking log $q$-PD envelope and log PD envelope:
\[
\begin{tikzcd}
    \fS_R(\bullet) \arrow[r] \arrow[d] &  A(R_{\mO})(\bullet) \arrow[d]\\
    \mS_R(\bullet) \arrow[r] & \Acrys(R_{\mO})(\bullet).
\end{tikzcd}
\]
Then (2) follows from (\ref{crystal over Acrys vs stratifications}) and the construction of crystalline realization.
\end{proof}

Applying the \'etale and crystalline comparison theorems to the derived relative log BKF module associated with semistable local systems, we arrive at the following comparison theorem.

\begin{theorem}\label{thm: semistable compsrison over A_crys}
    Let $\fY$ be a proper semistable formal scheme over $\mO_K$ with generic fiber $Y$. Let $\L$ be a semistable $\Z_p$-local system on $Y$ and let $\mM$ be the associated derived relative log BKF module. Then there is a natural isomorphism
    \[
        R\Gamma_{\ett}(Y_C, \L)\otimes_{\Z_p}^{\L}B_\crys \cong R\Gamma_{\crys}\big(\fY_{\mO/p}/\ul \Acrys,\sigma_{A_\crys}^* \mM\big)\otimes^{\L}_{A_{\crys}}B_{\crys}.
    \]
\end{theorem}

\vspace{0.1in}
\subsection{Infinitesimal cohomology over $B_{\dR}^+$}
\noindent
\vspace{0.1in}

\noindent
In the proof of the $C_{\st}$-conjecture with coefficients (Theorem \ref{mainthm}), we need to compare the filtrations on the \'etale and crystalline cohomology. The way to achieve this is to consider the infinitesimal cohomology over $B_{\dR}^+$ studied in \cite{BMS1} and \cite{Guo-Reinecke}. In what follows, we shall adopt the notation and results from \cite[\S 10]{Guo-Reinecke} (see also \cite[\S 7.3.1]{DMS}). Recall that $\fY$ is a semistable formal scheme over $\mO_K$ with adic generic fiber $Y$. 

Let $A_{\inf,K} = \Ainf \otimes_{W(k)} \mO_K$. Let $\tilde \theta = \theta\circ\vp^{-1} \colon \Ainf \ra \mO$, which extends to $\tilde\theta_K\colon \A_{\inf,K} \ra \mO$. Notice
that the $\ker(\tilde\theta_K)$-adic completion of $A_{\inf,K}[\frac 1p]$ is isomorphic to $B_{\dR}^+$. For
each positive integer $e$, write $B_{\dR,e}^+= B_{\dR}^+/(\ker(\tilde\theta_K))^e$. 

\begin{definition}
    We define the \emph{infinitesimal site} $Y_{C,\ett}/B_{\dR,\inf}^+$ as follows:
    \begin{enumerate}
        \item An object is a pair $(U, T)$, called an \emph{infinitesimal thickening} of $Y_C$, where $U \in Y_{C,\ett}$, and $T$ is an adic space that is topologically of finite type over $B_{\dR,e}^+$ for some $e\ge 1$ together with a Zariski closed immersion $U \ra T$ given by a nilpotent ideal.
        \item A morphism between objects $(U_1, T_1) \ra (U_2, T_2)$ is given by a morphism $U_1 \ra U_2$ in $Y_{C,\ett}$ and a compatible map of adic spaces $T_1 \ra T_2 $ over $B_{\dR,e}^+$ for some $e\ge 1$.
    \end{enumerate}
    We equip it with the \'etale topology. Let $\Vect(Y_{C,\ett}/B_{\dR,\inf}^+)$ denote the category of finite locally free crystals over $Y_{C,\ett}/B_{\dR,\inf}^+$. Let $D_{\perf}(Y_{C,\ett}/B_{\dR,\inf}^+)$ denote the category of crystals in perfect complexes over $Y_{C,\ett}/B_{\dR,\inf}^+$.
\end{definition}

Let $\mF$ be a finite locally free crystal over $Y_{C,\ett}/B_{\dR,\inf}^+$. There are two ways to compute the $B_{\dR}^+$-cohomology $R\Gamma(Y_{C,\ett}/B_{\dR,\inf}^+, \mF)$: through \v{C}ech–Alexander complexes, or through de Rham complexes, as in \cite[Construction 10.6]{Guo-Reinecke} and \cite[\S 7.3.1]{DMS}. Here we adopt the first approach. 

Let us write down the \v{C}ech–Alexander complex that computes the infinitesimal cohomology over $B_{\dR}^+$. Since the question is \'etale local, we may assume that $\fY= \spf \ul R$ is semistable small affine. Let $\fY_\mO= \spf(\underline{R_\mO})$ be its base change along $\underline{\mO_K} \ra \underline{\mO}$. We recall the notation from \S \ref{subsection: crystalline comparison}. Consider a triple $\Phi= (S,P_\Lambda, \iota)$ consisting of
\begin{enumerate}
    \item a finite set $S$ that indexes the coordinates of the formal $\mO$-torus (equipped with the trivial log structure)
    \[
        R^\square_S := \mO\langle X_s^{\pm 1}\rangle_{s\in S};
    \]
    \item a nonempty finite set of monoids $P_\Lambda= \{P_\lambda \mid \lambda \in \Lambda\}$ together with injective homomorphisms $u_\lambda \colon N \hookrightarrow P_\lambda$ satisfying the conditions in Definition \ref{definition: small affine},
    which induces a pre-log ring
    \[
        \ul{R^\square_\lambda} := (\mO\langle P_\lambda \rangle\htimes_{\mO\langle N \rangle} \mO, P_\lambda\sqcup_N N_\infty)
    \]
    over $\underline{\mO}$;
    \item an exact closed immersion of pre-log rings
    \[
        \iota \colon \fY_\mO= \spf \underline{R_\mO} \ra \spf (R^\square_S) \times \prod_{\lambda\in \Lambda} \spf(\ul{R^\square_\lambda})^a
    \]
    where the products are formed over $\spf(\ul\mO)^a$, satisfying
    \begin{itemize}
        \item the map $\iota_S \colon \fY_\mO \ra \spf (R^\square_S)$ is already a closed immersion,
        \item the map $\iota_\lambda \colon \fY_\mO \ra \spf(\ul{R^\square_\lambda})^a$ is strictly \'etale for all $\lambda\in \Lambda$,
        \item for any $\lambda \in \Lambda$, the image of $P_\lambda \ra R_\mO$ lies in $R_\mO\cap (R_\mO[p^{-1}])^{\times}$.\footnote{This holds, for example, if we take the standard semistable chart for the semistable small affine.}
    \end{itemize}
\end{enumerate}
For each $\lambda\in \Lambda$, choose $T_\lambda$ to be a finite set of generators of $P_\lambda$ as in \S \ref{subsection: crystalline comparison}. This induces a surjection $\N^{T_\lambda} \ra P_\lambda$. Let $T= \bigsqcup_{\lambda\in \Lambda} T_\lambda$. Then we obtain a surjection of pre-log rings
\[
    \Sigma= \Sigma(S, \ul T)= \mO\langle (X_s^{\pm 1})_{s\in S}, (X_t)_{t\in T}\rangle \ra R_S^\square \otimes_{\mO}\Big(\bigotimes_{\lambda\in \Lambda} \ul{R^\square_\lambda}\Big) \ra \ul {R_{\mO}.
}\]
Define
\[
    B_{\dR}^+(S,\ul T)= \varprojlim_n (B_{\dR}^+/\xi^n) \langle (X_s)_{s\in S},\ul T\rangle.
\]
It provides a surjection $\theta \colon B_{\dR}^+(S,\ul T) \ra R_\mO$.
Let $B_\dR^+(S,\ul T)(m)$ be the $(m+1)$-fold completed self product of $B_\dR^+(S,\ul T)$ over $B_\dR$. Then $\theta$ extends to 
\[
    \theta_m \colon B_\dR^+(S,\ul T)(m) \xrightarrow{\Delta} B_\dR^+(S,\ul T) \xrightarrow[]{\theta} R_\mO,
\]
where $\Delta$ denotes the multiplication map. Define
\[
    \mD_{\dR,\Sigma}(m)= \varprojlim_n \big(B_{\dR}^+(S,\ul T)(m)/(\ker(\theta_m))^n\big).
\]
Denote $\mF_{\Sigma}= \mF(\mD_{\dR,\Sigma})$ and $\mF_\Sigma(\bullet)= \mF(\mD_{\dR,\Sigma}(\bullet))$.
By \cite[Construction 10.6]{Guo-Reinecke}, $R\Gamma(Y_{C,\ett}/B_{\dR,\inf}^+, \mF)$ is computed by the $\check{\text{C}}$ech–Alexander complex $\mF_\Sigma(\bullet)$.

By the same construction as in the proof of \cite[Proposition 10.10]{Guo-Reinecke}, there is a natural morphism $\mD_{\crys,\Sigma}(\bullet) \ra \mD_{\dR,\Sigma}(\bullet)$ between simplicial rings. More precisely, let $J(n)$ denote the kernel of the surjection 
\[
    A_0(\Sigma)(n) = \Ainf\langle (X_s^{\pm 1})_{s\in S}, (X_t)_{t\in T}\rangle \ra R_\mO,
\]
and let $\mD_{\crys,\Sigma}(n)$ denote the $p$-completed log PD envelope of $A_0(\Sigma)(n)$ with respect to $J(n)$. Then there is a surjection
\[
    \mD_{\crys,\Sigma}(n)[p^{-1}] \ra A_0(\Sigma)(n)[p^{-1}]/J(n)^m,
\]
for all $m\ge 1$. Passing to the limit with respect to $m$, one obtains a natural morphism
\[
    \mD_{\crys,\Sigma}(n) \ra \mD_{\dR,\Sigma}(n) \cong \big(A_0(\Sigma)(n)[p^{-1}]\big)_{J(n)}^{\wedge}.
\]
It extends to a map of simplicial rings $\mD_{\crys,\Sigma}(\bullet) \ra \mD_{\dR,\Sigma}(\bullet)$ as $n$ varies, and this map factors as 
\[
    \mD_{\crys,\Sigma}(\bullet) \ra \mD_{\crys,\Sigma}(\bullet)\htimes_{\Acrys}B_{\dR}^+ \ra \mD_{\dR,\Sigma}(\bullet).
\]
Similarly, let $\mD_{\Sigma}(n)$ be the $p$-completed log $q$-PD envelope of $A_0(\Sigma)(n)$ with respect to the ideal $J(n)$. Then there is are natural maps of simplicial rings
\[
    \mD_{\Sigma}(\bullet) \ra \mD_{\Sigma}(\bullet)\htimes_{\Ainf}B_{\dR}^+ \ra \mD_{\dR,\Sigma}(\bullet).
\]
Recall that there is a natural map $ \mD_{\Sigma}(\bullet) \ra \mD_{\crys,\Sigma}(\bullet)$ and these simplicial rings fit into a commutative diagram 
\begin{equation}\label{eq: qcrys to crys to dR}
    \begin{tikzcd}
    \mD_{\Sigma}(\bullet) \ar[r] \ar[rr, bend left=15] & \mD_{\crys,\Sigma}(\bullet) \ar[r]  & \mD_{\dR,\Sigma}(\bullet).
\end{tikzcd}
\end{equation}

\begin{definition}
Consider the category $\DBKF^{\log}(\fY_{\mO},\vp)$ of derived relative log BKF modules over $\fY_{\mO}$. We define the \emph{$B_{\dR}^+$-infinitesimal specialization functor} 
\[
        \sigma_{B_{\dR}^+}^*\colon \DBKF^{\log}(\fY_{\mO},\vp) \ra D_{\perf}(Y_{C,\ett}/B_{\dR,\inf}^+)
\]
as follows. For any $\mM \in \DBKF^{\log}(\fY_{\mO},\vp)$, we put
\[
(\sigma_{B_{\dR}^+}^* \mM)(\mD_{\dR,\Sigma}(\bullet))= (\sigma_{\Acrys}^* \mM)(\mD_{\crys,\Sigma}(\bullet)) \otimes^{\L}_{\mD_{\crys,\Sigma}(\bullet)}\mD_{\dR,\Sigma}(\bullet).
\]
Note that $\sigma_{B_{\dR}^+}^*$ restricts to a specialization functor
    \[
        \sigma_{B_{\dR}^+}^*\colon \BKF^{\log}(\fY_{\mO},\vp) \ra \Vect(Y_{C,\ett}/B_{\dR,\inf}^+).
    \]
\end{definition}

\begin{remark}\label{rmk: define dR via qcrys}
    By (\ref{eq: qcrys to crys to dR}), if $\mM=(\bM,\bV)$, we can also define the $B_{\dR}^+$-specialization functor as
    \[
        (\sigma_{B_{\dR}^+}^* \mM)(\mD_{\dR,\Sigma}(\bullet))= 
        (\bM_{q\crys})(\mD_{\Sigma}(\bullet)) \otimes^{\L}_{\mD_{\Sigma}(\bullet)}\mD_{\dR,\Sigma}(\bullet).
    \]
\end{remark}

\begin{construction}\label{construction: commutativity of diagram (1)}
    Now we construct comparison maps between $\Ainf$-cohomology and $B_{\dR}^+$-cohomology. The question is \'etale local, so we may assume that $\fY$ is semistable small affine.
    We retain the setup in \ref{subsection: crystalline comparison}. 
    Let $\mM = (\bM,\bV) \in \DBKF^{\log}(\fY_{\mO},\vp)$ be a derived log BKF module on $\fY_{\mO}$ associated with a semistable local system $\L$. Given a triple $\Phi= (S,P_\Lambda, \iota)$.
    Let $R_{\infty,\Phi}(m)$ be the $p$-adic completion of the $(m+1)$-fold self-tensor product of $R_{\infty,\Phi}$ over $R_\mO$. 
    Then $\spa\big(R_{\infty,\Phi}(\bullet)[\frac 1p], R_{\infty,\Phi}(\bullet)\big)$ is the \v{C}ech nerve of $\spa(R_{\infty,\Phi}[\frac 1p], R_{\infty,\Phi})$ in $Y_{C,\proet}$. By evaluation, we obtain an isomorphism
    \[
        R\Gamma_{\proet}(Y_C,\bM\otimes_{\A_{\inf}}\B_{\dR}^+) \cong (\bM\otimes_{\A_{\inf}}\B_{\dR}^+) \big(\spa\big(R_{\infty,\Phi}(\bullet)[\frac 1p], R_{\infty,\Phi}(\bullet)\big)\big).
    \]
    On the other hand, starting from $\Phi$, we can construct a surjection $\Sigma= \Sigma(S,\ul T) \ra R_\mO$. The natural morphisms $B_{\dR}^+(S,\ul T)(m) \ra R_{\infty,\Phi}[\frac 1p](m)$
    induce a morphism between simplicial rings
    \[
        \mD_{\dR,\Sigma}(\bullet) \ra R_{\infty,\Phi}(\bullet)[\frac 1p].
    \]
    It fits into the following commutative diagram of simplicial rings
    \[
    \begin{tikzcd}
         \mD_{\crys,\Sigma}(\bullet)\htimes_{\Acrys}B_{\dR}^+ \ar[d]  & \mD_{\Sigma}(\bullet)\htimes_{\Ainf}B_{\dR}^+ \ar[l, "\cong"'] \ar[d]\\
        \mD_{\dR,\Sigma}(\bullet) \ar[r] & R_{\infty,\Phi}(\bullet)[\frac 1p].
    \end{tikzcd}
    \]
    By the construction of the absolute crystalline specialization functor and Remark \ref{rmk: define dR via qcrys}, the diagram above further yields a diagram
    \[
    \begin{tikzcd}
        (\sigma_{\Acrys}^*\mM)(\mD_{\crys,\Sigma}(\bullet))\htimes_{\Acrys}B_{\dR}^+ \ar[d] & \bM_{q\crys}(\mD_{\Sigma}(\bullet))\htimes_{\Ainf}B_{\dR}^+ \ar[l, "\cong"'] \ar[d] \\
        (\sigma_{B_{\dR}^+}^*\mM) (\mD_{\dR,\Sigma}(\bullet)) \ar[r] & (\bM\otimes_{\A_{\inf}}\B_{\dR}^+) \big(\spa\big(R_{\infty,\Phi}(\bullet)[\frac 1p], R_{\infty,\Phi}(\bullet)\big)\big).
    \end{tikzcd}
    \]
    Here, $(\sigma_{B_{\dR}^+}^*\mM) (\mD_{\dR,\Sigma}(\bullet))$ is identified with $ \bM_{q\crys}(\mD_{\Sigma}(\bullet))\otimes^{\L}_{\mD_{\Sigma}(\bullet)}\mD_{\dR,\Sigma}(\bullet)$ by Remark \ref{rmk: define dR via qcrys}. Consequently, the induced morphisms between \v{C}ech--Alexander complexes induces a diagram
    \begin{equation}\label{eq: diagram 1}
    \begin{tikzcd}
        R\Gamma_{\crys}(\fY_{\mO/p}/\Acrys, \sigma_{\Acrys}^*\mM)\htimes_{\Acrys}^{\L}B^+_{\dR} \ar[d] & R\Gamma_{\Ainf}(\fY_{\mO},\mM)\htimes_{\Ainf}^{\L}B_{\dR}^+ \ar[l, "\cong"'] \ar[d] \\
        R\Gamma(Y_{C,\ett}/B_{\dR,\inf}^+,\sigma_{B_{\dR}^+}^*\mM) \ar[r] & R\Gamma_{\proet}(\bM\otimes_{\A_{\inf,Y_C}}\B^+_{\dR,Y_C}).
    \end{tikzcd}      
    \end{equation}
\end{construction}

\begin{remark}\label{prop: GR prop 10.11}
    If we compute $B_{\dR,\inf}^+$-cohomology $R\Gamma\big(Y_{C,\ett}/B_{\dR,\inf}^+,\sigma_{B_{\dR}^+}^*\mM\big)$ via de Rham complexes as in \cite[Construction 10.6, Corollary 10.8]{Guo-Reinecke}), we can show that the left vertical arrow in \eqref{eq: diagram 1} is actually an isomorphism; namely, 
    \[
        R\Gamma_{\crys}(\fY_{\mO/p}/\Acrys, \sigma_{\Acrys}^*\mM)\htimes_{\Acrys}^{\L}B^+_{\dR} \xrightarrow[]{\sim} R\Gamma(Y_{C,\ett}/B_{\dR,\inf}^+,\sigma_{B_{\dR}^+}^*\mM).
    \]
    The proof is similar to those of \cite[Proposition 10.11]{Guo-Reinecke} and \cite[Proposition 7.21]{DMS}. But we do not need this result.
\end{remark}

\vspace{0.1in}
\subsection{The semistable comparison theorem}\label{subsection: proof of Cst conjecture}
\noindent
\vspace{0.1in}

\noindent
We finally finish the proof of the semistable comparison theorem (Theorem \ref{mainthm}). Recall that $C= \widehat{\overline K}$, the residue field $k$ of $\mO_K$ is assumed to be algebraically closed, and that $N_{\infty}=\mathbb{Q}_{\ge 0}$. To proceed, we recall Breuil's period rings in \cite{breuil1997representations}, which are ``$\fS$-deformations'' of the corresponding classical period rings. 
Firstly, let $\widehat A_{\inf}$ be the log envelop of $A_{\inf}\otimes_{W(k)}\fS$ with respect to the map $\theta \colon \underline{A_{\inf}}\otimes_{W(\underline{k})}\underline{\fS} \ra \underline{\mO}$ sending $u\mapsto \pi$ (see \cite[\S 2.3.3]{andreatta2012semistable} for the notion of log envelop). Explicitly, if we write $\N^{(1)}= \{(x,y)\in \Z\times \Z \mid x+y \ge 0\}$, then 
\[
    \widehat A_{\inf} \cong \left(A_{\inf}\otimes_{W(k)}\fS\right)\otimes_{\Z[\N \times \N]}\Z[\N^{(1)}]\cong A_{\inf}[\![u]\!][\frac u{[\pi^\flat]}, \frac{[\pi^\flat]}u], 
\]
where the map $\Z[\N \times \N]\ra A_{\inf}\otimes_{W(k)}\fS$ sends $(a,b)\mapsto [\pi^\flat]^a\otimes u^b$, and $\Z[\N \times \N]\ra\Z[\N^{(1)}]$ is the natural inclusion. Note that $\theta$ extends to $\theta_\fS \colon \widehat A_{\inf} \ra \mO$.

\begin{definition}[{\cite{breuil1997representations}}]\label{logperiodfield}
    \begin{enumerate}
        \item Let $\widehat{A}_\st$ be the $p$-adic completion of the PD-envelop of $\widehat A_{\inf}$ with respect to $\ker(\theta_\fS)$. Let $\widehat{B}_\st^+ = \widehat{A}_\st [p^{-1}]$ and $\widehat{B}_\st= \widehat{B}_\st^+[t^{-1}]$. We endow $\widehat A_\st$ with the pre-log structure $\N \ra \widehat A_\st$ sending $1\mapsto u$.
        \item All the rings $\widehat{A}_\st, \widehat{B}_\st^+, \widehat{B}_\st$ are endowed with natural $G_K$-actions, which act trivially on $u$. Moreover, they are endowed with natural $\vp$-actions induced from the $\vp$-action on $A_{\inf}$ and sending $u\mapsto u^p$.
        \item Define the \emph{monodromy operator} $N$ on $\widehat{B}_\st$ by $N=u\dfrac{d}{du}$.
    \end{enumerate}
\end{definition}

\begin{proposition}\label{Nnilplocal}
    \begin{enumerate}
        \item $\widehat{A}_\st \cong A_\crys \{\langle u/[\pi^\flat]-1 \rangle\}$.
        \item The monodromy operator $N$ on $\widehat{B}_{\st}$ commutes with the $G_K$-action, and $N \vp = p \vp N$.
        \item The $G_K$ action on $X= u/[\pi^\flat]-1$ is given by 
        \[
            g(X)= [\epsilon(g)]X+[\epsilon(g)]-1,
        \]
        where $\epsilon(g)$ is determined by $g(\pi^\flat)= \epsilon(g) \pi^\flat$.
        \item There is a natural $(\vp,G_K,N)$-equivariant inclusion $B_\st \hookrightarrow \widehat{B}_\st$ sending $\log([\pi^\flat]/\pi)$ to $\log([\pi^\flat]/u)$.
        Moreover, we have $\widehat{B}_\st^{N-\nilp} = B_\st$, where ``$N$-$\nilp$'' stands for the subset of $N$-nilpotent elements.
    \end{enumerate}
\end{proposition}

\begin{proof}
    For part $(1)$, see \cite[\S 2]{breuil1997representations}. Part $(2)$ follows from the definition. Part $(3)$ is a special case of \cite[Theorem 3.40]{andreatta2012semistable}.
    The proof of part $(4)$ can be found in \cite[3.7]{Katosemistable}. 
\end{proof}

There are two categories of filtered $(\vp,N)$-modules associated with the period rings $B_\st$ and $\hB_\st$ respectively. Let us briefly recall their definitions. Let $K_0= W(k)[\frac 1p] \subseteq K$.

\begin{definition}
    A \emph{filtered $(\vp,N)$-module} over $K$ is a quadruple $(M,\vp,N,\Fil^\bullet)$ where
    \begin{itemize}
        \item $M$ is a finite dimensional $K_0$-vector space;
        \item $\vp\colon M \ra M$ is a bijective $K_0$-semilinear morphism with respect to the absolute Frobenius $\vp$;
        \item $N \colon M \ra M$ is a $K_0$-linear morphism satisfying $N\vp= p \vp N$;
        \item $\Fil^\bullet$ is a decreasing, separated, and exhaustive filtration on $M_K=M\otimes_{K_0}K$.
    \end{itemize}
    A morphism between two filtered $(\vp,N)$-modules over $K$ is a morphism between $K_0$-vector spaces compatible with $(\vp,N)$-actions and filtrations. The category of filtered $(\vp,N)$-modules over $K$ is denoted by $\MF_K(\vp,N)$.
\end{definition}

\begin{definition} A \emph{filtered $(\vp,N)$-module over} $\widehat{B}_{\st}^{G_K}$ is a quadruple $(\widehat M,\vp,N,\Fil^\bullet)$ where
    \begin{itemize}
        \item $\widehat M$ is a free $\widehat{B}_{\st}^{G_K}$-module of finite rank;
        \item $\vp\colon \widehat M \ra \widehat M$ is a $\widehat{B}_{\st}^{G_K}$-semilinear morphism with respect to the Frobenius $\vp$ on $\widehat{B}_{\st}^{G_K}$, whose determinant is invertible in $\widehat{B}_{\st}^{G_K}$;
        \item $\Fil^\bullet$ is a decreasing, separated, and exhaustive filtration on $\widehat M$. Moreover, for any $f\in \Fil^i(\widehat{B}_{\st}^{G_K})$ and $x \in \Fil^j(\widehat M)$, we have $fx \in \Fil^{i+j}(\widehat M)$. 
        \item $N \colon \widehat M \ra \widehat M$ is a $K_0$-linear morphism satisfying
        \begin{itemize}
            \item $N(fx)= N(f)x+fN(x)$ for any $f\in \widehat{B}_{\st}^{G_K}$ and $x\in \widehat M$,
            \item $N \vp= p \vp N$,
            \item $N\big(\Fil^i\widehat M\big) \subseteq \Fil^{i-1}\widehat M$ for any $i \in \Z$.
        \end{itemize}
    \end{itemize}
    A morphism between two filtered $(\vp,N)$-modules over $\widehat{B}_{\st}^{G_K}$ is a morphism between $\widehat{B}_{\st}^{G_K}$-modules compatible with $(\vp,N)$-actions and filtrations. The category of filtered $(\vp,N)$-modules over $\widehat{B}_{\st}^{G_K}$ is denoted by $\hMF_{\widehat{B}_{\st}^{G_K}}(\vp,N)$.
\end{definition}

The two categories of filtered $(\vp,N)$-modules are actually equivalent, thanks to the following theorem.
\begin{theorem}[{\cite[Theorem 6.1.1]{breuil1997representations}}]\label{breuil thm 6.1.1}
    There is an equivalence of categories
    \[
    \begin{aligned}
        \varUpsilon\colon \MF_K(\vp,N) & \lra \hMF_{\widehat{B}_{\st}^{G_K}}(\vp,N)\\
        M & \longmapsto M\otimes_{K_0}\widehat{B}_{\st}^{G_K},
    \end{aligned}
    \]
\end{theorem}

Let $V$ be a finite dimensional continuous $\Q_p$-representation of $G_K$. Recall that $V$ is a semistable representation if $D_\st(V)= (V\otimes_{\Q_p}B_\st)^{G_K}$ has $K_0$-dimension equal to $\dim_{\Q_p} V$. The category of semistable representations is denoted by $\Rep^\st_{\Q_p}(G_K)$. On the other hand, we define 
\[
    \hD_\st(V)= (V\otimes_{\Q_p}\hB_\st)^{G_K},
\]
which is an object in $\hMF_{\widehat{B}_{\st}^{G_K}}(\vp,N)$ with $(\vp,N)$-action and filtration induced from those on $\hB_\st$. 

\begin{definition}
    Let $V$ be a finite dimensional continuous $\Q_p$-representation of $G_K$. We say that $V$ is \emph{$\hB_\st$-admissible} if the injection
    \[
        \alpha_{\st} \colon \hD_\st(V) \otimes_{\widehat{B}_{\st}^{G_K}}\hB_\st \ra V\otimes_{\Q_p} \hB_\st
    \]
    is an isomorphism of $\hB_\st$-modules. The category of $\hB_\st$-admissible representations is denoted by $\Rep^{\hB_\st}_{\Q_p}(G_K)$.
\end{definition}

The functor $\hD_\st \colon \Rep_{\Q_p}^{\hB_\st}(G_K) \ra \hMF_{\widehat{B}_{\st}^{G_K}}(\vp, N)$ is fully faithful by the results in \cite[\S 9]{breuil1997representations}.

Since the natural inclusion $B_\st \hookrightarrow \hB_\st$ is $G_K$-equivariant, every semistable representation is $\hB_\st$-admissible; namely, there is a natural inclusion
\[
    \iota \colon \Rep^\st_{\Q_p}(G_K) \hookrightarrow \Rep^{\hB_\st}_{\Q_p}(G_K).
\]
In fact, this is an equivalence of categories.

\begin{theorem}[{\cite[Theorem 3.3]{breuil1997representations}}]\label{breuil thm 3.3}
The inclusion $\iota$ is an equivalence of categories. 
\end{theorem}

We now return to the proof of the $C_\st$-conjecture. Note that taking base change along $(\Acrys, \N) \ra (\Acrys, N_\infty)$ induces an equivalence of categories
\[
    D_\perf^{\vp}\big((\fY_{\mO/p}/(\Acrys, \N))_{\crys}\big) \cong D_\perf^{\vp}\big((\fY_{\mO/p}/(\Acrys, N_\infty))_{\crys}\big),
\]
(see \eqref{eq: admissibly smooth 2}). By a slight abuse of notation, for any $F$-crystal $\mE\in D_\perf^{\vp}\big((\fY_{\mO/p}/(\Acrys, N_\infty))_{\crys}\big)$, we still use $\mE$ to denote the corresponding object in $D_\perf^{\vp}\big((\fY_{\mO/p}/(\Acrys, \N))_{\crys}\big)$. Then the base change along $(\mO, \N) \ra (\mO, N_\infty)$ induces a Frobenius-equivariant isomorphism (see \eqref{crys cohom N vs N_infty})
\begin{equation}\label{eq: N vs Ninfty}
     R\Gamma_{\crys}\big(\fY_{\mO/p}/(\Acrys, \N), \mE\big) \cong R\Gamma_{\crys}\big(\fY_{\mO/p}/(\Acrys, N_\infty), \mE\big).
\end{equation}
We equip $R\Gamma_{\crys}\big(\fY_{\mO/p}/(\Acrys, \N), \mE\big)$ with the monodromy operator induced by $\mathbb{N}$, which is also viewed as a monodromy operator on $R\Gamma_{\crys}\big(\fY_{\mO/p}/(\Acrys, N_\infty), \mE\big)$ via the identification \eqref{eq: N vs Ninfty}.

Similarly, taking base change along $(W(k), \N) \ra (W(k), N_\infty)$ induces an equivalence of categories
\[
    D_\perf^{\vp}\big((\fY_{k}/(W(k), \N))_{\crys}\big) \cong D_\perf^{\vp}\big((\fY_{k}/(W(k), N_\infty))_{\crys}\big).
\]
Furthermore, it induces isomorphisms between cohomologies as \eqref{eq: N vs Ninfty}. Therefore, the derived crystalline specialization functor in Definition \ref{defn: crystalline specialization} induces a functor
\begin{equation}\label{eq: crystalline specialization with N}
    \DBKF^{\log}(\fX,\vp) \xrightarrow[]{\sigma_\crys^*} D_{\perf}^{\vp}\big((\fX_k/(W(k),N_\infty))_{\crys}\big) \xrightarrow[]{\sim} D_{\perf}^{\vp}\big((\fX_k/(W(k),\N))_{\crys}\big),
\end{equation}
which we still denote by $\sigma_\crys^*$.

Now consider three crystalline pre-log prisms $(\mS,(p),\N)$, $(\Acrys,(p),\N)$, and $(\widehat{A}_\st, (p), \N)$. Recall that the pre-log structures on $\Acrys$ (resp. $\widehat{A}_\st$ and $\mS$) sends $1 \mapsto [\pi^\flat]$ (resp. $1\mapsto u$). Since $u/[\pi^\flat]$ is a unit in $\widehat A_\st$, we obtain maps of log-rings
\begin{equation}\label{eq: diagram crystalline log prisms}
\begin{tikzcd}
    (\mS,(p), \N)^a \arrow[r] & (\widehat{A}_\st, (p), \N)^a & (\Acrys,(p), \N)^a \arrow[l],
\end{tikzcd}
\end{equation}
where the left arrow sends $u \mapsto u$ and the right arrow is induced by the natural inclusion $\Acrys\rightarrow \widehat{A}_{\mathrm{st}}$. If we put the trivial $G_K$-action on $\mS$, then the two arrows are $G_K$-equivariant.\footnote{The morphism $(\mS,(p), \N)^a \ra (\Acrys,(p), \N)^a$ sending $u$ to $[\pi^\flat]$ is not $G_K$-equivariant (cf. \cite[Remark 3.8]{Yao_semistable}). But we do not use this map in this article.}

Let $\fY$ be a semistable formal scheme over $\mO_K$ and let $\fY_{\mO_K/p}$ (resp. $\fY_{\mO/p}$) be the base change of $\fY$ over $\mO_K/p$ (resp. $\mO/p$). Then \eqref{eq: diagram crystalline log prisms} induces a commutative diagram
    \begin{equation}\label{diagram: A_st}
    \begin{tikzcd}
        \Vect^{\an,\vp}\big(\fY_\Prism\big) \arrow[r] \arrow[d, "j_*"] \ar[dd, bend right =60]&\DBKF^{\log}(\fY_{\mO},\vp) \arrow[d,"\sigma^*_{\Acrys}"] \\
        D_{\perf}^{\vp}(\fY_{\Prism}) \arrow[d, "D_{\crys}"] &  D_{\perf}^{\vp}\big((\fY_{\mO/p}/(\Acrys,N_\infty))_{\crys}\big) \arrow[d, "\cong"] \\
        D_{\perf}^{\vp}\big((\fY_{\mO_K/p}/(\mS, \N))_{\crys}\big) \ar[d] \arrow[rd] &  D_{\perf}^{\vp}\big((\fY_{\mO/p}/(\Acrys,\N))_{\crys}\big) \arrow[d]\\ 
        D_{\perf}^{\vp}\big((\fY_k/(W(k), \N))_{\crys}\big) & D_{\perf}^{\vp}\big((\fY_{\mO/p}/(\widehat A_\st, \N))_{\crys}\big) \arrow[l]
    \end{tikzcd}
    \end{equation}
    For any $\mM\in \DBKF^{\log}(\fY_\mO,\vp)$ that lies in the essential image of $\Vect^{\an,\vp}\big(\fY_\Prism\big)$, we use $\sigma^*_{\widehat{A}_{\st}}\mM$ and $\sigma_{\mS}^* \mM$ to denote the images of $\mM$ in $D_{\perf}^{\vp}\big((\fY_{\mO/p}/(\widehat A_{\st}, \N))_{\crys}\big)$ and $D_{\perf}^{\vp}\big((\fY_{\mO_K/p}/(\mS, \N))_{\crys}\big)$ in the diagram above, respectively. Let $\sigma_\crys^*\mM$ be the image of $\mM$ in $D_{\perf}^{\vp}\big((\fY_k/(W(k), \N))_{\crys}\big)$ via the functor \eqref{eq: crystalline specialization with N}.

    For reader's convenience, let us clarify the monodromy operators on the crystalline cohomlogies of $\sigma_{\crys}^* \mM$, $\sigma_{\widehat A_\st}^*\mM$, and $\sigma_{\mS}^*\mM$. For simplicity, we use $\mE$ to denote either $\sigma_{\crys}^* \mM$, $\sigma_{\widehat A_\st}^*\mM$, or $\sigma_{\mS}^*\mM$. We follow the constructions in \cite[(3.5)]{HK}, \cite[(4.3)]{tsuji1999}, and \cite[Remark 6.45]{DMS}.
    \begin{enumerate}
        \item Let $(\mS^{(1)}, (p), \N)^a$ denote the self product of $(\mS, (p), \N)^a$ in the category of log prisms over $(\mO_K, (p), \N)^a$, with canonical maps $p_0,p_1\colon (\mS,(p),\N) \ra (\mS^{(1)},(p),\N)$. Then $\mS^{(1)}$ can be identified with $\mS\langle v-1\rangle_{\mathrm{PD}}$ with $v= \frac{p_0(u)}{p_1(u)}$. Let 
        \[
            \mK_\mS= R\Gamma_{\crys}\big(\fY_{\mO_K/p}/(\mS,\N), \mE\big), \ \text{and}\ \mK_\mS^{(1)}= R\Gamma_{\crys}\big(\fY_{\mO_K/p}/(\mS^{(1)},\N), \mE\big).
        \]
        Let $q\colon \mS^{(1)} \cong \widehat \bigoplus_{n\ge 0} \mS(v-1)^{[n]}  \ra \mS$ denote the projection to the component of $(v-1)^{[1]}$, where $(v-1)^{[n]}= \frac{(v-1)^n}{n!}$. Then we define the monodromy operator on $\mK_S$ as
        \[
            N \colon \mK_\mS \ra \mK_\mS\otimes_{\mS, p_0}^{\L}\mS^{(1)} \cong \mK_\mS^{(1)} \cong  \mK_\mS\otimes_{\mS, p_1}^{\L}\mS^{(1)} \xrightarrow[]{q} \mK_\mS,
        \]
        (cf. \cite[(4.3)]{tsuji1999} and \cite[Remark 6.45]{DMS}).
        \item Let $(\widehat A_\st^{(1)}, (p), \N)^a$ denote the self product of $(\widehat A_\st, (p), \N)^a$ in the category of log prisms over $(\Acrys, (p), \N)^a$ with canonical maps $p_0,p_1\colon (\widehat A_\st,(p), \N) \ra (\widehat A_\st^{(1)},(p), \N)$. Then $\widehat A_\st^{(1)}$ can also be identified with $\widehat A_\st\langle v-1\rangle_{\mathrm{PD}}$ with $v= \frac{p_0(u)}{p_1(u)}$. The monodromy operator $N$ on $\mK_\st= R\Gamma_{\crys}\big(\fY_{\mO/p}/(\widehat A_\st,\N), \mE\big)$ is defined in the same way as in (1).

        \item By construction, the natural morphisms of log-prisms
        \[
        \begin{tikzcd}
            (\mO_K, (p), \N)^a \ar[r] \ar[d] & (\mS, (p), \N)^a \ar[d] \ar[r,shift left=2pt, "p_0"] \ar[r,shift right=2pt, "p_1"'] & (\mS^{(1)}, (p), \N)^a \ar[d] \\
            (\Acrys, (p), \N)^a \ar[r] & (\widehat A_\st, (p), \N)^a \ar[r,shift left=2pt, "p_0"] \ar[r,shift right=2pt, "p_1"'] & (\widehat A_\st^{(1)}, (p), \N)^a
        \end{tikzcd}
        \]
        induce a map $\mK_S \ra \mK_\st$ that is compatible with monodromy operators.
        \item Let $(\mD,(p),\N)^a$ denote the self product of $(W(k), (p), \N)^a$ over $(W(k),(p), \{0\})^a$, where the first $W(k)$ is equipped with the pre-log structure $\N \ra W(k), 1\mapsto 0$, and the second one is equipped with the trivial log structure. The same construction as in (1) defines monodromy operators on $R\Gamma_{\crys}\big(\fY_k/(W(k),\N\big), \mE)$ and $\mK_0= R\Gamma_{\crys}\big(\fY_{\mO_K/p}/(W(k),\N), \mE\big)$ (cf. \cite[(3.5)]{HK}). 
        \item By the same argument as \cite[Proposition 4.13, Lemma 5.2]{HK}, which essentially uses the Dwork's trick as in Proposition \ref{crys vs absolute crys}, there is a natural isomorphism 
        \[
            \mK_0\otimes_{W(k)}^{\L} \mS \xrightarrow[]{\sim} \mK_S.
        \]
        Moreover, it is compatible with monodromy operators (cf. \cite[(5.5)]{HK}).
    \end{enumerate}

\begin{lemma}\label{lem: N=0 on Acrys}
    With the same notation as above, the image of 
    \[
        H_{\crys}^i\big(\fY_{\mO/p}/(\Acrys,\N), \mE\big) \ra H_\crys^i\big(\fY_{\mO/p}/(\widehat A_\st, \N),\mE\big)
    \]
    is contained in the $N=0$ part.
\end{lemma}
\begin{proof}
    Let $p_i^*$ denote the morphism $H_{\crys}^i\big(\fY_{\mO/p}/(\widehat A_\st,\N), \mE\big) \ra H_\crys^i\big(\fY_{\mO/p}/(\widehat A_\st^{(1)}, \N),\mE\big)$ induced by $p_i \colon \widehat A_{\st} \ra \widehat A_{\st}^{(1)}$ for $i=0,1$. Then the image of $H_{\crys}^i\big(\fY_{\mO/p}/(\Acrys,\N), \mE\big)$ is contained in the image of $p_0^*-p_1^*$, on which the monodromy $N$ is zero by construction (cf. \cite[Lemma 4.3.8]{tsuji1999}).
\end{proof}

\begin{proposition}\label{prop: compare two crystalline cohom tensor B_st}
    Let $\fY$ be a semistable formal scheme over $\mO_K$ and let $\mM\in \DBKF^{\log}(\fY_{\mO},\vp)$ that lies in the essential image of $\Vect^{\an,\vp}\big(\fY_\Prism\big)$. Then there is an isomorphism
    \[
        R\Gamma_{\crys}\big(\fY_k/(W(k), \N),\sigma^*_{\crys}\mM\big) \otimes^{\L}_{W(k)} \hB^+_\st \cong R\Gamma_{\crys}\big(\fY_{\mO/p}/(\Acrys, N_\infty),\sigma^*_{A_\crys}\mM\big)\otimes^{\L}_{A_\crys} \hB_\st^+
    \]
    compatible with $G_K$-actions, Frobenii, and monodromy operators.\footnote{Note that the left-hand side is equipped with the monodromy operator $N\otimes 1+ 1\otimes N_{\widehat B_\st}$. On the other hand, the right-hand side is equipped with the monodromy operator $1\otimes N_{\widehat B_\st}$ because the monodromy operator on $R\Gamma_{\crys}\big(\fY_{\mO/p}/(\Acrys, N_\infty),\sigma^*_{A_\crys}\mM\big)$ is trivial by \eqref{eq: N vs Ninfty} and Lemma \ref{lem: N=0 on Acrys}.}
\end{proposition}

\begin{proof}
     By \eqref{eq: N vs Ninfty}, we may replace $(\Acrys, N_\infty)$ by $(\Acrys, \N)$. Consider the base change along the arrows in diagram \eqref{diagram: A_st}. By the base change property of log crystalline cohomology along $(\Acrys, \N)^a \ra (\widehat A_\st, \N)^a$, we obtain an isomorphism
    \[
        R\Gamma_\crys\big(\fY_{\mO/p}/(\Acrys,\N), \sigma^*_{\Acrys}\mM\big) \otimes^{\L}_{\Acrys}\widehat B_\st^+ \cong R\Gamma_\crys\big(\fY_{\mO/p}/(\widehat A_\st,\N), \sigma^*_{\widehat A_\st}\mM\big)[\frac 1p],
    \]
    compatible with Galois actions and Frobenii. According to Lemma \ref{lem: N=0 on Acrys}, it is also compatible with monodromy operators. 
    Similarly, use the $G_K$-equivariant map $(\mS, \N)^a \ra (\widehat A_\st, \mathbb{N})^a$, we obtain another isomorphism
    \[
       R\Gamma_\crys\big(\fY_{\mO_K/p}/(\mS, \N), \sigma_\mS^*\mM\big) \otimes^{\L}_{\mS}\widehat B_\st^+ \cong R\Gamma_\crys\big(\fY_{\mO/p}/(\widehat A_\st,\N), \sigma^*_{\widehat A_\st}\mM\big)[\frac 1p],
    \]
    compatible with $G_K$-actions, Frobenii, and monodromy operators. 
    Notice that by Remark \ref{absolute crys vs relative crys} (cf. \cite[Remark B.20]{DLMS2}), the projection $\mO_K/p \ra k$ induces an equivalence of categories
    \[
        \Isoc^\vp((\fY_{\mO_K/p}/(\mS,\N))_\crys) \xrightarrow[]{\sim} \Isoc^\vp((\fY_k/(W(k),\N))_\crys),
    \]
    which sends $(\sigma_\mS^*\mM)_\Q$ to $(\sigma^*_{\crys}\mM)_\Q$. In particular, we may regard $(\sigma^*_{\crys}\mM)_\Q$ as an object in the category $\Isoc^\vp((\fY_{\mO_K/p}/(\mS,\N))_\crys)$.
    By the same argument as in \cite[Proposition 4.13, Lemma 5.2, (5.5)]{HK}, there is a natural isomorphism 
    \[
        R\Gamma_\crys\big(\fY_{\mO_K/p}/(\mS, \N), (\sigma_\mS^*\mM)_\Q\big) \cong R\Gamma_\crys\big(\fY_{\mO_K/p}/(W(k), \N), (\sigma^*_{\crys}\mM)_\Q\big)\otimes_{K_0}^{\L}\mS[\frac 1p],
    \]
    compatible with Frobenii and monodromy operators (and note that both sides possess trivial Galois actions). 
    
    Finally, we finish the proof following the same strategy as in the proof of Proposition \ref{crys vs absolute crys}. 
    \begin{enumerate}
        \item Let $\alpha\colon \N \ra \mO_K$ denote the standard pre-log structure sending $1 \mapsto \pi$, which induces pre-log structures $\alpha\colon \N \ra k$, $\alpha\colon \N \ra \mO_K/p$, and $\alpha\colon \N \ra \mO_K/p^{1/p^n}$ through natural projections.
        \item Let $\beta\colon \N \ra W(k)$ (resp. $\beta\colon \N \ra \mO_K$) be the pre-log structure on $W(k)$ (resp. $\mO_K$) that sends $1$ to $0$; it induces pre-log structures $\beta\colon \N \ra k$, $\beta\colon \N \ra \mO_K/p^{1/p^n}$ through natural projections.
    \end{enumerate}
    Notice that $(k, \N, \alpha)^a$ coincides with $(k, \N, \beta)^a$. Choose $n\ge 0$ such that $p^{1/p^n} \mid \pi$ in $\mO_K$, so that $(\mO_K/p^{1/p^n}, \N, \alpha)^a= (\mO_K/p^{1/p^n}, \N, \beta)^a$. Then we obtain an isomorphism
    \[
        \fY_k\times_{\spf (k,\N, \beta)^a} \spf (\mO_K/p^{1/p^n}, \N, \beta)^a \cong \fY_{\mO_K/p}\times_{\spf (\mO_K/p,\N, \alpha)^a} \spf (\mO_K/p^{1/p^n}, \N, \alpha)^a.
    \]
    Applying Frobenius twists to their log crystalline cohomologies, we obtain
    \[
        R\Gamma_\crys\big(\fY_{\mO_K/p}/(W(k), \N,\beta), (\sigma^*_{\crys}\mM)_\Q\big) \cong R\Gamma_\crys\big(\fY_k/(W(k), \N,\beta), (\sigma^*_{\crys}\mM)_\Q\big)
    \]
    compatible with Frobenii and Galois actions. The compatibility of monodromy operators can be deduced similarly as in \cite[(5.5)]{HK}. Putting everything together, we arrive at the desired isomorphism
    \[
        R\Gamma_\crys\big(\fY_{\mO/p}/(\Acrys, N_{\infty}), \sigma^*_{\Acrys}\mM\big) \otimes^{\L}_{\Acrys}\widehat B^+_\st \cong R\Gamma_\crys\big(\fY_k/(W(k),\N), \sigma^*_{\crys}\mM\big)\otimes^{\L}_{W(k)} \widehat B^+_\st,
    \]
    compatible with Galois action, Frobenii, and monodromy operators. 
\end{proof}

We now prove the $C_{\st}$-conjecture with coefficients (Theorem~\ref{mainthm}), except for the compatibility with filtrations.

\begin{theorem}\label{thm: semistable comparison}
    Let $\fY$ be a proper semistable formal scheme over $\mO_K$ with adic generic fiber $Y$. Let $\L$ be a semistable $\Z_p$-local system on $Y$ (viewed as a locally finite free sheaf of $\widehat{Z}_p$-modules on $Y_{\proet}$), and let $\mE_{\crys,\Q}$ be the log $F$-isocrystal associated with $\L$ in the sense of Definition \ref{defn: semistable local systems}. Then, there is a natural isomorphism
        \[
           R\nu_* \L \otimes^{\L}_{\Z_p}B_\st \cong R\upsilon_*(\mE_{\crys,\Q})\otimes^{\L}_{K_0}B_{\st},
        \]
        compatible with $G_K$-actions, Frobenii, and monodromy operators, where $\nu\colon Y_{C,\proet} \ra \fY_{\mO,\ett}$ and $\upsilon\colon (\fY_k/(W(k), \N))_{\crys} \ra \fY_{k,\ett}$ are the natural projections of sites (note that $C= \widehat{\overline K}$). As an immediate corollary, there are natural isomorphisms
        \[
           H^i_{\ett}(Y_C, \L) \otimes_{\Z_p}B_\st \cong H^i_{\mathrm{logcrys}}\big(\fY_k/(W(k), \N), \mE_{\crys,\Q}\big)\otimes_{K_0}B_{\st},
        \]
        compatible with $G_K$-actions, Frobenii, and monodromy operators.
\end{theorem}
\begin{proof}
    Let $\mM$ be the derived relative log BKF module associated with $\L$. By Proposition \ref{crystalline specialization vs crystalline realzation}, the crystalline specialization of $\mM$ recovers $\mE_{\crys,\Q}$. More precisely, let $\mE\in \Vect^{\an,\vp}(\fY_\Prism)$ denote the analytic prismatic crystal associated with $\L$ by Theorem \ref{thm: prismatic description of semistable local systems} and notice that, by Remark \ref{absolute crys vs relative crys} (cf. \cite[Remark B.20]{DLMS2}), there is an equivalence of categories 
    \[ 
        \Isoc^\vp((\fY_{\mO_K/p}/(\mS,\N))_\crys) \cong \Isoc^\vp((\fY_k/(W(k),\N))_\crys).
    \]
    Hence, by Proposition \ref{crystalline specialization vs crystalline realzation}, we obtain isomorphisms
    \[
        (\sigma_{\crys}^*\mM)_{\Q} \cong (D_{\crys}(j_*\mE))_{\Q} \cong \mE_{\crys,\Q}
    \]
    of $F$-isocrystals on $(\fY_k/(W(k), \N))_{\crys}$.
    
    By Theorem \ref{thm: semistable compsrison over A_crys} and Proposition \ref{prop: compare two crystalline cohom tensor B_st}, we already have a natural isomorphism
    \begin{equation}\label{comparison map B_st hat}
        R\upsilon_*(\sigma^*_{\crys}\mM)\otimes^{\L}_{W(k)}\hB_{\st} \simra R\nu_*\L \otimes^{\L}_{\Z_p}\hB_\st,
    \end{equation}
    compatible with $G_K$-actions, Frobenii, and monodromy operators.
    Since the monodromy is nilpotent on $R\upsilon_*(\sigma^*_{\crys}\mM)$ and $\widehat{B}_{\st}^{N-\nilp} =B_\st$, it induces a map
    \begin{equation}\label{comparison map B_st}
        R\upsilon_*(\sigma^*_{\crys}\mM)\otimes^{\L}_{W(k)}B_{\st} \rightarrow (R\nu_*\L \otimes^{\L}_{\Z_p}\widehat B_\st)^{N-\nilp}= R\nu_*\L \otimes^{\L}_{\Z_p} B_\st,
    \end{equation}
   which yields injections on the level of cohomology groups. To show that it is an isomorphism, we take cohomology groups on both sides. The isomorphism (\ref{comparison map B_st hat}) implies
    \begin{equation}\label{eq: H_crys vs widehat D_st}
         H^i_\crys\big(\fY_k/(W(k), \N), \sigma_{\crys}^*\mM\big)\otimes_{W(k)} \hB_\st^{G_K}\cong \hD_{\st}(H^i_{\ett}(Y_C, \L))
    \end{equation}
    as objects in $\hMF_{\hB_\st^{G_K}}(\vp,N)$; in particular, $H^i_{\ett}(Y_C, \L)$ is $\hB_\st$-admissible, and hence semistable by Theorem \ref{breuil thm 3.3}. Therefore, under the equivalence between $\MF_K(\vp,N)$ and $\hMF_{\hB_\st^{G_K}}(\vp,N)$, we obtain an identification of filtered $(\vp,N)$-modules
    \[
        \hD_{\st}(H^i_{\ett}(Y_C, \L))\cong D_{\st}(H^i_{\ett}(Y_C, \L))\otimes_{K_0}\hB_\st^{G_K}
    \]
    by Theorem \ref{breuil thm 6.1.1}. Taking $G_K$-invariants on both sides of $(\ref{comparison map B_st})$, we obtain an inclusion
    \[
        H^i_\crys\big(\fY_k/(W(k), \N), \sigma_{\crys}^*\mM\big)[\frac 1p]\subseteq D_\st(H^i_{\ett}(Y_C, \L)).
    \]
    Since
    \[
        \dim_{K_0}\big(H^i\big(\fY_k/(W(k), \N), \sigma_{\crys}^*\mM\big)[\frac 1p]\big) \xlongequal{\eqref{eq: H_crys vs widehat D_st} }\rank_{\hB_\st^{G_K}}(\hD_{\st}(V))= \dim_{K_0}(D_\st(V)),
    \]
    where $V= H^i_{\ett}(Y_C, \L)$, we conclude that the inclusion must be an identity. Consequently, (\ref{comparison map B_st}) is an isomorphism.
\end{proof}

It remains to deal with filtrations. By Proposition \ref{prop: F-isocrystal structure},  $(\sigma_{\crys}^*\mM)_{\Q}$ is associated with the filtered vector bundle $\big(D_{\dR}(\L),\nabla_{D_{\dR}(\L)},\Fil^\bullet_{D_{\dR}(\L)}\big)$ with integrable connection. The filtration $\Fil_{D_\dR(\L)}^\bullet$ induces a natural filtration on $R\upsilon_*(\sigma^*_{\crys}\mM)\otimes^{\L}_{W(k)}B_{\dR}$ via pulling back along
\begin{align*}
    R\upsilon_*(\sigma^*_{\crys}\mM)\otimes_{W(k)}^{\L} B_{\dR} & \ra R\upsilon_*(\sigma^*_{\Acrys}\mM) \otimes_{\Acrys}^{\L}B_\dR \\ 
     & \ra R\nu_*(\sigma_{B_{\dR}^+}^*\mM)\otimes_{B_\dR^+}^{\L}B_\dR \xrightarrow[]{\cong} \DR\big(D_{\dR}(\L), \nabla_{D_{\dR}(\L)}\big)\otimes_K^{\L} B_{\dR},
\end{align*}
where the first arrow is obtained from Proposition \ref{prop: compare two crystalline cohom tensor B_st}, the second one is obtained from diagram (\ref{eq: diagram 1}), and the last one is obtained from diagram (3) in \cite[Theorem 10.16]{Guo-Reinecke}.
We need to show that the isomorphism
    \begin{equation}\label{eq: C_st comparison tensor B_dR}
         R\nu_* \L \otimes^{\L}_{\Z_p}B_\dR \ra R\upsilon_*(\sigma^*_{\crys}\mM)\otimes^{\L}_{W(k)}B_{\dR}
    \end{equation}
    is compatible with filtrations.
    Let \[C_{\proet\text{-}\dR}\coloneqq R\nu_*\big(\DR(\L\otimes \mO\B_\dR, \mathrm{id}_{\L}\otimes\nabla_{\mO\B_\dR})\big)\] and $C_{\dR}\coloneqq \DR\big(D_{\dR}(\L), \nabla_{D_{\dR}(\L)}\big)$. Consider the diagram
    \begin{equation}\label{diagram: filtration}
    \begin{tikzcd}
        & A\Omega_{\fY}^{\log}(\mM)\otimes_{\Ainf}^{\L}B_\dR \ar[d,"\beta"'] \ar[r,"\alpha"] & R\upsilon_*(\sigma^*_{\Acrys}\mM)\otimes_{\Acrys}^{\L}B_\dR \ar[d] \arrow[phantom, from=1-2, to=2-3, "(1)"] \\
         R\nu_*\L\otimes_{\Z_p}^{\L}B_\dR \ar[r] \ar[ru] \ar[rd] & R\nu_*(\bM\otimes \B_\dR^+)\otimes_{B_\dR^+}B_\dR \ar[d] & R\nu_*(\sigma_{B_{\dR}^+}^*\mM)\otimes_{B_\dR^+}^{\L}B_\dR \ar[l] \ar[d] 
        \arrow[phantom, from=2-2, to=3-3, "(2)" ] \\
        & C_{\proet\text{-}\dR} & C_\dR\otimes_K^{\L} B_\dR. \ar[l]
    \end{tikzcd}
    \end{equation}
    Square (1) commutes due to the diagram (\ref{eq: diagram 1}) in Construction \ref{construction: commutativity of diagram (1)}, while square (2) commutes by the same argument as in the proof of \cite[Theorem 10.16]{Guo-Reinecke} and \cite[\S 7.3.2]{DMS}. 
    Consequently, the arrows in (2) are isomorphisms compatible with filtrations by \cite[Theorem 10.16]{Guo-Reinecke}. Therefore, since both $\alpha$ and $\beta$ in diagram \eqref{diagram: filtration} are already isomorphisms, the comparison map 
    \[
        R\upsilon_*(\sigma^*_{\Acrys}\mM) \otimes_{\Acrys}^{\L}B_\dR \ra  C_\dR\otimes_K^{\L} B_\dR
    \]
    is also an isomorphism. Then the diagram \eqref{diagram: filtration} yields a commutative diagram
    \begin{equation}\label{diagram: define filtrations}
    \begin{tikzcd}
        R\nu_* \L \otimes^{\L}_{\Z_p}B_\dR \arrow[r] \arrow[d, "\eqref{eq: C_st comparison tensor B_dR}"'] \ar[rd] & \DR\big(D_{\dR}(\L), \nabla_{D_{\dR}(\L)}\big)\otimes_K^{\L} B_{\dR} \\
        R\upsilon_*(\sigma^*_{\crys}\mM)\otimes_{W(k)}^{\L} B_{\dR} \ar[r,"\gamma"] & R\upsilon_*(\sigma^*_{\Acrys}\mM) \otimes_{\Acrys}^{\L}B_\dR \ar[u,"\delta"'],
    \end{tikzcd}
    \end{equation}
    where the arrows in the  upper triangle are compatible with filtrations, while the bottom arrow is obtained from Proposition \ref{prop: compare two crystalline cohom tensor B_st}. The morphism $\delta\circ \gamma$ in diagram \eqref{diagram: define filtrations} is exactly the morphism defining the filtrations on $R\upsilon_*(\sigma^*_{\crys}\mM)\otimes_{W(k)}^{\L} B_{\dR}$ (which is now proved to be an isomorphism).
    Therefore, the isomorphism \eqref{eq: C_st comparison tensor B_dR} is compatible with filtrations.
    
    Summing up, we have proved the following semistable comparison theorem; namely, Theorem \ref{mainthm}.

\begin{theorem}\label{thm: semistable comparison+ filtration}
    Let $\fY$ be a proper semistable formal scheme over $\mO_K$ with adic generic fiber $Y$. Let $\L$ be a semistable $\Z_p$-local system on $Y$ and let $\mE_{\crys,\Q}$ be the associated log $F$-isocrystal in the sense of Definition \ref{defn: semistable local systems}. Then, there is a natural isomorphism
        \[
           R\nu_* \L \otimes^{\L}_{\Z_p}B_\st \cong R\upsilon_*(\mE_{\crys,\Q})\otimes^{\L}_{K_0}B_{\st},
        \]
        compatible with $G_K$-actions, Frobenii, filtrations, and monodromy operators, where $\nu\colon Y_{C,\proet} \ra \fY_{\mO,\ett}$ and $\upsilon\colon (\fY_k/(W(k), \N))_{\crys} \ra \fY_{k,\ett}$ are natural projections of sites. As a corollary, there are natural isomorphisms
        \[
           H^i_{\ett}(Y_C, \L) \otimes_{\Z_p}B_\st \cong H^i_{\mathrm{logcrys}}\big(\fY_k/(W(k), \N), \mE_{\crys,\Q}\big)\otimes_{K_0}B_{\st},
        \]
        compatible with $G_K$-actions, Frobenii, filtrations, and monodromy operators.
\end{theorem}

\vspace{0.3in}

\bibliographystyle{alpha}
\bibliography{ref}

\vspace{15mm}

\begin{tabular}{l}
    H.D.\\
    Yau Mathematical Sciences Center \& Department of Mathematical Sciences,\\
    Tsinghua University   \\
    Beijing, China\\
    \textit{E-mail address: }\texttt{hdiao@mail.tsinghua.edu.cn }\\
    \\
    Z.D.\\
    Department of Mathematical Sciences,\\
    Tsinghua University\\
    Beijing, China\\
    \textit{E-mail address: }\texttt{zhefanduan@126.com}\\
    \\
    Z.Y.\\
    Department of Mathematics, \\
    University of California Santa Barbara\\
    Santa Barbara, California, USA\\
    \textit{E-mail address: }\texttt{yao@math.ucsb.edu }
\end{tabular}

\end{document}